\documentclass{birkjour}

\ifx\figforTeXisloaded\relax \else\global\let\figforTeXisloaded=\relax\fi
\catcode`\@=11
\ifx\ctr@ln@m\undefined\else%
    \immediate\write16{*** Fig4TeX WARNING : \string\ctr@ln@m\space already defined.}\fi
\def\ctr@ln@m#1{\ifx#1\undefined\else%
    \immediate\write16{*** Fig4TeX WARNING : \string#1 already defined.}\fi}
\ctr@ln@m\ctr@ld@f
\def\ctr@ld@f#1#2{\ctr@ln@m#2#1#2}
\ctr@ld@f\def\ctr@ln@w#1#2{\ctr@ln@m#2\csname#1\endcsname#2}
{\catcode`\/=0 \catcode`/\=12 /ctr@ld@f/gdef/BS@{\}}
\ctr@ld@f\def\ctr@lcsn@m#1{\expandafter\ifx\csname#1\endcsname\relax\else%
    \immediate\write16{*** Fig4TeX WARNING : \BS@\expandafter\string#1\space already defined.}\fi}
\ctr@ld@f\edef\colonc@tcode{\the\catcode`\:}
\ctr@ld@f\edef\semicolonc@tcode{\the\catcode`\;}
\ctr@ld@f\def\t@stc@tcodech@nge{{\let\c@tcodech@nged=\z@%
    \ifnum\colonc@tcode=\the\catcode`\:\else\let\c@tcodech@nged=\@ne\fi%
    \ifnum\semicolonc@tcode=\the\catcode`\;\else\let\c@tcodech@nged=\@ne\fi%
    \ifx\c@tcodech@nged\@ne%
    \immediate\write16{}
    \immediate\write16{!!!=============================================================!!!}
    \immediate\write16{ Fig4TeX WARNING:}
    \immediate\write16{ The category code of some characters has been changed, which will}
    \immediate\write16{ result in an error (message "Runaway argument?").}
    \immediate\write16{ This probably comes from another package that changed the category}
    \immediate\write16{ code after Fig4TeX was loaded. If that proves to be exact, the}
    \immediate\write16{ solution is to exchange the loading commands on top of your file}
    \immediate\write16{ so that Fig4TeX is loaded last. For example, in LaTeX, we should}
    \immediate\write16{ say :}
    \immediate\write16{\BS@ usepackage[french]{babel}}
    \immediate\write16{\BS@ usepackage{fig4tex}}
    \immediate\write16{!!!=============================================================!!!}
    \immediate\write16{}
    \fi}}
\ctr@ld@f\def\FigforTeX{F\kern-.05em i\kern-.05em g\kern-.1em\raise-.14em\hbox{4}\kern-.19em\TeX}
\ctr@ld@f\def\W@rnmesoldA#1{\W@rnmesold}
\ctr@ld@f\def\W@rnmesoldAB#1(#2){\W@rnmesold}
\ctr@ld@f\def\W@rnmesold{%
    \immediate\write16{}
    \immediate\write16{!!!=============================================================!!!}
    \immediate\write16{ Fig4TeX WARNING:}
    \immediate\write16{ The file to be compiled is not compatible with the current version}
    \immediate\write16{ of Fig4TeX. To fix that, upgrade the source file (mainly change \BS@ ps*}
    \immediate\write16{ macros by \BS@ fig* macros), or use fig4tex184.tex instead (\BS@ input fig4tex184}
    \immediate\write16{ or \BS@ usepackage{fig4tex184}).}
    \immediate\write16{!!!=============================================================!!!}
    \immediate\write16{}}
\ctr@ln@m\psbeginfig\let\psbeginfig\W@rnmesoldA
\ctr@ln@m\psset\let\psset\W@rnmesoldAB
\ctr@ln@m\pssetdefault\let\pssetdefault\W@rnmesoldAB
\ctr@ln@m\pssetupdate\let\pssetupdate\W@rnmesoldA
\ctr@ln@w{newdimen}\epsil@n\epsil@n=0.00005pt
\ctr@ln@w{newdimen}\Cepsil@n\Cepsil@n=0.005pt
\ctr@ln@w{newdimen}\dcq@\dcq@=254pt
\ctr@ln@w{newdimen}\PI@\PI@=3.141592pt
\ctr@ln@w{newdimen}\DemiPI@deg\DemiPI@deg=90pt
\ctr@ln@w{newdimen}\PI@deg\PI@deg=180pt
\ctr@ln@w{newdimen}\DePI@deg\DePI@deg=360pt
\ctr@ld@f\chardef\t@n=10
\ctr@ld@f\chardef\c@nt=100
\ctr@ld@f\chardef\@lxxiv=74
\ctr@ld@f\chardef\@xci=91
\ctr@ld@f\mathchardef\@nMnCQn=9949
\ctr@ld@f\chardef\@vi=6
\ctr@ld@f\chardef\@xxx=30
\ctr@ld@f\chardef\@lvi=56
\ctr@ld@f\chardef\@@lxxi=71
\ctr@ld@f\chardef\@lxxxv=85
\ctr@ld@f\mathchardef\@@mmmmlxviii=4068
\ctr@ld@f\mathchardef\@ccclx=360
\ctr@ld@f\mathchardef\@dccxx=720
\ctr@ln@w{newcount}\p@rtent \ctr@ln@w{newcount}\f@ctech \ctr@ln@w{newcount}\result@tent
\ctr@ln@w{newdimen}\v@lmin \ctr@ln@w{newdimen}\v@lmax \ctr@ln@w{newdimen}\v@leur
\ctr@ln@w{newdimen}\result@t\ctr@ln@w{newdimen}\result@@t
\ctr@ln@w{newdimen}\mili@u \ctr@ln@w{newdimen}\c@rre \ctr@ln@w{newdimen}\delt@
\ctr@ld@f\def\degT@rd{0.017453 }  
\ctr@ld@f\def\rdT@deg{57.295779 } 
\ctr@ln@m\v@leurseule
{\catcode`p=12 \catcode`t=12 \gdef\v@leurseule#1pt{#1}}
\ctr@ld@f\def\repdecn@mb#1{\expandafter\v@leurseule\the#1\space}
\ctr@ld@f\def\arct@n#1(#2,#3){{\v@lmin=#2\v@lmax=#3%
    \maxim@m{\mili@u}{-\v@lmin}{\v@lmin}\maxim@m{\c@rre}{-\v@lmax}{\v@lmax}%
    \delt@=\mili@u\m@ech\mili@u%
    \ifdim\c@rre>\@nMnCQn\mili@u\divide\v@lmax\tw@\c@lATAN\v@leur(\z@,\v@lmax)
    \else%
    \maxim@m{\mili@u}{-\v@lmin}{\v@lmin}\maxim@m{\c@rre}{-\v@lmax}{\v@lmax}%
    \m@ech\c@rre%
    \ifdim\mili@u>\@nMnCQn\c@rre\divide\v@lmin\tw@
    \maxim@m{\mili@u}{-\v@lmin}{\v@lmin}\c@lATAN\v@leur(\mili@u,\z@)%
    \else\c@lATAN\v@leur(\delt@,\v@lmax)\fi\fi%
    \ifdim\v@lmin<\z@\v@leur=-\v@leur\ifdim\v@lmax<\z@\advance\v@leur-\PI@%
    \else\advance\v@leur\PI@\fi\fi%
    \global\result@t=\v@leur}#1=\result@t}
\ctr@ld@f\def\m@ech#1{\ifdim#1>1.646pt\divide\mili@u\t@n\divide\c@rre\t@n\m@ech#1\fi}
\ctr@ld@f\def\c@lATAN#1(#2,#3){{\v@lmin=#2\v@lmax=#3\v@leur=\z@\delt@=\tw@ pt%
    \un@iter{0.785398}{\v@lmax<}%
    \un@iter{0.463648}{\v@lmax<}%
    \un@iter{0.244979}{\v@lmax<}%
    \un@iter{0.124355}{\v@lmax<}%
    \un@iter{0.062419}{\v@lmax<}%
    \un@iter{0.031240}{\v@lmax<}%
    \un@iter{0.015624}{\v@lmax<}%
    \un@iter{0.007812}{\v@lmax<}%
    \un@iter{0.003906}{\v@lmax<}%
    \un@iter{0.001953}{\v@lmax<}%
    \un@iter{0.000976}{\v@lmax<}%
    \un@iter{0.000488}{\v@lmax<}%
    \un@iter{0.000244}{\v@lmax<}%
    \un@iter{0.000122}{\v@lmax<}%
    \un@iter{0.000061}{\v@lmax<}%
    \un@iter{0.000030}{\v@lmax<}%
    \un@iter{0.000015}{\v@lmax<}%
    \global\result@t=\v@leur}#1=\result@t}
\ctr@ld@f\def\un@iter#1#2{%
    \divide\delt@\tw@\edef\dpmn@{\repdecn@mb{\delt@}}%
    \mili@u=\v@lmin%
    \ifdim#2\z@%
      \advance\v@lmin-\dpmn@\v@lmax\advance\v@lmax\dpmn@\mili@u%
      \advance\v@leur-#1pt%
    \else%
      \advance\v@lmin\dpmn@\v@lmax\advance\v@lmax-\dpmn@\mili@u%
      \advance\v@leur#1pt%
    \fi}
\ctr@ld@f\def\c@ssin#1#2#3{\expandafter\ifx\csname COS@\number#3\endcsname\relax\c@lCS{#3pt}%
    \expandafter\xdef\csname COS@\number#3\endcsname{\repdecn@mb\result@t}%
    \expandafter\xdef\csname SIN@\number#3\endcsname{\repdecn@mb\result@@t}\fi%
    \edef#1{\csname COS@\number#3\endcsname}\edef#2{\csname SIN@\number#3\endcsname}}
\ctr@ld@f\def\c@lCS#1{{\mili@u=#1\p@rtent=\@ne%
    \relax\ifdim\mili@u<\z@\red@ng<-\else\red@ng>+\fi\f@ctech=\p@rtent%
    \relax\ifdim\mili@u<\z@\mili@u=-\mili@u\f@ctech=-\f@ctech\fi\c@@lCS}}
\ctr@ld@f\def\c@@lCS{\v@lmin=\mili@u\c@rre=-\mili@u\advance\c@rre\DemiPI@deg\v@lmax=\c@rre%
    \mili@u\@@lxxi\mili@u\divide\mili@u\@@mmmmlxviii%
    \edef\v@larg{\repdecn@mb{\mili@u}}\mili@u=-\v@larg\mili@u%
    \edef\v@lmxde{\repdecn@mb{\mili@u}}%
    \c@rre\@@lxxi\c@rre\divide\c@rre\@@mmmmlxviii%
    \edef\v@largC{\repdecn@mb{\c@rre}}\c@rre=-\v@largC\c@rre%
    \edef\v@lmxdeC{\repdecn@mb{\c@rre}}%
    \fctc@s\mili@u\v@lmin\global\result@t\p@rtent\v@leur%
    \let\t@mp=\v@larg\let\v@larg=\v@largC\let\v@largC=\t@mp%
    \let\t@mp=\v@lmxde\let\v@lmxde=\v@lmxdeC\let\v@lmxdeC=\t@mp%
    \fctc@s\c@rre\v@lmax\global\result@@t\f@ctech\v@leur}
\ctr@ld@f\def\fctc@s#1#2{\v@leur=#1\relax\ifdim#2<\@lxxxv\p@\cosser@h\else\sinser@t\fi}
\ctr@ld@f\def\cosser@h{\advance\v@leur\@lvi\p@\divide\v@leur\@lvi%
    \v@leur=\v@lmxde\v@leur\advance\v@leur\@xxx\p@%
    \v@leur=\v@lmxde\v@leur\advance\v@leur\@ccclx\p@%
    \v@leur=\v@lmxde\v@leur\advance\v@leur\@dccxx\p@\divide\v@leur\@dccxx}
\ctr@ld@f\def\sinser@t{\v@leur=\v@lmxdeC\p@\advance\v@leur\@vi\p@%
    \v@leur=\v@largC\v@leur\divide\v@leur\@vi}
\ctr@ld@f\def\red@ng#1#2{\relax\ifdim\mili@u#1#2\DemiPI@deg\advance\mili@u#2-\PI@deg%
    \p@rtent=-\p@rtent\red@ng#1#2\fi}
\ctr@ld@f\def\pr@c@lCS#1#2#3{\ctr@lcsn@m{COS@\number#3 }%
    \expandafter\xdef\csname COS@\number#3\endcsname{#1}%
    \expandafter\xdef\csname SIN@\number#3\endcsname{#2}}
\pr@c@lCS{1}{0}{0}
\pr@c@lCS{0.7071}{0.7071}{45}\pr@c@lCS{0.7071}{-0.7071}{-45}
\pr@c@lCS{0}{1}{90}          \pr@c@lCS{0}{-1}{-90}
\pr@c@lCS{-1}{0}{180}        \pr@c@lCS{-1}{0}{-180}
\pr@c@lCS{0}{-1}{270}        \pr@c@lCS{0}{1}{-270}
\ctr@ld@f\def\invers@#1#2{{\v@leur=#2\maxim@m{\v@lmax}{-\v@leur}{\v@leur}%
    \f@ctech=\@ne\m@inv@rs%
    \multiply\v@leur\f@ctech\edef\v@lv@leur{\repdecn@mb{\v@leur}}%
    \p@rtentiere{\p@rtent}{\v@leur}\v@lmin=\p@\divide\v@lmin\p@rtent%
    \inv@rs@\multiply\v@lmax\f@ctech\global\result@t=\v@lmax}#1=\result@t}
\ctr@ld@f\def\m@inv@rs{\ifdim\v@lmax<\p@\multiply\v@lmax\t@n\multiply\f@ctech\t@n\m@inv@rs\fi}
\ctr@ld@f\def\inv@rs@{\v@lmax=-\v@lmin\v@lmax=\v@lv@leur\v@lmax%
    \advance\v@lmax\tw@ pt\v@lmax=\repdecn@mb{\v@lmin}\v@lmax%
    \delt@=\v@lmax\advance\delt@-\v@lmin\ifdim\delt@<\z@\delt@=-\delt@\fi%
    \ifdim\delt@>\epsil@n\v@lmin=\v@lmax\inv@rs@\fi}
\ctr@ld@f\def\minim@m#1#2#3{\relax\ifdim#2<#3#1=#2\else#1=#3\fi}
\ctr@ld@f\def\maxim@m#1#2#3{\relax\ifdim#2>#3#1=#2\else#1=#3\fi}
\ctr@ld@f\def\p@rtentiere#1#2{#1=#2\divide#1by65536 }
\ctr@ld@f\def\r@undint#1#2{{\v@leur=#2\divide\v@leur\t@n\p@rtentiere{\p@rtent}{\v@leur}%
    \v@leur=\p@rtent pt\global\result@t=\t@n\v@leur}#1=\result@t}
\ctr@ld@f\def\sqrt@#1#2{{\v@leur=#2%
    \minim@m{\v@lmin}{\p@}{\v@leur}\maxim@m{\v@lmax}{\p@}{\v@leur}%
    \f@ctech=\@ne\m@sqrt@\sqrt@@%
    \mili@u=\v@lmin\advance\mili@u\v@lmax\divide\mili@u\tw@\multiply\mili@u\f@ctech%
    \global\result@t=\mili@u}#1=\result@t}
\ctr@ld@f\def\m@sqrt@{\ifdim\v@leur>\dcq@\divide\v@leur\c@nt\v@lmax=\v@leur%
    \multiply\f@ctech\t@n\m@sqrt@\fi}
\ctr@ld@f\def\sqrt@@{\mili@u=\v@lmin\advance\mili@u\v@lmax\divide\mili@u\tw@%
    \c@rre=\repdecn@mb{\mili@u}\mili@u%
    \ifdim\c@rre<\v@leur\v@lmin=\mili@u\else\v@lmax=\mili@u\fi%
    \delt@=\v@lmax\advance\delt@-\v@lmin\ifdim\delt@>\epsil@n\sqrt@@\fi}
\ctr@ld@f\def\extrairelepremi@r#1\de#2{\expandafter\lepremi@r#2@#1#2}
\ctr@ld@f\def\lepremi@r#1,#2@#3#4{\def#3{#1}\def#4{#2}\ignorespaces}
\ctr@ld@f\def\@cfor#1:=#2\do#3{%
  \edef\@fortemp{#2}%
  \ifx\@fortemp\empty\else\@cforloop#2,\@nil,\@nil\@@#1{#3}\fi}
\ctr@ln@m\@nextwhile
\ctr@ld@f\def\@cforloop#1,#2\@@#3#4{%
  \def#3{#1}%
  \ifx#3\Fig@nnil\let\@nextwhile=\Fig@fornoop\else#4\relax\let\@nextwhile=\@cforloop\fi%
  \@nextwhile#2\@@#3{#4}}

\ctr@ld@f\def\@ecfor#1:=#2\do#3{%
  \def\@@cfor{\@cfor#1:=}%
  \edef\@@@cfor{#2}%
  \expandafter\@@cfor\@@@cfor\do{#3}}
\ctr@ld@f\def\Fig@nnil{\@nil}
\ctr@ld@f\def\Fig@fornoop#1\@@#2#3{}
\ctr@ln@m\list@@rg
\ctr@ld@f\def\trtlis@rg#1#2{\def\list@@rg{#1}%
    \@ecfor\p@rv@l:=\list@@rg\do{\expandafter#2\p@rv@l|}}
\ctr@ld@f\def\trtlis@rgtok#1{\let@xte={}\let\n@xt\addt@t@xt\addt@t@xt #1}
\ctr@ln@m\M@cro
\ctr@ln@m\n@xt
\ctr@ld@f\def\addt@t@xt#1{\if#1|\let\n@xt\relax\else%
    \if#1,\expandafter\M@cro\the\let@xte|\let@xte={}%
    \else\let@xte=\expandafter{\the\let@xte #1}\fi\fi\n@xt}
\ctr@ln@w{newbox}\b@xvisu
\ctr@ln@w{newtoks}\let@xte
\ctr@ln@w{newif}\ifitis@K
\ctr@ln@w{newcount}\s@mme
\ctr@ln@w{newcount}\l@mbd@un \ctr@ln@w{newcount}\l@mbd@de
\ctr@ln@w{newcount}\superc@ntr@l\superc@ntr@l=\@ne        
\ctr@ln@w{newcount}\typec@ntr@l\typec@ntr@l=\superc@ntr@l 
\ctr@ln@w{newdimen}\v@lX  \ctr@ln@w{newdimen}\v@lY  \ctr@ln@w{newdimen}\v@lZ
\ctr@ln@w{newdimen}\v@lXa \ctr@ln@w{newdimen}\v@lYa \ctr@ln@w{newdimen}\v@lZa
\ctr@ln@w{newdimen}\unit@\unit@=\p@ 
\ctr@ld@f\def\unit@util{pt}
\ctr@ld@f\def\ptT@ptps{0.996264 }
\ctr@ld@f\def\ptpsT@pt{1.00375 }
\ctr@ld@f\def\ptT@unit@{1} 
\ctr@ld@f\def\setunit@#1{\def\unit@util{#1}\setunit@@#1:\invers@{\result@t}{\unit@}%
    \edef\ptT@unit@{\repdecn@mb\result@t}}
\ctr@ld@f\def\setunit@@#1#2:{\ifcat#1a\unit@=\@ne#1#2\else\unit@=#1#2\fi}
\ctr@ld@f\def\d@fm@cdim#1#2{{\v@leur=#2\v@leur=\ptT@unit@\v@leur\xdef#1{\repdecn@mb\v@leur}}}
\ctr@ln@w{newif}\ifBdingB@x\BdingB@xtrue
\ctr@ln@w{newdimen}\c@@rdXmin \ctr@ln@w{newdimen}\c@@rdYmin  
\ctr@ln@w{newdimen}\c@@rdXmax \ctr@ln@w{newdimen}\c@@rdYmax
\ctr@ld@f\def\b@undb@x#1#2{\ifBdingB@x%
    \relax\ifdim#1<\c@@rdXmin\global\c@@rdXmin=#1\fi%
    \relax\ifdim#2<\c@@rdYmin\global\c@@rdYmin=#2\fi%
    \relax\ifdim#1>\c@@rdXmax\global\c@@rdXmax=#1\fi%
    \relax\ifdim#2>\c@@rdYmax\global\c@@rdYmax=#2\fi\fi}
\ctr@ld@f\def\b@undb@xP#1{{\Figg@tXY{#1}\b@undb@x{\v@lX}{\v@lY}}}
\ctr@ld@f\def\ellBB@x#1;#2,#3(#4,#5,#6){{\s@uvc@ntr@l\et@tellBB@x%
    \setc@ntr@l{2}\figptell-2::#1;#2,#3(#4,#6)\b@undb@xP{-2}%
    \figptell-2::#1;#2,#3(#5,#6)\b@undb@xP{-2}%
    \c@ssin{\C@}{\S@}{#6}\v@lmin=\C@ pt\v@lmax=\S@ pt%
    \mili@u=#3\v@lmin\delt@=#2\v@lmax\arct@n\v@leur(\delt@,\mili@u)%
    \mili@u=-#3\v@lmax\delt@=#2\v@lmin\arct@n\c@rre(\delt@,\mili@u)%
    \v@leur=\rdT@deg\v@leur\advance\v@leur-\DePI@deg%
    \c@rre=\rdT@deg\c@rre\advance\c@rre-\DePI@deg%
    \v@lmin=#4pt\v@lmax=#5pt%
    \loop\ifdim\v@leur<\v@lmax\ifdim\v@leur>\v@lmin%
    \edef\@ngle{\repdecn@mb\v@leur}\figptell-2::#1;#2,#3(\@ngle,#6)%
    \b@undb@xP{-2}\fi\advance\v@leur\PI@deg\repeat%
    \loop\ifdim\c@rre<\v@lmax\ifdim\c@rre>\v@lmin%
    \edef\@ngle{\repdecn@mb\c@rre}\figptell-2::#1;#2,#3(\@ngle,#6)%
    \b@undb@xP{-2}\fi\advance\c@rre\PI@deg\repeat%
    \resetc@ntr@l\et@tellBB@x}\ignorespaces}
\ctr@ld@f\def\initb@undb@x{\c@@rdXmin=\maxdimen\c@@rdYmin=\maxdimen%
    \c@@rdXmax=-\maxdimen\c@@rdYmax=-\maxdimen}
\ctr@ld@f\def\c@ntr@lnum#1{%
    \relax\ifnum\typec@ntr@l=\@ne%
    \ifnum#1<\z@%
    \immediate\write16{*** Forbidden point number (#1). Abort.}\end\fi\fi%
    \set@bjc@de{#1}}
\ctr@ln@m\objc@de
\ctr@ld@f\def\set@bjc@de#1{\edef\objc@de{@BJ\ifnum#1<\z@ M\romannumeral-#1\else\romannumeral#1\fi}}
\s@mme=\m@ne\loop\ifnum\s@mme>-19
  \set@bjc@de{\s@mme}\ctr@lcsn@m\objc@de\ctr@lcsn@m{\objc@de T}
\advance\s@mme\m@ne\repeat
\s@mme=\@ne\loop\ifnum\s@mme<6
  \set@bjc@de{\s@mme}\ctr@lcsn@m\objc@de\ctr@lcsn@m{\objc@de T}
\advance\s@mme\@ne\repeat
\ctr@ld@f\def\setc@ntr@l#1{\ifnum\superc@ntr@l>#1\typec@ntr@l=\superc@ntr@l%
    \else\typec@ntr@l=#1\fi}
\ctr@ld@f\def\resetc@ntr@l#1{\global\superc@ntr@l=#1\setc@ntr@l{#1}}
\ctr@ld@f\def\s@uvc@ntr@l#1{\edef#1{\the\superc@ntr@l}}
\ctr@ln@m\c@lproscal
\ctr@ld@f\def\c@lproscalDD#1[#2,#3]{{\Figg@tXY{#2}%
    \edef\Xu@{\repdecn@mb{\v@lX}}\edef\Yu@{\repdecn@mb{\v@lY}}\Figg@tXY{#3}%
    \global\result@t=\Xu@\v@lX\global\advance\result@t\Yu@\v@lY}#1=\result@t}
\ctr@ld@f\def\c@lproscalTD#1[#2,#3]{{\Figg@tXY{#2}\edef\Xu@{\repdecn@mb{\v@lX}}%
    \edef\Yu@{\repdecn@mb{\v@lY}}\edef\Zu@{\repdecn@mb{\v@lZ}}%
    \Figg@tXY{#3}\global\result@t=\Xu@\v@lX\global\advance\result@t\Yu@\v@lY%
    \global\advance\result@t\Zu@\v@lZ}#1=\result@t}
\ctr@ld@f\def\c@lprovec#1{%
    \det@rmC\v@lZa(\v@lX,\v@lY,\v@lmin,\v@lmax)%
    \det@rmC\v@lXa(\v@lY,\v@lZ,\v@lmax,\v@leur)%
    \det@rmC\v@lYa(\v@lZ,\v@lX,\v@leur,\v@lmin)%
    \Figv@ctCreg#1(\v@lXa,\v@lYa,\v@lZa)}
\ctr@ld@f\def\det@rm#1[#2,#3]{{\Figg@tXY{#2}\Figg@tXYa{#3}%
    \delt@=\repdecn@mb{\v@lX}\v@lYa\advance\delt@-\repdecn@mb{\v@lY}\v@lXa%
    \global\result@t=\delt@}#1=\result@t}
\ctr@ld@f\def\det@rmC#1(#2,#3,#4,#5){{\global\result@t=\repdecn@mb{#2}#5%
    \global\advance\result@t-\repdecn@mb{#3}#4}#1=\result@t}
\ctr@ld@f\def\getredf@ctDD#1(#2,#3){{\maxim@m{\v@lXa}{-#2}{#2}\maxim@m{\v@lYa}{-#3}{#3}%
    \maxim@m{\v@lXa}{\v@lXa}{\v@lYa}
    \ifdim\v@lXa>\@xci pt\divide\v@lXa\@xci%
    \p@rtentiere{\p@rtent}{\v@lXa}\advance\p@rtent\@ne\else\p@rtent=\@ne\fi%
    \global\result@tent=\p@rtent}#1=\result@tent\ignorespaces}
\ctr@ld@f\def\getredf@ctTD#1(#2,#3,#4){{\maxim@m{\v@lXa}{-#2}{#2}\maxim@m{\v@lYa}{-#3}{#3}%
    \maxim@m{\v@lZa}{-#4}{#4}\maxim@m{\v@lXa}{\v@lXa}{\v@lYa}%
    \maxim@m{\v@lXa}{\v@lXa}{\v@lZa}
    \ifdim\v@lXa>\@lxxiv pt\divide\v@lXa\@lxxiv%
    \p@rtentiere{\p@rtent}{\v@lXa}\advance\p@rtent\@ne\else\p@rtent=\@ne\fi%
    \global\result@tent=\p@rtent}#1=\result@tent\ignorespaces}
\ctr@ln@m\getredf@ctB
\ctr@ld@f\def\getredf@ctBDD#1{\getredf@ctDD#1(\v@lX,\v@lY)}
\ctr@ld@f\def\getredf@ctBTD#1{\getredf@ctTD#1(\v@lX,\v@lY,\v@lZ)}
\ctr@ld@f\def\FigptintercircB@zDD#1:#2:#3,#4[#5,#6,#7,#8]{{\s@uvc@ntr@l\et@tfigptintercircB@zDD%
    \setc@ntr@l{2}\figvectPDD-1[#5,#8]\Figg@tXY{-1}\getredf@ctDD\f@ctech(\v@lX,\v@lY)%
    \mili@u=#4\unit@\divide\mili@u\f@ctech\c@rre=\repdecn@mb{\mili@u}\mili@u%
    \figptBezierDD-5::#3[#5,#6,#7,#8]%
    \v@lmin=#3\p@\v@lmax=\v@lmin\advance\v@lmax0.1\p@%
    \loop\edef\T@{\repdecn@mb{\v@lmax}}\figptBezierDD-2::\T@[#5,#6,#7,#8]%
    \figvectPDD-1[-5,-2]\n@rmeucCDD{\delt@}{-1}\ifdim\delt@<\c@rre\v@lmin=\v@lmax%
    \advance\v@lmax0.1\p@\repeat%
    \loop\mili@u=\v@lmin\advance\mili@u\v@lmax%
    \divide\mili@u\tw@\edef\T@{\repdecn@mb{\mili@u}}\figptBezierDD-2::\T@[#5,#6,#7,#8]%
    \figvectPDD-1[-5,-2]\n@rmeucCDD{\delt@}{-1}\ifdim\delt@>\c@rre\v@lmax=\mili@u%
    \else\v@lmin=\mili@u\fi\v@leur=\v@lmax\advance\v@leur-\v@lmin%
    \ifdim\v@leur>\epsil@n\repeat\figptcopyDD#1:#2/-2/%
    \resetc@ntr@l\et@tfigptintercircB@zDD}\ignorespaces}
\ctr@ln@m\figptinterlines
\ctr@ld@f\def\inters@cDD#1:#2[#3,#4;#5,#6]{{\s@uvc@ntr@l\et@tinters@cDD%
    \setc@ntr@l{2}\vecunit@{-1}{#4}\vecunit@{-2}{#6}%
    \Figg@tXY{-1}\setc@ntr@l{1}\Figg@tXYa{#3}%
    \edef\A@{\repdecn@mb{\v@lX}}\edef\B@{\repdecn@mb{\v@lY}}%
    \v@lmin=\B@\v@lXa\advance\v@lmin-\A@\v@lYa%
    \Figg@tXYa{#5}\setc@ntr@l{2}\Figg@tXY{-2}%
    \edef\C@{\repdecn@mb{\v@lX}}\edef\D@{\repdecn@mb{\v@lY}}%
    \v@lmax=\D@\v@lXa\advance\v@lmax-\C@\v@lYa%
    \delt@=\A@\v@lY\advance\delt@-\B@\v@lX%
    \invers@{\v@leur}{\delt@}\edef\v@ldelta{\repdecn@mb{\v@leur}}%
    \v@lXa=\A@\v@lmax\advance\v@lXa-\C@\v@lmin%
    \v@lYa=\B@\v@lmax\advance\v@lYa-\D@\v@lmin%
    \v@lXa=\v@ldelta\v@lXa\v@lYa=\v@ldelta\v@lYa%
    \setc@ntr@l{1}\Figp@intregDD#1:{#2}(\v@lXa,\v@lYa)%
    \resetc@ntr@l\et@tinters@cDD}\ignorespaces}
\ctr@ld@f\def\inters@cTD#1:#2[#3,#4;#5,#6]{{\s@uvc@ntr@l\et@tinters@cTD%
    \setc@ntr@l{2}\figvectNVTD-1[#4,#6]\figvectNVTD-2[#6,-1]\figvectPTD-1[#3,#5]%
    \r@pPSTD\v@leur[-2,-1,#4]\edef\v@lcoef{\repdecn@mb{\v@leur}}%
    \figpttraTD#1:{#2}=#3/\v@lcoef,#4/\resetc@ntr@l\et@tinters@cTD}\ignorespaces}
\ctr@ld@f\def\r@pPSTD#1[#2,#3,#4]{{\Figg@tXY{#2}\edef\Xu@{\repdecn@mb{\v@lX}}%
    \edef\Yu@{\repdecn@mb{\v@lY}}\edef\Zu@{\repdecn@mb{\v@lZ}}%
    \Figg@tXY{#3}\v@lmin=\Xu@\v@lX\advance\v@lmin\Yu@\v@lY\advance\v@lmin\Zu@\v@lZ%
    \Figg@tXY{#4}\v@lmax=\Xu@\v@lX\advance\v@lmax\Yu@\v@lY\advance\v@lmax\Zu@\v@lZ%
    \invers@{\v@leur}{\v@lmax}\global\result@t=\repdecn@mb{\v@leur}\v@lmin}%
    #1=\result@t}
\ctr@ln@m\n@rminf
\ctr@ld@f\def\n@rminfDD#1#2{{\Figg@tXY{#2}\maxim@m{\v@lX}{\v@lX}{-\v@lX}%
    \maxim@m{\v@lY}{\v@lY}{-\v@lY}\maxim@m{\global\result@t}{\v@lX}{\v@lY}}%
    #1=\result@t}
\ctr@ld@f\def\n@rminfTD#1#2{{\Figg@tXY{#2}\maxim@m{\v@lX}{\v@lX}{-\v@lX}%
    \maxim@m{\v@lY}{\v@lY}{-\v@lY}\maxim@m{\v@lZ}{\v@lZ}{-\v@lZ}%
    \maxim@m{\v@lX}{\v@lX}{\v@lY}\maxim@m{\global\result@t}{\v@lX}{\v@lZ}}%
    #1=\result@t}
\ctr@ln@m\n@rmeucC
\ctr@ld@f\def\n@rmeucCDD#1#2{\Figg@tXY{#2}\divide\v@lX\f@ctech\divide\v@lY\f@ctech%
    #1=\repdecn@mb{\v@lX}\v@lX\v@lX=\repdecn@mb{\v@lY}\v@lY\advance#1\v@lX}
\ctr@ld@f\def\n@rmeucCTD#1#2{\Figg@tXY{#2}%
    \divide\v@lX\f@ctech\divide\v@lY\f@ctech\divide\v@lZ\f@ctech%
    #1=\repdecn@mb{\v@lX}\v@lX\v@lX=\repdecn@mb{\v@lY}\v@lY\advance#1\v@lX%
    \v@lX=\repdecn@mb{\v@lZ}\v@lZ\advance#1\v@lX}
\ctr@ln@m\n@rmeucSV
\ctr@ld@f\def\n@rmeucSVDD#1#2{{\Figg@tXY{#2}%
    \v@lXa=\repdecn@mb{\v@lX}\v@lX\v@lYa=\repdecn@mb{\v@lY}\v@lY%
    \advance\v@lXa\v@lYa\sqrt@{\global\result@t}{\v@lXa}}#1=\result@t}
\ctr@ld@f\def\n@rmeucSVTD#1#2{{\Figg@tXY{#2}\v@lXa=\repdecn@mb{\v@lX}\v@lX%
    \v@lYa=\repdecn@mb{\v@lY}\v@lY\v@lZa=\repdecn@mb{\v@lZ}\v@lZ%
    \advance\v@lXa\v@lYa\advance\v@lXa\v@lZa\sqrt@{\global\result@t}{\v@lXa}}#1=\result@t}
\ctr@ln@m\n@rmeuc
\ctr@ld@f\def\n@rmeucDD#1#2{{\Figg@tXY{#2}\getredf@ctDD\f@ctech(\v@lX,\v@lY)%
    \divide\v@lX\f@ctech\divide\v@lY\f@ctech%
    \v@lXa=\repdecn@mb{\v@lX}\v@lX\v@lYa=\repdecn@mb{\v@lY}\v@lY%
    \advance\v@lXa\v@lYa\sqrt@{\global\result@t}{\v@lXa}%
    \global\multiply\result@t\f@ctech}#1=\result@t}
\ctr@ld@f\def\n@rmeucTD#1#2{{\Figg@tXY{#2}\getredf@ctTD\f@ctech(\v@lX,\v@lY,\v@lZ)%
    \divide\v@lX\f@ctech\divide\v@lY\f@ctech\divide\v@lZ\f@ctech%
    \v@lXa=\repdecn@mb{\v@lX}\v@lX%
    \v@lYa=\repdecn@mb{\v@lY}\v@lY\v@lZa=\repdecn@mb{\v@lZ}\v@lZ%
    \advance\v@lXa\v@lYa\advance\v@lXa\v@lZa\sqrt@{\global\result@t}{\v@lXa}%
    \global\multiply\result@t\f@ctech}#1=\result@t}
\ctr@ln@m\vecunit@
\ctr@ld@f\def\vecunit@DD#1#2{{\Figg@tXY{#2}\getredf@ctDD\f@ctech(\v@lX,\v@lY)%
    \divide\v@lX\f@ctech\divide\v@lY\f@ctech%
    \Figv@ctCreg#1(\v@lX,\v@lY)\n@rmeucSV{\v@lYa}{#1}%
    \invers@{\v@lXa}{\v@lYa}\edef\v@lv@lXa{\repdecn@mb{\v@lXa}}%
    \v@lX=\v@lv@lXa\v@lX\v@lY=\v@lv@lXa\v@lY%
    \Figv@ctCreg#1(\v@lX,\v@lY)\multiply\v@lYa\f@ctech\global\result@t=\v@lYa}}
\ctr@ld@f\def\vecunit@TD#1#2{{\Figg@tXY{#2}\getredf@ctTD\f@ctech(\v@lX,\v@lY,\v@lZ)%
    \divide\v@lX\f@ctech\divide\v@lY\f@ctech\divide\v@lZ\f@ctech%
    \Figv@ctCreg#1(\v@lX,\v@lY,\v@lZ)\n@rmeucSV{\v@lYa}{#1}%
    \invers@{\v@lXa}{\v@lYa}\edef\v@lv@lXa{\repdecn@mb{\v@lXa}}%
    \v@lX=\v@lv@lXa\v@lX\v@lY=\v@lv@lXa\v@lY\v@lZ=\v@lv@lXa\v@lZ%
    \Figv@ctCreg#1(\v@lX,\v@lY,\v@lZ)\multiply\v@lYa\f@ctech\global\result@t=\v@lYa}}
\ctr@ld@f\def\vecunitC@TD[#1,#2]{\Figg@tXYa{#1}\Figg@tXY{#2}%
    \advance\v@lX-\v@lXa\advance\v@lY-\v@lYa\advance\v@lZ-\v@lZa\c@lvecunitTD}
\ctr@ld@f\def\vecunitCV@TD#1{\Figg@tXY{#1}\c@lvecunitTD}
\ctr@ld@f\def\c@lvecunitTD{\getredf@ctTD\f@ctech(\v@lX,\v@lY,\v@lZ)%
    \divide\v@lX\f@ctech\divide\v@lY\f@ctech\divide\v@lZ\f@ctech%
    \v@lXa=\repdecn@mb{\v@lX}\v@lX%
    \v@lYa=\repdecn@mb{\v@lY}\v@lY\v@lZa=\repdecn@mb{\v@lZ}\v@lZ%
    \advance\v@lXa\v@lYa\advance\v@lXa\v@lZa\sqrt@{\v@lYa}{\v@lXa}%
    \invers@{\v@lXa}{\v@lYa}\edef\v@lv@lXa{\repdecn@mb{\v@lXa}}%
    \v@lX=\v@lv@lXa\v@lX\v@lY=\v@lv@lXa\v@lY\v@lZ=\v@lv@lXa\v@lZ}
\ctr@ln@m\figgetangle
\ctr@ld@f\def\figgetangleDD#1[#2,#3,#4]{\ifGR@cri{\s@uvc@ntr@l\et@tfiggetangleDD\setc@ntr@l{2}%
    \figvectPDD-1[#2,#3]\figvectPDD-2[#2,#4]\vecunit@{-1}{-1}%
    \c@lproscalDD\delt@[-2,-1]\figvectNVDD-1[-1]\c@lproscalDD\v@leur[-2,-1]%
    \arct@n\v@lmax(\delt@,\v@leur)\v@lmax=\rdT@deg\v@lmax%
    \ifdim\v@lmax<\z@\advance\v@lmax\DePI@deg\fi\xdef#1{\repdecn@mb{\v@lmax}}%
    \resetc@ntr@l\et@tfiggetangleDD}\ignorespaces\fi}
\ctr@ld@f\def\figgetangleTD#1[#2,#3,#4,#5]{\ifGR@cri{\s@uvc@ntr@l\et@tfiggetangleTD\setc@ntr@l{2}%
    \figvectPTD-1[#2,#3]\figvectPTD-2[#2,#5]\figvectNVTD-3[-1,-2]%
    \figvectPTD-2[#2,#4]\figvectNVTD-4[-3,-1]%
    \vecunit@{-1}{-1}\c@lproscalTD\delt@[-2,-1]\c@lproscalTD\v@leur[-2,-4]%
    \arct@n\v@lmax(\delt@,\v@leur)\v@lmax=\rdT@deg\v@lmax%
    \ifdim\v@lmax<\z@\advance\v@lmax\DePI@deg\fi\xdef#1{\repdecn@mb{\v@lmax}}%
    \resetc@ntr@l\et@tfiggetangleTD}\ignorespaces\fi}    
\ctr@ld@f\def\figgetdist#1[#2,#3]{\ifGR@cri{\s@uvc@ntr@l\et@tfiggetdist\setc@ntr@l{2}%
    \figvectP-1[#2,#3]\n@rmeuc{\v@lX}{-1}\v@lX=\ptT@unit@\v@lX\xdef#1{\repdecn@mb{\v@lX}}%
    \resetc@ntr@l\et@tfiggetdist}\ignorespaces\fi}
\ctr@ld@f\def\figget#1=#2[#3]{\keln@mun#1|%
    \def\n@mref{a}\ifx\l@debut\n@mref\figgetangle#2[#3]\else
    \def\n@mref{d}\ifx\l@debut\n@mref\figgetdist#2[#3]\else
    \W@rnmeskwd{figget}{#1}\fi\fi\ignorespaces}
\ctr@ld@f\def\Figg@tT#1{\c@ntr@lnum{#1}%
    {\expandafter\expandafter\expandafter\extr@ctT\csname\objc@de\endcsname:%
     \ifnum\B@@ltxt=\z@\ptn@me{#1}\else\csname\objc@de T\endcsname\fi}}
\ctr@ld@f\def\extr@ctT#1,#2,#3/#4:{\def\B@@ltxt{#3}}
\ctr@ld@f\def\Figg@tXY#1{\c@ntr@lnum{#1}%
    \expandafter\expandafter\expandafter\extr@ctC\csname\objc@de\endcsname:}
\ctr@ln@m\extr@ctC
\ctr@ld@f\def\extr@ctCDD#1/#2,#3,#4:{\v@lX=#2\v@lY=#3}
\ctr@ld@f\def\extr@ctCTD#1/#2,#3,#4:{\v@lX=#2\v@lY=#3\v@lZ=#4}
\ctr@ld@f\def\Figg@tXYa#1{\c@ntr@lnum{#1}%
    \expandafter\expandafter\expandafter\extr@ctCa\csname\objc@de\endcsname:}
\ctr@ln@m\extr@ctCa
\ctr@ld@f\def\extr@ctCaDD#1/#2,#3,#4:{\v@lXa=#2\v@lYa=#3}
\ctr@ld@f\def\extr@ctCaTD#1/#2,#3,#4:{\v@lXa=#2\v@lYa=#3\v@lZa=#4}
\ctr@ln@m\t@xt@
\ctr@ld@f\def\figinit#1{\t@stc@tcodech@nge\initpr@lim\Figinit@#1,:\initpss@ttings\ignorespaces}
\ctr@ld@f\def\Figinit@#1,#2:{\setunit@{#1}\def\t@xt@{#2}\ifx\t@xt@\empty\else\Figinit@@#2:\fi}
\ctr@ld@f\def\Figinit@@#1#2:{\if#12 \else\Figs@tproj{#1}\initTD@\fi}
\ctr@ln@w{newif}\ifTr@isDim
\ctr@ld@f\def\UnD@fined{UNDEFINED}
\ctr@ln@m\@utoFN
\ctr@ln@m\@utoFInDone
\ctr@ln@m\disob@unit
\ctr@ld@f\def\initpr@lim{\initb@undb@x\figsetmark{}\figsetptname{$A_{##1}$}\def\Sc@leFact{1}%
    \initDD@\figsetroundcoord{yes}\GR@critrue\expandafter\setupd@te\D@FTupdate:%
    \edef\disob@unit{\UnD@fined}\edef\t@rgetpt{\UnD@fined}\gdef\@utoFInDone{1}\gdef\@utoFN{0}}
\ctr@ld@f\def\initDD@{\Tr@isDimfalse%
    \ifPDFm@ke%
     \let\Ps@rcerc=\Ps@rcercBz%
     \let\Ps@rell=\Ps@rellBz%
    \fi
    \let\c@lDCUn=\c@lDCUnDD%
    \let\c@lDCDeux=\c@lDCDeuxDD%
    \let\c@ldefproj=\relax%
    \let\c@lproscal=\c@lproscalDD%
    \let\c@lprojSP=\relax%
    \let\extr@ctC=\extr@ctCDD%
    \let\extr@ctCa=\extr@ctCaDD%
    \let\extr@ctCF=\extr@ctCFDD%
    \let\Figp@intreg=\Figp@intregDD%
    \let\Figpts@xes=\Figpts@xesDD%
    \let\getredf@ctB=\getredf@ctBDD%
    \let\n@rmeucSV=\n@rmeucSVDD\let\n@rmeuc=\n@rmeucDD\let\n@rmeucC\n@rmeucCDD\let\n@rminf=\n@rminfDD%
    \let\pr@dMatV=\pr@dMatVDD%
    \let\Q@@xes=\Q@@xesDD%
    \let\vecunit@=\vecunit@DD%
    \let\figcoord=\figcoordDD%
    \let\figgetangle=\figgetangleDD%
    \let\figpt=\figptDD%
    \let\figptBezier=\figptBezierDD%
    \let\figptbary=\figptbaryDD%
    \let\figptcirc=\figptcircDD%
    \let\figptcircumcenter=\figptcircumcenterDD%
    \let\figptcopy=\figptcopyDD%
    \let\figptcurvcenter=\figptcurvcenterDD%
    \let\figptell=\figptellDD%
    \let\figptendnormal=\figptendnormalDD%
    \let\figptinterlineplane=\figptinterlineplaneDD%
    \let\figptinterlines=\inters@cDD%
    \let\figptorthocenter=\figptorthocenterDD%
    \let\figptorthoprojline=\figptorthoprojlineDD%
    \let\figptorthoprojplane=\figptorthoprojplaneDD%
    \let\figptrot=\figptrotDD%
    \let\figptscontrol=\figptscontrolDD%
    \let\figptsintercirc=\figptsintercircDD%
    \let\figptsinterlinell=\figptsinterlinellDD%
    \let\figptsorthoprojline=\figptsorthoprojlineDD%
    \let\figptorthoprojplane=\figptorthoprojplaneDD%
    \let\figptsrot=\figptsrotDD%
    \let\figptssym=\figptssymDD%
    \let\figptstra=\figptstraDD%
    \let\figptsym=\figptsymDD%
    \let\figpttraC=\figpttraCDD%
    \let\figpttra=\figpttraDD%
    \let\figptvisilimSL=\figptvisilimSLDD%
    \let\figsetobdist=\figsetobdistDD%
    \let\figsettarget=\figsettargetDD%
    \let\figsetview=\figsetviewDD%
    \let\figvectDBezier=\figvectDBezierDD%
    \let\figvectN=\figvectNDD%
    \let\figvectNV=\figvectNVDD%
    \let\figvectP=\figvectPDD%
    \let\figvectU=\figvectUDD%
    \let\figdrawarccircP=\Q@arccircPDD%
    \let\figdrawarccirc=\Q@arccircDD%
    \let\figdrawarcell=\Q@arcellDD%
    \let\figdrawarcellPA=\Q@arcellPADD%
    \let\figdrawarrowBezier=\Q@arrowBezierDD%
    \let\figdrawarrowcircP=\Q@arrowcircPDD%
    \let\figdrawarrowcirc=\Q@arrowcircDD%
    \let\figdrawarrowhead=\Q@arrowheadDD%
    \let\figdrawarrow=\Q@arrowDD%
    \let\figdrawBezier=\Q@BezierDD%
    \let\figdrawcirc=\Q@circDD%
    \let\figdrawcurve=\Q@curveDD%
    \let\figdrawnormal=\Q@normalDD%
    }
\ctr@ld@f\def\initTD@{\Tr@isDimtrue\initb@undb@xTD\newt@rgetptfalse\newdis@bfalse%
    \let\c@lDCUn=\c@lDCUnTD%
    \let\c@lDCDeux=\c@lDCDeuxTD%
    \let\c@ldefproj=\c@ldefprojTD%
    \let\c@lproscal=\c@lproscalTD%
    \let\extr@ctC=\extr@ctCTD%
    \let\extr@ctCa=\extr@ctCaTD%
    \let\extr@ctCF=\extr@ctCFTD%
    \let\Figp@intreg=\Figp@intregTD%
    \let\Figpts@xes=\Figpts@xesTD%
    \let\getredf@ctB=\getredf@ctBTD%
    \let\n@rmeucSV=\n@rmeucSVTD\let\n@rmeuc=\n@rmeucTD\let\n@rmeucC\n@rmeucCTD\let\n@rminf=\n@rminfTD%
    \let\pr@dMatV=\pr@dMatVTD%
    \let\Q@@xes=\Q@@xesTD%
    \let\vecunit@=\vecunit@TD%
    \let\figcoord=\figcoordTD%
    \let\figgetangle=\figgetangleTD%
    \let\figpt=\figptTD%
    \let\figptBezier=\figptBezierTD%
    \let\figptbary=\figptbaryTD%
    \let\figptcirc=\figptcircTD%
    \let\figptcircumcenter=\figptcircumcenterTD%
    \let\figptcopy=\figptcopyTD%
    \let\figptcurvcenter=\figptcurvcenterTD%
    \let\figptinterlineplane=\figptinterlineplaneTD%
    \let\figptinterlines=\inters@cTD%
    \let\figptorthocenter=\figptorthocenterTD%
    \let\figptorthoprojline=\figptorthoprojlineTD%
    \let\figptorthoprojplane=\figptorthoprojplaneTD%
    \let\figptrot=\figptrotTD%
    \let\figptscontrol=\figptscontrolTD%
    \let\figptsintercirc=\figptsintercircTD%
    \let\figptsorthoprojline=\figptsorthoprojlineTD%
    \let\figptsorthoprojplane=\figptsorthoprojplaneTD%
    \let\figptsrot=\figptsrotTD%
    \let\figptssym=\figptssymTD%
    \let\figptstra=\figptstraTD%
    \let\figptsym=\figptsymTD%
    \let\figpttraC=\figpttraCTD%
    \let\figpttra=\figpttraTD%
    \let\figptvisilimSL=\figptvisilimSLTD%
    \let\figsetobdist=\figsetobdistTD%
    \let\figsettarget=\figsettargetTD%
    \let\figsetview=\figsetviewTD%
    \let\figvectDBezier=\figvectDBezierTD%
    \let\figvectN=\figvectNTD%
    \let\figvectNV=\figvectNVTD%
    \let\figvectP=\figvectPTD%
    \let\figvectU=\figvectUTD%
    \let\figdrawarccircP=\Q@arccircPTD%
    \let\figdrawarccirc=\Q@arccircTD%
    \let\figdrawarcell=\Q@arcellTD%
    \let\figdrawarcellPA=\Q@arcellPATD%
    \let\figdrawarrowBezier=\Q@arrowBezierTD%
    \let\figdrawarrowcircP=\Q@arrowcircPTD%
    \let\figdrawarrowcirc=\Q@arrowcircTD%
    \let\figdrawarrowhead=\Q@arrowheadTD%
    \let\figdrawarrow=\Q@arrowTD%
    \let\figdrawBezier=\Q@BezierTD%
    \let\figdrawcirc=\Q@circTD%
    \let\figdrawcurve=\Q@curveTD%
    }
\ctr@ld@f\def\un@v@ilable#1{\immediate\write16{*** The macro #1 is not available in the current context.}}
\ctr@ld@f\def\figinsert#1{{\def\t@xt@{#1}\relax%
    \ifx\t@xt@\empty\ifnum\@utoFInDone>\z@\Figinsert@\DefGIfilen@me,:\fi%
    \else\expandafter\FiginsertNu@#1 :\fi}\ignorespaces}
\ctr@ld@f\def\FiginsertNu@#1 #2:{\def\t@xt@{#1}\relax\ifx\t@xt@\empty\def\t@xt@{#2}%
    \ifx\t@xt@\empty\ifnum\@utoFInDone>\z@\Figinsert@\DefGIfilen@me,:\fi%
    \else\FiginsertNu@#2:\fi\else\expandafter\FiginsertNd@#1 #2:\fi}
\ctr@ld@f\def\FiginsertNd@#1#2:{\ifcat#1a\Figinsert@#1#2,:\else%
    \ifnum\@utoFInDone>\z@\Figinsert@\DefGIfilen@me,#1#2,:\fi\fi}
\ctr@ln@m\Sc@leFact
\ctr@ld@f\def\Figinsert@#1,#2:{\def\t@xt@{#2}\ifx\t@xt@\empty\xdef\Sc@leFact{1}\else%
    \X@rgdeux@#2\xdef\Sc@leFact{\@rgdeux}\fi%
    \Figdisc@rdLTS{#1}{\t@xt@}\@psfgetbb{\t@xt@}%
    \v@lX=\@psfllx\p@\v@lX=\ptpsT@pt\v@lX\v@lX=\Sc@leFact\v@lX%
    \v@lY=\@psflly\p@\v@lY=\ptpsT@pt\v@lY\v@lY=\Sc@leFact\v@lY%
    \b@undb@x{\v@lX}{\v@lY}%
    \v@lX=\@psfurx\p@\v@lX=\ptpsT@pt\v@lX\v@lX=\Sc@leFact\v@lX%
    \v@lY=\@psfury\p@\v@lY=\ptpsT@pt\v@lY\v@lY=\Sc@leFact\v@lY%
    \b@undb@x{\v@lX}{\v@lY}%
    \ifPDFm@ke\Figinclud@PDF{\t@xt@}{\Sc@leFact}\else%
    \v@lX=\c@nt pt\v@lX=\Sc@leFact\v@lX\edef\F@ct{\repdecn@mb{\v@lX}}%
    \ifx\TeXturesonMacOSltX\special{postscriptfile #1 vscale=\F@ct\space hscale=\F@ct}%
    \else\includegraphics{#1}\fi\fi%
    \message{[\t@xt@]}\ignorespaces}
\ctr@ld@f\def\Figdisc@rdLTS#1#2{\expandafter\Figdisc@rdLTS@#1 :#2}
\ctr@ld@f\def\Figdisc@rdLTS@#1 #2:#3{\def#3{#1}\relax\ifx#3\empty\expandafter\Figdisc@rdLTS@#2:#3\fi}
\ctr@ld@f\def\figinsertE#1{\FiginsertE@#1,:\ignorespaces}
\ctr@ld@f\def\FiginsertE@#1,#2:{{\def\t@xt@{#2}\ifx\t@xt@\empty\xdef\Sc@leFact{1}\else%
    \X@rgdeux@#2\xdef\Sc@leFact{\@rgdeux}\fi%
    \Figdisc@rdLTS{#1}{\t@xt@}\pdfximage{\t@xt@}%
    \setbox\Gb@x=\hbox{\pdfrefximage\pdflastximage}%
    \v@lX=\z@\v@lY=-\Sc@leFact\dp\Gb@x\b@undb@x{\v@lX}{\v@lY}%
    \advance\v@lX\Sc@leFact\wd\Gb@x\advance\v@lY\Sc@leFact\dp\Gb@x%
    \advance\v@lY\Sc@leFact\ht\Gb@x\b@undb@x{\v@lX}{\v@lY}%
    \v@lX=\Sc@leFact\wd\Gb@x\pdfximage width \v@lX {\t@xt@}%
    \rlap{\pdfrefximage\pdflastximage}\message{[\t@xt@]}}\ignorespaces}
\ctr@ld@f\def\X@rgdeux@#1,{\edef\@rgdeux{#1}}
\ctr@ln@m\figpt
\ctr@ld@f\def\figptDD#1:#2(#3,#4){\ifGR@cri\c@ntr@lnum{#1}%
    {\v@lX=#3\unit@\v@lY=#4\unit@\Fig@dmpt{#2}{\z@}}\ignorespaces\fi}
\ctr@ld@f\def\Fig@dmpt#1#2{\def\t@xt@{#1}\ifx\t@xt@\empty\def\B@@ltxt{\z@}%
    \else\expandafter\gdef\csname\objc@de T\endcsname{#1}\def\B@@ltxt{\@ne}\fi%
    \expandafter\xdef\csname\objc@de\endcsname{\ifitis@vect@r\C@dCl@svect%
    \else\C@dCl@spt\fi,\z@,\B@@ltxt/\the\v@lX,\the\v@lY,#2}}
\ctr@ld@f\def\C@dCl@spt{P}
\ctr@ld@f\def\C@dCl@svect{V}
\ctr@ln@m\c@@rdYZ
\ctr@ln@m\c@@rdY
\ctr@ld@f\def\figptTD#1:#2(#3,#4){\ifGR@cri\c@ntr@lnum{#1}%
    \def\c@@rdYZ{#4,0,0}\extrairelepremi@r\c@@rdY\de\c@@rdYZ%
    \extrairelepremi@r\c@@rdZ\de\c@@rdYZ%
    {\v@lX=#3\unit@\v@lY=\c@@rdY\unit@\v@lZ=\c@@rdZ\unit@\Fig@dmpt{#2}{\the\v@lZ}%
    \b@undb@xTD{\v@lX}{\v@lY}{\v@lZ}}\ignorespaces\fi}
\ctr@ln@m\Figp@intreg
\ctr@ld@f\def\Figp@intregDD#1:#2(#3,#4){\c@ntr@lnum{#1}%
    {\result@t=#4\v@lX=#3\v@lY=\result@t\Fig@dmpt{#2}{\z@}}\ignorespaces}
\ctr@ld@f\def\Figp@intregTD#1:#2(#3,#4){\c@ntr@lnum{#1}%
    \def\c@@rdYZ{#4,\z@,\z@}\extrairelepremi@r\c@@rdY\de\c@@rdYZ%
    \extrairelepremi@r\c@@rdZ\de\c@@rdYZ%
    {\v@lX=#3\v@lY=\c@@rdY\v@lZ=\c@@rdZ\Fig@dmpt{#2}{\the\v@lZ}%
    \b@undb@xTD{\v@lX}{\v@lY}{\v@lZ}}\ignorespaces}
\ctr@ln@m\figptBezier
\ctr@ld@f\def\figptBezierDD#1:#2:#3[#4,#5,#6,#7]{\ifGR@cri{\s@uvc@ntr@l\et@tfigptBezierDD%
    \FigptBezier@#3[#4,#5,#6,#7]\Figp@intregDD#1:{#2}(\v@lX,\v@lY)%
    \resetc@ntr@l\et@tfigptBezierDD}\ignorespaces\fi}
\ctr@ld@f\def\figptBezierTD#1:#2:#3[#4,#5,#6,#7]{\ifGR@cri{\s@uvc@ntr@l\et@tfigptBezierTD%
    \FigptBezier@#3[#4,#5,#6,#7]\Figp@intregTD#1:{#2}(\v@lX,\v@lY,\v@lZ)%
    \resetc@ntr@l\et@tfigptBezierTD}\ignorespaces\fi}
\ctr@ld@f\def\FigptBezier@#1[#2,#3,#4,#5]{\setc@ntr@l{2}%
    \edef\T@{#1}\v@leur=\p@\advance\v@leur-#1pt\edef\UNmT@{\repdecn@mb{\v@leur}}%
    \figptcopy-4:/#2/\figptcopy-3:/#3/\figptcopy-2:/#4/\figptcopy-1:/#5/%
    \l@mbd@un=-4 \l@mbd@de=-\thr@@\p@rtent=\m@ne\c@lDecast%
    \l@mbd@un=-4 \l@mbd@de=-\thr@@\p@rtent=-\tw@\c@lDecast%
    \l@mbd@un=-4 \l@mbd@de=-\thr@@\p@rtent=-\thr@@\c@lDecast\Figg@tXY{-4}}
\ctr@ln@m\c@lDCUn
\ctr@ld@f\def\c@lDCUnDD#1#2{\Figg@tXY{#1}\v@lX=\UNmT@\v@lX\v@lY=\UNmT@\v@lY%
    \Figg@tXYa{#2}\advance\v@lX\T@\v@lXa\advance\v@lY\T@\v@lYa%
    \Figp@intregDD#1:(\v@lX,\v@lY)}
\ctr@ld@f\def\c@lDCUnTD#1#2{\Figg@tXY{#1}\v@lX=\UNmT@\v@lX\v@lY=\UNmT@\v@lY\v@lZ=\UNmT@\v@lZ%
    \Figg@tXYa{#2}\advance\v@lX\T@\v@lXa\advance\v@lY\T@\v@lYa\advance\v@lZ\T@\v@lZa%
    \Figp@intregTD#1:(\v@lX,\v@lY,\v@lZ)}
\ctr@ld@f\def\c@lDecast{\relax\ifnum\l@mbd@un<\p@rtent\c@lDCUn{\l@mbd@un}{\l@mbd@de}%
    \advance\l@mbd@un\@ne\advance\l@mbd@de\@ne\c@lDecast\fi}
\ctr@ld@f\def\figptmap#1:#2=#3/#4/#5/{\ifGR@cri{\s@uvc@ntr@l\et@tfigptmap%
    \setc@ntr@l{2}\figvectP-1[#4,#3]\Figg@tXY{-1}%
    \pr@dMatV/#5/\figpttra#1:{#2}=#4/1,-1/%
    \resetc@ntr@l\et@tfigptmap}\ignorespaces\fi}
\ctr@ln@m\pr@dMatV
\ctr@ld@f\def\pr@dMatVDD/#1,#2;#3,#4/{\v@lXa=#1\v@lX\advance\v@lXa#2\v@lY%
    \v@lYa=#3\v@lX\advance\v@lYa#4\v@lY\Figv@ctCreg-1(\v@lXa,\v@lYa)}
\ctr@ld@f\def\pr@dMatVTD/#1,#2,#3;#4,#5,#6;#7,#8,#9/{%
    \v@lXa=#1\v@lX\advance\v@lXa#2\v@lY\advance\v@lXa#3\v@lZ%
    \v@lYa=#4\v@lX\advance\v@lYa#5\v@lY\advance\v@lYa#6\v@lZ%
    \v@lZa=#7\v@lX\advance\v@lZa#8\v@lY\advance\v@lZa#9\v@lZ%
    \Figv@ctCreg-1(\v@lXa,\v@lYa,\v@lZa)}
\ctr@ln@m\figptbary
\ctr@ld@f\def\figptbaryDD#1:#2[#3;#4]{\ifGR@cri{\edef\list@num{#3}\extrairelepremi@r\p@int\de\list@num%
    \s@mme=\z@\@ecfor\c@ef:=#4\do{\advance\s@mme\c@ef}%
    \edef\listec@ef{#4,0}\extrairelepremi@r\c@ef\de\listec@ef%
    \Figg@tXY{\p@int}\divide\v@lX\s@mme\divide\v@lY\s@mme%
    \multiply\v@lX\c@ef\multiply\v@lY\c@ef%
    \@ecfor\p@int:=\list@num\do{\extrairelepremi@r\c@ef\de\listec@ef%
           \Figg@tXYa{\p@int}\divide\v@lXa\s@mme\divide\v@lYa\s@mme%
           \multiply\v@lXa\c@ef\multiply\v@lYa\c@ef%
           \advance\v@lX\v@lXa\advance\v@lY\v@lYa}%
    \Figp@intregDD#1:{#2}(\v@lX,\v@lY)}\ignorespaces\fi}
\ctr@ld@f\def\figptbaryTD#1:#2[#3;#4]{\ifGR@cri{\edef\list@num{#3}\extrairelepremi@r\p@int\de\list@num%
    \s@mme=\z@\@ecfor\c@ef:=#4\do{\advance\s@mme\c@ef}%
    \edef\listec@ef{#4,0}\extrairelepremi@r\c@ef\de\listec@ef%
    \Figg@tXY{\p@int}\divide\v@lX\s@mme\divide\v@lY\s@mme\divide\v@lZ\s@mme%
    \multiply\v@lX\c@ef\multiply\v@lY\c@ef\multiply\v@lZ\c@ef%
    \@ecfor\p@int:=\list@num\do{\extrairelepremi@r\c@ef\de\listec@ef%
           \Figg@tXYa{\p@int}\divide\v@lXa\s@mme\divide\v@lYa\s@mme\divide\v@lZa\s@mme%
           \multiply\v@lXa\c@ef\multiply\v@lYa\c@ef\multiply\v@lZa\c@ef%
           \advance\v@lX\v@lXa\advance\v@lY\v@lYa\advance\v@lZ\v@lZa}%
    \Figp@intregTD#1:{#2}(\v@lX,\v@lY,\v@lZ)}\ignorespaces\fi}
\ctr@ld@f\def\figptbaryR#1:#2[#3;#4]{\ifGR@cri{%
    \v@leur=\z@\@ecfor\c@ef:=#4\do{\maxim@m{\v@lmax}{\c@ef pt}{-\c@ef pt}%
    \ifdim\v@lmax>\v@leur\v@leur=\v@lmax\fi}%
    \ifdim\v@leur<\p@\f@ctech=\@M\else\ifdim\v@leur<\t@n\p@\f@ctech=\@m\else%
    \ifdim\v@leur<\c@nt\p@\f@ctech=\c@nt\else\ifdim\v@leur<\@m\p@\f@ctech=\t@n\else%
    \f@ctech=\@ne\fi\fi\fi\fi%
    \def\listec@ef{0}%
    \@ecfor\c@ef:=#4\do{\sc@lec@nvRI{\c@ef pt}\edef\listec@ef{\listec@ef,\the\s@mme}}%
    \extrairelepremi@r\c@ef\de\listec@ef\figptbary#1:#2[#3;\listec@ef]}\ignorespaces\fi}
\ctr@ld@f\def\sc@lec@nvRI#1{\v@leur=#1\p@rtentiere{\s@mme}{\v@leur}\advance\v@leur-\s@mme\p@%
    \multiply\v@leur\f@ctech\p@rtentiere{\p@rtent}{\v@leur}%
    \multiply\s@mme\f@ctech\advance\s@mme\p@rtent}
\ctr@ln@m\figptcirc
\ctr@ld@f\def\figptcircDD#1:#2:#3;#4(#5){\ifGR@cri{\s@uvc@ntr@l\et@tfigptcircDD%
    \c@lptellDD#1:{#2}:#3;#4,#4(#5)\resetc@ntr@l\et@tfigptcircDD}\ignorespaces\fi}
\ctr@ld@f\def\figptcircTD#1:#2:#3,#4,#5;#6(#7){\ifGR@cri{\s@uvc@ntr@l\et@tfigptcircTD%
    \setc@ntr@l{2}\c@lExtAxes#3,#4,#5(#6)\figptellP#1:{#2}:#3,-4,-5(#7)%
    \resetc@ntr@l\et@tfigptcircTD}\ignorespaces\fi}
\ctr@ln@m\figptcircumcenter
\ctr@ld@f\def\figptcircumcenterDD#1:#2[#3,#4,#5]{\ifGR@cri{\s@uvc@ntr@l\et@tfigptcircumcenterDD%
    \setc@ntr@l{2}\figvectNDD-5[#3,#4]\figptbaryDD-3:[#3,#4;1,1]%
                  \figvectNDD-6[#4,#5]\figptbaryDD-4:[#4,#5;1,1]%
    \resetc@ntr@l{2}\inters@cDD#1:{#2}[-3,-5;-4,-6]%
    \resetc@ntr@l\et@tfigptcircumcenterDD}\ignorespaces\fi}
\ctr@ld@f\def\figptcircumcenterTD#1:#2[#3,#4,#5]{\ifGR@cri{\s@uvc@ntr@l\et@tfigptcircumcenterTD%
    \setc@ntr@l{2}\figvectNTD-1[#3,#4,#5]%
    \figvectPTD-3[#3,#4]\figvectNVTD-5[-1,-3]\figptbaryTD-3:[#3,#4;1,1]%
    \figvectPTD-4[#4,#5]\figvectNVTD-6[-1,-4]\figptbaryTD-4:[#4,#5;1,1]%
    \resetc@ntr@l{2}\inters@cTD#1:{#2}[-3,-5;-4,-6]%
    \resetc@ntr@l\et@tfigptcircumcenterTD}\ignorespaces\fi}
\ctr@ln@m\figptcopy
\ctr@ld@f\def\figptcopyDD#1:#2/#3/{\ifGR@cri{\Figg@tXY{#3}%
    \Figp@intregDD#1:{#2}(\v@lX,\v@lY)}\ignorespaces\fi}
\ctr@ld@f\def\figptcopyTD#1:#2/#3/{\ifGR@cri{\Figg@tXY{#3}%
    \Figp@intregTD#1:{#2}(\v@lX,\v@lY,\v@lZ)}\ignorespaces\fi}
\ctr@ln@m\figptcurvcenter
\ctr@ld@f\def\figptcurvcenterDD#1:#2:#3[#4,#5,#6,#7]{\ifGR@cri{\s@uvc@ntr@l\et@tfigptcurvcenterDD%
    \setc@ntr@l{2}\c@lcurvradDD#3[#4,#5,#6,#7]\edef\Sprim@{\repdecn@mb{\result@t}}%
    \figptBezierDD-1::#3[#4,#5,#6,#7]\figpttraDD#1:{#2}=-1/\Sprim@,-5/%
    \resetc@ntr@l\et@tfigptcurvcenterDD}\ignorespaces\fi}
\ctr@ld@f\def\figptcurvcenterTD#1:#2:#3[#4,#5,#6,#7]{\ifGR@cri{\s@uvc@ntr@l\et@tfigptcurvcenterTD%
    \setc@ntr@l{2}\figvectDBezierTD -5:1,#3[#4,#5,#6,#7]%
    \figvectDBezierTD -6:2,#3[#4,#5,#6,#7]\vecunit@TD{-5}{-5}%
    \edef\Sprim@{\repdecn@mb{\result@t}}\figvectNVTD-1[-6,-5]%
    \figvectNVTD-5[-5,-1]\c@lproscalTD\v@leur[-6,-5]%
    \invers@{\v@leur}{\v@leur}\v@leur=\Sprim@\v@leur\v@leur=\Sprim@\v@leur%
    \figptBezierTD-1::#3[#4,#5,#6,#7]\edef\Sprim@{\repdecn@mb{\v@leur}}%
    \figpttraTD#1:{#2}=-1/\Sprim@,-5/\resetc@ntr@l\et@tfigptcurvcenterTD}\ignorespaces\fi}
\ctr@ld@f\def\c@lcurvradDD#1[#2,#3,#4,#5]{{\figvectDBezierDD -5:1,#1[#2,#3,#4,#5]%
    \figvectDBezierDD -6:2,#1[#2,#3,#4,#5]\vecunit@DD{-5}{-5}%
    \edef\Sprim@{\repdecn@mb{\result@t}}\figvectNVDD-5[-5]\c@lproscalDD\v@leur[-6,-5]%
    \invers@{\v@leur}{\v@leur}\v@leur=\Sprim@\v@leur\v@leur=\Sprim@\v@leur%
    \global\result@t=\v@leur}}
\ctr@ln@m\figptell
\ctr@ld@f\def\figptellDD#1:#2:#3;#4,#5(#6,#7){\ifGR@cri{\s@uvc@ntr@l\et@tfigptell%
    \c@lptellDD#1::#3;#4,#5(#6)\figptrotDD#1:{#2}=#1/#3,#7/%
    \resetc@ntr@l\et@tfigptell}\ignorespaces\fi}
\ctr@ld@f\def\c@lptellDD#1:#2:#3;#4,#5(#6){\c@ssin{\C@}{\S@}{#6}\v@lmin=\C@ pt\v@lmax=\S@ pt%
    \v@lmin=#4\v@lmin\v@lmax=#5\v@lmax%
    \edef\Xc@mp{\repdecn@mb{\v@lmin}}\edef\Yc@mp{\repdecn@mb{\v@lmax}}%
    \setc@ntr@l{2}\figvectC-1(\Xc@mp,\Yc@mp)\figpttraDD#1:{#2}=#3/1,-1/}
\ctr@ld@f\def\figptellP#1:#2:#3,#4,#5(#6){\ifGR@cri{\s@uvc@ntr@l\et@tfigptellP%
    \setc@ntr@l{2}\figvectP-1[#3,#4]\figvectP-2[#3,#5]%
    \v@leur=#6pt\c@lptellP{#3}{-1}{-2}\figptcopy#1:{#2}/-3/%
    \resetc@ntr@l\et@tfigptellP}\ignorespaces\fi}
\ctr@ln@m\@ngle
\ctr@ld@f\def\c@lptellP#1#2#3{\edef\@ngle{\repdecn@mb\v@leur}\c@ssin{\C@}{\S@}{\@ngle}%
    \figpttra-3:=#1/\C@,#2/\figpttra-3:=-3/\S@,#3/}
\ctr@ln@m\figptendnormal
\ctr@ld@f\def\figptendnormalDD#1:#2:#3,#4[#5,#6]{\ifGR@cri{\s@uvc@ntr@l\et@tfigptendnormal%
    \Figg@tXYa{#5}\Figg@tXY{#6}%
    \advance\v@lX-\v@lXa\advance\v@lY-\v@lYa%
    \setc@ntr@l{2}\Figv@ctCreg-1(\v@lX,\v@lY)\vecunit@{-1}{-1}\Figg@tXY{-1}%
    \delt@=#3\unit@\maxim@m{\delt@}{\delt@}{-\delt@}\edef\l@ngueur{\repdecn@mb{\delt@}}%
    \v@lX=\l@ngueur\v@lX\v@lY=\l@ngueur\v@lY%
    \delt@=\p@\advance\delt@-#4pt\edef\l@ngueur{\repdecn@mb{\delt@}}%
    \figptbaryR-1:[#5,#6;#4,\l@ngueur]\Figg@tXYa{-1}%
    \advance\v@lXa\v@lY\advance\v@lYa-\v@lX%
    \setc@ntr@l{1}\Figp@intregDD#1:{#2}(\v@lXa,\v@lYa)\resetc@ntr@l\et@tfigptendnormal}%
    \ignorespaces\fi}
\ctr@ld@f\def\figptexcenter#1:#2[#3,#4,#5]{\ifGR@cri{\let@xte={-}
    \Figptexinsc@nter#1:#2[#3,#4,#5]}\ignorespaces\fi}
\ctr@ld@f\def\figptincenter#1:#2[#3,#4,#5]{\ifGR@cri{\let@xte={}
    \Figptexinsc@nter#1:#2[#3,#4,#5]}\ignorespaces\fi}
\ctr@ld@f
\ctr@ld@f\def\Figptexinsc@nter#1:#2[#3,#4,#5]{%
    \figgetdist\LA@[#4,#5]\figgetdist\LB@[#3,#5]\figgetdist\LC@[#3,#4]%
    \figptbaryR#1:{#2}[#3,#4,#5;\the\let@xte\LA@,\LB@,\LC@]}
\ctr@ln@m\figptinterlineplane
\ctr@ld@f\def\figptinterlineplaneDD{\un@v@ilable{figptinterlineplane}}
\ctr@ld@f\def\figptinterlineplaneTD#1:#2[#3,#4;#5,#6]{\ifGR@cri{\s@uvc@ntr@l\et@tfigptinterlineplane%
    \setc@ntr@l{2}\figvectPTD-1[#3,#5]\vecunit@TD{-2}{#6}%
    \r@pPSTD\v@leur[-2,-1,#4]\edef\v@lcoef{\repdecn@mb{\v@leur}}%
    \figpttraTD#1:{#2}=#3/\v@lcoef,#4/\resetc@ntr@l\et@tfigptinterlineplane}\ignorespaces\fi}
\ctr@ln@m\figptorthocenter
\ctr@ld@f\def\figptorthocenterDD#1:#2[#3,#4,#5]{\ifGR@cri{\s@uvc@ntr@l\et@tfigptorthocenterDD%
    \setc@ntr@l{2}\figvectNDD-3[#3,#4]\figvectNDD-4[#4,#5]%
    \resetc@ntr@l{2}\inters@cDD#1:{#2}[#5,-3;#3,-4]%
    \resetc@ntr@l\et@tfigptorthocenterDD}\ignorespaces\fi}
\ctr@ld@f\def\figptorthocenterTD#1:#2[#3,#4,#5]{\ifGR@cri{\s@uvc@ntr@l\et@tfigptorthocenterTD%
    \setc@ntr@l{2}\figvectNTD-1[#3,#4,#5]%
    \figvectPTD-2[#3,#4]\figvectNVTD-3[-1,-2]%
    \figvectPTD-2[#4,#5]\figvectNVTD-4[-1,-2]%
    \resetc@ntr@l{2}\inters@cTD#1:{#2}[#5,-3;#3,-4]%
    \resetc@ntr@l\et@tfigptorthocenterTD}\ignorespaces\fi}
\ctr@ln@m\figptorthoprojline
\ctr@ld@f\def\figptorthoprojlineDD#1:#2=#3/#4,#5/{\ifGR@cri{\s@uvc@ntr@l\et@tfigptorthoprojlineDD%
    \setc@ntr@l{2}\figvectPDD-3[#4,#5]\figvectNVDD-4[-3]\resetc@ntr@l{2}%
    \inters@cDD#1:{#2}[#3,-4;#4,-3]\resetc@ntr@l\et@tfigptorthoprojlineDD}\ignorespaces\fi}
\ctr@ld@f\def\figptorthoprojlineTD#1:#2=#3/#4,#5/{\ifGR@cri{\s@uvc@ntr@l\et@tfigptorthoprojlineTD%
    \setc@ntr@l{2}\figvectPTD-1[#4,#3]\figvectPTD-2[#4,#5]\vecunit@TD{-2}{-2}%
    \c@lproscalTD\v@leur[-1,-2]\edef\v@lcoef{\repdecn@mb{\v@leur}}%
    \figpttraTD#1:{#2}=#4/\v@lcoef,-2/\resetc@ntr@l\et@tfigptorthoprojlineTD}\ignorespaces\fi}
\ctr@ln@m\figptorthoprojplane
\ctr@ld@f\def\figptorthoprojplaneDD{\un@v@ilable{figptorthoprojplane}}
\ctr@ld@f\def\figptorthoprojplaneTD#1:#2=#3/#4,#5/{\ifGR@cri{\s@uvc@ntr@l\et@tfigptorthoprojplane%
    \setc@ntr@l{2}\figvectPTD-1[#3,#4]\vecunit@TD{-2}{#5}%
    \c@lproscalTD\v@leur[-1,-2]\edef\v@lcoef{\repdecn@mb{\v@leur}}%
    \figpttraTD#1:{#2}=#3/\v@lcoef,-2/\resetc@ntr@l\et@tfigptorthoprojplane}\ignorespaces\fi}
\ctr@ld@f\def\figpthom#1:#2=#3/#4,#5/{\ifGR@cri{\s@uvc@ntr@l\et@tfigpthom%
    \setc@ntr@l{2}\figvectP-1[#4,#3]\figpttra#1:{#2}=#4/#5,-1/%
    \resetc@ntr@l\et@tfigpthom}\ignorespaces\fi}
\ctr@ld@f\def\figptinv#1:#2=#3/#4,#5/{\ifGR@cri{\s@uvc@ntr@l\et@tfigptinv%
    \setc@ntr@l{2}\figvectP-1[#4,#3]\Figg@tXY{-1}%
    \getredf@ctB\f@ctech\n@rmeucC{\delt@}{-1}%
    \delt@=\ptT@unit@\delt@\delt@=\ptT@unit@\delt@%
    \invers@{\delt@}{\delt@}\multiply\f@ctech\f@ctech\divide\delt@\f@ctech%
    \delt@=#5\delt@\edef\v@lcoef{\repdecn@mb{\delt@}}\figpttra#1:{#2}=#4/\v@lcoef,-1/%
    \resetc@ntr@l\et@tfigptinv}\ignorespaces\fi}
\ctr@ln@m\figptrot
\ctr@ld@f\def\figptrotDD#1:#2=#3/#4,#5/{\ifGR@cri{\s@uvc@ntr@l\et@tfigptrotDD%
    \c@ssin{\C@}{\S@}{#5}\setc@ntr@l{2}\figvectPDD-1[#4,#3]\Figg@tXY{-1}%
    \v@lXa=\C@\v@lX\advance\v@lXa-\S@\v@lY%
    \v@lYa=\S@\v@lX\advance\v@lYa\C@\v@lY%
    \Figv@ctCreg-1(\v@lXa,\v@lYa)\figpttraDD#1:{#2}=#4/1,-1/%
    \resetc@ntr@l\et@tfigptrotDD}\ignorespaces\fi}
\ctr@ld@f\def\figptrotTD#1:#2=#3/#4,#5,#6/{\ifGR@cri{\s@uvc@ntr@l\et@tfigptrotTD%
    \c@ssin{\C@}{\S@}{#5}%
    \setc@ntr@l{2}\figptorthoprojplaneTD-3:=#4/#3,#6/\figvectPTD-2[-3,#3]%
    \n@rmeucTD\v@leur{-2}\ifdim\v@leur<\Cepsil@n\Figg@tXYa{#3}\else%
    \edef\v@lcoef{\repdecn@mb{\v@leur}}\figvectNVTD-1[#6,-2]%
    \Figg@tXYa{-1}\v@lXa=\v@lcoef\v@lXa\v@lYa=\v@lcoef\v@lYa\v@lZa=\v@lcoef\v@lZa%
    \v@lXa=\S@\v@lXa\v@lYa=\S@\v@lYa\v@lZa=\S@\v@lZa\Figg@tXY{-2}%
    \advance\v@lXa\C@\v@lX\advance\v@lYa\C@\v@lY\advance\v@lZa\C@\v@lZ%
    \Figg@tXY{-3}\advance\v@lXa\v@lX\advance\v@lYa\v@lY\advance\v@lZa\v@lZ\fi%
    \Figp@intregTD#1:{#2}(\v@lXa,\v@lYa,\v@lZa)\resetc@ntr@l\et@tfigptrotTD}\ignorespaces\fi}
\ctr@ln@m\figptsym
\ctr@ld@f\def\figptsymDD#1:#2=#3/#4,#5/{\ifGR@cri{\s@uvc@ntr@l\et@tfigptsymDD%
    \resetc@ntr@l{2}\figptorthoprojlineDD-5:=#3/#4,#5/\figvectPDD-2[#3,-5]%
    \figpttraDD#1:{#2}=#3/2,-2/\resetc@ntr@l\et@tfigptsymDD}\ignorespaces\fi}
\ctr@ld@f\def\figptsymTD#1:#2=#3/#4,#5/{\ifGR@cri{\s@uvc@ntr@l\et@tfigptsymTD%
    \resetc@ntr@l{2}\figptorthoprojplaneTD-3:=#3/#4,#5/\figvectPTD-2[#3,-3]%
    \figpttraTD#1:{#2}=#3/2,-2/\resetc@ntr@l\et@tfigptsymTD}\ignorespaces\fi}
\ctr@ln@m\figpttra
\ctr@ld@f\def\figpttraDD#1:#2=#3/#4,#5/{\ifGR@cri{\Figg@tXYa{#5}\v@lXa=#4\v@lXa\v@lYa=#4\v@lYa%
    \Figg@tXY{#3}\advance\v@lX\v@lXa\advance\v@lY\v@lYa%
    \Figp@intregDD#1:{#2}(\v@lX,\v@lY)}\ignorespaces\fi}
\ctr@ld@f\def\figpttraTD#1:#2=#3/#4,#5/{\ifGR@cri{\Figg@tXYa{#5}\v@lXa=#4\v@lXa\v@lYa=#4\v@lYa%
    \v@lZa=#4\v@lZa\Figg@tXY{#3}\advance\v@lX\v@lXa\advance\v@lY\v@lYa%
    \advance\v@lZ\v@lZa\Figp@intregTD#1:{#2}(\v@lX,\v@lY,\v@lZ)}\ignorespaces\fi}
\ctr@ln@m\figpttraC
\ctr@ld@f\def\figpttraCDD#1:#2=#3/#4,#5/{\ifGR@cri{\v@lXa=#4\unit@\v@lYa=#5\unit@%
    \Figg@tXY{#3}\advance\v@lX\v@lXa\advance\v@lY\v@lYa%
    \Figp@intregDD#1:{#2}(\v@lX,\v@lY)}\ignorespaces\fi}
\ctr@ld@f\def\figpttraCTD#1:#2=#3/#4,#5,#6/{\ifGR@cri{\v@lXa=#4\unit@\v@lYa=#5\unit@\v@lZa=#6\unit@%
    \Figg@tXY{#3}\advance\v@lX\v@lXa\advance\v@lY\v@lYa\advance\v@lZ\v@lZa%
    \Figp@intregTD#1:{#2}(\v@lX,\v@lY,\v@lZ)}\ignorespaces\fi}
\ctr@ld@f\def\figptsaxes#1:#2(#3){\ifGR@cri{\an@lys@xes#3,:\ifx\t@xt@\empty%
    \ifTr@isDim\Figpts@xes#1:#2(0,#3,0,#3,0,#3)\else\Figpts@xes#1:#2(0,#3,0,#3)\fi%
    \else\Figpts@xes#1:#2(#3)\fi}\ignorespaces\fi}
\ctr@ln@m\Figpts@xes
\ctr@ld@f\def\Figpts@xesDD#1:#2(#3,#4,#5,#6){%
    \s@mme=#1\figpttraC\the\s@mme:$x$=#2/#4,0/%
    \advance\s@mme\@ne\figpttraC\the\s@mme:$y$=#2/0,#6/}
\ctr@ld@f\def\Figpts@xesTD#1:#2(#3,#4,#5,#6,#7,#8){%
    \s@mme=#1\figpttraC\the\s@mme:$x$=#2/#4,0,0/%
    \advance\s@mme\@ne\figpttraC\the\s@mme:$y$=#2/0,#6,0/%
    \advance\s@mme\@ne\figpttraC\the\s@mme:$z$=#2/0,0,#8/}
\ctr@ld@f\def\figptsmap#1=#2/#3/#4/{\ifGR@cri{\s@uvc@ntr@l\et@tfigptsmap%
    \setc@ntr@l{2}\def\list@num{#2}\s@mme=#1%
    \@ecfor\p@int:=\list@num\do{\figvectP-1[#3,\p@int]\Figg@tXY{-1}%
    \pr@dMatV/#4/\figpttra\the\s@mme:=#3/1,-1/\advance\s@mme\@ne}%
    \resetc@ntr@l\et@tfigptsmap}\ignorespaces\fi}
\ctr@ln@m\figptscontrol
\ctr@ld@f\def\figptscontrolDD#1[#2,#3,#4,#5]{\ifGR@cri{\s@uvc@ntr@l\et@tfigptscontrolDD\setc@ntr@l{2}%
    \v@lX=\z@\v@lY=\z@\Figtr@nptDD{-5}{#2}\Figtr@nptDD{2}{#5}%
    \divide\v@lX\@vi\divide\v@lY\@vi%
    \Figtr@nptDD{3}{#3}\Figtr@nptDD{-1.5}{#4}\Figp@intregDD-1:(\v@lX,\v@lY)%
    \v@lX=\z@\v@lY=\z@\Figtr@nptDD{2}{#2}\Figtr@nptDD{-5}{#5}%
    \divide\v@lX\@vi\divide\v@lY\@vi\Figtr@nptDD{-1.5}{#3}\Figtr@nptDD{3}{#4}%
    \s@mme=#1\advance\s@mme\@ne\Figp@intregDD\the\s@mme:(\v@lX,\v@lY)%
    \figptcopyDD#1:/-1/\resetc@ntr@l\et@tfigptscontrolDD}\ignorespaces\fi}
\ctr@ld@f\def\figptscontrolTD#1[#2,#3,#4,#5]{\ifGR@cri{\s@uvc@ntr@l\et@tfigptscontrolTD\setc@ntr@l{2}%
    \v@lX=\z@\v@lY=\z@\v@lZ=\z@\Figtr@nptTD{-5}{#2}\Figtr@nptTD{2}{#5}%
    \divide\v@lX\@vi\divide\v@lY\@vi\divide\v@lZ\@vi%
    \Figtr@nptTD{3}{#3}\Figtr@nptTD{-1.5}{#4}\Figp@intregTD-1:(\v@lX,\v@lY,\v@lZ)%
    \v@lX=\z@\v@lY=\z@\v@lZ=\z@\Figtr@nptTD{2}{#2}\Figtr@nptTD{-5}{#5}%
    \divide\v@lX\@vi\divide\v@lY\@vi\divide\v@lZ\@vi\Figtr@nptTD{-1.5}{#3}\Figtr@nptTD{3}{#4}%
    \s@mme=#1\advance\s@mme\@ne\Figp@intregTD\the\s@mme:(\v@lX,\v@lY,\v@lZ)%
    \figptcopyTD#1:/-1/\resetc@ntr@l\et@tfigptscontrolTD}\ignorespaces\fi}
\ctr@ld@f\def\Figtr@nptDD#1#2{\Figg@tXYa{#2}\v@lXa=#1\v@lXa\v@lYa=#1\v@lYa%
    \advance\v@lX\v@lXa\advance\v@lY\v@lYa}
\ctr@ld@f\def\Figtr@nptTD#1#2{\Figg@tXYa{#2}\v@lXa=#1\v@lXa\v@lYa=#1\v@lYa\v@lZa=#1\v@lZa%
    \advance\v@lX\v@lXa\advance\v@lY\v@lYa\advance\v@lZ\v@lZa}
\ctr@ld@f\def\figptscontrolcurve#1,#2[#3]{\ifGR@cri{\s@uvc@ntr@l\et@tfigptscontrolcurve%
    \def\list@num{#3}\extrairelepremi@r\Ak@\de\list@num%
    \extrairelepremi@r\Ai@\de\list@num\extrairelepremi@r\Aj@\de\list@num%
    \s@mme=#1\figptcopy\the\s@mme:/\Ai@/%
    \setc@ntr@l{2}\figvectP -1[\Ak@,\Aj@]%
    \@ecfor\Ak@:=\list@num\do{\advance\s@mme\@ne\figpttra\the\s@mme:=\Ai@/\curv@roundness,-1/%
       \figvectP -1[\Ai@,\Ak@]\advance\s@mme\@ne\figpttra\the\s@mme:=\Aj@/-\curv@roundness,-1/%
       \advance\s@mme\@ne\figptcopy\the\s@mme:/\Aj@/%
       \edef\Ai@{\Aj@}\edef\Aj@{\Ak@}}\advance\s@mme-#1\divide\s@mme\thr@@%
       \xdef#2{\the\s@mme}%
    \resetc@ntr@l\et@tfigptscontrolcurve}\ignorespaces\fi}
\ctr@ln@m\figptsintercirc
\ctr@ld@f\def\figptsintercircDD#1[#2,#3;#4,#5]{\ifGR@cri{\s@uvc@ntr@l\et@tfigptsintercircDD%
    \setc@ntr@l{2}\let\c@lNVintc=\c@lNVintcDD\Figptsintercirc@#1[#2,#3;#4,#5]%
    \resetc@ntr@l\et@tfigptsintercircDD}\ignorespaces\fi}
\ctr@ld@f\def\figptsintercircTD#1[#2,#3;#4,#5;#6]{\ifGR@cri{\s@uvc@ntr@l\et@tfigptsintercircTD%
    \setc@ntr@l{2}\let\c@lNVintc=\c@lNVintcTD\vecunitC@TD[#2,#6]%
    \Figv@ctCreg-3(\v@lX,\v@lY,\v@lZ)\Figptsintercirc@#1[#2,#3;#4,#5]%
    \resetc@ntr@l\et@tfigptsintercircTD}\ignorespaces\fi}
\ctr@ld@f\def\Figptsintercirc@#1[#2,#3;#4,#5]{\figvectP-1[#2,#4]%
    \vecunit@{-1}{-1}\delt@=\result@t\f@ctech=\result@tent%
    \s@mme=#1\advance\s@mme\@ne\figptcopy#1:/#2/\figptcopy\the\s@mme:/#4/%
    \ifdim\delt@=\z@\else%
    \v@lmin=#3\unit@\v@lmax=#5\unit@\v@leur=\v@lmin\advance\v@leur\v@lmax%
    \ifdim\v@leur>\delt@%
    \v@leur=\v@lmin\advance\v@leur-\v@lmax\maxim@m{\v@leur}{\v@leur}{-\v@leur}%
    \ifdim\v@leur<\delt@%
    \divide\v@lmin\f@ctech\divide\v@lmax\f@ctech\divide\delt@\f@ctech%
    \v@lmin=\repdecn@mb{\v@lmin}\v@lmin\v@lmax=\repdecn@mb{\v@lmax}\v@lmax%
    \invers@{\v@leur}{\delt@}\advance\v@lmax-\v@lmin%
    \v@lmax=-\repdecn@mb{\v@leur}\v@lmax\advance\delt@\v@lmax\delt@=.5\delt@%
    \v@lmax=\delt@\multiply\v@lmax\f@ctech%
    \edef\t@ille{\repdecn@mb{\v@lmax}}\figpttra-2:=#2/\t@ille,-1/%
    \delt@=\repdecn@mb{\delt@}\delt@\advance\v@lmin-\delt@%
    \sqrt@{\v@leur}{\v@lmin}\multiply\v@leur\f@ctech\edef\t@ille{\repdecn@mb{\v@leur}}%
    \c@lNVintc\figpttra#1:=-2/-\t@ille,-1/\figpttra\the\s@mme:=-2/\t@ille,-1/\fi\fi\fi}
\ctr@ld@f\def\c@lNVintcDD{\Figg@tXY{-1}\Figv@ctCreg-1(-\v@lY,\v@lX)} 
\ctr@ld@f\def\c@lNVintcTD{{\Figg@tXY{-3}\v@lmin=\v@lX\v@lmax=\v@lY\v@leur=\v@lZ%
    \Figg@tXY{-1}\c@lprovec{-3}\vecunit@{-3}{-3}
    \Figg@tXY{-1}\v@lmin=\v@lX\v@lmax=\v@lY%
    \v@leur=\v@lZ\Figg@tXY{-3}\c@lprovec{-1}}} 
\ctr@ln@m\figptsinterlinell
\ctr@ld@f\def\figptsinterlinellDD#1[#2,#3,#4,#5;#6,#7]{\ifGR@cri{\s@uvc@ntr@l\et@tfigptsinterlinellDD%
    \figptcopy#1:/#6/\s@mme=#1\advance\s@mme\@ne\figptcopy\the\s@mme:/#7/%
    \v@lmin=#3\unit@\v@lmax=#4\unit@
    \setc@ntr@l{2}\figptbaryDD-4:[#6,#7;1,1]\figptsrotDD-3=-4,#7/#2,-#5/
    \Figg@tXY{-3}\Figg@tXYa{#2}\advance\v@lX-\v@lXa\advance\v@lY-\v@lYa
    \figvectP-1[-3,-2]\Figg@tXYa{-1}\figvectP-3[-4,#7]\Figptsint@rLE{#1}
    \resetc@ntr@l\et@tfigptsinterlinellDD}\ignorespaces\fi}
\ctr@ld@f\def\figptsinterlinellP#1[#2,#3,#4;#5,#6]{\ifGR@cri{\s@uvc@ntr@l\et@tfigptsinterlinellP%
    \figptcopy#1:/#5/\s@mme=#1\advance\s@mme\@ne\figptcopy\the\s@mme:/#6/\setc@ntr@l{2}%
    \figvectP-1[#2,#3]\vecunit@{-1}{-1}\v@lmin=\result@t
    \figvectP-2[#2,#4]\vecunit@{-2}{-2}\v@lmax=\result@t
    \figptbary-4:[#5,#6;1,1]
    \figvectP-3[#2,-4]\c@lproscal\v@lX[-3,-1]\c@lproscal\v@lY[-3,-2]
    \figvectP-3[-4,#6]\c@lproscal\v@lXa[-3,-1]\c@lproscal\v@lYa[-3,-2]
    \Figptsint@rLE{#1}\resetc@ntr@l\et@tfigptsinterlinellP}\ignorespaces\fi}
\ctr@ld@f\def\Figptsint@rLE#1{%
    \getredf@ctDD\f@ctech(\v@lmin,\v@lmax)%
    \getredf@ctDD\p@rtent(\v@lX,\v@lY)\ifnum\p@rtent>\f@ctech\f@ctech=\p@rtent\fi%
    \getredf@ctDD\p@rtent(\v@lXa,\v@lYa)\ifnum\p@rtent>\f@ctech\f@ctech=\p@rtent\fi%
    \divide\v@lmin\f@ctech\divide\v@lmax\f@ctech\divide\v@lX\f@ctech\divide\v@lY\f@ctech%
    \divide\v@lXa\f@ctech\divide\v@lYa\f@ctech%
    \c@rre=\repdecn@mb\v@lXa\v@lmax\mili@u=\repdecn@mb\v@lYa\v@lmin%
    \getredf@ctDD\f@ctech(\c@rre,\mili@u)%
    \c@rre=\repdecn@mb\v@lX\v@lmax\mili@u=\repdecn@mb\v@lY\v@lmin%
    \getredf@ctDD\p@rtent(\c@rre,\mili@u)\ifnum\p@rtent>\f@ctech\f@ctech=\p@rtent\fi%
    \divide\v@lmin\f@ctech\divide\v@lmax\f@ctech\divide\v@lX\f@ctech\divide\v@lY\f@ctech%
    \divide\v@lXa\f@ctech\divide\v@lYa\f@ctech%
    \v@lmin=\repdecn@mb{\v@lmin}\v@lmin\v@lmax=\repdecn@mb{\v@lmax}\v@lmax%
    \edef\G@xde{\repdecn@mb\v@lmin}\edef\P@xde{\repdecn@mb\v@lmax}%
    \c@rre=-\v@lmax\v@leur=\repdecn@mb\v@lY\v@lY\advance\c@rre\v@leur\c@rre=\G@xde\c@rre%
    \v@leur=\repdecn@mb\v@lX\v@lX\v@leur=\P@xde\v@leur\advance\c@rre\v@leur
    \v@lmin=\repdecn@mb\v@lYa\v@lmin\v@lmax=\repdecn@mb\v@lXa\v@lmax%
    \mili@u=\repdecn@mb\v@lX\v@lmax\advance\mili@u\repdecn@mb\v@lY\v@lmin
    \v@lmax=\repdecn@mb\v@lXa\v@lmax\advance\v@lmax\repdecn@mb\v@lYa\v@lmin
    \ifdim\v@lmax>\epsil@n%
    \maxim@m{\v@leur}{\c@rre}{-\c@rre}\maxim@m{\v@lmin}{\mili@u}{-\mili@u}%
    \maxim@m{\v@leur}{\v@leur}{\v@lmin}\maxim@m{\v@lmin}{\v@lmax}{-\v@lmax}%
    \maxim@m{\v@leur}{\v@leur}{\v@lmin}\p@rtentiere{\p@rtent}{\v@leur}\advance\p@rtent\@ne%
    \divide\c@rre\p@rtent\divide\mili@u\p@rtent\divide\v@lmax\p@rtent%
    \delt@=\repdecn@mb{\mili@u}\mili@u\v@leur=\repdecn@mb{\v@lmax}\c@rre%
    \advance\delt@-\v@leur\ifdim\delt@<\z@\else\sqrt@\delt@\delt@%
    \invers@\v@lmax\v@lmax\edef\Uns@rAp{\repdecn@mb\v@lmax}%
    \v@leur=-\mili@u\advance\v@leur-\delt@\v@leur=\Uns@rAp\v@leur%
    \edef\t@ille{\repdecn@mb\v@leur}\figpttra#1:=-4/\t@ille,-3/\s@mme=#1\advance\s@mme\@ne%
    \v@leur=-\mili@u\advance\v@leur\delt@\v@leur=\Uns@rAp\v@leur%
    \edef\t@ille{\repdecn@mb\v@leur}\figpttra\the\s@mme:=-4/\t@ille,-3/\fi\fi}
\ctr@ln@m\figptsorthoprojline
\ctr@ld@f\def\figptsorthoprojlineDD#1=#2/#3,#4/{\ifGR@cri{\s@uvc@ntr@l\et@tfigptsorthoprojlineDD%
    \setc@ntr@l{2}\figvectPDD-3[#3,#4]\figvectNVDD-4[-3]\resetc@ntr@l{2}%
    \def\list@num{#2}\s@mme=#1\@ecfor\p@int:=\list@num\do{%
    \inters@cDD\the\s@mme:[\p@int,-4;#3,-3]\advance\s@mme\@ne}%
    \resetc@ntr@l\et@tfigptsorthoprojlineDD}\ignorespaces\fi}
\ctr@ld@f\def\figptsorthoprojlineTD#1=#2/#3,#4/{\ifGR@cri{\s@uvc@ntr@l\et@tfigptsorthoprojlineTD%
    \setc@ntr@l{2}\figvectPTD-2[#3,#4]\vecunit@TD{-2}{-2}%
    \def\list@num{#2}\s@mme=#1\@ecfor\p@int:=\list@num\do{%
    \figvectPTD-1[#3,\p@int]\c@lproscalTD\v@leur[-1,-2]%
    \edef\v@lcoef{\repdecn@mb{\v@leur}}\figpttraTD\the\s@mme:=#3/\v@lcoef,-2/%
    \advance\s@mme\@ne}\resetc@ntr@l\et@tfigptsorthoprojlineTD}\ignorespaces\fi}
\ctr@ln@m\figptsorthoprojplane
\ctr@ld@f\def\figptsorthoprojplaneDD{\un@v@ilable{figptsorthoprojplane}}
\ctr@ld@f\def\figptsorthoprojplaneTD#1=#2/#3,#4/{\ifGR@cri{\s@uvc@ntr@l\et@tfigptsorthoprojplane%
    \setc@ntr@l{2}\vecunit@TD{-2}{#4}%
    \def\list@num{#2}\s@mme=#1\@ecfor\p@int:=\list@num\do{\figvectPTD-1[\p@int,#3]%
    \c@lproscalTD\v@leur[-1,-2]\edef\v@lcoef{\repdecn@mb{\v@leur}}%
    \figpttraTD\the\s@mme:=\p@int/\v@lcoef,-2/\advance\s@mme\@ne}%
    \resetc@ntr@l\et@tfigptsorthoprojplane}\ignorespaces\fi}
\ctr@ld@f\def\figptshom#1=#2/#3,#4/{\ifGR@cri{\s@uvc@ntr@l\et@tfigptshom%
    \setc@ntr@l{2}\def\list@num{#2}\s@mme=#1%
    \@ecfor\p@int:=\list@num\do{\figvectP-1[#3,\p@int]%
    \figpttra\the\s@mme:=#3/#4,-1/\advance\s@mme\@ne}%
    \resetc@ntr@l\et@tfigptshom}\ignorespaces\fi}
\ctr@ld@f\def\figptsinv#1=#2/#3,#4/{\ifGR@cri{\s@uvc@ntr@l\et@tfigptsinv%
    \setc@ntr@l{2}\def\list@num{#2}\s@mme=#1%
    \@ecfor\p@int:=\list@num\do{\figvectP-1[#3,\p@int]\Figg@tXY{-1}%
    \getredf@ctB\f@ctech\n@rmeucC{\delt@}{-1}%
    \delt@=\ptT@unit@\delt@\delt@=\ptT@unit@\delt@%
    \invers@{\delt@}{\delt@}\multiply\f@ctech\f@ctech\divide\delt@\f@ctech%
    \delt@=#4\delt@\edef\v@lcoef{\repdecn@mb{\delt@}}\figpttra\the\s@mme:=#3/\v@lcoef,-1/%
    \advance\s@mme\@ne}\resetc@ntr@l\et@tfigptsinv}\ignorespaces\fi}
\ctr@ln@m\figptsrot
\ctr@ld@f\def\figptsrotDD#1=#2/#3,#4/{\ifGR@cri{\s@uvc@ntr@l\et@tfigptsrotDD%
    \c@ssin{\C@}{\S@}{#4}\setc@ntr@l{2}\def\list@num{#2}\s@mme=#1%
    \@ecfor\p@int:=\list@num\do{\figvectPDD-1[#3,\p@int]\Figg@tXY{-1}%
    \v@lXa=\C@\v@lX\advance\v@lXa-\S@\v@lY%
    \v@lYa=\S@\v@lX\advance\v@lYa\C@\v@lY%
    \Figv@ctCreg-1(\v@lXa,\v@lYa)\figpttraDD\the\s@mme:=#3/1,-1/\advance\s@mme\@ne}%
    \resetc@ntr@l\et@tfigptsrotDD}\ignorespaces\fi}
\ctr@ld@f\def\figptsrotTD#1=#2/#3,#4,#5/{\ifGR@cri{\s@uvc@ntr@l\et@tfigptsrotTD%
    \c@ssin{\C@}{\S@}{#4}%
    \setc@ntr@l{2}\def\list@num{#2}\s@mme=#1%
    \@ecfor\p@int:=\list@num\do{\figptorthoprojplaneTD-3:=#3/\p@int,#5/%
    \figvectPTD-2[-3,\p@int]%
    \figvectNVTD-1[#5,-2]\n@rmeucTD\v@leur{-2}\edef\v@lcoef{\repdecn@mb{\v@leur}}%
    \Figg@tXYa{-1}\v@lXa=\v@lcoef\v@lXa\v@lYa=\v@lcoef\v@lYa\v@lZa=\v@lcoef\v@lZa%
    \v@lXa=\S@\v@lXa\v@lYa=\S@\v@lYa\v@lZa=\S@\v@lZa\Figg@tXY{-2}%
    \advance\v@lXa\C@\v@lX\advance\v@lYa\C@\v@lY\advance\v@lZa\C@\v@lZ%
    \Figg@tXY{-3}\advance\v@lXa\v@lX\advance\v@lYa\v@lY\advance\v@lZa\v@lZ%
    \Figp@intregTD\the\s@mme:(\v@lXa,\v@lYa,\v@lZa)\advance\s@mme\@ne}%
    \resetc@ntr@l\et@tfigptsrotTD}\ignorespaces\fi}
\ctr@ln@m\figptssym
\ctr@ld@f\def\figptssymDD#1=#2/#3,#4/{\ifGR@cri{\s@uvc@ntr@l\et@tfigptssymDD%
    \setc@ntr@l{2}\figvectPDD-3[#3,#4]\Figg@tXY{-3}\Figv@ctCreg-4(-\v@lY,\v@lX)%
    \resetc@ntr@l{2}\def\list@num{#2}\s@mme=#1%
    \@ecfor\p@int:=\list@num\do{\inters@cDD-5:[#3,-3;\p@int,-4]\figvectPDD-2[\p@int,-5]%
    \figpttraDD\the\s@mme:=\p@int/2,-2/\advance\s@mme\@ne}%
    \resetc@ntr@l\et@tfigptssymDD}\ignorespaces\fi}
\ctr@ld@f\def\figptssymTD#1=#2/#3,#4/{\ifGR@cri{\s@uvc@ntr@l\et@tfigptssymTD%
    \setc@ntr@l{2}\vecunit@TD{-2}{#4}\def\list@num{#2}\s@mme=#1%
    \@ecfor\p@int:=\list@num\do{\figvectPTD-1[\p@int,#3]%
    \c@lproscalTD\v@leur[-1,-2]\v@leur=2\v@leur\edef\v@lcoef{\repdecn@mb{\v@leur}}%
    \figpttraTD\the\s@mme:=\p@int/\v@lcoef,-2/\advance\s@mme\@ne}%
    \resetc@ntr@l\et@tfigptssymTD}\ignorespaces\fi}
\ctr@ln@m\figptstra
\ctr@ld@f\def\figptstraDD#1=#2/#3,#4/{\ifGR@cri{\Figg@tXYa{#4}\v@lXa=#3\v@lXa\v@lYa=#3\v@lYa%
    \def\list@num{#2}\s@mme=#1\@ecfor\p@int:=\list@num\do{\Figg@tXY{\p@int}%
    \advance\v@lX\v@lXa\advance\v@lY\v@lYa%
    \Figp@intregDD\the\s@mme:(\v@lX,\v@lY)\advance\s@mme\@ne}}\ignorespaces\fi}
\ctr@ld@f\def\figptstraTD#1=#2/#3,#4/{\ifGR@cri{\Figg@tXYa{#4}\v@lXa=#3\v@lXa\v@lYa=#3\v@lYa%
    \v@lZa=#3\v@lZa\def\list@num{#2}\s@mme=#1\@ecfor\p@int:=\list@num\do{\Figg@tXY{\p@int}%
    \advance\v@lX\v@lXa\advance\v@lY\v@lYa\advance\v@lZ\v@lZa%
    \Figp@intregTD\the\s@mme:(\v@lX,\v@lY,\v@lZ)\advance\s@mme\@ne}}\ignorespaces\fi}
\ctr@ln@m\figptvisilimSL
\ctr@ld@f\def\figptvisilimSLDD{\un@v@ilable{figptvisilimSL}}
\ctr@ld@f\def\figptvisilimSLTD#1:#2[#3,#4;#5,#6]{\ifGR@cri{\s@uvc@ntr@l\et@tfigptvisilimSLTD%
    \setc@ntr@l{2}\figvectP-1[#3,#4]\n@rminf{\delt@}{-1}%
    \ifcase\CUR@proj\v@lX=\cxa@\p@\v@lY=-\p@\v@lZ=\cxb@\p@
    \Figv@ctCreg-2(\v@lX,\v@lY,\v@lZ)\figvectP-3[#5,#6]\figvectNV-1[-2,-3]%
    \or\figvectP-1[#5,#6]\vecunitCV@TD{-1}\v@lmin=\v@lX\v@lmax=\v@lY
    \v@leur=\v@lZ\v@lX=\cza@\p@\v@lY=\czb@\p@\v@lZ=\czc@\p@\c@lprovec{-1}%
    \or\c@ley@pt{-2}\figvectN-1[#5,#6,-2]\fi
    \edef\Ai@{#3}\edef\Aj@{#4}\figvectP-2[#5,\Ai@]\c@lproscal\v@leur[-1,-2]%
    \ifdim\v@leur>\z@\p@rtent=\@ne\else\p@rtent=\m@ne\fi%
    \figvectP-2[#5,\Aj@]\c@lproscal\v@leur[-1,-2]%
    \ifdim\p@rtent\v@leur>\z@\figptcopy#1:#2/#3/%
    \message{*** \BS@ figptvisilimSL: points are on the same side.}\else%
    \figptcopy-3:/#3/\figptcopy-4:/#4/%
    \loop\figptbary-5:[-3,-4;1,1]\figvectP-2[#5,-5]\c@lproscal\v@leur[-1,-2]%
    \ifdim\p@rtent\v@leur>\z@\figptcopy-3:/-5/\else\figptcopy-4:/-5/\fi%
    \divide\delt@\tw@\ifdim\delt@>\epsil@n\repeat%
    \figptbary#1:#2[-3,-4;1,1]\fi\resetc@ntr@l\et@tfigptvisilimSLTD}\ignorespaces\fi}
\ctr@ld@f\def\c@ley@pt#1{\t@stp@r\ifitis@K\v@lX=\cza@\p@\v@lY=\czb@\p@\v@lZ=\czc@\p@%
    \Figv@ctCreg-1(\v@lX,\v@lY,\v@lZ)\Figp@intreg-2:(\wd\Bt@rget,\ht\Bt@rget,\dp\Bt@rget)%
    \figpttra#1:=-2/-\disob@intern,-1/\else\end\fi}
\ctr@ld@f\def\t@stp@r{\itis@Ktrue\ifnewt@rgetpt\else\itis@Kfalse%
    \message{*** \BS@ figptvisilimXX: target point undefined.}\fi\ifnewdis@b\else%
    \itis@Kfalse\message{*** \BS@ figptvisilimXX: observation distance undefined.}\fi%
    \ifitis@K\else\message{*** This macro must be called after \BS@ figdrawbegin or after
    having set the missing parameter(s) with \BS@ figset proj()}\fi}
\ctr@ld@f\def\figscan#1(#2,#3){{\s@uvc@ntr@l\et@tfigscan\@psfgetbb{#1}\if@psfbbfound\else%
    \def\@psfllx{0}\def\@psflly{20}\def\@psfurx{540}\def\@psfury{640}\fi\figscan@{#2}{#3}%
    \resetc@ntr@l\et@tfigscan}\ignorespaces}
\ctr@ld@f\def\figscan@#1#2{%
    \unit@=\@ne bp\setc@ntr@l{2}\figsetmark{}%
    \def\minst@p{20pt}%
    \v@lX=\@psfllx\p@\v@lX=\Sc@leFact\v@lX\r@undint\v@lX\v@lX%
    \v@lY=\@psflly\p@\v@lY=\Sc@leFact\v@lY\ifdim\v@lY>\z@\r@undint\v@lY\v@lY\fi%
    \delt@=\@psfury\p@\delt@=\Sc@leFact\delt@%
    \advance\delt@-\v@lY\v@lXa=\@psfurx\p@\v@lXa=\Sc@leFact\v@lXa\v@leur=\minst@p%
    \edef\valv@lY{\repdecn@mb{\v@lY}}\edef\LgTr@it{\the\delt@}%
    \loop\ifdim\v@lX<\v@lXa\edef\valv@lX{\repdecn@mb{\v@lX}}%
    \figptDD -1:(\valv@lX,\valv@lY)\figwriten -1:\hbox{\vrule height\LgTr@it}(0)%
    \ifdim\v@leur<\minst@p\else\figsetmark{\raise-8bp\hbox{$\scriptscriptstyle\triangle$}}%
    \figwrites -1:\@ffichnb{0}{\valv@lX}(6)\v@leur=\z@\figsetmark{}\fi%
    \advance\v@leur#1pt\advance\v@lX#1pt\repeat%
    \def\minst@p{10pt}%
    \v@lX=\@psfllx\p@\v@lX=\Sc@leFact\v@lX\ifdim\v@lX>\z@\r@undint\v@lX\v@lX\fi%
    \v@lY=\@psflly\p@\v@lY=\Sc@leFact\v@lY\r@undint\v@lY\v@lY%
    \delt@=\@psfurx\p@\delt@=\Sc@leFact\delt@%
    \advance\delt@-\v@lX\v@lYa=\@psfury\p@\v@lYa=\Sc@leFact\v@lYa\v@leur=\minst@p%
    \edef\valv@lX{\repdecn@mb{\v@lX}}\edef\LgTr@it{\the\delt@}%
    \loop\ifdim\v@lY<\v@lYa\edef\valv@lY{\repdecn@mb{\v@lY}}%
    \figptDD -1:(\valv@lX,\valv@lY)\figwritee -1:\vbox{\hrule width\LgTr@it}(0)%
    \ifdim\v@leur<\minst@p\else\figsetmark{$\triangleright$\kern4bp}%
    \figwritew -1:\@ffichnb{0}{\valv@lY}(6)\v@leur=\z@\figsetmark{}\fi%
    \advance\v@leur#2pt\advance\v@lY#2pt\repeat}
\ctr@ld@f
\ctr@ld@f\def\figscan@E#1(#2,#3){{\s@uvc@ntr@l\et@tfigscan@E%
    \Figdisc@rdLTS{#1}{\t@xt@}\pdfximage{\t@xt@}%
    \setbox\Gb@x=\hbox{\pdfrefximage\pdflastximage}%
    \edef\@psfllx{0}\v@lY=-\dp\Gb@x\edef\@psflly{\repdecn@mb{\v@lY}}%
    \edef\@psfurx{\repdecn@mb{\wd\Gb@x}}%
    \v@lY=\dp\Gb@x\advance\v@lY\ht\Gb@x\edef\@psfury{\repdecn@mb{\v@lY}}%
    \figscan@{#2}{#3}\resetc@ntr@l\et@tfigscan@E}\ignorespaces}
\ctr@ld@f\def\figshowpts[#1,#2]{{\figsetmark{$\bullet$}\figsetptname{\bf ##1}%
    \p@rtent=#2\relax\ifnum\p@rtent<\z@\p@rtent=\z@\fi%
    \s@mme=#1\relax\ifnum\s@mme<\z@\s@mme=\z@\fi%
    \loop\ifnum\s@mme<\p@rtent\pt@rvect{\s@mme}%
    \ifitis@K\figwriten{\the\s@mme}:(4pt)\fi\advance\s@mme\@ne\repeat%
    \pt@rvect{\s@mme}\ifitis@K\figwriten{\the\s@mme}:(4pt)\fi}\ignorespaces}
\ctr@ld@f\def\pt@rvect#1{\set@bjc@de{#1}%
    \expandafter\expandafter\expandafter\inqpt@rvec\csname\objc@de\endcsname:}
\ctr@ld@f\def\inqpt@rvec#1#2:{\if#1\C@dCl@spt\itis@Ktrue\else\itis@Kfalse\fi}
\ctr@ld@f\def\figshowsettings{{%
    \immediate\write16{====================================================================}%
    \immediate\write16{ Current settings are (DDV means "with dynamic default value"):}%
    \immediate\write16{ --- GENERAL ---}%
    \immediate\write16{Scale factor and Unit = \unit@util\space (\the\unit@)
     \space -> \BS@ figinit{ScaleFactorUnit}}%
    \immediate\write16{Update mode = \ifGRupdatem@de yes\else no\fi
     \space-> \BS@ figset(update=yes/no) or \BS@ figsetdefault(update=yes/no)}%
    \immediate\write16{ --- WRITING ---}%
    \immediate\write16{Implicit point name = \ptn@me{i} \space-> \BS@ figset write(ptname={Name})}%
    \immediate\write16{Point marker = \the\c@nsymb \space -> \BS@ figset write(mark=Mark)}%
    \immediate\write16{Print rounded coordinates = \ifr@undcoord yes\else no\fi
     \space-> \BS@ figset write(roundcoord=yes/no)}%
    \immediate\write16{ --- GRAPHICAL (general) ---}%
    \immediate\write16{Color = \CUR@color \space-> \BS@ figset(color=ColorDefinition)}%
    \immediate\write16{Filling mode = \iffillm@de yes\else no\fi
     \space-> \BS@ figset(fillmode=yes/no)}%
    \immediate\write16{Line join = \CUR@join \space-> \BS@ figset(join=miter/round/bevel)}%
    \immediate\write16{Line style = \CUR@dash \space-> \BS@ figset(dash=Index/Pattern)}%
    \immediate\write16{Line width = \CUR@width
     \space-> \BS@ figset(width=real in PostScript units)}%
    \immediate\write16{ --- GRAPHICAL (specific) ---}%
    \immediate\write16{Altitude (all the following attributes are DDV):}%
    \immediate\write16{ Base line color =
     \ifx\DDV@blcolor\D@FTref general color\else\DDV@blcolor\fi
     \space-> \BS@ figset altitude(blcolor=ColorDefinition)}%
    \immediate\write16{ Base line style =
     \ifx\DDV@bldash\D@FTref general style\else\DDV@bldash\fi
     \space-> \BS@ figset altitude(bldash=Index/Pattern)}%
    \immediate\write16{ Base line width =
     \ifx\DDV@blwidth\D@FTref general width\else\DDV@blwidth\fi
     \space-> \BS@ figset altitude(blwidth=real in PostScript units)}%
    \immediate\write16{ Square line color =
     \ifx\DDV@sqcolor\D@FTref general color\else\DDV@sqcolor\fi
     \space-> \BS@ figset altitude(sqcolor=ColorDefinition)}%
    \immediate\write16{ Square line style =
     \ifx\DDV@sqdash\D@FTref general style\else\DDV@sqdash\fi
     \space-> \BS@ figset altitude(sqdash=Index/Pattern)}%
    \immediate\write16{ Square line width =
     \ifx\DDV@sqwidth\D@FTref general width\else\DDV@sqwidth\fi
     \space-> \BS@ figset altitude(sqwidth=real in PostScript units)}%
    \immediate\write16{Arrowhead:}%
    \immediate\write16{ (half-)Angle = \@rrowheadangle
     \space-> \BS@ figset arrowhead(angle=real in degrees)}%
    \immediate\write16{ Filling mode = \if@rrowhfill yes\else no\fi
     \space-> \BS@ figset arrowhead(fillmode=yes/no)}%
    \immediate\write16{ "Outside" = \if@rrowhout yes\else no\fi
     \space-> \BS@ figset arrowhead(out=yes/no)}%
    \immediate\write16{ Length = \@rrowheadlength
     \if@rrowratio\space(not active)\else\space(active)\fi
     \space-> \BS@ figset arrowhead(length=real in user coord.)}%
    \immediate\write16{ Ratio = \@rrowheadratio
     \if@rrowratio\space(active)\else\space(not active)\fi
     \space-> \BS@ figset arrowhead(ratio=real in [0,1])}%
    \immediate\write16{Curve:}%
    \immediate\write16{ Roundness = \curv@roundness
     \space-> \BS@ figset curve(roundness=real in [0,0.5])}%
    \immediate\write16{Flow chart:}%
    \immediate\write16{ Arrow position = \@rrowp@s
     \space-> \BS@ figset flowchart(arrowposition=real in [0,1])}%
    \immediate\write16{ Arrow reference point = \ifcase\@rrowr@fpt start\else end\fi
     \space-> \BS@ figset flowchart(arrowrefpt = start/end)}%
    \immediate\write16{ Background color = \fcbgc@lor
     \space-> \BS@ figset flowchart(bgcolor=ColorDefinition)}%
    \immediate\write16{ Line type = \ifcase\fclin@typ@ curve\else polygon\fi
     \space-> \BS@ figset flowchart(line=polygon/curve)}%
    \immediate\write16{ Padding = (\Xp@dd, \Yp@dd)
     \space-> \BS@ figset flowchart(padding = real in user coord.)}%
    \immediate\write16{\space\space\space\space(or
     \BS@ figset flowchart(xpadding=real, ypadding=real) )}%
    \immediate\write16{ Radius = \fclin@r@d
     \space-> \BS@ figset flowchart(radius=positive real in user coord.)}%
    \immediate\write16{ Shape = \fcsh@pe
     \space-> \BS@ figset flowchart(shape = rectangle, ellipse or lozenge)}%
    \immediate\write16{ Thickness color (DDV) = 
     \ifx\DDV@thickcolor\D@FTref general color\else\DDV@thickcolor\fi
     \space-> \BS@ figset flowchart(thickcolor=ColorDefinition)}%
    \immediate\write16{ Thickness = \thickn@ss
     \space-> \BS@ figset flowchart(thickness = real in user coord.)}%
    \immediate\write16{Mesh:}%
    \immediate\write16{ Diagonal = \c@ntrolmesh
     \space-> \BS@ figset mesh(diag=integer in {-1,0,1})}%
    \immediate\write16{ Lines color (DDV) =
     \ifx\DDV@meshcolor\D@FTref general color\else\DDV@meshcolor\fi
     \space-> \BS@ figset mesh(color=ColorDefinition)}%
    \immediate\write16{ Lines style (DDV) =
     \ifx\DDV@meshdash\D@FTref general style\else\DDV@meshdash\fi
     \space-> \BS@ figset mesh(dash=Index/Pattern)}%
    \immediate\write16{ Lines width (DDV) =
     \ifx\DDV@meshwidth\D@FTref general width\else\DDV@meshwidth\fi
     \space-> \BS@ figset mesh(width=real in PostScript units)}%
    \immediate\write16{Trimesh:}%
    \immediate\write16{ Lines color (DDV) =
     \ifx\DDV@tmeshcolor\D@FTref general color\else\DDV@tmeshcolor\fi
     \space-> \BS@ figset trimesh(color=ColorDefinition)}%
    \immediate\write16{ Lines style (DDV) =
     \ifx\DDV@tmeshdash\D@FTref general style\else\DDV@tmeshdash\fi
     \space-> \BS@ figset trimesh(dash=Index/Pattern)}%
    \immediate\write16{ Lines width (DDV) =
     \ifx\DDV@tmeshwidth\D@FTref general width\else\DDV@tmeshwidth\fi
     \space-> \BS@ figset trimesh(width=real in PostScript units)}%
    \ifTr@isDim%
    \immediate\write16{ --- 3D to 2D PROJECTION ---}%
    \immediate\write16{Projection : \typ@proj \space-> \BS@ figinit{ScaleFactorUnit, ProjType}}%
    \immediate\write16{Longitude (psi) = \v@lPsi \space-> \BS@ figset proj(psi=real in degrees)}%
    \ifcase\CUR@proj\immediate\write16{Depth coeff. (Lambda)
     \space = \v@lTheta \space-> \BS@ figset proj(lambda=real in [0,1])}%
    \else\immediate\write16{Latitude (theta)
     \space = \v@lTheta \space-> \BS@ figset proj(theta=real in degrees)}%
    \fi%
    \ifnum\CUR@proj=\tw@%
    \immediate\write16{Observation distance = \disob@unit
     \space-> \BS@ figset proj(dist=real in user coord.)}%
    \immediate\write16{Target point = \t@rgetpt \space-> \BS@ figset proj(targetpt=pt number)}%
     \v@lX=\ptT@unit@\wd\Bt@rget\v@lY=\ptT@unit@\ht\Bt@rget\v@lZ=\ptT@unit@\dp\Bt@rget%
    \immediate\write16{ Its coordinates are
     (\repdecn@mb{\v@lX}, \repdecn@mb{\v@lY}, \repdecn@mb{\v@lZ})}%
    \fi%
    \fi%
    \immediate\write16{====================================================================}%
    \ignorespaces}}
\ctr@ln@w{newif}\ifitis@vect@r
\ctr@ld@f\def\figvectC#1(#2,#3){{\itis@vect@rtrue\figpt#1:(#2,#3)}\ignorespaces}
\ctr@ld@f\def\Figv@ctCreg#1(#2,#3){{\itis@vect@rtrue\Figp@intreg#1:(#2,#3)}\ignorespaces}
\ctr@ln@m\figvectDBezier
\ctr@ld@f\def\figvectDBezierDD#1:#2,#3[#4,#5,#6,#7]{\ifGR@cri{\s@uvc@ntr@l\et@tfigvectDBezierDD%
    \FigvectDBezier@#2,#3[#4,#5,#6,#7]\v@lX=\c@ef\v@lX\v@lY=\c@ef\v@lY%
    \Figv@ctCreg#1(\v@lX,\v@lY)\resetc@ntr@l\et@tfigvectDBezierDD}\ignorespaces\fi}
\ctr@ld@f\def\figvectDBezierTD#1:#2,#3[#4,#5,#6,#7]{\ifGR@cri{\s@uvc@ntr@l\et@tfigvectDBezierTD%
    \FigvectDBezier@#2,#3[#4,#5,#6,#7]\v@lX=\c@ef\v@lX\v@lY=\c@ef\v@lY\v@lZ=\c@ef\v@lZ%
    \Figv@ctCreg#1(\v@lX,\v@lY,\v@lZ)\resetc@ntr@l\et@tfigvectDBezierTD}\ignorespaces\fi}
\ctr@ld@f\def\FigvectDBezier@#1,#2[#3,#4,#5,#6]{\setc@ntr@l{2}%
    \edef\T@{#2}\v@leur=\p@\advance\v@leur-#2pt\edef\UNmT@{\repdecn@mb{\v@leur}}%
    \ifnum#1=\tw@\def\c@ef{6}\else\def\c@ef{3}\fi%
    \figptcopy-4:/#3/\figptcopy-3:/#4/\figptcopy-2:/#5/\figptcopy-1:/#6/%
    \l@mbd@un=-4 \l@mbd@de=-\thr@@\p@rtent=\m@ne\c@lDecast%
    \ifnum#1=\tw@\c@lDCDeux{-4}{-3}\c@lDCDeux{-3}{-2}\c@lDCDeux{-4}{-3}\else%
    \l@mbd@un=-4 \l@mbd@de=-\thr@@\p@rtent=-\tw@\c@lDecast%
    \c@lDCDeux{-4}{-3}\fi\Figg@tXY{-4}}
\ctr@ln@m\c@lDCDeux
\ctr@ld@f\def\c@lDCDeuxDD#1#2{\Figg@tXY{#2}\Figg@tXYa{#1}%
    \advance\v@lX-\v@lXa\advance\v@lY-\v@lYa\Figp@intregDD#1:(\v@lX,\v@lY)}
\ctr@ld@f\def\c@lDCDeuxTD#1#2{\Figg@tXY{#2}\Figg@tXYa{#1}\advance\v@lX-\v@lXa%
    \advance\v@lY-\v@lYa\advance\v@lZ-\v@lZa\Figp@intregTD#1:(\v@lX,\v@lY,\v@lZ)}
\ctr@ln@m\figvectN
\ctr@ld@f\def\figvectNDD#1[#2,#3]{\ifGR@cri{\Figg@tXYa{#2}\Figg@tXY{#3}%
    \advance\v@lX-\v@lXa\advance\v@lY-\v@lYa%
    \Figv@ctCreg#1(-\v@lY,\v@lX)}\ignorespaces\fi}
\ctr@ld@f\def\figvectNTD#1[#2,#3,#4]{\ifGR@cri{\vecunitC@TD[#2,#4]\v@lmin=\v@lX\v@lmax=\v@lY%
    \v@leur=\v@lZ\vecunitC@TD[#2,#3]\c@lprovec{#1}}\ignorespaces\fi}
\ctr@ln@m\figvectNV
\ctr@ld@f\def\figvectNVDD#1[#2]{\ifGR@cri{\Figg@tXY{#2}\Figv@ctCreg#1(-\v@lY,\v@lX)}\ignorespaces\fi}
\ctr@ld@f\def\figvectNVTD#1[#2,#3]{\ifGR@cri{\vecunitCV@TD{#3}\v@lmin=\v@lX\v@lmax=\v@lY%
    \v@leur=\v@lZ\vecunitCV@TD{#2}\c@lprovec{#1}}\ignorespaces\fi}
\ctr@ln@m\figvectP
\ctr@ld@f\def\figvectPDD#1[#2,#3]{\ifGR@cri{\Figg@tXYa{#2}\Figg@tXY{#3}%
    \advance\v@lX-\v@lXa\advance\v@lY-\v@lYa%
    \Figv@ctCreg#1(\v@lX,\v@lY)}\ignorespaces\fi}
\ctr@ld@f\def\figvectPTD#1[#2,#3]{\ifGR@cri{\Figg@tXYa{#2}\Figg@tXY{#3}%
    \advance\v@lX-\v@lXa\advance\v@lY-\v@lYa\advance\v@lZ-\v@lZa%
    \Figv@ctCreg#1(\v@lX,\v@lY,\v@lZ)}\ignorespaces\fi}
\ctr@ln@m\figvectU
\ctr@ld@f\def\figvectUDD#1[#2]{\ifGR@cri{\n@rmeuc\v@leur{#2}\invers@\v@leur\v@leur%
    \delt@=\repdecn@mb{\v@leur}\unit@\edef\v@ldelt@{\repdecn@mb{\delt@}}%
    \Figg@tXY{#2}\v@lX=\v@ldelt@\v@lX\v@lY=\v@ldelt@\v@lY%
    \Figv@ctCreg#1(\v@lX,\v@lY)}\ignorespaces\fi}
\ctr@ld@f\def\figvectUTD#1[#2]{\ifGR@cri{\n@rmeuc\v@leur{#2}\invers@\v@leur\v@leur%
    \delt@=\repdecn@mb{\v@leur}\unit@\edef\v@ldelt@{\repdecn@mb{\delt@}}%
    \Figg@tXY{#2}\v@lX=\v@ldelt@\v@lX\v@lY=\v@ldelt@\v@lY\v@lZ=\v@ldelt@\v@lZ%
    \Figv@ctCreg#1(\v@lX,\v@lY,\v@lZ)}\ignorespaces\fi}
\ctr@ld@f\def\figvisu#1#2#3{\c@ldefproj\initb@undb@x\xdef\figforTeXFigno{\figforTeXnextFigno}%
    \s@mme=\figforTeXnextFigno\advance\s@mme\@ne\xdef\figforTeXnextFigno{\number\s@mme}%
    \setbox\b@xvisu=\hbox{\ifnum\@utoFN>\z@\figinsert{}\gdef\@utoFInDone{0}\fi\ignorespaces#3}%
    \gdef\@utoFInDone{1}\gdef\@utoFN{0}%
    \v@lXa=-\c@@rdYmin\v@lYa=\c@@rdYmax\advance\v@lYa-\c@@rdYmin%
    \v@lX=\c@@rdXmax\advance\v@lX-\c@@rdXmin%
    \setbox#1=\hbox{#2}\v@lY=-\v@lX\maxim@m{\v@lX}{\v@lX}{\wd#1}%
    \advance\v@lY\v@lX\divide\v@lY\tw@\advance\v@lY-\c@@rdXmin%
    \setbox#1=\vbox{\parindent\z@\hsize=\v@lX\vskip\v@lYa%
    \rlap{\hskip\v@lY\smash{\raise\v@lXa\box\b@xvisu}}%
    \def\t@xt@{#2}\ifx\t@xt@\empty\else\medskip\centerline{#2}\fi}\wd#1=\v@lX}
\ctr@ld@f\def\figDecrementFigno{{\xdef\figforTeXnextFigno{\figforTeXFigno}%
    \s@mme=\figforTeXFigno\advance\s@mme\m@ne\xdef\figforTeXFigno{\number\s@mme}}}
\ctr@ln@w{newbox}\Bt@rget\setbox\Bt@rget=\null
\ctr@ln@w{newbox}\BminTD@\setbox\BminTD@=\null
\ctr@ln@w{newbox}\BmaxTD@\setbox\BmaxTD@=\null
\ctr@ln@w{newif}\ifnewt@rgetpt\ctr@ln@w{newif}\ifnewdis@b
\ctr@ld@f\def\b@undb@xTD#1#2#3{%
    \relax\ifdim#1<\wd\BminTD@\global\wd\BminTD@=#1\fi%
    \relax\ifdim#2<\ht\BminTD@\global\ht\BminTD@=#2\fi%
    \relax\ifdim#3<\dp\BminTD@\global\dp\BminTD@=#3\fi%
    \relax\ifdim#1>\wd\BmaxTD@\global\wd\BmaxTD@=#1\fi%
    \relax\ifdim#2>\ht\BmaxTD@\global\ht\BmaxTD@=#2\fi%
    \relax\ifdim#3>\dp\BmaxTD@\global\dp\BmaxTD@=#3\fi}
\ctr@ld@f\def\c@ldefdisob{{\ifdim\wd\BminTD@<\maxdimen\v@leur=\wd\BmaxTD@\advance\v@leur-\wd\BminTD@%
    \delt@=\ht\BmaxTD@\advance\delt@-\ht\BminTD@\maxim@m{\v@leur}{\v@leur}{\delt@}%
    \delt@=\dp\BmaxTD@\advance\delt@-\dp\BminTD@\maxim@m{\v@leur}{\v@leur}{\delt@}%
    \v@leur=5\v@leur\else\v@leur=800pt\fi\c@ldefdisob@{\v@leur}}}
\ctr@ln@m\disob@intern
\ctr@ln@m\disob@
\ctr@ln@m\divf@ctproj
\ctr@ld@f\def\c@ldefdisob@#1{{\v@leur=#1\ifdim\v@leur<\p@\v@leur=800pt\fi%
    \xdef\disob@intern{\repdecn@mb{\v@leur}}%
    \delt@=\ptT@unit@\v@leur\xdef\disob@unit{\repdecn@mb{\delt@}}%
    \f@ctech=\@ne\loop\ifdim\v@leur>\t@n pt\divide\v@leur\t@n\multiply\f@ctech\t@n\repeat%
    \xdef\disob@{\repdecn@mb{\v@leur}}\xdef\divf@ctproj{\the\f@ctech}}%
    \global\newdis@btrue}
\ctr@ln@m\t@rgetpt
\ctr@ld@f\def\c@ldeft@rgetpt{\newt@rgetpttrue\def\t@rgetpt{CenterBoundBox}{%
    \delt@=\wd\BmaxTD@\advance\delt@-\wd\BminTD@\divide\delt@\tw@%
    \v@leur=\wd\BminTD@\advance\v@leur\delt@\global\wd\Bt@rget=\v@leur%
    \delt@=\ht\BmaxTD@\advance\delt@-\ht\BminTD@\divide\delt@\tw@%
    \v@leur=\ht\BminTD@\advance\v@leur\delt@\global\ht\Bt@rget=\v@leur%
    \delt@=\dp\BmaxTD@\advance\delt@-\dp\BminTD@\divide\delt@\tw@%
    \v@leur=\dp\BminTD@\advance\v@leur\delt@\global\dp\Bt@rget=\v@leur}}
\ctr@ln@m\c@ldefproj
\ctr@ld@f\def\c@ldefprojTD{\ifnewt@rgetpt\else\c@ldeft@rgetpt\fi\ifnewdis@b\else\c@ldefdisob\fi}
\ctr@ld@f\def\c@lprojcav{
    \v@lZa=\cxa@\v@lY\advance\v@lX\v@lZa%
    \v@lZa=\cxb@\v@lY\v@lY=\v@lZ\advance\v@lY\v@lZa\ignorespaces}
\ctr@ln@m\v@lcoef
\ctr@ld@f\def\c@lprojrea{
    \advance\v@lX-\wd\Bt@rget\advance\v@lY-\ht\Bt@rget\advance\v@lZ-\dp\Bt@rget%
    \v@lZa=\cza@\v@lX\advance\v@lZa\czb@\v@lY\advance\v@lZa\czc@\v@lZ%
    \divide\v@lZa\divf@ctproj\advance\v@lZa\disob@ pt\invers@{\v@lZa}{\v@lZa}%
    \v@lZa=\disob@\v@lZa\edef\v@lcoef{\repdecn@mb{\v@lZa}}%
    \v@lXa=\cxa@\v@lX\advance\v@lXa\cxb@\v@lY\v@lXa=\v@lcoef\v@lXa%
    \v@lY=\cyb@\v@lY\advance\v@lY\cya@\v@lX\advance\v@lY\cyc@\v@lZ%
    \v@lY=\v@lcoef\v@lY\v@lX=\v@lXa\ignorespaces}
\ctr@ld@f\def\c@lprojort{
    \v@lXa=\cxa@\v@lX\advance\v@lXa\cxb@\v@lY%
    \v@lY=\cyb@\v@lY\advance\v@lY\cya@\v@lX\advance\v@lY\cyc@\v@lZ%
    \v@lX=\v@lXa\ignorespaces}
\ctr@ld@f\def\Figptpr@j#1:#2/#3/{{\Figg@tXY{#3}\superc@lprojSP%
    \Figp@intregDD#1:{#2}(\v@lX,\v@lY)}\ignorespaces}
\ctr@ln@m\figsetobdist
\ctr@ld@f\def\figsetobdistDD{\un@v@ilable{figsetobdist}}
\ctr@ld@f\def\figsetobdistTD(#1){{\ifCUR@PS\W@rnmesIgn{figset proj(dist=...)}%
    \else\v@leur=#1\unit@\c@ldefdisob@{\v@leur}\fi}\ignorespaces}
\ctr@ln@m\c@lprojSP
\ctr@ln@m\CUR@proj
\ctr@ln@m\typ@proj
\ctr@ln@m\superc@lprojSP
\ctr@ld@f\def\Figs@tproj#1{%
    \if#13 \def@ultproj\else\if#1c\def@ultproj%
    \else\if#1o\xdef\CUR@proj{1}\xdef\typ@proj{orthogonal}%
         \figsetviewTD(\def@ultpsi,\def@ulttheta)%
         \global\let\c@lprojSP=\c@lprojort\global\let\superc@lprojSP=\c@lprojort%
    \else\if#1r\xdef\CUR@proj{2}\xdef\typ@proj{realistic}%
         \figsetviewTD(\def@ultpsi,\def@ulttheta)%
         \global\let\c@lprojSP=\c@lprojrea\global\let\superc@lprojSP=\c@lprojrea%
    \else\def@ultproj\message{*** Unknown projection. Cavalier projection assumed.}%
    \fi\fi\fi\fi}
\ctr@ld@f\def\def@ultproj{\xdef\CUR@proj{0}\xdef\typ@proj{cavalier}\figsetviewTD(\def@ultpsi,0.5)%
         \global\let\c@lprojSP=\c@lprojcav\global\let\superc@lprojSP=\c@lprojcav}
\ctr@ln@m\figsettarget
\ctr@ld@f\def\figsettargetDD{\un@v@ilable{figsettarget}}
\ctr@ld@f\def\figsettargetTD[#1]{{\ifCUR@PS\W@rnmesIgn{figset proj(targetpt=...)}%
    \else\global\newt@rgetpttrue\xdef\t@rgetpt{#1}\Figg@tXY{#1}\global\wd\Bt@rget=\v@lX%
    \global\ht\Bt@rget=\v@lY\global\dp\Bt@rget=\v@lZ\fi}\ignorespaces}
\ctr@ln@m\figsetview
\ctr@ld@f\def\figsetviewDD{\un@v@ilable{figsetview}}
\ctr@ld@f\def\figsetviewTD(#1){\ifCUR@PS\W@rnmesIgn{figset proj(Psi|Theta|Lambda=...)}%
     \else\Figsetview@#1,:\fi\ignorespaces}
\ctr@ld@f\def\Figsetview@#1,#2:{{\xdef\v@lPsi{#1}\def\t@xt@{#2}%
    \ifx\t@xt@\empty\def\@rgdeux{\v@lTheta}\else\X@rgdeux@#2\fi%
    \c@ssin{\costhet@}{\sinthet@}{#1}\v@lmin=\costhet@ pt\v@lmax=\sinthet@ pt%
    \ifcase\CUR@proj%
    \v@leur=\@rgdeux\v@lmin\xdef\cxa@{\repdecn@mb{\v@leur}}%
    \v@leur=\@rgdeux\v@lmax\xdef\cxb@{\repdecn@mb{\v@leur}}\v@leur=\@rgdeux pt%
    \relax\ifdim\v@leur>\p@\message{*** Lambda too large ! See \BS@ figset proj() !}\fi%
    \else%
    \v@lmax=-\v@lmax\xdef\cxa@{\repdecn@mb{\v@lmax}}\xdef\cxb@{\costhet@}%
    \ifx\t@xt@\empty\edef\@rgdeux{\def@ulttheta}\fi\c@ssin{\C@}{\S@}{\@rgdeux}%
    \v@lmax=-\S@ pt%
    \v@leur=\v@lmax\v@leur=\costhet@\v@leur\xdef\cya@{\repdecn@mb{\v@leur}}%
    \v@leur=\v@lmax\v@leur=\sinthet@\v@leur\xdef\cyb@{\repdecn@mb{\v@leur}}%
    \xdef\cyc@{\C@}\v@lmin=-\C@ pt%
    \v@leur=\v@lmin\v@leur=\costhet@\v@leur\xdef\cza@{\repdecn@mb{\v@leur}}%
    \v@leur=\v@lmin\v@leur=\sinthet@\v@leur\xdef\czb@{\repdecn@mb{\v@leur}}%
    \xdef\czc@{\repdecn@mb{\v@lmax}}\fi%
    \xdef\v@lTheta{\@rgdeux}}}
\ctr@ld@f\def\def@ultpsi{40}
\ctr@ld@f\def\def@ulttheta{25}
\ctr@ln@m\l@debut
\ctr@ln@m\n@mref
\ctr@ld@f\def\Figsetpr@j#1=#2|{\keln@mtr#1|%
    \def\n@mref{dep}\ifx\l@debut\n@mref\Figsetd@p{#2}\else
    \def\n@mref{dis}\ifx\l@debut\n@mref%
     \ifnum\CUR@proj=\tw@\figsetobdist(#2)\else\Figset@rr\fi\else
    \def\n@mref{lam}\ifx\l@debut\n@mref\Figsetd@p{#2}\else
    \def\n@mref{lat}\ifx\l@debut\n@mref\Figsetth@{#2}\else
    \def\n@mref{lon}\ifx\l@debut\n@mref\figsetview(#2)\else
    \def\n@mref{psi}\ifx\l@debut\n@mref\figsetview(#2)\else
    \def\n@mref{tar}\ifx\l@debut\n@mref%
     \ifnum\CUR@proj=\tw@\figsettarget[#2]\else\Figset@rr\fi\else
    \def\n@mref{the}\ifx\l@debut\n@mref\Figsetth@{#2}\else
    \W@rnmesAttr{figset proj}{#1}\fi\fi\fi\fi\fi\fi\fi\fi}
\ctr@ld@f\def\Figsetd@p#1{\ifnum\CUR@proj=\z@\figsetview(\v@lPsi,#1)\else\Figset@rr\fi}
\ctr@ld@f\def\Figsetth@#1{\ifnum\CUR@proj=\z@\Figset@rr\else\figsetview(\v@lPsi,#1)\fi}
\ctr@ld@f\def\Figset@rr{\message{*** \BS@ figset proj(): Attribute "\n@mref" ignored, incompatible
    with current projection}}
\ctr@ld@f\def\initb@undb@xTD{\wd\BminTD@=\maxdimen\ht\BminTD@=\maxdimen\dp\BminTD@=\maxdimen%
    \wd\BmaxTD@=-\maxdimen\ht\BmaxTD@=-\maxdimen\dp\BmaxTD@=-\maxdimen}
\ctr@ln@w{newbox}\Gb@x      
\ctr@ln@w{newbox}\Gb@xSC    
\ctr@ln@w{newtoks}\c@nsymb  
\ctr@ln@w{newif}\ifr@undcoord\ctr@ln@w{newif}\ifunitpr@sent
\ctr@ld@f\def\unssqrttw@{0.707106 }
\ctr@ld@f\def\figAst{\raise-1.15ex\hbox{$\ast$}}
\ctr@ld@f\def\figBullet{\raise-1.15ex\hbox{$\bullet$}}
\ctr@ld@f\def\figCirc{\raise-1.15ex\hbox{$\circ$}}
\ctr@ld@f\def\figDiamond{\raise-1.15ex\hbox{$\diamond$}}%
\ctr@ld@f\def\boxit#1#2{\leavevmode\hbox{\vrule\vbox{\hrule\vglue#1%
    \vtop{\hbox{\kern#1{#2}\kern#1}\vglue#1\hrule}}\vrule}}
\ctr@ld@f
\ctr@ld@f
\ctr@ld@f\def\c@nterpt{\ignorespaces%
    \kern-.5\wd\Gb@xSC%
    \raise-.5\ht\Gb@xSC\rlap{\hbox{\raise.5\dp\Gb@xSC\hbox{\copy\Gb@xSC}}}%
    \kern .5\wd\Gb@xSC\ignorespaces}
\ctr@ld@f\def\b@undb@xSC#1#2{{\v@lXa=#1\v@lYa=#2%
    \v@leur=\ht\Gb@xSC\advance\v@leur\dp\Gb@xSC%
    \advance\v@lXa-.5\wd\Gb@xSC\advance\v@lYa-.5\v@leur\b@undb@x{\v@lXa}{\v@lYa}%
    \advance\v@lXa\wd\Gb@xSC\advance\v@lYa\v@leur\b@undb@x{\v@lXa}{\v@lYa}}}
\ctr@ln@m\Dist@n
\ctr@ln@m\l@suite
\ctr@ld@f\def\@keldist#1#2{\edef\Dist@n{#2}\y@tiunit{\Dist@n}%
    \ifunitpr@sent#1=\Dist@n\else#1=\Dist@n\unit@\fi}
\ctr@ld@f\def\y@tiunit#1{\unitpr@sentfalse\expandafter\y@tiunit@#1:}
\ctr@ld@f\def\y@tiunit@#1#2:{\ifcat#1a\unitpr@senttrue\else\def\l@suite{#2}%
    \ifx\l@suite\empty\else\y@tiunit@#2:\fi\fi}
\ctr@ln@m\figcoord
\ctr@ld@f\def\figcoordDD#1{{\v@lX=\ptT@unit@\v@lX\v@lY=\ptT@unit@\v@lY%
    \ifr@undcoord\ifcase#1\v@leur=0.5pt\or\v@leur=0.05pt\or\v@leur=0.005pt%
    \or\v@leur=0.0005pt\else\v@leur=\z@\fi%
    \ifdim\v@lX<\z@\advance\v@lX-\v@leur\else\advance\v@lX\v@leur\fi%
    \ifdim\v@lY<\z@\advance\v@lY-\v@leur\else\advance\v@lY\v@leur\fi\fi%
    (\@ffichnb{#1}{\repdecn@mb{\v@lX}},\ifmmode\else\thinspace\fi%
    \@ffichnb{#1}{\repdecn@mb{\v@lY}})}}
\ctr@ld@f\def\@ffichnb#1#2{{\def\@@ffich{\@ffich#1(}\edef\n@mbre{#2}%
    \expandafter\@@ffich\n@mbre)}}
\ctr@ld@f\def\@ffich#1(#2.#3){{#2\ifnum#1>\z@.\fi\def\dig@ts{#3}\s@mme=\z@%
    \loop\ifnum\s@mme<#1\expandafter\@ffichdec\dig@ts:\advance\s@mme\@ne\repeat}}
\ctr@ld@f\def\@ffichdec#1#2:{\relax#1\def\dig@ts{#20}}
\ctr@ld@f\def\figcoordTD#1{{\v@lX=\ptT@unit@\v@lX\v@lY=\ptT@unit@\v@lY\v@lZ=\ptT@unit@\v@lZ%
    \ifr@undcoord\ifcase#1\v@leur=0.5pt\or\v@leur=0.05pt\or\v@leur=0.005pt%
    \or\v@leur=0.0005pt\else\v@leur=\z@\fi%
    \ifdim\v@lX<\z@\advance\v@lX-\v@leur\else\advance\v@lX\v@leur\fi%
    \ifdim\v@lY<\z@\advance\v@lY-\v@leur\else\advance\v@lY\v@leur\fi%
    \ifdim\v@lZ<\z@\advance\v@lZ-\v@leur\else\advance\v@lZ\v@leur\fi\fi%
    (\@ffichnb{#1}{\repdecn@mb{\v@lX}},\ifmmode\else\thinspace\fi%
     \@ffichnb{#1}{\repdecn@mb{\v@lY}},\ifmmode\else\thinspace\fi%
     \@ffichnb{#1}{\repdecn@mb{\v@lZ}})}}
\ctr@ld@f\def\figsetroundcoord#1{\expandafter\Figsetr@undcoord#1:\ignorespaces}
\ctr@ld@f\def\Figsetr@undcoord#1#2:{\if#1n\r@undcoordfalse\else\r@undcoordtrue\fi}
\ctr@ld@f\def\Figsetwr@te#1=#2|{\keln@mun#1|%
    \def\n@mref{m}\ifx\l@debut\n@mref\figsetmark{#2}\else
    \def\n@mref{p}\ifx\l@debut\n@mref\figsetptname{#2}\else
    \def\n@mref{r}\ifx\l@debut\n@mref\figsetroundcoord{#2}\else
    \W@rnmesAttr{figset write}{#1}\fi\fi\fi}
\ctr@ld@f\def\figsetmark#1{\c@nsymb={#1}\setbox\Gb@xSC=\hbox{\the\c@nsymb}\ignorespaces}
\ctr@ln@m\ptn@me
\ctr@ld@f\def\figsetptname#1{\def\ptn@me##1{#1}\ignorespaces}
\ctr@ld@f\def\FigWrit@L#1:#2(#3,#4){\ignorespaces\@keldist\v@leur{#3}\@keldist\delt@{#4}%
    \C@rp@r@m\def\list@num{#1}\@ecfor\p@int:=\list@num\do{\FigWrit@pt{\p@int}{#2}}}
\ctr@ld@f\def\FigWrit@pt#1#2{\FigWp@r@m{#1}{#2}\Vc@rrect\figWp@si%
    \ifdim\wd\Gb@xSC>\z@\b@undb@xSC{\v@lX}{\v@lY}\fi\figWBB@x}
\ctr@ld@f\def\FigWp@r@m#1#2{\Figg@tXY{#1}%
    \setbox\Gb@x=\hbox{\def\t@xt@{#2}\ifx\t@xt@\empty\Figg@tT{#1}\else#2\fi}\c@lprojSP}
\ctr@ld@f\let\Vc@rrect=\relax
\ctr@ld@f\let\C@rp@r@m=\relax
\ctr@ld@f\def\figwrite[#1]#2{{\ignorespaces\def\list@num{#1}\@ecfor\p@int:=\list@num\do{%
    \setbox\Gb@x=\hbox{\def\t@xt@{#2}\ifx\t@xt@\empty\Figg@tT{\p@int}\else#2\fi}%
    \Figwrit@{\p@int}}}\ignorespaces}
\ctr@ld@f\def\Figwrit@#1{\Figg@tXY{#1}\c@lprojSP%
    \rlap{\kern\v@lX\raise\v@lY\hbox{\unhcopy\Gb@x}}\v@leur=\v@lY%
    \advance\v@lY\ht\Gb@x\b@undb@x{\v@lX}{\v@lY}\advance\v@lX\wd\Gb@x%
    \v@lY=\v@leur\advance\v@lY-\dp\Gb@x\b@undb@x{\v@lX}{\v@lY}}
\ctr@ld@f\def\figwritec[#1]#2{{\ignorespaces\def\list@num{#1}%
    \@ecfor\p@int:=\list@num\do{\Figwrit@c{\p@int}{#2}}}\ignorespaces}
\ctr@ld@f\def\Figwrit@c#1#2{\FigWp@r@m{#1}{#2}%
    \rlap{\kern\v@lX\raise\v@lY\hbox{\rlap{\kern-.5\wd\Gb@x%
    \raise-.5\ht\Gb@x\hbox{\raise.5\dp\Gb@x\hbox{\unhcopy\Gb@x}}}}}%
    \v@leur=\ht\Gb@x\advance\v@leur\dp\Gb@x%
    \advance\v@lX-.5\wd\Gb@x\advance\v@lY-.5\v@leur\b@undb@x{\v@lX}{\v@lY}%
    \advance\v@lX\wd\Gb@x\advance\v@lY\v@leur\b@undb@x{\v@lX}{\v@lY}}
\ctr@ld@f\def\figwritep[#1]{{\ignorespaces\def\list@num{#1}\setbox\Gb@x=\hbox{\c@nterpt}%
    \@ecfor\p@int:=\list@num\do{\Figwrit@{\p@int}}}\ignorespaces}
\ctr@ld@f\def\figwritew#1:#2(#3){\figwritegcw#1:{#2}(#3,0pt)}
\ctr@ld@f\def\figwritee#1:#2(#3){\figwritegce#1:{#2}(#3,0pt)}
\ctr@ld@f\def\figwriten#1:#2(#3){{\def\Vc@rrect{\v@lZ=\v@leur\advance\v@lZ\dp\Gb@x}%
    \Figwrit@NS#1:{#2}(#3)}\ignorespaces}
\ctr@ld@f\def\figwrites#1:#2(#3){{\def\Vc@rrect{\v@lZ=-\v@leur\advance\v@lZ-\ht\Gb@x}%
    \Figwrit@NS#1:{#2}(#3)}\ignorespaces}
\ctr@ld@f\def\Figwrit@NS#1:#2(#3){\let\figWp@si=\FigWp@siNS\let\figWBB@x=\FigWBB@xNS%
    \FigWrit@L#1:{#2}(#3,0pt)}
\ctr@ld@f\def\FigWp@siNS{\rlap{\kern\v@lX\raise\v@lY\hbox{\rlap{\kern-.5\wd\Gb@x%
    \raise\v@lZ\hbox{\unhcopy\Gb@x}}\c@nterpt}}}
\ctr@ld@f\def\FigWBB@xNS{\advance\v@lY\v@lZ%
    \advance\v@lY-\dp\Gb@x\advance\v@lX-.5\wd\Gb@x\b@undb@x{\v@lX}{\v@lY}%
    \advance\v@lY\ht\Gb@x\advance\v@lY\dp\Gb@x%
    \advance\v@lX\wd\Gb@x\b@undb@x{\v@lX}{\v@lY}}
\ctr@ld@f\def\figwritenw#1:#2(#3){{\let\figWp@si=\FigWp@sigW\let\figWBB@x=\FigWBB@xgWE%
    \def\C@rp@r@m{\v@leur=\unssqrttw@\v@leur\delt@=\v@leur%
    \ifdim\delt@=\z@\delt@=\epsil@n\fi}\let@xte={-}\FigWrit@L#1:{#2}(#3,0pt)}\ignorespaces}
\ctr@ld@f\def\figwritesw#1:#2(#3){{\let\figWp@si=\FigWp@sigW\let\figWBB@x=\FigWBB@xgWE%
    \def\C@rp@r@m{\v@leur=\unssqrttw@\v@leur\delt@=-\v@leur%
    \ifdim\delt@=\z@\delt@=-\epsil@n\fi}\let@xte={-}\FigWrit@L#1:{#2}(#3,0pt)}\ignorespaces}
\ctr@ld@f\def\figwritene#1:#2(#3){{\let\figWp@si=\FigWp@sigE\let\figWBB@x=\FigWBB@xgWE%
    \def\C@rp@r@m{\v@leur=\unssqrttw@\v@leur\delt@=\v@leur%
    \ifdim\delt@=\z@\delt@=\epsil@n\fi}\let@xte={}\FigWrit@L#1:{#2}(#3,0pt)}\ignorespaces}
\ctr@ld@f\def\figwritese#1:#2(#3){{\let\figWp@si=\FigWp@sigE\let\figWBB@x=\FigWBB@xgWE%
    \def\C@rp@r@m{\v@leur=\unssqrttw@\v@leur\delt@=-\v@leur%
    \ifdim\delt@=\z@\delt@=-\epsil@n\fi}\let@xte={}\FigWrit@L#1:{#2}(#3,0pt)}\ignorespaces}
\ctr@ld@f\def\figwritegw#1:#2(#3,#4){{\let\figWp@si=\FigWp@sigW\let\figWBB@x=\FigWBB@xgWE%
    \let@xte={-}\FigWrit@L#1:{#2}(#3,#4)}\ignorespaces}
\ctr@ld@f\def\figwritege#1:#2(#3,#4){{\let\figWp@si=\FigWp@sigE\let\figWBB@x=\FigWBB@xgWE%
    \let@xte={}\FigWrit@L#1:{#2}(#3,#4)}\ignorespaces}
\ctr@ld@f\def\FigWp@sigW{\v@lXa=\z@\v@lYa=\ht\Gb@x\advance\v@lYa\dp\Gb@x%
    \ifdim\delt@>\z@\relax%
    \rlap{\kern\v@lX\raise\v@lY\hbox{\rlap{\kern-\wd\Gb@x\kern-\v@leur%
          \raise\delt@\hbox{\raise\dp\Gb@x\hbox{\unhcopy\Gb@x}}}\c@nterpt}}%
    \else\ifdim\delt@<\z@\relax\v@lYa=-\v@lYa%
    \rlap{\kern\v@lX\raise\v@lY\hbox{\rlap{\kern-\wd\Gb@x\kern-\v@leur%
          \raise\delt@\hbox{\raise-\ht\Gb@x\hbox{\unhcopy\Gb@x}}}\c@nterpt}}%
    \else\v@lXa=-.5\v@lYa%
    \rlap{\kern\v@lX\raise\v@lY\hbox{\rlap{\kern-\wd\Gb@x\kern-\v@leur%
          \raise-.5\ht\Gb@x\hbox{\raise.5\dp\Gb@x\hbox{\unhcopy\Gb@x}}}\c@nterpt}}%
    \fi\fi}
\ctr@ld@f\def\FigWp@sigE{\v@lXa=\z@\v@lYa=\ht\Gb@x\advance\v@lYa\dp\Gb@x%
    \ifdim\delt@>\z@\relax%
    \rlap{\kern\v@lX\raise\v@lY\hbox{\c@nterpt\kern\v@leur%
          \raise\delt@\hbox{\raise\dp\Gb@x\hbox{\unhcopy\Gb@x}}}}%
    \else\ifdim\delt@<\z@\relax\v@lYa=-\v@lYa%
    \rlap{\kern\v@lX\raise\v@lY\hbox{\c@nterpt\kern\v@leur%
          \raise\delt@\hbox{\raise-\ht\Gb@x\hbox{\unhcopy\Gb@x}}}}%
    \else\v@lXa=-.5\v@lYa%
    \rlap{\kern\v@lX\raise\v@lY\hbox{\c@nterpt\kern\v@leur%
          \raise-.5\ht\Gb@x\hbox{\raise.5\dp\Gb@x\hbox{\unhcopy\Gb@x}}}}%
    \fi\fi}
\ctr@ld@f\def\FigWBB@xgWE{\advance\v@lY\delt@%
    \advance\v@lX\the\let@xte\v@leur\advance\v@lY\v@lXa\b@undb@x{\v@lX}{\v@lY}%
    \advance\v@lX\the\let@xte\wd\Gb@x\advance\v@lY\v@lYa\b@undb@x{\v@lX}{\v@lY}}
\ctr@ld@f\def\figwritegcw#1:#2(#3,#4){{\let\figWp@si=\FigWp@sigcW\let\figWBB@x=\FigWBB@xgcWE%
    \let@xte={-}\FigWrit@L#1:{#2}(#3,#4)}\ignorespaces}
\ctr@ld@f\def\figwritegce#1:#2(#3,#4){{\let\figWp@si=\FigWp@sigcE\let\figWBB@x=\FigWBB@xgcWE%
    \let@xte={}\FigWrit@L#1:{#2}(#3,#4)}\ignorespaces}
\ctr@ld@f\def\FigWp@sigcW{\rlap{\kern\v@lX\raise\v@lY\hbox{\rlap{\kern-\wd\Gb@x\kern-\v@leur%
     \raise-.5\ht\Gb@x\hbox{\raise\delt@\hbox{\raise.5\dp\Gb@x\hbox{\unhcopy\Gb@x}}}}%
     \c@nterpt}}}
\ctr@ld@f\def\FigWp@sigcE{\rlap{\kern\v@lX\raise\v@lY\hbox{\c@nterpt\kern\v@leur%
    \raise-.5\ht\Gb@x\hbox{\raise\delt@\hbox{\raise.5\dp\Gb@x\hbox{\unhcopy\Gb@x}}}}}}
\ctr@ld@f\def\FigWBB@xgcWE{\v@lZ=\ht\Gb@x\advance\v@lZ\dp\Gb@x%
    \advance\v@lX\the\let@xte\v@leur\advance\v@lY\delt@\advance\v@lY.5\v@lZ%
    \b@undb@x{\v@lX}{\v@lY}%
    \advance\v@lX\the\let@xte\wd\Gb@x\advance\v@lY-\v@lZ\b@undb@x{\v@lX}{\v@lY}}
\ctr@ld@f\def\figwritebn#1:#2(#3){{\def\Vc@rrect{\v@lZ=\v@leur}\Figwrit@NS#1:{#2}(#3)}\ignorespaces}
\ctr@ld@f\def\figwritebs#1:#2(#3){{\def\Vc@rrect{\v@lZ=-\v@leur}\Figwrit@NS#1:{#2}(#3)}\ignorespaces}
\ctr@ld@f\def\figwritebw#1:#2(#3){{\let\figWp@si=\FigWp@sibW\let\figWBB@x=\FigWBB@xbWE%
    \let@xte={-}\FigWrit@L#1:{#2}(#3,0pt)}\ignorespaces}
\ctr@ld@f\def\figwritebe#1:#2(#3){{\let\figWp@si=\FigWp@sibE\let\figWBB@x=\FigWBB@xbWE%
    \let@xte={}\FigWrit@L#1:{#2}(#3,0pt)}\ignorespaces}
\ctr@ld@f\def\FigWp@sibW{\rlap{\kern\v@lX\raise\v@lY\hbox{\rlap{\kern-\wd\Gb@x\kern-\v@leur%
          \hbox{\unhcopy\Gb@x}}\c@nterpt}}}
\ctr@ld@f\def\FigWp@sibE{\rlap{\kern\v@lX\raise\v@lY\hbox{\c@nterpt\kern\v@leur%
          \hbox{\unhcopy\Gb@x}}}}
\ctr@ld@f\def\FigWBB@xbWE{\v@lZ=\ht\Gb@x\advance\v@lZ\dp\Gb@x%
    \advance\v@lX\the\let@xte\v@leur\advance\v@lY\ht\Gb@x\b@undb@x{\v@lX}{\v@lY}%
    \advance\v@lX\the\let@xte\wd\Gb@x\advance\v@lY-\v@lZ\b@undb@x{\v@lX}{\v@lY}}
\ctr@ln@w{newread}\frf@g  \ctr@ln@w{newwrite}\fwf@g
\ctr@ln@w{newif}\ifCUR@PS
\ctr@ln@w{newif}\ifGR@cri
\ctr@ln@w{newif}\ifUse@llipse
\ctr@ln@w{newif}\ifGRdebugm@de \GRdebugm@defalse 
\ctr@ln@w{newif}\ifPDFm@ke
\ifx\pdfliteral\undefined\else\ifnum\pdfoutput>\z@\PDFm@ketrue\fi\fi
\ctr@ld@f\def\initPDF@rDVI{%
\ifPDFm@ke
 \let\figscan=\figscan@E
 \let\newGr@FN=\newGr@FNPDF
 \ctr@ld@f\def\c@mcurveto{c}
 \ctr@ld@f\def\c@mfill{f}
 \ctr@ld@f\def\c@mgsave{q}
 \ctr@ld@f\def\c@mgrestore{Q}
 \ctr@ld@f\def\c@mlineto{l}
 \ctr@ld@f\def\c@mmoveto{m}
 \ctr@ld@f\def\c@msetgray{g}     \ctr@ld@f\def\c@msetgrayStroke{G}
 \ctr@ld@f\def\c@msetcmykcolor{k}\ctr@ld@f\def\c@msetcmykcolorStroke{K}
 \ctr@ld@f\def\c@msetrgbcolor{rg}\ctr@ld@f\def\c@msetrgbcolorStroke{RG}
 \ctr@ld@f\def\d@fprimarC@lor{\CUR@color\space\CUR@colorc@md%
               \space\CUR@color\space\CUR@colorc@mdStroke}
 \ctr@ld@f\def\c@msetdash{d}
 \ctr@ld@f\def\c@msetlinejoin{j}
 \ctr@ld@f\def\c@msetlinewidth{w}
 \ctr@ld@f\def\f@gclosestroke{\immediate\write\fwf@g{s}}
 \ctr@ld@f\def\f@gfill{\immediate\write\fwf@g{\fillc@md}}
 \ctr@ld@f\def\f@gnewpath{}
 \ctr@ld@f\def\f@gstroke{\immediate\write\fwf@g{S}}
\else
 \let\figinsertE=\figinsert
 \let\newGr@FN=\newGr@FNDVI
 \ctr@ld@f\def\c@mcurveto{curveto}
 \ctr@ld@f\def\c@mfill{fill}
 \ctr@ld@f\def\c@mgsave{gsave}
 \ctr@ld@f\def\c@mgrestore{grestore}
 \ctr@ld@f\def\c@mlineto{lineto}
 \ctr@ld@f\def\c@mmoveto{moveto}
 \ctr@ld@f\def\c@msetgray{setgray}          \ctr@ld@f\def\c@msetgrayStroke{}
 \ctr@ld@f\def\c@msetcmykcolor{setcmykcolor}\ctr@ld@f\def\c@msetcmykcolorStroke{}
 \ctr@ld@f\def\c@msetrgbcolor{setrgbcolor}  \ctr@ld@f\def\c@msetrgbcolorStroke{}
 \ctr@ld@f\def\d@fprimarC@lor{\CUR@color\space\CUR@colorc@md}
 \ctr@ld@f\def\c@msetdash{setdash}
 \ctr@ld@f\def\c@msetlinejoin{setlinejoin}
 \ctr@ld@f\def\c@msetlinewidth{setlinewidth}
 \ctr@ld@f\def\f@gclosestroke{\immediate\write\fwf@g{closepath\space stroke}}
 \ctr@ld@f\def\f@gfill{\immediate\write\fwf@g{\fillc@md}}
 \ctr@ld@f\def\f@gnewpath{\immediate\write\fwf@g{newpath}}
 \ctr@ld@f\def\f@gstroke{\immediate\write\fwf@g{stroke}}
\fi}
\ctr@ld@f\def\c@pypsfile#1#2{\c@pyfil@{\immediate\write#1}{#2}}
\ctr@ld@f\def\Figinclud@PDF#1#2{\openin\frf@g=#1\pdfliteral{q #2 0 0 #2 0 0 cm}%
    \c@pyfil@{\pdfliteral}{\frf@g}\pdfliteral{Q}\closein\frf@g}
\ctr@ln@w{newif}\ifmored@ta
\ctr@ln@m\bl@nkline
\ctr@ld@f\def\c@pyfil@#1#2{\def\bl@nkline{\par}{\catcode`\%=12
    \loop\ifeof#2\mored@tafalse\else\mored@tatrue\immediate\read#2 to\tr@c
    \ifx\tr@c\bl@nkline\else#1{\tr@c}\fi\fi\ifmored@ta\repeat}}
\ctr@ld@f\def\keln@mun#1#2|{\def\l@debut{#1}\def\l@suite{#2}}
\ctr@ld@f\def\keln@mde#1#2#3|{\def\l@debut{#1#2}\def\l@suite{#3}}
\ctr@ld@f\def\keln@mtr#1#2#3#4|{\def\l@debut{#1#2#3}\def\l@suite{#4}}
\ctr@ld@f\def\keln@mqu#1#2#3#4#5|{\def\l@debut{#1#2#3#4}\def\l@suite{#5}}
\ctr@ld@f\let\@psffilein=\frf@g 
\ctr@ln@w{newif}\if@psffileok    
\ctr@ln@w{newif}\if@psfbbfound   
\ctr@ln@w{newif}\if@psfverbose   
\@psfverbosetrue
\ctr@ln@m\@psfllx \ctr@ln@m\@psflly
\ctr@ln@m\@psfurx \ctr@ln@m\@psfury
\ctr@ln@m\resetcolonc@tcode
\ctr@ld@f\def\@psfgetbb#1{\global\@psfbbfoundfalse%
\global\def\@psfllx{0}\global\def\@psflly{0}%
\global\def\@psfurx{30}\global\def\@psfury{30}%
\openin\@psffilein=#1\relax
\ifeof\@psffilein\errmessage{I couldn't open #1, will ignore it}\else
   \edef\resetcolonc@tcode{\catcode`\noexpand\:\the\catcode`\:\relax}%
   {\@psffileoktrue \chardef\other=12
    \def\do##1{\catcode`##1=\other}\dospecials \catcode`\ =10 \resetcolonc@tcode
    \loop
       \read\@psffilein to \@psffileline
       \ifeof\@psffilein\@psffileokfalse\else
          \expandafter\@psfaux\@psffileline:. \\%
       \fi
   \if@psffileok\repeat
   \if@psfbbfound\else
    \if@psfverbose\message{No bounding box comment in #1; using defaults}\fi\fi
   }\closein\@psffilein\fi}%
\ctr@ln@m\@psfbblit
\ctr@ln@m\@psfpercent
{\catcode`\%=12 \global\let\@psfpercent=
\ctr@ln@m\@psfaux
\long\def\@psfaux#1#2:#3\\{\ifx#1\@psfpercent
   \def\testit{#2}\ifx\testit\@psfbblit
      \@psfgrab #3 . . . \\%
      \@psffileokfalse
      \global\@psfbbfoundtrue
   \fi\else\ifx#1\par\else\@psffileokfalse\fi\fi}%
\ctr@ld@f\def\@psfempty{}%
\ctr@ld@f\def\@psfgrab #1 #2 #3 #4 #5\\{%
\global\def\@psfllx{#1}\ifx\@psfllx\@psfempty
      \@psfgrab #2 #3 #4 #5 .\\\else
   \global\def\@psflly{#2}%
   \global\def\@psfurx{#3}\global\def\@psfury{#4}\fi}%
\ctr@ld@f\def\PSwrit@cmd#1#2#3{{\Figg@tXY{#1}\c@lprojSP\b@undb@x{\v@lX}{\v@lY}%
    \v@lX=\ptT@ptps\v@lX\v@lY=\ptT@ptps\v@lY%
    \immediate\write#3{\repdecn@mb{\v@lX}\space\repdecn@mb{\v@lY}\space#2}}}
\ctr@ld@f\def\PSwrit@cmdS#1#2#3#4#5{{\Figg@tXY{#1}\c@lprojSP\b@undb@x{\v@lX}{\v@lY}%
    \global\result@t=\v@lX\global\result@@t=\v@lY%
    \v@lX=\ptT@ptps\v@lX\v@lY=\ptT@ptps\v@lY%
    \immediate\write#3{\repdecn@mb{\v@lX}\space\repdecn@mb{\v@lY}\space#2}}%
    \edef#4{\the\result@t}\edef#5{\the\result@@t}}
\ctr@ld@f\def\update@ttr#1#2#3{\Figdisc@rdLTS{#3}{\n@mref}%
    \ifx\n@mref\D@FTref#2{#1}\else#2{#3}\fi}
\ctr@ld@f\def\D@FTref{default}
\ctr@ld@f\def\W@rnmesAttr#1#2{%
    \immediate\write16{*** Unknown attribute: \BS@ #1(..., #2=...)}}
\ctr@ld@f\def\W@rnmeskwd#1#2{%
    \immediate\write16{*** Unknown keyword #2 in \BS@ #1}}
\ctr@ld@f\def\W@rnmesIgn#1{\immediate\write16{*** \BS@ #1 is ignored inside a
     \BS@ figdrawbegin-\BS@ figdrawend block.}}
\ctr@ld@f\def\Psset@lti#1=#2|{\keln@mtr#1|%
    \def\n@mref{blc}\ifx\l@debut\n@mref\update@ttr\D@FTref\P@setblcolor{#2}\else
    \def\n@mref{bld}\ifx\l@debut\n@mref\update@ttr\D@FTref\P@setbldash{#2}\else
    \def\n@mref{blw}\ifx\l@debut\n@mref\update@ttr\D@FTref\P@setblwidth{#2}\else
    \def\n@mref{sqc}\ifx\l@debut\n@mref\update@ttr\D@FTref\P@setsqcolor{#2}\else
    \def\n@mref{sqd}\ifx\l@debut\n@mref\update@ttr\D@FTref\P@setsqdash{#2}\else
    \def\n@mref{sqw}\ifx\l@debut\n@mref\update@ttr\D@FTref\P@setsqwidth{#2}\else
    \W@rnmesAttr{figset altitude}{#1}\fi\fi\fi\fi\fi\fi}
\ctr@ln@m\DDV@blcolor
\ctr@ld@f\def\P@setblcolor#1{\edef\DDV@blcolor{#1}}
\ctr@ln@m\DDV@bldash
\ctr@ld@f\def\P@setbldash#1{\edef\DDV@bldash{#1}}
\ctr@ln@m\DDV@blwidth
\ctr@ld@f\def\P@setblwidth#1{\edef\DDV@blwidth{#1}}
\ctr@ln@m\DDV@sqcolor
\ctr@ld@f\def\P@setsqcolor#1{\edef\DDV@sqcolor{#1}}
\ctr@ln@m\DDV@sqdash
\ctr@ld@f\def\P@setsqdash#1{\edef\DDV@sqdash{#1}}
\ctr@ln@m\DDV@sqwidth
\ctr@ld@f\def\P@setsqwidth#1{\edef\DDV@sqwidth{#1}}
\ctr@ld@f\def\figdrawaltitude#1[#2,#3,#4]{{\ifCUR@PS\ifGR@cri%
    \PSc@mment{altitude Square Dim=#1, Triangle=[#2 / #3,#4]}%
    \s@uvc@ntr@l\et@tpsaltitude\resetc@ntr@l{2}\figptorthoprojline-5:=#2/#3,#4/%
    \figvectP -1[#3,#4]\n@rminf{\v@leur}{-1}\vecunit@{-3}{-1}%
    \figvectP -1[-5,#3]\n@rminf{\v@lmin}{-1}\figvectP -2[-5,#4]\n@rminf{\v@lmax}{-2}%
    \ifdim\v@lmin<\v@lmax\s@mme=#3\else\v@lmax=\v@lmin\s@mme=#4\fi%
    \figvectP -4[-5,#2]\vecunit@{-4}{-4}\delt@=#1\unit@%
    \edef\t@ille{\repdecn@mb{\delt@}}\figpttra-1:=-5/\t@ille,-3/%
    \figptstra-3=-5,-1/\t@ille,-4/\figdrawline[#2,-5]%
    \Pss@tspecifSt{color=\DDV@sqcolor,dash=\DDV@sqdash,width=\DDV@sqwidth}%
    \figdrawline[-1,-2,-3]%
    \Psrest@reSt{color=\DDV@sqcolor,dash=\DDV@sqdash,width=\DDV@sqwidth}%
    \ifdim\v@leur<\v@lmax%
    \Pss@tspecifSt{color=\DDV@blcolor,dash=\DDV@bldash,width=\DDV@blwidth}%
    \figdrawline[-5,\the\s@mme]%
    \Psrest@reSt{color=\DDV@blcolor,dash=\DDV@bldash,width=\DDV@blwidth}%
    \fi\PSc@mment{End altitude}\resetc@ntr@l\et@tpsaltitude\fi\fi}}
\ctr@ld@f\def\Ps@rcerc#1;#2(#3,#4){\ellBB@x#1;#2,#2(#3,#4,0)%
    \f@gnewpath{\delt@=#2\unit@\delt@=\ptT@ptps\delt@%
    \BdingB@xfalse%
    \PSwrit@cmd{#1}{\repdecn@mb{\delt@}\space #3\space #4\space arc}{\fwf@g}}}
\ctr@ln@m\figdrawarccirc
\ctr@ld@f\def\Q@arccircDD#1;#2(#3,#4){\ifCUR@PS\ifGR@cri%
    \PSc@mment{arccircDD Center=#1 ; Radius=#2 (Ang1=#3, Ang2=#4)}%
    \iffillm@de\Ps@rcerc#1;#2(#3,#4)%
    \f@gfill%
    \else\Ps@rcerc#1;#2(#3,#4)\f@gstroke\fi%
    \PSc@mment{End arccircDD}\fi\fi}
\ctr@ld@f\def\Q@arccircTD#1,#2,#3;#4(#5,#6){{\ifCUR@PS\ifGR@cri\s@uvc@ntr@l\et@tpsarccircTD%
    \PSc@mment{arccircTD Center=#1,P1=#2,P2=#3 ; Radius=#4 (Ang1=#5, Ang2=#6)}%
    \setc@ntr@l{2}\c@lExtAxes#1,#2,#3(#4)\Q@arcellPATD#1,-4,-5(#5,#6)%
    \PSc@mment{End arccircTD}\resetc@ntr@l\et@tpsarccircTD\fi\fi}}
\ctr@ld@f\def\c@lExtAxes#1,#2,#3(#4){%
    \figvectPTD-5[#1,#2]\vecunit@{-5}{-5}\figvectNTD-4[#1,#2,#3]\vecunit@{-4}{-4}%
    \figvectNVTD-3[-4,-5]\delt@=#4\unit@\edef\r@yon{\repdecn@mb{\delt@}}%
    \figpttra-4:=#1/\r@yon,-5/\figpttra-5:=#1/\r@yon,-3/}
\ctr@ln@m\figdrawarccircP
\ctr@ld@f\def\Q@arccircPDD#1;#2[#3,#4]{{\ifCUR@PS\ifGR@cri\s@uvc@ntr@l\et@tpsarccircPDD%
    \PSc@mment{arccircPDD Center=#1; Radius=#2, [P1=#3, P2=#4]}%
    \Ps@ngleparam#1;#2[#3,#4]\ifdim\v@lmin>\v@lmax\advance\v@lmax\DePI@deg\fi%
    \edef\@ngdeb{\repdecn@mb{\v@lmin}}\edef\@ngfin{\repdecn@mb{\v@lmax}}%
    \figdrawarccirc#1;\r@dius(\@ngdeb,\@ngfin)%
    \PSc@mment{End arccircPDD}\resetc@ntr@l\et@tpsarccircPDD\fi\fi}}
\ctr@ld@f\def\Q@arccircPTD#1;#2[#3,#4,#5]{{\ifCUR@PS\ifGR@cri\s@uvc@ntr@l\et@tpsarccircPTD%
    \PSc@mment{arccircPTD Center=#1; Radius=#2, [P1=#3, P2=#4, P3=#5]}%
    \setc@ntr@l{2}\c@lExtAxes#1,#3,#5(#2)\figdrawarcellPP#1,-4,-5[#3,#4]%
    \PSc@mment{End arccircPTD}\resetc@ntr@l\et@tpsarccircPTD\fi\fi}}
\ctr@ld@f\def\Ps@ngleparam#1;#2[#3,#4]{\setc@ntr@l{2}%
    \figvectPDD-1[#1,#3]\vecunit@{-1}{-1}\Figg@tXY{-1}\arct@n\v@lmin(\v@lX,\v@lY)%
    \figvectPDD-2[#1,#4]\vecunit@{-2}{-2}\Figg@tXY{-2}\arct@n\v@lmax(\v@lX,\v@lY)%
    \v@lmin=\rdT@deg\v@lmin\v@lmax=\rdT@deg\v@lmax%
    \v@leur=#2pt\maxim@m{\mili@u}{-\v@leur}{\v@leur}%
    \edef\r@dius{\repdecn@mb{\mili@u}}}
\ctr@ld@f\def\Ps@rcercBz#1;#2(#3,#4){\Ps@rellBz#1;#2,#2(#3,#4,0)}
\ctr@ld@f\def\Ps@rellBz#1;#2,#3(#4,#5,#6){%
    \ellBB@x#1;#2,#3(#4,#5,#6)\BdingB@xfalse%
    \c@lNbarcs{#4}{#5}\v@leur=#4pt\setc@ntr@l{2}\figptell-13::#1;#2,#3(#4,#6)%
    \f@gnewpath\PSwrit@cmd{-13}{\c@mmoveto}{\fwf@g}%
    \s@mme=\z@\bcl@rellBz#1;#2,#3(#6)\BdingB@xtrue}
\ctr@ld@f\def\bcl@rellBz#1;#2,#3(#4){\relax%
    \ifnum\s@mme<\p@rtent\advance\s@mme\@ne%
    \advance\v@leur\delt@\edef\@ngle{\repdecn@mb\v@leur}\figptell-14::#1;#2,#3(\@ngle,#4)%
    \advance\v@leur\delt@\edef\@ngle{\repdecn@mb\v@leur}\figptell-15::#1;#2,#3(\@ngle,#4)%
    \advance\v@leur\delt@\edef\@ngle{\repdecn@mb\v@leur}\figptell-16::#1;#2,#3(\@ngle,#4)%
    \figptscontrolDD-18[-13,-14,-15,-16]%
    \PSwrit@cmd{-18}{}{\fwf@g}\PSwrit@cmd{-17}{}{\fwf@g}%
    \PSwrit@cmd{-16}{\c@mcurveto}{\fwf@g}%
    \figptcopyDD-13:/-16/\bcl@rellBz#1;#2,#3(#4)\fi}
\ctr@ld@f\def\Ps@rell#1;#2,#3(#4,#5,#6){\ellBB@x#1;#2,#3(#4,#5,#6)%
    \f@gnewpath{\v@lmin=#2\unit@\v@lmin=\ptT@ptps\v@lmin%
    \v@lmax=#3\unit@\v@lmax=\ptT@ptps\v@lmax\BdingB@xfalse%
    \PSwrit@cmd{#1}%
    {#6\space\repdecn@mb{\v@lmin}\space\repdecn@mb{\v@lmax}\space #4\space #5\space ellipse}{\fwf@g}}%
    \global\Use@llipsetrue}
\ctr@ln@m\figdrawarcell
\ctr@ld@f\def\Q@arcellDD#1;#2,#3(#4,#5,#6){{\ifCUR@PS\ifGR@cri%
    \PSc@mment{arcellDD Center=#1 ; XRad=#2, YRad=#3 (Ang1=#4, Ang2=#5, Inclination=#6)}%
    \iffillm@de\Ps@rell#1;#2,#3(#4,#5,#6)%
    \f@gfill%
    \else\Ps@rell#1;#2,#3(#4,#5,#6)\f@gstroke\fi%
    \PSc@mment{End arcellDD}\fi\fi}}
\ctr@ld@f\def\Q@arcellTD#1;#2,#3(#4,#5,#6){{\ifCUR@PS\ifGR@cri\s@uvc@ntr@l\et@tpsarcellTD%
    \PSc@mment{arcellTD Center=#1 ; XRad=#2, YRad=#3 (Ang1=#4, Ang2=#5, Inclination=#6)}%
    \setc@ntr@l{2}\figpttraC -8:=#1/#2,0,0/\figpttraC -7:=#1/0,#3,0/%
    \figvectC -4(0,0,1)\figptsrot -8=-8,-7/#1,#6,-4/\Q@arcellPATD#1,-8,-7(#4,#5)%
    \PSc@mment{End arcellTD}\resetc@ntr@l\et@tpsarcellTD\fi\fi}}
\ctr@ln@m\figdrawarcellPA
\ctr@ld@f\def\Q@arcellPADD#1,#2,#3(#4,#5){{\ifCUR@PS\ifGR@cri\s@uvc@ntr@l\et@tpsarcellPADD%
    \PSc@mment{arcellPADD Center=#1,PtAxis1=#2,PtAxis2=#3 (Ang1=#4, Ang2=#5)}%
    \setc@ntr@l{2}\figvectPDD-1[#1,#2]\vecunit@DD{-1}{-1}\v@lX=\ptT@unit@\result@t%
    \edef\XR@d{\repdecn@mb{\v@lX}}\Figg@tXY{-1}\arct@n\v@lmin(\v@lX,\v@lY)%
    \v@lmin=\rdT@deg\v@lmin\edef\Inclin@{\repdecn@mb{\v@lmin}}%
    \figgetdist\YR@d[#1,#3]\Q@arcellDD#1;\XR@d,\YR@d(#4,#5,\Inclin@)%
    \PSc@mment{End arcellPADD}\resetc@ntr@l\et@tpsarcellPADD\fi\fi}}
\ctr@ld@f\def\Q@arcellPATD#1,#2,#3(#4,#5){{\ifCUR@PS\ifGR@cri\s@uvc@ntr@l\et@tpsarcellPATD%
    \PSc@mment{arcellPATD Center=#1,PtAxis1=#2,PtAxis2=#3 (Ang1=#4, Ang2=#5)}%
    \iffillm@de\Ps@rellPATD#1,#2,#3(#4,#5)%
    \f@gfill%
    \else\Ps@rellPATD#1,#2,#3(#4,#5)\f@gstroke\fi%
    \PSc@mment{End arcellPATD}\resetc@ntr@l\et@tpsarcellPATD\fi\fi}}
\ctr@ld@f\def\Ps@rellPATD#1,#2,#3(#4,#5){\let\c@lprojSP=\relax%
    \setc@ntr@l{2}\figvectPTD-1[#1,#2]\figvectPTD-2[#1,#3]\c@lNbarcs{#4}{#5}%
    \v@leur=#4pt\c@lptellP{#1}{-1}{-2}\Figptpr@j-5:/-3/%
    \f@gnewpath\PSwrit@cmdS{-5}{\c@mmoveto}{\fwf@g}{\X@un}{\Y@un}%
    \edef\C@nt@r{#1}\s@mme=\z@\bcl@rellPATD}
\ctr@ld@f\def\bcl@rellPATD{\relax%
    \ifnum\s@mme<\p@rtent\advance\s@mme\@ne%
    \advance\v@leur\delt@\c@lptellP{\C@nt@r}{-1}{-2}\Figptpr@j-4:/-3/%
    \advance\v@leur\delt@\c@lptellP{\C@nt@r}{-1}{-2}\Figptpr@j-6:/-3/%
    \advance\v@leur\delt@\c@lptellP{\C@nt@r}{-1}{-2}\Figptpr@j-3:/-3/%
    \v@lX=\z@\v@lY=\z@\Figtr@nptDD{-5}{-5}\Figtr@nptDD{2}{-3}%
    \divide\v@lX\@vi\divide\v@lY\@vi%
    \Figtr@nptDD{3}{-4}\Figtr@nptDD{-1.5}{-6}\v@lmin=\v@lX\v@lmax=\v@lY%
    \v@lX=\z@\v@lY=\z@\Figtr@nptDD{2}{-5}\Figtr@nptDD{-5}{-3}%
    \divide\v@lX\@vi\divide\v@lY\@vi\Figtr@nptDD{-1.5}{-4}\Figtr@nptDD{3}{-6}%
    \BdingB@xfalse%
    \Figp@intregDD-4:(\v@lmin,\v@lmax)\PSwrit@cmdS{-4}{}{\fwf@g}{\X@de}{\Y@de}%
    \Figp@intregDD-4:(\v@lX,\v@lY)\PSwrit@cmdS{-4}{}{\fwf@g}{\X@tr}{\Y@tr}%
    \BdingB@xtrue\PSwrit@cmdS{-3}{\c@mcurveto}{\fwf@g}{\X@qu}{\Y@qu}%
    \B@zierBB@x{1}{\Y@un}(\X@un,\X@de,\X@tr,\X@qu)%
    \B@zierBB@x{2}{\X@un}(\Y@un,\Y@de,\Y@tr,\Y@qu)%
    \edef\X@un{\X@qu}\edef\Y@un{\Y@qu}\figptcopyDD-5:/-3/\bcl@rellPATD\fi}
\ctr@ld@f\def\c@lNbarcs#1#2{%
    \delt@=#2pt\advance\delt@-#1pt\maxim@m{\v@lmax}{\delt@}{-\delt@}%
    \v@leur=\v@lmax\divide\v@leur45 \p@rtentiere{\p@rtent}{\v@leur}\advance\p@rtent\@ne%
    \s@mme=\p@rtent\multiply\s@mme\thr@@\divide\delt@\s@mme}
\ctr@ld@f\def\figdrawarcellPP#1,#2,#3[#4,#5]{{\ifCUR@PS\ifGR@cri\s@uvc@ntr@l\et@tpsarcellPP%
    \PSc@mment{arcellPP Center=#1,PtAxis1=#2,PtAxis2=#3 [Point1=#4, Point2=#5]}%
    \setc@ntr@l{2}\figvectP-2[#1,#3]\vecunit@{-2}{-2}\v@lmin=\result@t%
    \invers@{\v@lmax}{\v@lmin}%
    \figvectP-1[#1,#2]\vecunit@{-1}{-1}\v@leur=\result@t%
    \v@leur=\repdecn@mb{\v@lmax}\v@leur\edef\AsB@{\repdecn@mb{\v@leur}}
    \c@lAngle{#1}{#4}{\v@lmin}\edef\@ngdeb{\repdecn@mb{\v@lmin}}%
    \c@lAngle{#1}{#5}{\v@lmax}\ifdim\v@lmin>\v@lmax\advance\v@lmax\DePI@deg\fi%
    \edef\@ngfin{\repdecn@mb{\v@lmax}}\figdrawarcellPA#1,#2,#3(\@ngdeb,\@ngfin)%
    \PSc@mment{End arcellPP}\resetc@ntr@l\et@tpsarcellPP\fi\fi}}
\ctr@ld@f\def\c@lAngle#1#2#3{\figvectP-3[#1,#2]%
    \c@lproscal\delt@[-3,-1]\c@lproscal\v@leur[-3,-2]%
    \v@leur=\AsB@\v@leur\arct@n#3(\delt@,\v@leur)#3=\rdT@deg#3}
\ctr@ln@w{newif}\if@rrowratio\@rrowratiotrue
\ctr@ln@w{newif}\if@rrowhfill
\ctr@ln@w{newif}\if@rrowhout
\ctr@ld@f\def\Psset@rrowhe@d#1=#2|{\keln@mun#1|%
    \def\n@mref{a}\ifx\l@debut\n@mref\update@ttr\D@FTarrowheadangle\Q@s@tarrowheadangle{#2}\else
    \def\n@mref{f}\ifx\l@debut\n@mref\update@ttr\D@FTarrowheadfill\Q@s@tarrowheadfill{#2}\else
    \def\n@mref{l}\ifx\l@debut\n@mref\update@ttr\D@FTarrowheadlength\Q@s@tarrowheadlength{#2}\else
    \def\n@mref{o}\ifx\l@debut\n@mref\update@ttr\D@FTarrowheadout\Q@s@tarrowheadout{#2}\else
    \def\n@mref{r}\ifx\l@debut\n@mref\update@ttr\D@FTarrowheadratio\Q@s@tarrowheadratio{#2}\else
    \W@rnmesAttr{figset arrowhead}{#1}\fi\fi\fi\fi\fi}
\ctr@ln@m\@rrowheadangle
\ctr@ln@m\C@AHANG \ctr@ln@m\S@AHANG \ctr@ln@m\UNSS@N
\ctr@ld@f\def\Q@s@tarrowheadangle#1{\edef\@rrowheadangle{#1}{\c@ssin{\C@}{\S@}{#1}%
    \xdef\C@AHANG{\C@}\xdef\S@AHANG{\S@}\v@lmax=\S@ pt%
    \invers@{\v@leur}{\v@lmax}\maxim@m{\v@leur}{\v@leur}{-\v@leur}%
    \xdef\UNSS@N{\the\v@leur}}}
\ctr@ld@f\def\Q@s@tarrowheadfill#1{\expandafter\set@rrowhfill#1:}
\ctr@ld@f\def\set@rrowhfill#1#2:{\if#1n\@rrowhfillfalse\else\@rrowhfilltrue\fi}
\ctr@ld@f\def\Q@s@tarrowheadout#1{\expandafter\set@rrowhout#1:}
\ctr@ld@f\def\set@rrowhout#1#2:{\if#1n\@rrowhoutfalse\else\@rrowhouttrue\fi}
\ctr@ln@m\@rrowheadlength
\ctr@ld@f\def\Q@s@tarrowheadlength#1{\edef\@rrowheadlength{#1}\@rrowratiofalse}
\ctr@ln@m\@rrowheadratio
\ctr@ld@f\def\Q@s@tarrowheadratio#1{\edef\@rrowheadratio{#1}\@rrowratiotrue}
\ctr@ln@m\D@FTarrowheadlength
\ctr@ld@f\def\figresetarrowhead{%
    \Q@s@tarrowheadangle{\D@FTarrowheadangle}%
    \Q@s@tarrowheadfill{\D@FTarrowheadfill}%
    \Q@s@tarrowheadout{\D@FTarrowheadout}%
    \Q@s@tarrowheadratio{\D@FTarrowheadratio}%
    \d@fm@cdim\D@FTarrowheadlength{\D@FTh@rdahlength}
    \Q@s@tarrowheadlength{\D@FTarrowheadlength}}
\ctr@ld@f\def\D@FTarrowheadratio{0.1}
\ctr@ld@f\def\D@FTarrowheadangle{20}
\ctr@ld@f\def\D@FTarrowheadfill{no}
\ctr@ld@f\def\D@FTarrowheadout{no}
\ctr@ld@f\def\D@FTh@rdahlength{8pt}
\ctr@ln@m\figdrawarrow
\ctr@ld@f\def\Q@arrowDD[#1,#2]{{\ifCUR@PS\ifGR@cri\s@uvc@ntr@l\et@tpsarrow%
    \PSc@mment{arrowDD [Pt1,Pt2]=[#1,#2]}\Q@s@tfillmode{no}%
    \Q@arrowheadDD[#1,#2]\setc@ntr@l{2}\figdrawline[#1,-3]%
    \PSc@mment{End arrowDD}\resetc@ntr@l\et@tpsarrow\fi\fi}}
\ctr@ld@f\def\Q@arrowTD[#1,#2]{{\ifCUR@PS\ifGR@cri\s@uvc@ntr@l\et@tpsarrowTD%
    \PSc@mment{arrowTD [Pt1,Pt2]=[#1,#2]}\resetc@ntr@l{2}%
    \Figptpr@j-5:/#1/\Figptpr@j-6:/#2/\let\c@lprojSP=\relax\Q@arrowDD[-5,-6]%
    \PSc@mment{End arrowTD}\resetc@ntr@l\et@tpsarrowTD\fi\fi}}
\ctr@ln@m\figdrawarrowhead
\ctr@ld@f\def\Q@arrowheadDD[#1,#2]{{\ifCUR@PS\ifGR@cri\s@uvc@ntr@l\et@tpsarrowheadDD%
    \if@rrowhfill\def\@hangle{-\@rrowheadangle}\else\def\@hangle{\@rrowheadangle}\fi%
    \if@rrowratio%
    \if@rrowhout\def\@hratio{-\@rrowheadratio}\else\def\@hratio{\@rrowheadratio}\fi%
    \PSc@mment{arrowheadDD Ratio=\@hratio, Angle=\@hangle, [Pt1,Pt2]=[#1,#2]}%
    \Ps@rrowhead\@hratio,\@hangle[#1,#2]%
    \else%
    \if@rrowhout\def\@hlength{-\@rrowheadlength}\else\def\@hlength{\@rrowheadlength}\fi%
    \PSc@mment{arrowheadDD Length=\@hlength, Angle=\@hangle, [Pt1,Pt2]=[#1,#2]}%
    \Ps@rrowheadfd\@hlength,\@hangle[#1,#2]%
    \fi%
    \PSc@mment{End arrowheadDD}\resetc@ntr@l\et@tpsarrowheadDD\fi\fi}}
\ctr@ld@f\def\Q@arrowheadTD[#1,#2]{{\ifCUR@PS\ifGR@cri\s@uvc@ntr@l\et@tpsarrowheadTD%
    \PSc@mment{arrowheadTD [Pt1,Pt2]=[#1,#2]}\resetc@ntr@l{2}%
    \Figptpr@j-5:/#1/\Figptpr@j-6:/#2/\let\c@lprojSP=\relax\Q@arrowheadDD[-5,-6]%
    \PSc@mment{End arrowheadTD}\resetc@ntr@l\et@tpsarrowheadTD\fi\fi}}
\ctr@ld@f\def\Ps@rrowhead#1,#2[#3,#4]{\v@leur=#1\p@\maxim@m{\v@leur}{\v@leur}{-\v@leur}%
    \ifdim\v@leur>\Cepsil@n{
    \PSc@mment{@rrowhead Ratio=#1, Angle=#2, [Pt1,Pt2]=[#3,#4]}\v@leur=\UNSS@N%
    \v@leur=\CUR@width\v@leur\v@leur=\ptpsT@pt\v@leur\delt@=.5\v@leur
    \setc@ntr@l{2}\figvectPDD-3[#4,#3]%
    \Figg@tXY{-3}\v@lX=#1\v@lX\v@lY=#1\v@lY\Figv@ctCreg-3(\v@lX,\v@lY)%
    \vecunit@{-4}{-3}\mili@u=\result@t%
    \ifdim#2pt>\z@\v@lXa=-\C@AHANG\delt@%
     \edef\c@ef{\repdecn@mb{\v@lXa}}\figpttraDD-3:=-3/\c@ef,-4/\fi%
    \edef\c@ef{\repdecn@mb{\delt@}}%
    \v@lXa=\mili@u\v@lXa=\C@AHANG\v@lXa%
    \v@lYa=\ptpsT@pt\p@\v@lYa=\CUR@width\v@lYa\v@lYa=\sDcc@ngle\v@lYa%
    \advance\v@lXa-\v@lYa\gdef\sDcc@ngle{0}%
    \ifdim\v@lXa>\v@leur\edef\c@efendpt{\repdecn@mb{\v@leur}}%
    \else\edef\c@efendpt{\repdecn@mb{\v@lXa}}\fi%
    \Figg@tXY{-3}\v@lmin=\v@lX\v@lmax=\v@lY%
    \v@lXa=\C@AHANG\v@lmin\v@lYa=\S@AHANG\v@lmax\advance\v@lXa\v@lYa%
    \v@lYa=-\S@AHANG\v@lmin\v@lX=\C@AHANG\v@lmax\advance\v@lYa\v@lX%
    \setc@ntr@l{1}\Figg@tXY{#4}\advance\v@lX\v@lXa\advance\v@lY\v@lYa%
    \setc@ntr@l{2}\Figp@intregDD-2:(\v@lX,\v@lY)%
    \v@lXa=\C@AHANG\v@lmin\v@lYa=-\S@AHANG\v@lmax\advance\v@lXa\v@lYa%
    \v@lYa=\S@AHANG\v@lmin\v@lX=\C@AHANG\v@lmax\advance\v@lYa\v@lX%
    \setc@ntr@l{1}\Figg@tXY{#4}\advance\v@lX\v@lXa\advance\v@lY\v@lYa%
    \setc@ntr@l{2}\Figp@intregDD-1:(\v@lX,\v@lY)%
    \ifdim#2pt<\z@\fillm@detrue\figdrawline[-2,#4,-1]
    \else\figptstraDD-3=#4,-2,-1/\c@ef,-4/\s@uvdash{\typ@dash}\Q@s@tdash{\D@FTdash}%
    \figdrawline[-2,-3,-1]\Q@s@tdash{\typ@dash}\fi
    \ifdim#1pt>\z@\figpttraDD-3:=#4/\c@efendpt,-4/\else\figptcopyDD-3:/#4/\fi%
    \PSc@mment{End @rrowhead}}\fi}
\ctr@ld@f\def\sDcc@ngle{0}
\ctr@ld@f\def\Ps@rrowheadfd#1,#2[#3,#4]{{%
    \PSc@mment{@rrowheadfd Length=#1, Angle=#2, [Pt1,Pt2]=[#3,#4]}%
    \setc@ntr@l{2}\figvectPDD-1[#3,#4]\n@rmeucDD{\v@leur}{-1}\v@leur=\ptT@unit@\v@leur%
    \invers@{\v@leur}{\v@leur}\v@leur=#1\v@leur\edef\R@tio{\repdecn@mb{\v@leur}}%
    \Ps@rrowhead\R@tio,#2[#3,#4]\PSc@mment{End @rrowheadfd}}}
\ctr@ln@m\figdrawarrowBezier
\ctr@ld@f\def\Q@arrowBezierDD[#1,#2,#3,#4]{{\ifCUR@PS\ifGR@cri\s@uvc@ntr@l\et@tpsarrowBezierDD%
    \PSc@mment{arrowBezierDD Control points=#1,#2,#3,#4}\setc@ntr@l{2}%
    \if@rrowratio\c@larclengthDD\v@leur,10[#1,#2,#3,#4]\else\v@leur=\z@\fi%
    \Ps@rrowB@zDD\v@leur[#1,#2,#3,#4]%
    \PSc@mment{End arrowBezierDD}\resetc@ntr@l\et@tpsarrowBezierDD\fi\fi}}
\ctr@ld@f\def\Q@arrowBezierTD[#1,#2,#3,#4]{{\ifCUR@PS\ifGR@cri\s@uvc@ntr@l\et@tpsarrowBezierTD%
    \PSc@mment{arrowBezierTD Control points=#1,#2,#3,#4}\resetc@ntr@l{2}%
    \Figptpr@j-7:/#1/\Figptpr@j-8:/#2/\Figptpr@j-9:/#3/\Figptpr@j-10:/#4/%
    \let\c@lprojSP=\relax\ifnum\CUR@proj<\tw@\Q@arrowBezierDD[-7,-8,-9,-10]%
    \else\f@gnewpath\PSwrit@cmd{-7}{\c@mmoveto}{\fwf@g}%
    \if@rrowratio\c@larclengthDD\mili@u,10[-7,-8,-9,-10]\else\mili@u=\z@\fi%
    \p@rtent=\NBz@rcs\advance\p@rtent\m@ne\subB@zierTD\p@rtent[#1,#2,#3,#4]%
    \f@gstroke%
    \advance\v@lmin\p@rtent\delt@
    \v@leur=\v@lmin\advance\v@leur0.33333 \delt@\edef\unti@rs{\repdecn@mb{\v@leur}}%
    \v@leur=\v@lmin\advance\v@leur0.66666 \delt@\edef\deti@rs{\repdecn@mb{\v@leur}}%
    \figptcopyDD-8:/-10/\c@lsubBzarc\unti@rs,\deti@rs[#1,#2,#3,#4]%
    \figptcopyDD-8:/-4/\figptcopyDD-9:/-3/\Ps@rrowB@zDD\mili@u[-7,-8,-9,-10]\fi%
    \PSc@mment{End arrowBezierTD}\resetc@ntr@l\et@tpsarrowBezierTD\fi\fi}}
\ctr@ld@f\def\c@larclengthDD#1,#2[#3,#4,#5,#6]{{\p@rtent=#2\figptcopyDD-5:/#3/%
    \delt@=\p@\divide\delt@\p@rtent\c@rre=\z@\v@leur=\z@\s@mme=\z@%
    \loop\ifnum\s@mme<\p@rtent\advance\s@mme\@ne\advance\v@leur\delt@%
    \edef\T@{\repdecn@mb{\v@leur}}\figptBezierDD-6::\T@[#3,#4,#5,#6]%
    \figvectPDD-1[-5,-6]\n@rmeucDD{\mili@u}{-1}\advance\c@rre\mili@u%
    \figptcopyDD-5:/-6/\repeat\global\result@t=\ptT@unit@\c@rre}#1=\result@t}
\ctr@ld@f\def\Ps@rrowB@zDD#1[#2,#3,#4,#5]{{\Q@s@tfillmode{no}%
    \if@rrowratio\delt@=\@rrowheadratio#1\else\delt@=\@rrowheadlength pt\fi%
    \v@leur=\C@AHANG\delt@\edef\R@dius{\repdecn@mb{\v@leur}}%
    \FigptintercircB@zDD-5::0,\R@dius[#5,#4,#3,#2]%
    \Q@s@tarrowheadlength{\repdecn@mb{\delt@}}\Q@arrowheadDD[-5,#5]%
    \let\n@rmeuc=\n@rmeucDD\figgetdist\R@dius[#5,-3]%
    \FigptintercircB@zDD-6::0,\R@dius[#5,#4,#3,#2]%
    \figptBezierDD-5::0.33333[#5,#4,#3,#2]\figptBezierDD-3::0.66666[#5,#4,#3,#2]%
    \figptscontrolDD-5[-6,-5,-3,#2]\Q@BezierDD1[-6,-5,-4,#2]}}
\ctr@ln@m\figdrawarrowcirc
\ctr@ld@f\def\Q@arrowcircDD#1;#2(#3,#4){{\ifCUR@PS\ifGR@cri\s@uvc@ntr@l\et@tpsarrowcircDD%
    \PSc@mment{arrowcircDD Center=#1 ; Radius=#2 (Ang1=#3,Ang2=#4)}%
    \Q@s@tfillmode{no}\Pscirc@rrowhead#1;#2(#3,#4)%
    \setc@ntr@l{2}\figvectPDD -4[#1,-3]\vecunit@{-4}{-4}%
    \Figg@tXY{-4}\arct@n\v@lmin(\v@lX,\v@lY)%
    \v@lmin=\rdT@deg\v@lmin\v@leur=#4pt\advance\v@leur-\v@lmin%
    \maxim@m{\v@leur}{\v@leur}{-\v@leur}%
    \ifdim\v@leur>\DemiPI@deg\relax\ifdim\v@lmin<#4pt\advance\v@lmin\DePI@deg%
    \else\advance\v@lmin-\DePI@deg\fi\fi\edef\ar@ngle{\repdecn@mb{\v@lmin}}%
    \ifdim#3pt<#4pt\figdrawarccirc#1;#2(#3,\ar@ngle)\else\figdrawarccirc#1;#2(\ar@ngle,#3)\fi%
    \PSc@mment{End arrowcircDD}\resetc@ntr@l\et@tpsarrowcircDD\fi\fi}}
\ctr@ld@f\def\Q@arrowcircTD#1,#2,#3;#4(#5,#6){{\ifCUR@PS\ifGR@cri\s@uvc@ntr@l\et@tpsarrowcircTD%
    \PSc@mment{arrowcircTD Center=#1,P1=#2,P2=#3 ; Radius=#4 (Ang1=#5, Ang2=#6)}%
    \resetc@ntr@l{2}\c@lExtAxes#1,#2,#3(#4)\let\c@lprojSP=\relax%
    \figvectPTD-11[#1,-4]\figvectPTD-12[#1,-5]\c@lNbarcs{#5}{#6}%
    \if@rrowratio\v@lmax=\degT@rd\v@lmax\edef\D@lpha{\repdecn@mb{\v@lmax}}\fi%
    \advance\p@rtent\m@ne\mili@u=\z@%
    \v@leur=#5pt\c@lptellP{#1}{-11}{-12}\Figptpr@j-9:/-3/%
    \f@gnewpath\PSwrit@cmdS{-9}{\c@mmoveto}{\fwf@g}{\X@un}{\Y@un}%
    \edef\C@nt@r{#1}\s@mme=\z@\bcl@rcircTD\f@gstroke%
    \advance\v@leur\delt@\c@lptellP{#1}{-11}{-12}\Figptpr@j-5:/-3/%
    \advance\v@leur\delt@\c@lptellP{#1}{-11}{-12}\Figptpr@j-6:/-3/%
    \advance\v@leur\delt@\c@lptellP{#1}{-11}{-12}\Figptpr@j-10:/-3/%
    \figptscontrolDD-8[-9,-5,-6,-10]%
    \if@rrowratio\c@lcurvradDD0.5[-9,-8,-7,-10]\advance\mili@u\result@t%
    \maxim@m{\mili@u}{\mili@u}{-\mili@u}\mili@u=\ptT@unit@\mili@u%
    \mili@u=\D@lpha\mili@u\advance\p@rtent\@ne\divide\mili@u\p@rtent\fi%
    \Ps@rrowB@zDD\mili@u[-9,-8,-7,-10]%
    \PSc@mment{End arrowcircTD}\resetc@ntr@l\et@tpsarrowcircTD\fi\fi}}
\ctr@ld@f\def\bcl@rcircTD{\relax%
    \ifnum\s@mme<\p@rtent\advance\s@mme\@ne%
    \advance\v@leur\delt@\c@lptellP{\C@nt@r}{-11}{-12}\Figptpr@j-5:/-3/%
    \advance\v@leur\delt@\c@lptellP{\C@nt@r}{-11}{-12}\Figptpr@j-6:/-3/%
    \advance\v@leur\delt@\c@lptellP{\C@nt@r}{-11}{-12}\Figptpr@j-10:/-3/%
    \figptscontrolDD-8[-9,-5,-6,-10]\BdingB@xfalse%
    \PSwrit@cmdS{-8}{}{\fwf@g}{\X@de}{\Y@de}\PSwrit@cmdS{-7}{}{\fwf@g}{\X@tr}{\Y@tr}%
    \BdingB@xtrue\PSwrit@cmdS{-10}{\c@mcurveto}{\fwf@g}{\X@qu}{\Y@qu}%
    \if@rrowratio\c@lcurvradDD0.5[-9,-8,-7,-10]\advance\mili@u\result@t\fi%
    \B@zierBB@x{1}{\Y@un}(\X@un,\X@de,\X@tr,\X@qu)%
    \B@zierBB@x{2}{\X@un}(\Y@un,\Y@de,\Y@tr,\Y@qu)%
    \edef\X@un{\X@qu}\edef\Y@un{\Y@qu}\figptcopyDD-9:/-10/\bcl@rcircTD\fi}
\ctr@ld@f\def\Pscirc@rrowhead#1;#2(#3,#4){{%
    \PSc@mment{circ@rrowhead Center=#1 ; Radius=#2 (Ang1=#3,Ang2=#4)}%
    \v@leur=#2\unit@\edef\s@glen{\repdecn@mb{\v@leur}}\v@lY=\z@\v@lX=\v@leur%
    \resetc@ntr@l{2}\Figv@ctCreg-3(\v@lX,\v@lY)\figpttraDD-5:=#1/1,-3/%
    \figptrotDD-5:=-5/#1,#4/%
    \figvectPDD-3[#1,-5]\Figg@tXY{-3}\v@leur=\v@lX%
    \ifdim#3pt<#4pt\v@lX=\v@lY\v@lY=-\v@leur\else\v@lX=-\v@lY\v@lY=\v@leur\fi%
    \Figv@ctCreg-3(\v@lX,\v@lY)\vecunit@{-3}{-3}%
    \if@rrowratio\v@leur=#4pt\advance\v@leur-#3pt\maxim@m{\mili@u}{-\v@leur}{\v@leur}%
    \mili@u=\degT@rd\mili@u\v@leur=\s@glen\mili@u\edef\s@glen{\repdecn@mb{\v@leur}}%
    \mili@u=#2\mili@u\mili@u=\@rrowheadratio\mili@u\else\mili@u=\@rrowheadlength pt\fi%
    \figpttraDD-6:=-5/\s@glen,-3/\v@leur=#2pt\v@leur=2\v@leur%
    \invers@{\v@leur}{\v@leur}\c@rre=\repdecn@mb{\v@leur}\mili@u
    \mili@u=\c@rre\mili@u=\repdecn@mb{\c@rre}\mili@u%
    \v@leur=\p@\advance\v@leur-\mili@u
    \invers@{\mili@u}{2\v@leur}\delt@=\c@rre\delt@=\repdecn@mb{\mili@u}\delt@%
    \xdef\sDcc@ngle{\repdecn@mb{\delt@}}
    \sqrt@{\mili@u}{\v@leur}\arct@n\v@leur(\mili@u,\c@rre)%
    \v@leur=\rdT@deg\v@leur
    \ifdim#3pt<#4pt\v@leur=-\v@leur\fi%
    \if@rrowhout\v@leur=-\v@leur\fi\edef\cor@ngle{\repdecn@mb{\v@leur}}%
    \figptrotDD-6:=-6/-5,\cor@ngle/\Q@arrowheadDD[-6,-5]%
    \PSc@mment{End circ@rrowhead}}}
\ctr@ln@m\figdrawarrowcircP
\ctr@ld@f\def\Q@arrowcircPDD#1;#2[#3,#4]{{\ifCUR@PS\ifGR@cri%
    \PSc@mment{arrowcircPDD Center=#1; Radius=#2, [P1=#3,P2=#4]}%
    \s@uvc@ntr@l\et@tpsarrowcircPDD\Ps@ngleparam#1;#2[#3,#4]%
    \ifdim\v@leur>\z@\ifdim\v@lmin>\v@lmax\advance\v@lmax\DePI@deg\fi%
    \else\ifdim\v@lmin<\v@lmax\advance\v@lmin\DePI@deg\fi\fi%
    \edef\@ngdeb{\repdecn@mb{\v@lmin}}\edef\@ngfin{\repdecn@mb{\v@lmax}}%
    \figdrawarrowcirc#1;\r@dius(\@ngdeb,\@ngfin)%
    \PSc@mment{End arrowcircPDD}\resetc@ntr@l\et@tpsarrowcircPDD\fi\fi}}
\ctr@ld@f\def\Q@arrowcircPTD#1;#2[#3,#4,#5]{{\ifCUR@PS\ifGR@cri\s@uvc@ntr@l\et@tpsarrowcircPTD%
    \PSc@mment{arrowcircPTD Center=#1; Radius=#2, [P1=#3,P2=#4,P3=#5]}%
    \figgetangleTD\@ngfin[#1,#3,#4,#5]\v@leur=#2pt%
    \maxim@m{\mili@u}{-\v@leur}{\v@leur}\edef\r@dius{\repdecn@mb{\mili@u}}%
    \ifdim\v@leur<\z@\v@lmax=\@ngfin pt\advance\v@lmax-\DePI@deg%
    \edef\@ngfin{\repdecn@mb{\v@lmax}}\fi\Q@arrowcircTD#1,#3,#5;\r@dius(0,\@ngfin)%
    \PSc@mment{End arrowcircPTD}\resetc@ntr@l\et@tpsarrowcircPTD\fi\fi}}
\ctr@ld@f\def\figdrawaxes#1(#2){{\ifCUR@PS\ifGR@cri\s@uvc@ntr@l\et@tpsaxes%
    \PSc@mment{axes Origin=#1 Range=(#2)}\an@lys@xes#2,:\resetc@ntr@l{2}%
    \ifx\t@xt@\empty\ifTr@isDim\Q@@xes#1(0,#2,0,#2,0,#2)\else\Q@@xes#1(0,#2,0,#2)\fi%
    \else\Q@@xes#1(#2)\fi\PSc@mment{End axes}\resetc@ntr@l\et@tpsaxes\fi\fi}}
\ctr@ld@f\def\an@lys@xes#1,#2:{\def\t@xt@{#2}}
\ctr@ln@m\Q@@xes
\ctr@ld@f\def\Q@@xesDD#1(#2,#3,#4,#5){%
    \figpttraC-5:=#1/#2,0/\figpttraC-6:=#1/#3,0/\Q@arrowDD[-5,-6]%
    \figpttraC-5:=#1/0,#4/\figpttraC-6:=#1/0,#5/\Q@arrowDD[-5,-6]}
\ctr@ld@f\def\Q@@xesTD#1(#2,#3,#4,#5,#6,#7){%
    \figpttraC-7:=#1/#2,0,0/\figpttraC-8:=#1/#3,0,0/\Q@arrowTD[-7,-8]%
    \figpttraC-7:=#1/0,#4,0/\figpttraC-8:=#1/0,#5,0/\Q@arrowTD[-7,-8]%
    \figpttraC-7:=#1/0,0,#6/\figpttraC-8:=#1/0,0,#7/\Q@arrowTD[-7,-8]}
\ctr@ln@m\newGr@FN
\ctr@ld@f\def\newGr@FNPDF#1{\s@mme=\Gr@FNb\advance\s@mme\@ne\xdef\Gr@FNb{\number\s@mme}}
\ctr@ld@f\def\newGr@FNDVI#1{\newGr@FNPDF{}\xdef#1{\jobname GI\Gr@FNb.anx}}
\ctr@ld@f\def\figdrawbegin#1{\newGr@FN\DefGIfilen@me\gdef\@utoFN{0}%
    \def\t@xt@{#1}\relax\ifx\t@xt@\empty\GRupdatem@detrue%
    \gdef\@utoFN{1}\Psb@ginfig\DefGIfilen@me\else\expandafter\Psb@ginfigNu@#1 :\fi}
\ctr@ld@f\def\Psb@ginfigNu@#1 #2:{\def\t@xt@{#1}\relax\ifx\t@xt@\empty\def\t@xt@{#2}%
    \ifx\t@xt@\empty\GRupdatem@detrue\gdef\@utoFN{1}\Psb@ginfig\DefGIfilen@me%
    \else\Psb@ginfigNu@#2:\fi\else\Psb@ginfig{#1}\fi}
\ctr@ln@m\PSfilen@me \ctr@ln@m\auxfilen@me
\ctr@ld@f\def\Psb@ginfig#1{\ifCUR@PS\else%
    \edef\PSfilen@me{#1}\edef\auxfilen@me{\jobname.anx}%
    \ifGRupdatem@de\GR@critrue\else\openin\frf@g=\PSfilen@me\relax%
    \ifeof\frf@g\GR@critrue\else\GR@crifalse\fi\closein\frf@g\fi%
    \CUR@PStrue\c@ldefproj\expandafter\setupd@te\D@FTupdate:%
    \ifGR@cri\initb@undb@x%
    \immediate\openout\fwf@g=\auxfilen@me\initpss@ttings\fi%
    \fi}
\ctr@ld@f\def\Gr@FNb{0}
\ctr@ld@f\def\figforTeXFileno{\Gr@FNb}
\ctr@ld@f\def\figforTeXFigno{0 }
\ctr@ld@f\def\figforTeXnextFigno{1 }
\ctr@ld@f\edef\DefGIfilen@me{\jobname GI.anx}
\ctr@ld@f\def\initpss@ttings{\figreset{altitude,arrowhead,curve,general,flowchart,mesh,trimesh}%
    \Use@llipsefalse}
\ctr@ld@f\def\B@zierBB@x#1#2(#3,#4,#5,#6){{\c@rre=\t@n\epsil@n
    \v@lmax=#4\advance\v@lmax-#5\v@lmax=\thr@@\v@lmax\advance\v@lmax#6\advance\v@lmax-#3%
    \mili@u=#4\mili@u=-\tw@\mili@u\advance\mili@u#3\advance\mili@u#5%
    \v@lmin=#4\advance\v@lmin-#3\maxim@m{\v@leur}{-\v@lmax}{\v@lmax}%
    \maxim@m{\delt@}{-\mili@u}{\mili@u}\maxim@m{\v@leur}{\v@leur}{\delt@}%
    \maxim@m{\delt@}{-\v@lmin}{\v@lmin}\maxim@m{\v@leur}{\v@leur}{\delt@}%
    \ifdim\v@leur>\c@rre\invers@{\v@leur}{\v@leur}\edef\Uns@rM@x{\repdecn@mb{\v@leur}}%
    \v@lmax=\Uns@rM@x\v@lmax\mili@u=\Uns@rM@x\mili@u\v@lmin=\Uns@rM@x\v@lmin%
    \maxim@m{\v@leur}{-\v@lmax}{\v@lmax}\ifdim\v@leur<\c@rre%
    \maxim@m{\v@leur}{-\mili@u}{\mili@u}\ifdim\v@leur<\c@rre\else%
    \invers@{\mili@u}{\mili@u}\v@leur=-0.5\v@lmin%
    \v@leur=\repdecn@mb{\mili@u}\v@leur\m@jBBB@x{\v@leur}{#1}{#2}(#3,#4,#5,#6)\fi%
    \else\delt@=\repdecn@mb{\mili@u}\mili@u\v@leur=\repdecn@mb{\v@lmax}\v@lmin%
    \advance\delt@-\v@leur\ifdim\delt@<\z@\else\invers@{\v@lmax}{\v@lmax}%
    \edef\Uns@rAp{\repdecn@mb{\v@lmax}}\sqrt@{\delt@}{\delt@}%
    \v@leur=-\mili@u\advance\v@leur\delt@\v@leur=\Uns@rAp\v@leur%
    \m@jBBB@x{\v@leur}{#1}{#2}(#3,#4,#5,#6)%
    \v@leur=-\mili@u\advance\v@leur-\delt@\v@leur=\Uns@rAp\v@leur%
    \m@jBBB@x{\v@leur}{#1}{#2}(#3,#4,#5,#6)\fi\fi\fi}}
\ctr@ld@f\def\m@jBBB@x#1#2#3(#4,#5,#6,#7){{\relax\ifdim#1>\z@\ifdim#1<\p@%
    \edef\T@{\repdecn@mb{#1}}\v@lX=\p@\advance\v@lX-#1\edef\UNmT@{\repdecn@mb{\v@lX}}%
    \v@lX=#4\v@lY=#5\v@lZ=#6\v@lXa=#7\v@lX=\UNmT@\v@lX\advance\v@lX\T@\v@lY%
    \v@lY=\UNmT@\v@lY\advance\v@lY\T@\v@lZ\v@lZ=\UNmT@\v@lZ\advance\v@lZ\T@\v@lXa%
    \v@lX=\UNmT@\v@lX\advance\v@lX\T@\v@lY\v@lY=\UNmT@\v@lY\advance\v@lY\T@\v@lZ%
    \v@lX=\UNmT@\v@lX\advance\v@lX\T@\v@lY%
    \ifcase#2\or\v@lY=#3\or\v@lY=\v@lX\v@lX=#3\fi\b@undb@x{\v@lX}{\v@lY}\fi\fi}}
\ctr@ld@f\def\PsB@zier#1[#2]{{\f@gnewpath%
    \s@mme=\z@\def\list@num{#2,0}\extrairelepremi@r\p@int\de\list@num%
    \PSwrit@cmdS{\p@int}{\c@mmoveto}{\fwf@g}{\X@un}{\Y@un}\p@rtent=#1\bclB@zier}}
\ctr@ld@f\def\bclB@zier{\relax%
    \ifnum\s@mme<\p@rtent\advance\s@mme\@ne\BdingB@xfalse%
    \extrairelepremi@r\p@int\de\list@num\PSwrit@cmdS{\p@int}{}{\fwf@g}{\X@de}{\Y@de}%
    \extrairelepremi@r\p@int\de\list@num\PSwrit@cmdS{\p@int}{}{\fwf@g}{\X@tr}{\Y@tr}%
    \BdingB@xtrue%
    \extrairelepremi@r\p@int\de\list@num\PSwrit@cmdS{\p@int}{\c@mcurveto}{\fwf@g}{\X@qu}{\Y@qu}%
    \B@zierBB@x{1}{\Y@un}(\X@un,\X@de,\X@tr,\X@qu)%
    \B@zierBB@x{2}{\X@un}(\Y@un,\Y@de,\Y@tr,\Y@qu)%
    \edef\X@un{\X@qu}\edef\Y@un{\Y@qu}\bclB@zier\fi}
\ctr@ln@m\figdrawBezier
\ctr@ld@f\def\Q@BezierDD#1[#2]{\ifCUR@PS\ifGR@cri%
    \PSc@mment{BezierDD N arcs=#1, Control points=#2}%
    \iffillm@de\PsB@zier#1[#2]%
    \f@gfill%
    \else\PsB@zier#1[#2]\f@gstroke\fi%
    \PSc@mment{End BezierDD}\fi\fi}
\ctr@ln@m\et@tpsBezierTD
\ctr@ld@f\def\Q@BezierTD#1[#2]{\ifCUR@PS\ifGR@cri\s@uvc@ntr@l\et@tpsBezierTD%
    \PSc@mment{BezierTD N arcs=#1, Control points=#2}%
    \iffillm@de\PsB@zierTD#1[#2]%
    \f@gfill%
    \else\PsB@zierTD#1[#2]\f@gstroke\fi%
    \PSc@mment{End BezierTD}\resetc@ntr@l\et@tpsBezierTD\fi\fi}
\ctr@ld@f\def\PsB@zierTD#1[#2]{\ifnum\CUR@proj<\tw@\PsB@zier#1[#2]\else\PsB@zier@TD#1[#2]\fi}
\ctr@ld@f\def\PsB@zier@TD#1[#2]{{\f@gnewpath%
    \s@mme=\z@\def\list@num{#2,0}\extrairelepremi@r\p@int\de\list@num%
    \let\c@lprojSP=\relax\setc@ntr@l{2}\Figptpr@j-7:/\p@int/%
    \PSwrit@cmd{-7}{\c@mmoveto}{\fwf@g}%
    \loop\ifnum\s@mme<#1\advance\s@mme\@ne\extrairelepremi@r\p@intun\de\list@num%
    \extrairelepremi@r\p@intde\de\list@num\extrairelepremi@r\p@inttr\de\list@num%
    \subB@zierTD\NBz@rcs[\p@int,\p@intun,\p@intde,\p@inttr]\edef\p@int{\p@inttr}\repeat}}
\ctr@ld@f\def\subB@zierTD#1[#2,#3,#4,#5]{\delt@=\p@\divide\delt@\NBz@rcs\v@lmin=\z@%
    {\Figg@tXY{-7}\edef\X@un{\the\v@lX}\edef\Y@un{\the\v@lY}%
    \s@mme=\z@\loop\ifnum\s@mme<#1\advance\s@mme\@ne%
    \v@leur=\v@lmin\advance\v@leur0.33333 \delt@\edef\unti@rs{\repdecn@mb{\v@leur}}%
    \v@leur=\v@lmin\advance\v@leur0.66666 \delt@\edef\deti@rs{\repdecn@mb{\v@leur}}%
    \advance\v@lmin\delt@\edef\trti@rs{\repdecn@mb{\v@lmin}}%
    \figptBezierTD-8::\trti@rs[#2,#3,#4,#5]\Figptpr@j-8:/-8/%
    \c@lsubBzarc\unti@rs,\deti@rs[#2,#3,#4,#5]\BdingB@xfalse%
    \PSwrit@cmdS{-4}{}{\fwf@g}{\X@de}{\Y@de}\PSwrit@cmdS{-3}{}{\fwf@g}{\X@tr}{\Y@tr}%
    \BdingB@xtrue\PSwrit@cmdS{-8}{\c@mcurveto}{\fwf@g}{\X@qu}{\Y@qu}%
    \B@zierBB@x{1}{\Y@un}(\X@un,\X@de,\X@tr,\X@qu)%
    \B@zierBB@x{2}{\X@un}(\Y@un,\Y@de,\Y@tr,\Y@qu)%
    \edef\X@un{\X@qu}\edef\Y@un{\Y@qu}\figptcopyDD-7:/-8/\repeat}}
\ctr@ld@f\def\NBz@rcs{2}
\ctr@ld@f\def\c@lsubBzarc#1,#2[#3,#4,#5,#6]{\figptBezierTD-5::#1[#3,#4,#5,#6]%
    \figptBezierTD-6::#2[#3,#4,#5,#6]\Figptpr@j-4:/-5/\Figptpr@j-5:/-6/%
    \figptscontrolDD-4[-7,-4,-5,-8]}
\ctr@ln@m\figdrawcirc
\ctr@ld@f\def\Q@circDD#1(#2){\ifCUR@PS\ifGR@cri\PSc@mment{circDD Center=#1 (Radius=#2)}%
    \Q@arccircDD#1;#2(0,360)\PSc@mment{End circDD}\fi\fi}
\ctr@ld@f\def\Q@circTD#1,#2,#3(#4){\ifCUR@PS\ifGR@cri%
    \PSc@mment{circTD Center=#1,P1=#2,P2=#3 (Radius=#4)}%
    \Q@arccircTD#1,#2,#3;#4(0,360)\PSc@mment{End circTD}\fi\fi}
\ctr@ln@m\p@urcent
{\catcode`\%=12\gdef\p@urcent{
\ctr@ld@f\def\PSc@mment#1{\ifGRdebugm@de\immediate\write\fwf@g{\p@urcent\space#1}\fi}
\ctr@ln@m\acc@louv \ctr@ln@m\acc@lfer
{\catcode`\[=1\catcode`\{=12\gdef\acc@louv[{}}
{\catcode`\]=2\catcode`\}=12\gdef\acc@lfer{}]]
\ctr@ld@f\def\PSdict@{\ifUse@llipse%
    \immediate\write\fwf@g{/ellipsedict 9 dict def ellipsedict /mtrx matrix put}%
    \immediate\write\fwf@g{/ellipse \acc@louv ellipsedict begin}%
    \immediate\write\fwf@g{ /endangle exch def /startangle exch def}%
    \immediate\write\fwf@g{ /yrad exch def /xrad exch def}%
    \immediate\write\fwf@g{ /rotangle exch def /y exch def /x exch def}%
    \immediate\write\fwf@g{ /savematrix mtrx currentmatrix def}%
    \immediate\write\fwf@g{ x y translate rotangle rotate xrad yrad scale}%
    \immediate\write\fwf@g{ 0 0 1 startangle endangle arc}%
    \immediate\write\fwf@g{ savematrix setmatrix end\acc@lfer def}%
    \fi\PShe@der{EndProlog}}
\ctr@ld@f\def\Pssetc@rve#1=#2|{\keln@mun#1|%
    \def\n@mref{r}\ifx\l@debut\n@mref\update@ttr\D@FTroundness\Q@s@troundness{#2}\else
    \W@rnmesAttr{figset curve}{#1}\fi}
\ctr@ln@m\curv@roundness
\ctr@ld@f\def\Q@s@troundness#1{\edef\curv@roundness{#1}}
\ctr@ld@f\def\D@FTroundness{0.2} 
\ctr@ln@m\figdrawcurve
\ctr@ld@f\def\Q@curveDD[#1]{{\ifCUR@PS\ifGR@cri\PSc@mment{curveDD Points=#1}%
    \s@uvc@ntr@l\et@tpscurveDD%
    \iffillm@de\Psc@rveDD\curv@roundness[#1]%
    \f@gfill%
    \else\Psc@rveDD\curv@roundness[#1]\f@gstroke\fi%
    \PSc@mment{End curveDD}\resetc@ntr@l\et@tpscurveDD\fi\fi}}
\ctr@ld@f\def\Q@curveTD[#1]{{\ifCUR@PS\ifGR@cri%
    \PSc@mment{curveTD Points=#1}\s@uvc@ntr@l\et@tpscurveTD\let\c@lprojSP=\relax%
    \iffillm@de\Psc@rveTD\curv@roundness[#1]%
    \f@gfill%
    \else\Psc@rveTD\curv@roundness[#1]\f@gstroke\fi%
    \PSc@mment{End curveTD}\resetc@ntr@l\et@tpscurveTD\fi\fi}}
\ctr@ld@f\def\Psc@rveDD#1[#2]{%
    \def\list@num{#2}\extrairelepremi@r\Ak@\de\list@num%
    \extrairelepremi@r\Ai@\de\list@num\extrairelepremi@r\Aj@\de\list@num%
    \f@gnewpath\PSwrit@cmdS{\Ai@}{\c@mmoveto}{\fwf@g}{\X@un}{\Y@un}%
    \setc@ntr@l{2}\figvectPDD -1[\Ak@,\Aj@]%
    \@ecfor\Ak@:=\list@num\do{\figpttraDD-2:=\Ai@/#1,-1/\BdingB@xfalse%
       \PSwrit@cmdS{-2}{}{\fwf@g}{\X@de}{\Y@de}%
       \figvectPDD -1[\Ai@,\Ak@]\figpttraDD-2:=\Aj@/-#1,-1/%
       \PSwrit@cmdS{-2}{}{\fwf@g}{\X@tr}{\Y@tr}\BdingB@xtrue%
       \PSwrit@cmdS{\Aj@}{\c@mcurveto}{\fwf@g}{\X@qu}{\Y@qu}%
       \B@zierBB@x{1}{\Y@un}(\X@un,\X@de,\X@tr,\X@qu)%
       \B@zierBB@x{2}{\X@un}(\Y@un,\Y@de,\Y@tr,\Y@qu)%
       \edef\X@un{\X@qu}\edef\Y@un{\Y@qu}\edef\Ai@{\Aj@}\edef\Aj@{\Ak@}}}
\ctr@ld@f\def\Psc@rveTD#1[#2]{\ifnum\CUR@proj<\tw@\Psc@rvePPTD#1[#2]\else\Psc@rveCPTD#1[#2]\fi}
\ctr@ld@f\def\Psc@rvePPTD#1[#2]{\setc@ntr@l{2}%
    \def\list@num{#2}\extrairelepremi@r\Ak@\de\list@num\Figptpr@j-5:/\Ak@/%
    \extrairelepremi@r\Ai@\de\list@num\Figptpr@j-3:/\Ai@/%
    \extrairelepremi@r\Aj@\de\list@num\Figptpr@j-4:/\Aj@/%
    \f@gnewpath\PSwrit@cmdS{-3}{\c@mmoveto}{\fwf@g}{\X@un}{\Y@un}%
    \figvectPDD -1[-5,-4]%
    \@ecfor\Ak@:=\list@num\do{\Figptpr@j-5:/\Ak@/\figpttraDD-2:=-3/#1,-1/%
       \BdingB@xfalse\PSwrit@cmdS{-2}{}{\fwf@g}{\X@de}{\Y@de}%
       \figvectPDD -1[-3,-5]\figpttraDD-2:=-4/-#1,-1/%
       \PSwrit@cmdS{-2}{}{\fwf@g}{\X@tr}{\Y@tr}\BdingB@xtrue%
       \PSwrit@cmdS{-4}{\c@mcurveto}{\fwf@g}{\X@qu}{\Y@qu}%
       \B@zierBB@x{1}{\Y@un}(\X@un,\X@de,\X@tr,\X@qu)%
       \B@zierBB@x{2}{\X@un}(\Y@un,\Y@de,\Y@tr,\Y@qu)%
       \edef\X@un{\X@qu}\edef\Y@un{\Y@qu}\figptcopyDD-3:/-4/\figptcopyDD-4:/-5/}}
\ctr@ld@f\def\Psc@rveCPTD#1[#2]{\setc@ntr@l{2}%
    \def\list@num{#2}\extrairelepremi@r\Ak@\de\list@num%
    \extrairelepremi@r\Ai@\de\list@num\extrairelepremi@r\Aj@\de\list@num%
    \Figptpr@j-7:/\Ai@/%
    \f@gnewpath\PSwrit@cmd{-7}{\c@mmoveto}{\fwf@g}%
    \figvectPTD -9[\Ak@,\Aj@]%
    \@ecfor\Ak@:=\list@num\do{\figpttraTD-10:=\Ai@/#1,-9/%
       \figvectPTD -9[\Ai@,\Ak@]\figpttraTD-11:=\Aj@/-#1,-9/%
       \subB@zierTD\NBz@rcs[\Ai@,-10,-11,\Aj@]\edef\Ai@{\Aj@}\edef\Aj@{\Ak@}}}
\ctr@ld@f\def\figdrawend{\ifCUR@PS\ifGR@cri\immediate\closeout\fwf@g%
    \immediate\openout\fwf@g=\PSfilen@me\relax%
    \ifPDFm@ke\PSBdingB@x\else%
    \immediate\write\fwf@g{\p@urcent\string!PS-Adobe-2.0 EPSF-2.0}%
    \PShe@der{Creator\string: TeX (fig4tex.tex)}%
    \PShe@der{Title\string: \PSfilen@me}%
    \PShe@der{CreationDate\string: \the\day/\the\month/\the\year}%
    \PSBdingB@x%
    \PShe@der{EndComments}\PSdict@\fi%
    \immediate\write\fwf@g{\c@mgsave}%
    \openin\frf@g=\auxfilen@me\c@pypsfile\fwf@g\frf@g\closein\frf@g%
    \immediate\write\fwf@g{\c@mgrestore}%
    \PSc@mment{End of file.}\immediate\closeout\fwf@g%
    \immediate\openout\fwf@g=\auxfilen@me\immediate\closeout\fwf@g%
    \immediate\write16{File \PSfilen@me\space created.}\fi\fi\CUR@PSfalse\GR@critrue}
\ctr@ld@f\def\PShe@der#1{\immediate\write\fwf@g{\p@urcent\p@urcent#1}}
\ctr@ld@f\def\PSBdingB@x{{\v@lX=\ptT@ptps\c@@rdXmin\v@lY=\ptT@ptps\c@@rdYmin%
     \v@lXa=\ptT@ptps\c@@rdXmax\v@lYa=\ptT@ptps\c@@rdYmax%
     \PShe@der{BoundingBox\string: \repdecn@mb{\v@lX}\space\repdecn@mb{\v@lY}%
     \space\repdecn@mb{\v@lXa}\space\repdecn@mb{\v@lYa}}}}
\ctr@ld@f\def\figdrawfcconnect[#1]{{\ifCUR@PS\ifGR@cri\PSc@mment{fcconnect Points=#1}%
    \Q@s@tfillmode{no}\s@uvc@ntr@l\et@tpsfcconnect\resetc@ntr@l{2}%
    \fcc@nnect@[#1]\resetc@ntr@l\et@tpsfcconnect\PSc@mment{End fcconnect}\fi\fi}}
\ctr@ld@f\def\fcc@nnect@[#1]{\let\N@rm=\n@rmeucDD\def\list@num{#1}%
    \extrairelepremi@r\Ai@\de\list@num\edef\pr@m{\Ai@}\v@leur=\z@\p@rtent=\@ne\c@llgtot%
    \ifcase\fclin@typ@\edef\list@num{[\pr@m,#1,\Ai@}\expandafter\figdrawcurve\list@num]%
    \else\ifdim\fclin@r@d\p@>\z@\Pslin@conge[#1]\else\figdrawline[#1]\fi\fi%
    \v@leur=\@rrowp@s\v@leur\edef\list@num{#1,\Ai@,0}%
    \extrairelepremi@r\Ai@\de\list@num\mili@u=\epsil@n\c@llgpart%
    \advance\mili@u-\epsil@n\advance\mili@u-\delt@\advance\v@leur-\mili@u%
    \ifcase\fclin@typ@\invers@\mili@u\delt@%
    \ifnum\@rrowr@fpt>\z@\advance\delt@-\v@leur\v@leur=\delt@\fi%
    \v@leur=\repdecn@mb\v@leur\mili@u\edef\v@lt{\repdecn@mb\v@leur}%
    \extrairelepremi@r\Ak@\de\list@num%
    \figvectPDD-1[\pr@m,\Aj@]\figpttraDD-6:=\Ai@/\curv@roundness,-1/%
    \figvectPDD-1[\Ak@,\Ai@]\figpttraDD-7:=\Aj@/\curv@roundness,-1/%
    \delt@=\@rrowheadlength\p@\delt@=\C@AHANG\delt@\edef\R@dius{\repdecn@mb{\delt@}}%
    \ifcase\@rrowr@fpt%
    \FigptintercircB@zDD-8::\v@lt,\R@dius[\Ai@,-6,-7,\Aj@]\Q@arrowheadDD[-5,-8]\else%
    \FigptintercircB@zDD-8::\v@lt,\R@dius[\Aj@,-7,-6,\Ai@]\Q@arrowheadDD[-8,-5]\fi%
    \else\advance\delt@-\v@leur%
    \p@rtentiere{\p@rtent}{\delt@}\edef\C@efun{\the\p@rtent}%
    \p@rtentiere{\p@rtent}{\v@leur}\edef\C@efde{\the\p@rtent}%
    \figptbaryDD-5:[\Ai@,\Aj@;\C@efun,\C@efde]\ifcase\@rrowr@fpt%
    \delt@=\@rrowheadlength\unit@\delt@=\C@AHANG\delt@\edef\t@ille{\repdecn@mb{\delt@}}%
    \figvectPDD-2[\Ai@,\Aj@]\vecunit@{-2}{-2}\figpttraDD-5:=-5/\t@ille,-2/\fi%
    \Q@arrowheadDD[\Ai@,-5]\fi}
\ctr@ld@f\def\c@llgtot{\@ecfor\Aj@:=\list@num\do{\figvectP-1[\Ai@,\Aj@]\N@rm\delt@{-1}%
    \advance\v@leur\delt@\advance\p@rtent\@ne\edef\Ai@{\Aj@}}}
\ctr@ld@f\def\c@llgpart{\extrairelepremi@r\Aj@\de\list@num\figvectP-1[\Ai@,\Aj@]\N@rm\delt@{-1}%
    \advance\mili@u\delt@\ifdim\mili@u<\v@leur\edef\pr@m{\Ai@}\edef\Ai@{\Aj@}\c@llgpart\fi}
\ctr@ld@f\def\Pslin@conge[#1]{\ifnum\p@rtent>\tw@{\def\list@num{#1}%
    \extrairelepremi@r\Ai@\de\list@num\extrairelepremi@r\Aj@\de\list@num%
    \figptcopy-6:/\Ai@/\figvectP-3[\Ai@,\Aj@]\vecunit@{-3}{-3}\v@lmax=\result@t%
    \@ecfor\Ak@:=\list@num\do{\figvectP-4[\Aj@,\Ak@]\vecunit@{-4}{-4}%
    \minim@m\v@lmin\v@lmax\result@t\v@lmax=\result@t%
    \det@rm\delt@[-3,-4]\maxim@m\mili@u{\delt@}{-\delt@}\ifdim\mili@u>\Cepsil@n%
    \ifdim\delt@>\z@\figgetangleDD\Angl@[\Aj@,\Ak@,\Ai@]\else%
    \figgetangleDD\Angl@[\Aj@,\Ai@,\Ak@]\fi%
    \v@leur=\PI@deg\advance\v@leur-\Angl@\p@\divide\v@leur\tw@%
    \edef\Angl@{\repdecn@mb\v@leur}\c@ssin{\C@}{\S@}{\Angl@}\v@leur=\fclin@r@d\unit@%
    \v@leur=\S@\v@leur\mili@u=\C@\p@\invers@\mili@u\mili@u%
    \v@leur=\repdecn@mb{\mili@u}\v@leur%
    \minim@m\v@leur\v@leur\v@lmin\edef\t@ille{\repdecn@mb{\v@leur}}%
    \figpttra-5:=\Aj@/-\t@ille,-3/\figdrawline[-6,-5]\figpttra-6:=\Aj@/\t@ille,-4/%
    \figvectNVDD-3[-3]\figvectNVDD-8[-4]\inters@cDD-7:[-5,-3;-6,-8]%
    \ifdim\delt@>\z@\figdrawarccircP-7;\fclin@r@d[-5,-6]\else\figdrawarccircP-7;\fclin@r@d[-6,-5]\fi%
    \else\figdrawline[-6,\Aj@]\figptcopy-6:/\Aj@/\fi
    \edef\Ai@{\Aj@}\edef\Aj@{\Ak@}\figptcopy-3:/-4/}\figdrawline[-6,\Aj@]}\else\figdrawline[#1]\fi}
\ctr@ld@f\def\figdrawfcnode[#1]#2{{\ifCUR@PS\ifGR@cri\PSc@mment{fcnode Points=#1}%
    \s@uvc@ntr@l\et@tpsfcnode\resetc@ntr@l{2}%
    \def\t@xt@{#2}\ifx\t@xt@\empty\def\g@tt@xt{\setbox\Gb@x=\hbox{\Figg@tT{\p@int}}}%
    \else\def\g@tt@xt{\setbox\Gb@x=\hbox{#2}}\fi%
    \v@lmin=\h@rdfcXp@dd\advance\v@lmin\Xp@dd\unit@\multiply\v@lmin\tw@%
    \v@lmax=\h@rdfcYp@dd\advance\v@lmax\Yp@dd\unit@\multiply\v@lmax\tw@%
    \Figv@ctCreg-8(\unit@,-\unit@)\def\list@num{#1}%
    \delt@=\CUR@width bp\divide\delt@\tw@%
    \fcn@de\PSc@mment{End fcnode}\resetc@ntr@l\et@tpsfcnode\fi\fi}}
\ctr@ld@f\def\d@butn@de{\g@tt@xt\v@lX=\wd\Gb@x%
    \v@lY=\ht\Gb@x\advance\v@lY\dp\Gb@x\advance\v@lX\v@lmin\advance\v@lY\v@lmax}
\ctr@ld@f\def\fcn@deE{%
    \@ecfor\p@int:=\list@num\do{\d@butn@de\v@lX=\unssqrttw@\v@lX\v@lY=\unssqrttw@\v@lY%
    \ifdim\thickn@ss\p@>\z@
    \v@lXa=\v@lX\advance\v@lXa\delt@\v@lXa=\ptT@unit@\v@lXa\edef\XR@d{\repdecn@mb\v@lXa}%
    \v@lYa=\v@lY\advance\v@lYa\delt@\v@lYa=\ptT@unit@\v@lYa\edef\YR@d{\repdecn@mb\v@lYa}%
    \arct@n\v@leur(\v@lXa,\v@lYa)\v@leur=\rdT@deg\v@leur\edef\@nglde{\repdecn@mb\v@leur}%
    {\c@lptellDD-2::\p@int;\XR@d,\YR@d(\@nglde)}
    \advance\v@leur-\PI@deg\edef\@nglun{\repdecn@mb\v@leur}%
    {\c@lptellDD-3::\p@int;\XR@d,\YR@d(\@nglun)}%
    \figptstra-6=-3,-2,\p@int/\thickn@ss,-8/\Q@s@tfillmode{yes}%
    \Pss@tspecifSt{color=\DDV@thickcolor}%
    \figdrawline[-2,-3,-6,-5]\figdrawarcell-4;\XR@d,\YR@d(\@nglun,\@nglde,0)%
    \Psrest@reSt{color=\DDV@thickcolor}\fi
    \v@lX=\ptT@unit@\v@lX\v@lY=\ptT@unit@\v@lY%
    \edef\XR@d{\repdecn@mb\v@lX}\edef\YR@d{\repdecn@mb\v@lY}%
    \Q@s@tfillmode{yes}\Pss@tspecifSt{color=\fcbgc@lor}%
    \figdrawarcell\p@int;\XR@d,\YR@d(0,360,0)%
    \Q@s@tfillmode{no}\Psrest@reSt{color=\fcbgc@lor}\figdrawarcell\p@int;\XR@d,\YR@d(0,360,0)}}
\ctr@ld@f\def\fcn@deL{\delt@=\ptT@unit@\delt@\edef\t@ille{\repdecn@mb\delt@}%
    \@ecfor\p@int:=\list@num\do{\Figg@tXYa{\p@int}\d@butn@de%
    \ifdim\v@lX>\v@lY\itis@Ktrue\else\itis@Kfalse\fi%
    \advance\v@lXa-\v@lX\Figp@intreg-1:(\v@lXa,\v@lYa)%
    \advance\v@lXa\v@lX\advance\v@lYa-\v@lY\Figp@intreg-2:(\v@lXa,\v@lYa)%
    \advance\v@lXa\v@lX\advance\v@lYa\v@lY\Figp@intreg-3:(\v@lXa,\v@lYa)%
    \advance\v@lXa-\v@lX\advance\v@lYa\v@lY\Figp@intreg-4:(\v@lXa,\v@lYa)%
    \ifdim\thickn@ss\p@>\z@
    \Figg@tXYa{\p@int}\Q@s@tfillmode{yes}\Pss@tspecifSt{color=\DDV@thickcolor}%
    \c@lpt@xt{-1}{-4}\c@lpt@xt@\v@lXa\v@lYa\v@lX\v@lY\c@rre\delt@%
    \Figp@intregDD-9:(\v@lZ,\v@lYa)\Figp@intregDD-11:(\v@lZa,\v@lYa)%
    \c@lpt@xt{-4}{-3}\c@lpt@xt@\v@lYa\v@lXa\v@lY\v@lX\delt@\c@rre%
    \Figp@intregDD-12:(\v@lXa,\v@lZ)\Figp@intregDD-10:(\v@lXa,\v@lZa)%
    \ifitis@K\figptstra-7=-9,-10,-11/\thickn@ss,-8/\figdrawline[-9,-11,-5,-6,-7]\else%
    \figptstra-7=-10,-11,-12/\thickn@ss,-8/\figdrawline[-10,-12,-5,-6,-7]\fi%
    \Psrest@reSt{color=\DDV@thickcolor}\fi
    \Q@s@tfillmode{yes}\Pss@tspecifSt{color=\fcbgc@lor}\figdrawline[-1,-2,-3,-4]%
    \Q@s@tfillmode{no}\Psrest@reSt{color=\fcbgc@lor}\figdrawline[-1,-2,-3,-4,-1]}}
\ctr@ld@f\def\c@lpt@xt#1#2{\figvectN-7[#1,#2]\vecunit@{-7}{-7}\figpttra-5:=#1/\t@ille,-7/%
    \figvectP-7[#1,#2]\Figg@tXY{-7}\c@rre=\v@lX\delt@=\v@lY\Figg@tXY{-5}}
\ctr@ld@f\def\c@lpt@xt@#1#2#3#4#5#6{\v@lZ=#6\invers@{\v@lZ}{\v@lZ}\v@leur=\repdecn@mb{#5}\v@lZ%
    \v@lZ=#2\advance\v@lZ-#4\mili@u=\repdecn@mb{\v@leur}\v@lZ%
    \v@lZ=#3\advance\v@lZ\mili@u\v@lZa=-\v@lZ\advance\v@lZa\tw@#1}
\ctr@ld@f\def\fcn@deR{\@ecfor\p@int:=\list@num\do{\Figg@tXYa{\p@int}\d@butn@de%
    \advance\v@lXa-0.5\v@lX\advance\v@lYa-0.5\v@lY\Figp@intreg-1:(\v@lXa,\v@lYa)%
    \advance\v@lXa\v@lX\Figp@intreg-2:(\v@lXa,\v@lYa)%
    \advance\v@lYa\v@lY\Figp@intreg-3:(\v@lXa,\v@lYa)%
    \advance\v@lXa-\v@lX\Figp@intreg-4:(\v@lXa,\v@lYa)%
    \ifdim\thickn@ss\p@>\z@
    \Q@s@tfillmode{yes}\Pss@tspecifSt{color=\DDV@thickcolor}%
    \Figv@ctCreg-5(-\delt@,-\delt@)\figpttra-9:=-1/1,-5/%
    \Figv@ctCreg-5(\delt@,-\delt@)\figpttra-10:=-2/1,-5/%
    \Figv@ctCreg-5(\delt@,\delt@)\figpttra-11:=-3/1,-5/%
    \figptstra-7=-9,-10,-11/\thickn@ss,-8/\figdrawline[-9,-11,-5,-6,-7]%
    \Psrest@reSt{color=\DDV@thickcolor}\fi
    \Q@s@tfillmode{yes}\Pss@tspecifSt{color=\fcbgc@lor}\figdrawline[-1,-2,-3,-4]%
    \Q@s@tfillmode{no}\Psrest@reSt{color=\fcbgc@lor}\figdrawline[-1,-2,-3,-4,-1]}}
\ctr@ld@f\def\Pssetfl@wchart#1=#2|{\keln@mtr#1|%
    \def\n@mref{arr}\ifx\l@debut\n@mref\expandafter\keln@mtr\l@suite|%
     \def\n@mref{owp}\ifx\l@debut\n@mref\update@ttr\D@FTfcarrowposition\P@setfcarrowposition{#2}\else
     \def\n@mref{owr}\ifx\l@debut\n@mref\update@ttr\D@FTfcarrowrefpt\P@setfcarrowrefpt{#2}\else
     \W@rnmesAttr{figset flowchart}{#1}\fi\fi\else%
    \def\n@mref{bgc}\ifx\l@debut\n@mref\update@ttr\D@FTfcbgcolor\P@setfcbgcolor{#2}\else
    \def\n@mref{lin}\ifx\l@debut\n@mref\update@ttr\D@FTfcline\P@setfcline{#2}\else
    \def\n@mref{pad}\ifx\l@debut\n@mref\update@ttr\D@FTfcxpadding\P@setfcxpadding{#2}%
                                       \update@ttr\D@FTfcypadding\P@setfcypadding{#2}\else
    \def\n@mref{rad}\ifx\l@debut\n@mref\update@ttr\D@FTfcradius\P@setfcradius{#2}\else
    \def\n@mref{sha}\ifx\l@debut\n@mref\update@ttr\D@FTfcshape\P@setfcshape{#2}\else
    \def\n@mref{thi}\ifx\l@debut\n@mref\expandafter\keln@mtr\l@suite|%
     \def\n@mref{ckc}\ifx\l@debut\n@mref\update@ttr\D@FTref\P@setfcthickcolor{#2}\else
     \def\n@mref{ckn}\ifx\l@debut\n@mref\update@ttr\D@FTfcthickness\P@setfcthickness{#2}\else
     \W@rnmesAttr{figset flowchart}{#1}\fi\fi\else%
    \def\n@mref{xpa}\ifx\l@debut\n@mref\update@ttr\D@FTfcxpadding\P@setfcxpadding{#2}\else
    \def\n@mref{ypa}\ifx\l@debut\n@mref\update@ttr\D@FTfcypadding\P@setfcypadding{#2}\else
    \W@rnmesAttr{figset flowchart}{#1}\fi\fi\fi\fi\fi\fi\fi\fi\fi}
\ctr@ln@m\@rrowp@s
\ctr@ld@f\def\P@setfcarrowposition#1{\edef\@rrowp@s{#1}}
\ctr@ln@m\@rrowr@fpt
\ctr@ld@f\def\P@setfcarrowrefpt#1{\setfcr@fpt#1|}
\ctr@ld@f\def\setfcr@fpt#1#2|{\if#1e\def\@rrowr@fpt{1}\else\def\@rrowr@fpt{0}\fi}
\ctr@ln@m\fcbgc@lor
\ctr@ld@f\def\P@setfcbgcolor#1{\edef\fcbgc@lor{#1}}
\ctr@ln@m\fclin@typ@
\ctr@ld@f\def\P@setfcline#1{\setfccurv@#1|}
\ctr@ld@f\def\setfccurv@#1#2|{\if#1c\def\fclin@typ@{0}\else\def\fclin@typ@{1}\fi}
\ctr@ln@m\fclin@r@d
\ctr@ld@f\def\P@setfcradius#1{\edef\fclin@r@d{#1}}
\ctr@ln@m\fcn@de \ctr@ln@m\fcsh@pe
\ctr@ln@m\h@rdfcXp@dd \ctr@ln@m\h@rdfcYp@dd
\ctr@ld@f\def\P@setfcshape#1{\setfcshap@#1|}
\ctr@ld@f\def\setfcshap@#1#2|{%
    \if#1e\let\fcn@de=\fcn@deE\def\h@rdfcXp@dd{4pt}\def\h@rdfcYp@dd{4pt}%
     \edef\fcsh@pe{ellipse}\else%
    \if#1l\let\fcn@de=\fcn@deL\def\h@rdfcXp@dd{4pt}\def\h@rdfcYp@dd{4pt}%
     \edef\fcsh@pe{lozenge}\else%
          \let\fcn@de=\fcn@deR\def\h@rdfcXp@dd{6pt}\def\h@rdfcYp@dd{6pt}%
     \edef\fcsh@pe{rectangle}\fi\fi}
\ctr@ln@m\DDV@thickcolor
\ctr@ld@f\def\P@setfcthickcolor#1{\edef\DDV@thickcolor{#1}}
\ctr@ln@m\thickn@ss
\ctr@ld@f\def\P@setfcthickness#1{\edef\thickn@ss{#1}}
\ctr@ln@m\Xp@dd
\ctr@ld@f\def\P@setfcxpadding#1{\edef\Xp@dd{#1}}
\ctr@ln@m\Yp@dd
\ctr@ld@f\def\P@setfcypadding#1{\edef\Yp@dd{#1}}
\ctr@ld@f\def\figdrawline[#1]{{\ifCUR@PS\ifGR@cri\PSc@mment{line Points=#1}%
    \let\figdrawlign@=\Pslign@P\Pslin@{#1}\PSc@mment{End line}\fi\fi}}
\ctr@ld@f\def\figdrawlineF#1{{\ifCUR@PS\ifGR@cri\PSc@mment{lineF Filename=#1}%
    \let\figdrawlign@=\Pslign@F\Pslin@{#1}\PSc@mment{End lineF}\fi\fi}}
\ctr@ld@f\def\figdrawlineC(#1){{\ifCUR@PS\ifGR@cri\PSc@mment{lineC}%
    \let\figdrawlign@=\Pslign@C\Pslin@{#1}\PSc@mment{End lineC}\fi\fi}}
\ctr@ld@f\def\Pslin@#1{\iffillm@de\figdrawlign@{#1}%
    \f@gfill%
    \else\figdrawlign@{#1}\ifx\derp@int\premp@int%
    \f@gclosestroke%
    \else\f@gstroke\fi\fi}
\ctr@ld@f\def\Pslign@P#1{\def\list@num{#1}\extrairelepremi@r\p@int\de\list@num%
    \edef\premp@int{\p@int}\f@gnewpath%
    \PSwrit@cmd{\p@int}{\c@mmoveto}{\fwf@g}%
    \@ecfor\p@int:=\list@num\do{\PSwrit@cmd{\p@int}{\c@mlineto}{\fwf@g}%
    \edef\derp@int{\p@int}}}
\ctr@ld@f\def\Pslign@F#1{\s@uvc@ntr@l\et@tPslign@F\setc@ntr@l{2}\openin\frf@g=#1\relax%
    \ifeof\frf@g\message{*** File #1 not found !}\end\else%
    \read\frf@g to\tr@c\edef\premp@int{\tr@c}\expandafter\extr@ctCF\tr@c:%
    \f@gnewpath\PSwrit@cmd{-1}{\c@mmoveto}{\fwf@g}%
    \loop\read\frf@g to\tr@c\ifeof\frf@g\mored@tafalse\else\mored@tatrue\fi%
    \ifmored@ta\expandafter\extr@ctCF\tr@c:\PSwrit@cmd{-1}{\c@mlineto}{\fwf@g}%
    \edef\derp@int{\tr@c}\repeat\fi\closein\frf@g\resetc@ntr@l\et@tPslign@F}
\ctr@ln@m\extr@ctCF
\ctr@ld@f\def\extr@ctCFDD#1 #2:{\v@lX=#1\unit@\v@lY=#2\unit@\Figp@intregDD-1:(\v@lX,\v@lY)}
\ctr@ld@f\def\extr@ctCFTD#1 #2 #3:{\v@lX=#1\unit@\v@lY=#2\unit@\v@lZ=#3\unit@%
    \Figp@intregTD-1:(\v@lX,\v@lY,\v@lZ)}
\ctr@ld@f\def\Pslign@C#1{\s@uvc@ntr@l\et@tPslign@C\setc@ntr@l{2}%
    \def\list@num{#1}\extrairelepremi@r\p@int\de\list@num%
    \edef\premp@int{\p@int}\f@gnewpath%
    \expandafter\Pslign@C@\p@int:\PSwrit@cmd{-1}{\c@mmoveto}{\fwf@g}%
    \@ecfor\p@int:=\list@num\do{\expandafter\Pslign@C@\p@int:%
    \PSwrit@cmd{-1}{\c@mlineto}{\fwf@g}\edef\derp@int{\p@int}}%
    \resetc@ntr@l\et@tPslign@C}
\ctr@ld@f\def\Pslign@C@#1 #2:{{\def\t@xt@{#1}\ifx\t@xt@\empty\Pslign@C@#2:
    \else\extr@ctCF#1 #2:\fi}}
\ctr@ld@f\def\Pssetm@sh#1=#2|{\keln@mde#1|%
    \def\n@mref{co}\ifx\l@debut\n@mref\update@ttr\D@FTref\P@setmeshcolor{#2}\else
    \def\n@mref{da}\ifx\l@debut\n@mref\update@ttr\D@FTref\P@setmeshdash{#2}\else
    \def\n@mref{di}\ifx\l@debut\n@mref\update@ttr\D@FTmeshdiag\Q@s@tmeshdiag{#2}\else
    \def\n@mref{wi}\ifx\l@debut\n@mref\update@ttr\D@FTref\P@setmeshwidth{#2}\else
    \W@rnmesAttr{figset mesh}{#1}\fi\fi\fi\fi}
\ctr@ln@m\c@ntrolmesh
\ctr@ld@f\def\Q@s@tmeshdiag#1{\edef\c@ntrolmesh{#1}}
\ctr@ld@f\def\D@FTmeshdiag{0}    
\ctr@ln@m\DDV@meshcolor
\ctr@ld@f\def\P@setmeshcolor#1{\edef\DDV@meshcolor{#1}}
\ctr@ln@m\DDV@meshdash
\ctr@ld@f\def\P@setmeshdash#1{\edef\DDV@meshdash{#1}}
\ctr@ln@m\DDV@meshwidth
\ctr@ld@f\def\P@setmeshwidth#1{\edef\DDV@meshwidth{#1}}
\ctr@ld@f\def\figdrawmesh#1,#2[#3,#4,#5,#6]{{\ifCUR@PS\ifGR@cri%
    \PSc@mment{mesh N1=#1, N2=#2, Quadrangle=[#3,#4,#5,#6]}\s@uvc@ntr@l\et@tpsmesh%
    \Pss@tspecifSt{color=\DDV@meshcolor,dash=\DDV@meshdash,width=\DDV@meshwidth}%
    \setc@ntr@l{2}%
    \ifnum#1>\@ne\Psmeshp@rt#1[#3,#4,#5,#6]\fi%
    \ifnum#2>\@ne\Psmeshp@rt#2[#4,#5,#6,#3]\fi%
    \ifnum\c@ntrolmesh>\z@\Psmeshdi@g#1,#2[#3,#4,#5,#6]\fi%
    \ifnum\c@ntrolmesh<\z@\Psmeshdi@g#2,#1[#4,#5,#6,#3]\fi%
    \Psrest@reSt{color=\DDV@meshcolor,dash=\DDV@meshdash,width=\DDV@meshwidth}%
    \figdrawline[#3,#4,#5,#6,#3]\PSc@mment{End mesh}\resetc@ntr@l\et@tpsmesh\fi\fi}}
\ctr@ld@f\def\Psmeshp@rt#1[#2,#3,#4,#5]{{\l@mbd@un=\@ne\l@mbd@de=#1\loop%
    \ifnum\l@mbd@un<#1\advance\l@mbd@de\m@ne\figptbary-1:[#2,#3;\l@mbd@de,\l@mbd@un]%
    \figptbary-2:[#5,#4;\l@mbd@de,\l@mbd@un]\figdrawline[-1,-2]\advance\l@mbd@un\@ne\repeat}}
\ctr@ld@f\def\Psmeshdi@g#1,#2[#3,#4,#5,#6]{\figptcopy-2:/#3/\figptcopy-3:/#6/%
    \l@mbd@un=\z@\l@mbd@de=#1\loop\ifnum\l@mbd@un<#1%
    \advance\l@mbd@un\@ne\advance\l@mbd@de\m@ne\figptcopy-1:/-2/\figptcopy-4:/-3/%
    \figptbary-2:[#3,#4;\l@mbd@de,\l@mbd@un]%
    \figptbary-3:[#6,#5;\l@mbd@de,\l@mbd@un]\Psmeshdi@gp@rt#2[-1,-2,-3,-4]\repeat}
\ctr@ld@f\def\Psmeshdi@gp@rt#1[#2,#3,#4,#5]{{\l@mbd@un=\z@\l@mbd@de=#1\loop%
    \ifnum\l@mbd@un<#1\figptbary-5:[#2,#5;\l@mbd@de,\l@mbd@un]%
    \advance\l@mbd@de\m@ne\advance\l@mbd@un\@ne%
    \figptbary-6:[#3,#4;\l@mbd@de,\l@mbd@un]\figdrawline[-5,-6]\repeat}}
\ctr@ln@m\figdrawnormal
\ctr@ld@f\def\Q@normalDD#1,#2[#3,#4]{{\ifCUR@PS\ifGR@cri%
    \PSc@mment{normal Length=#1, Lambda=#2 [Pt1,Pt2]=[#3,#4]}%
    \s@uvc@ntr@l\et@tpsnormal\resetc@ntr@l{2}\figptendnormal-6::#1,#2[#3,#4]%
    \figptcopyDD-5:/-1/\figdrawarrow[-5,-6]%
    \PSc@mment{End normal}\resetc@ntr@l\et@tpsnormal\fi\fi}}
\ctr@ld@f\def\figreset#1{\trtlis@rg{#1}{\Psreset@}}
\ctr@ld@f\def\Psreset@#1|{\def\t@xt@{#1}\ifx\t@xt@\empty\P@resetg@n
    \else\keln@mde#1|%
    \def\n@mref{al}\ifx\l@debut\n@mref%
        \figset altitude(blcolor=default,bldash=default,blwidth=default,%
        sqcolor=default,sqdash=default,sqwidth=default)\else
    \def\n@mref{ar}\ifx\l@debut\n@mref\figresetarrowhead\else
    \def\n@mref{cu}\ifx\l@debut\n@mref\figset curve(roundness=\D@FTroundness)\else
    \def\n@mref{ge}\ifx\l@debut\n@mref\P@resetg@n\else
    \def\n@mref{fl}\ifx\l@debut\n@mref%
        \figset flowchart(arrowp=\D@FTfcarrowposition,arrowr=\D@FTfcarrowrefpt,%
	bgcolor=\D@FTfcbgcolor,line=\D@FTfcline,radius=\D@FTfcradius,%
	shape=\D@FTfcshape,thickcolor=default,thickness=\D@FTfcthickness,%
	xpadd=\D@FTfcxpadding,ypadd=\D@FTfcypadding)\else
    \def\n@mref{me}\ifx\l@debut\n@mref\figset mesh(diag=\D@FTmeshdiag,%
        color=default,dash=default,width=default)\else
    \def\n@mref{tr}\ifx\l@debut\n@mref%
        \figset trimesh(color=default,dash=default,width=default)\else
    \W@rnmeskwd{figreset}{#1}\fi\fi\fi\fi\fi\fi\fi\fi}
\ctr@ld@f\def\P@resetg@n{\figset (color=\D@FTcolor,dash=\D@FTdash,fill=\D@FTfill,%
    join=\D@FTjoin,width=\D@FTwidth)}
\ctr@ld@f\def\figset#1(#2){\def\t@xt@{#1}\ifx\t@xt@\empty\trtlis@rg{#2}{\Pssetg@n}
    \else\keln@mde#1|%
    \def\n@mref{al}\ifx\l@debut\n@mref\trtlis@rg{#2}{\Psset@lti}\else
    \def\n@mref{ar}\ifx\l@debut\n@mref\trtlis@rg{#2}{\Psset@rrowhe@d}\else
    \def\n@mref{cu}\ifx\l@debut\n@mref\trtlis@rg{#2}{\Pssetc@rve}\else
    \def\n@mref{fl}\ifx\l@debut\n@mref\trtlis@rg{#2}{\Pssetfl@wchart}\else
    \def\n@mref{ge}\ifx\l@debut\n@mref\trtlis@rg{#2}{\Pssetg@n}\else
    \def\n@mref{me}\ifx\l@debut\n@mref\trtlis@rg{#2}{\Pssetm@sh}\else
    \def\n@mref{pr}\ifx\l@debut\n@mref\ifCUR@PS\W@rnmesIgn{figset proj(...)}%
     \else\trtlis@rg{#2}{\Figsetpr@j}\fi\else
    \def\n@mref{tr}\ifx\l@debut\n@mref\trtlis@rg{#2}{\Pssettrim@sh}\else
    \def\n@mref{wr}\ifx\l@debut\n@mref\let\M@cro=\Figsetwr@te\trtlis@rgtok{#2,|}\else
    \W@rnmeskwd{figset}{#1}\fi\fi\fi\fi\fi\fi\fi\fi\fi\fi\ignorespaces}
\ctr@ld@f\def\figsetdefault#1(#2){\ifCUR@PS\W@rnmesIgn{figsetdefault}\else%
    \def\t@xt@{#1}\ifx\t@xt@\empty\trtlis@rg{#2}{\Pssd@g@n}\else\keln@mun#1|
    \def\n@mref{a}\ifx\l@debut\n@mref\trtlis@rg{#2}{\Pssd@@rrowhe@d}\else
    \def\n@mref{c}\ifx\l@debut\n@mref\trtlis@rg{#2}{\Pssd@c@rve}\else
    \def\n@mref{g}\ifx\l@debut\n@mref\trtlis@rg{#2}{\Pssd@g@n}\else
    \def\n@mref{f}\ifx\l@debut\n@mref\trtlis@rg{#2}{\Pssd@fl@wchart}\else
    \def\n@mref{m}\ifx\l@debut\n@mref\trtlis@rg{#2}{\Pssd@m@sh}\else
    \W@rnmeskwd{figsetdefault}{#1}\fi\fi\fi\fi\fi\fi\initpss@ttings\fi}
\ctr@ld@f\def\Pssd@g@n#1=#2|{\keln@mun#1|%
    \def\n@mref{c}\ifx\l@debut\n@mref\edef\D@FTcolor{#2}\else
    \def\n@mref{d}\ifx\l@debut\n@mref\edef\D@FTdash{#2}\else
    \def\n@mref{f}\ifx\l@debut\n@mref\edef\D@FTfill{#2}\else
    \def\n@mref{j}\ifx\l@debut\n@mref\edef\D@FTjoin{#2}\else
    \def\n@mref{u}\ifx\l@debut\n@mref\edef\D@FTupdate{#2}\Q@s@tupdate{#2}\else
    \def\n@mref{w}\ifx\l@debut\n@mref\edef\D@FTwidth{#2}\else
    \W@rnmesAttr{figsetdefault}{#1}\fi\fi\fi\fi\fi\fi}
\ctr@ld@f\def\Pssd@@rrowhe@d#1=#2|{\keln@mun#1|%
    \def\n@mref{a}\ifx\l@debut\n@mref\edef\D@FTarrowheadangle{#2}\else
    \def\n@mref{f}\ifx\l@debut\n@mref\edef\D@FTarrowheadfill{#2}\else
    \def\n@mref{l}\ifx\l@debut\n@mref\y@tiunit{#2}\ifunitpr@sent%
     \edef\D@FTh@rdahlength{#2}\else\edef\D@FTh@rdahlength{#2pt}%
     \message{*** \BS@ figsetdefault (..., #1=#2, ...) : unit is missing, pt is assumed.}%
     \fi\else
    \def\n@mref{o}\ifx\l@debut\n@mref\edef\D@FTarrowheadout{#2}\else
    \def\n@mref{r}\ifx\l@debut\n@mref\edef\D@FTarrowheadratio{#2}\else
    \W@rnmesAttr{figsetdefault arrowhead}{#1}\fi\fi\fi\fi\fi}
\ctr@ld@f\def\Pssd@c@rve#1=#2|{\keln@mun#1|%
    \def\n@mref{r}\ifx\l@debut\n@mref\edef\D@FTroundness{#2}\else%
    \W@rnmesAttr{figsetdefault curve}{#1}\fi}
\ctr@ld@f\def\Pssd@fl@wchart#1=#2|{\keln@mtr#1|%
    \def\n@mref{arr}\ifx\l@debut\n@mref\expandafter\keln@mtr\l@suite|%
     \def\n@mref{owp}\ifx\l@debut\n@mref\edef\D@FTfcarrowposition{#2}\else
     \def\n@mref{owr}\ifx\l@debut\n@mref\edef\D@FTfcarrowrefpt{#2}\else
                     \W@rnmesAttr{figsetdefault flowchart}{#1}\fi\fi\else%
    \def\n@mref{bgc}\ifx\l@debut\n@mref\edef\D@FTfcbgcolor{#2}\else
    \def\n@mref{lin}\ifx\l@debut\n@mref\edef\D@FTfcline{#2}\else
    \def\n@mref{pad}\ifx\l@debut\n@mref\edef\D@FTfcxpadding{#2}%
                    \edef\D@FTfcypadding{#2}\else
    \def\n@mref{rad}\ifx\l@debut\n@mref\edef\D@FTfcradius{#2}\else
    \def\n@mref{sha}\ifx\l@debut\n@mref\edef\D@FTfcshape{#2}\else
    \def\n@mref{thi}\ifx\l@debut\n@mref\expandafter\keln@mtr\l@suite|%
     \def\n@mref{ckn}\ifx\l@debut\n@mref\edef\D@FTfcthickness{#2}\else
                     \W@rnmesAttr{figsetdefault flowchart}{#1}\fi\else%
    \def\n@mref{xpa}\ifx\l@debut\n@mref\edef\D@FTfcxpadding{#2}\else
    \def\n@mref{ypa}\ifx\l@debut\n@mref\edef\D@FTfcypadding{#2}\else
    \W@rnmesAttr{figsetdefault flowchart}{#1}\fi\fi\fi\fi\fi\fi\fi\fi\fi}
\ctr@ld@f\def\D@FTfcarrowposition{0.5}
\ctr@ld@f\def\D@FTfcarrowrefpt{start}
\ctr@ld@f\def\D@FTfcbgcolor{1}
\ctr@ld@f\def\D@FTfcline{polygon}
\ctr@ld@f\def\D@FTfcradius{0}
\ctr@ld@f\def\D@FTfcshape{rectangle}
\ctr@ld@f\def\D@FTfcthickness{0}
\ctr@ld@f\def\D@FTfcxpadding{0}
\ctr@ld@f\def\D@FTfcypadding{0}
\ctr@ld@f\def\Pssd@m@sh#1=#2|{\keln@mun#1|%
    \def\n@mref{d}\ifx\l@debut\n@mref\edef\D@FTmeshdiag{#2}\else%
    \W@rnmesAttr{figsetdefault mesh}{#1}\fi}
\ctr@ln@w{newif}\iffillm@de
\ctr@ld@f\def\Q@s@tfillmode#1{\expandafter\setfillm@de#1:}
\ctr@ld@f\def\setfillm@de#1#2:{\if#1n\fillm@defalse\else\fillm@detrue\fi}
\ctr@ld@f\def\D@FTfill{no}     
\ctr@ln@w{newif}\ifGRupdatem@de
\ctr@ld@f\def\Q@s@tupdate#1{\ifCUR@PS\W@rnmesIgn{figset (update=...)}%
    \else\expandafter\setupd@te#1:\fi}
\ctr@ld@f\def\setupd@te#1#2:{\if#1n\GRupdatem@defalse\else\GRupdatem@detrue\fi}
\ctr@ld@f\def\D@FTupdate{no}     
\ctr@ln@m\CUR@color \ctr@ln@m\CUR@colorc@md
\ctr@ld@f\def\s@uvcolor#1{\edef#1{\CUR@color}}
\ctr@ld@f\def\D@FTcolor{0}       
\ctr@ld@f\def\Pssetc@lor#1{\ifGR@cri\result@tent=\@ne\expandafter\c@lnbV@l#1 :%
    \def\CUR@color{}\def\CUR@colorc@md{}%
    \ifcase\result@tent\or\Q@s@tgray{#1}\or\or\Q@s@trgb{#1}\or\Q@s@tcmyk{#1}\fi\fi}
\ctr@ln@m\CUR@colorc@mdStroke
\ctr@ld@f\def\Q@s@tcmyk#1{\ifGR@cri\def\CUR@color{#1}\def\CUR@colorc@md{\c@msetcmykcolor}%
    \def\CUR@colorc@mdStroke{\c@msetcmykcolorStroke}%
    \ifCUR@PS\PSc@mment{setcmyk Color=#1}\us@primarC@lor\fi\fi}
\ctr@ld@f\def\Q@s@trgb#1{\ifGR@cri\def\CUR@color{#1}\def\CUR@colorc@md{\c@msetrgbcolor}%
    \def\CUR@colorc@mdStroke{\c@msetrgbcolorStroke}%
    \ifCUR@PS\PSc@mment{setrgb Color=#1}\us@primarC@lor\fi\fi}
\ctr@ld@f\def\Q@s@tgray#1{\ifGR@cri\def\CUR@color{#1}\def\CUR@colorc@md{\c@msetgray}%
    \def\CUR@colorc@mdStroke{\c@msetgrayStroke}%
    \ifCUR@PS\PSc@mment{setgray Gray level=#1}\us@primarC@lor\fi\fi}
\ctr@ln@m\fillc@md
\ctr@ld@f\def\us@primarC@lor{\immediate\write\fwf@g{\d@fprimarC@lor}%
    \let\fillc@md=\prfillc@md}
\ctr@ld@f\def\prfillc@md{\d@fprimarC@lor\space\c@mfill}
\ctr@ld@f\def\c@lnbV@l#1 #2:{\def\t@xt@{#1}\relax\ifx\t@xt@\empty\c@lnbV@l#2:
    \else\c@lnbV@l@#1 #2:\fi}
\ctr@ld@f\def\c@lnbV@l@#1 #2:{\def\t@xt@{#2}\ifx\t@xt@\empty%
    \def\t@xt@{#1}\ifx\t@xt@\empty\advance\result@tent\m@ne\fi
    \else\advance\result@tent\@ne\c@lnbV@l@#2:\fi}
\ctr@ld@f\def\Blackcmyk{0 0 0 1}
\ctr@ld@f\def\Whitecmyk{0 0 0 0}
\ctr@ld@f\def\Cyancmyk{1 0 0 0}
\ctr@ld@f\def\Magentacmyk{0 1 0 0}
\ctr@ld@f\def\Yellowcmyk{0 0 1 0}
\ctr@ld@f\def\Redcmyk{0 1 1 0}
\ctr@ld@f\def\Greencmyk{1 0 1 0}
\ctr@ld@f\def\Bluecmyk{1 1 0 0}
\ctr@ld@f\def\Graycmyk{0 0 0 0.50}
\ctr@ld@f\def\BrickRedcmyk{0 0.89 0.94 0.28} 
\ctr@ld@f\def\Browncmyk{0 0.81 1 0.60} 
\ctr@ld@f\def\ForestGreencmyk{0.91 0 0.88 0.12} 
\ctr@ld@f\def\Goldenrodcmyk{ 0 0.10 0.84 0} 
\ctr@ld@f\def\Marooncmyk{0 0.87 0.68 0.32} 
\ctr@ld@f\def\Orangecmyk{0 0.61 0.87 0} 
\ctr@ld@f\def\Purplecmyk{0.45 0.86 0 0} 
\ctr@ld@f\def\RoyalBluecmyk{1. 0.50 0 0} 
\ctr@ld@f\def\Violetcmyk{0.79 0.88 0 0} 
\ctr@ld@f\def\Blackrgb{0 0 0}
\ctr@ld@f\def\Whitergb{1 1 1}
\ctr@ld@f\def\Redrgb{1 0 0}
\ctr@ld@f\def\Greenrgb{0 1 0}
\ctr@ld@f\def\Bluergb{0 0 1}
\ctr@ld@f\def\Cyanrgb{0 1 1}
\ctr@ld@f\def\Magentargb{1 0 1}
\ctr@ld@f\def\Yellowrgb{1 1 0}
\ctr@ld@f\def\Grayrgb{0.5 0.5 0.5}
\ctr@ld@f\def\Chocolatergb{0.824 0.412 0.118}
\ctr@ld@f\def\DarkGoldenrodrgb{0.722 0.525 0.043}
\ctr@ld@f\def\DarkOrangergb{1 0.549 0}
\ctr@ld@f\def\Firebrickrgb{0.698 0.133 0.133}
\ctr@ld@f\def\ForestGreenrgb{0.133 0.545 0.133}
\ctr@ld@f\def\Goldrgb{1 0.843 0}
\ctr@ld@f\def\HotPinkrgb{1 0.412 0.706}
\ctr@ld@f\def\Maroonrgb{0.690 0.188 0.376}
\ctr@ld@f\def\Pinkrgb{1 0.753 0.796}
\ctr@ld@f\def\RoyalBluergb{0.255 0.412 0.882}
\ctr@ld@f\def\Pssetg@n#1=#2|{\keln@mun#1|%
    \def\n@mref{c}\ifx\l@debut\n@mref\update@ttr\D@FTcolor\Pssetc@lor{#2}\else
    \def\n@mref{d}\ifx\l@debut\n@mref\update@ttr\D@FTdash\Q@s@tdash{#2}\else
    \def\n@mref{f}\ifx\l@debut\n@mref\update@ttr\D@FTfill\Q@s@tfillmode{#2}\else
    \def\n@mref{j}\ifx\l@debut\n@mref\update@ttr\D@FTjoin\Q@s@tjoin{#2}\else
    \def\n@mref{u}\ifx\l@debut\n@mref\update@ttr\D@FTupdate\Q@s@tupdate{#2}\else
    \def\n@mref{w}\ifx\l@debut\n@mref\update@ttr\D@FTwidth\Q@s@twidth{#2}\else
    \W@rnmesAttr{figset}{#1}\fi\fi\fi\fi\fi\fi}
\ctr@ln@m\CUR@dash
\ctr@ld@f\def\s@uvdash#1{\edef#1{\CUR@dash}}
\ctr@ld@f\def\D@FTdash{1}        
\ctr@ld@f\def\Q@s@tdash#1{\ifGR@cri\edef\CUR@dash{#1}\ifCUR@PS\expandafter\Pssetd@sh#1 :\fi\fi}
\ctr@ld@f\def\Pssetd@shI#1{\PSc@mment{setdash Index=#1}\ifcase#1%
    \or\immediate\write\fwf@g{[] 0 \c@msetdash}
    \or\immediate\write\fwf@g{[6 2] 0 \c@msetdash}
    \or\immediate\write\fwf@g{[4 2] 0 \c@msetdash}
    \or\immediate\write\fwf@g{[2 2] 0 \c@msetdash}
    \or\immediate\write\fwf@g{[1 2] 0 \c@msetdash}
    \or\immediate\write\fwf@g{[2 4] 0 \c@msetdash}
    \or\immediate\write\fwf@g{[3 5] 0 \c@msetdash}
    \or\immediate\write\fwf@g{[3 3] 0 \c@msetdash}
    \or\immediate\write\fwf@g{[3 5 1 5] 0 \c@msetdash}
    \or\immediate\write\fwf@g{[6 4 2 4] 0 \c@msetdash}
    \fi}
\ctr@ld@f\def\Pssetd@sh#1 #2:{{\def\t@xt@{#1}\ifx\t@xt@\empty\Pssetd@sh#2:
    \else\def\t@xt@{#2}\ifx\t@xt@\empty\Pssetd@shI{#1}\else\s@mme=\@ne\def\debutp@t{#1}%
    \an@lysd@sh#2:\ifodd\s@mme\edef\debutp@t{\debutp@t\space\finp@t}\def\finp@t{0}\fi%
    \PSc@mment{setdash Pattern=#1 #2}%
    \immediate\write\fwf@g{[\debutp@t] \finp@t\space\c@msetdash}\fi\fi}}
\ctr@ld@f\def\an@lysd@sh#1 #2:{\def\t@xt@{#2}\ifx\t@xt@\empty\def\finp@t{#1}\else%
    \edef\debutp@t{\debutp@t\space#1}\advance\s@mme\@ne\an@lysd@sh#2:\fi}
\ctr@ln@m\CUR@width
\ctr@ld@f\def\s@uvwidth#1{\edef#1{\CUR@width}}
\ctr@ld@f\def\D@FTwidth{0.4}     
\ctr@ld@f\def\Q@s@twidth#1{\ifGR@cri\edef\CUR@width{#1}\ifCUR@PS%
    \PSc@mment{setwidth Width=#1}\immediate\write\fwf@g{#1 \c@msetlinewidth}\fi\fi}
\ctr@ln@m\CUR@join
\ctr@ld@f\def\s@uvjoin#1{\edef#1{\CUR@join}}
\ctr@ld@f\def\D@FTjoin{miter}   
\ctr@ld@f\def\Q@s@tjoin#1{\ifGR@cri\edef\CUR@join{#1}\ifCUR@PS\expandafter\Pssetj@in#1:\fi\fi}
\ctr@ld@f\def\Pssetj@in#1#2:{\PSc@mment{setjoin join=#1}%
    \if#1r\def\t@xt@{1}\else\if#1b\def\t@xt@{2}\else\def\t@xt@{0}\fi\fi%
    \immediate\write\fwf@g{\t@xt@\space\c@msetlinejoin}}
\ctr@ld@f\def\Pss@tspecifSt#1{\trtlis@rg{#1}{\Pss@tspecifSt@}}
\ctr@ld@f\def\Pss@tspecifSt@#1=#2|{\keln@mun#1|%
    \def\n@mref{c}\ifx\l@debut\n@mref\def\n@mref{#2}\ifx\n@mref\D@FTref\else%
     \s@uvcolor{\typ@color}\Pssetc@lor{#2}\fi\else
    \def\n@mref{d}\ifx\l@debut\n@mref\def\n@mref{#2}\ifx\n@mref\D@FTref\else%
     \s@uvdash{\typ@dash}\Q@s@tdash{#2}\fi\else
    \def\n@mref{j}\ifx\l@debut\n@mref\def\n@mref{#2}\ifx\n@mref\D@FTref\else%
     \s@uvjoin{\typ@join}\Q@s@tjoin{#2}\fi\else
    \def\n@mref{w}\ifx\l@debut\n@mref\def\n@mref{#2}\ifx\n@mref\D@FTref\else%
     \s@uvwidth{\typ@width}\Q@s@twidth{#2}\fi\else
    \W@rnmeskwd{Pss@tspecifSt}{#1}\fi\fi\fi\fi}
\ctr@ld@f\def\Psrest@reSt#1{\trtlis@rg{#1}{\Psrest@reSt@}}
\ctr@ld@f\def\Psrest@reSt@#1=#2|{\keln@mun#1|%
    \def\n@mref{c}\ifx\l@debut\n@mref\def\n@mref{#2}\ifx\n@mref\D@FTref\else%
     \Pssetc@lor{\typ@color}\fi\else
    \def\n@mref{d}\ifx\l@debut\n@mref\def\n@mref{#2}\ifx\n@mref\D@FTref\else%
     \Q@s@tdash{\typ@dash}\fi\else
    \def\n@mref{j}\ifx\l@debut\n@mref\def\n@mref{#2}\ifx\n@mref\D@FTref\else%
     \Q@s@tjoin{\typ@join}\fi\else
    \def\n@mref{w}\ifx\l@debut\n@mref\def\n@mref{#2}\ifx\n@mref\D@FTref\else%
     \Q@s@twidth{\typ@width}\fi\else
    \W@rnmeskwd{Psrest@reSt}{#1}\fi\fi\fi\fi}
\ctr@ld@f\def\Pssettrim@sh#1=#2|{\keln@mde#1|%
    \def\n@mref{co}\ifx\l@debut\n@mref\update@ttr\D@FTref\P@settmeshcolor{#2}\else
    \def\n@mref{da}\ifx\l@debut\n@mref\update@ttr\D@FTref\P@settmeshdash{#2}\else
    \def\n@mref{wi}\ifx\l@debut\n@mref\update@ttr\D@FTref\P@settmeshwidth{#2}\else
    \W@rnmesAttr{figset trimesh}{#1}\fi\fi\fi}
\ctr@ln@m\DDV@tmeshcolor
\ctr@ld@f\def\P@settmeshcolor#1{\edef\DDV@tmeshcolor{#1}}
\ctr@ln@m\DDV@tmeshdash
\ctr@ld@f\def\P@settmeshdash#1{\edef\DDV@tmeshdash{#1}}
\ctr@ln@m\DDV@tmeshwidth
\ctr@ld@f\def\P@settmeshwidth#1{\edef\DDV@tmeshwidth{#1}}
\ctr@ld@f\def\figdrawtrimesh#1[#2,#3,#4]{{\ifCUR@PS\ifGR@cri%
    \PSc@mment{trimesh Type=#1, Triangle=[#2,#3,#4]}%
    \s@uvc@ntr@l\et@tpstrimesh\ifnum#1>\@ne%
    \Pss@tspecifSt{color=\DDV@tmeshcolor,dash=\DDV@tmeshdash,width=\DDV@tmeshwidth}%
    \setc@ntr@l{2}%
    \Pstrimeshp@rt#1[#2,#3,#4]\Pstrimeshp@rt#1[#3,#4,#2]\Pstrimeshp@rt#1[#4,#2,#3]%
    \Psrest@reSt{color=\DDV@tmeshcolor,dash=\DDV@tmeshdash,width=\DDV@tmeshwidth}%
    \fi\figdrawline[#2,#3,#4,#2]%
    \PSc@mment{End trimesh}\resetc@ntr@l\et@tpstrimesh\fi\fi}}
\ctr@ld@f\def\Pstrimeshp@rt#1[#2,#3,#4]{{\l@mbd@un=\@ne\l@mbd@de=#1\loop\ifnum\l@mbd@de>\@ne%
    \advance\l@mbd@de\m@ne\figptbary-1:[#2,#3;\l@mbd@de,\l@mbd@un]%
    \figptbary-2:[#2,#4;\l@mbd@de,\l@mbd@un]\figdrawline[-1,-2]%
    \advance\l@mbd@un\@ne\repeat}}
\initpr@lim\initpss@ttings\initPDF@rDVI
\ctr@ln@w{newbox}\figBoxA
\ctr@ln@w{newbox}\figBoxB
\ctr@ln@w{newbox}\figBoxC
\catcode`\@=12

\figsetdefault(update=yes)

\usepackage{amsmath}
\usepackage{amsfonts}
\usepackage{amsbsy}
\usepackage{amssymb}
\usepackage[normalem]{ulem}
\usepackage{times}
\usepackage{mathrsfs}
\usepackage{exscale}
\usepackage{graphicx}
\usepackage[colorlinks,citecolor=blue,linkcolor=rouge,
            bookmarksopen,
            bookmarksnumbered
           ]{hyperref}

\usepackage{times}
\renewcommand{\arraystretch}{1.25}

\usepackage{color}
\definecolor{gr}{rgb}   {0.,   0.6,   0.25 }
\definecolor{mg}{rgb}   {0.85,  0.,    0.85}
\definecolor{marin}{rgb}   {0.,   0.,   0.8}
\definecolor{rouge}{rgb}   {0.8,   0.,   0.}
\definecolor{orange}{rgb}   {0.8,   0.4,   0.}

\newcommand{\Gr}{\color{gr}}
\newcommand{\Mg}{\color{mg}}
\newcommand{\Bk}{\color{black}}
\newcommand{\Rd}{\color{red}}
\newcommand{\Bl}{\color{marin}}
\newcommand{\Or}{\color{orange}}

\newtheorem{theorem}{Theorem}[section]
\newtheorem{lemma}[theorem]{Lemma}
\newtheorem{proposition}[theorem]{Proposition}
\newtheorem{corollary}[theorem]{Corollary}

\theoremstyle{definition}
\newtheorem{definition}[theorem]{Definition}
\newtheorem{notation}[theorem]{Notation}

\theoremstyle{remark}
\newtheorem{remark}[theorem]{Remark}

\numberwithin{equation}{section}

\newcommand{\ee}{\hskip0.15ex}
\newcommand{\me}{\hskip-0.15ex}
\newcommand{\dd}[1]{_{\raise-0.3ex\hbox{$\scriptstyle #1$}}}
\newcommand{\di}{\displaystyle}
\newcommand{\on}[1]{\raise-.5ex\hbox{\big|}_{#1}}
\renewcommand{\Re}{\operatorname{\rm Re}}
\renewcommand{\Im}{\operatorname{\rm Im}}

\newcommand{\round}[1]{\lfloor #1\ee\rceil}

\newcommand{\vpha}{\left.\vphantom{T^{j_0}_{j_0}}\!\!\right.}

\newcommand {\Norm}[2]{ \mathchoice
    {|\ee #1\ee|\dd{#2}}
    {| #1 |_{#2}}
    {| #1 |_{#2}}
    {| #1 |_{#2}} }
\newcommand {\DNorm}[2]{ \mathchoice
    {\|\ee #1\ee\|\dd{#2}}
    {\| #1 \|_{#2}}
    {\| #1 \|_{#2}}
    {\| #1 \|_{#2}} }
\newcommand {\Normc}[2]{ \mathchoice
    {|\ee #1\ee|\dd{#2}^2}
    {| #1 |_{#2}^2}
    {| #1 |_{#2}^2}
    {| #1 |_{#2}^2} }
\newcommand {\DNormc}[2]{ \mathchoice
    {\|\ee #1\ee\|\dd{#2}^2}
    {\| #1 \|_{#2}^2}
    {\| #1 \|_{#2}^2}
    {\| #1 \|_{#2}^2} }

\renewcommand{\div}{\operatorname{\rm div}}
\newcommand{\grad}{\operatorname{\textbf{grad}}}
\newcommand{\curl}{\operatorname{\textbf{curl}}}
\newcommand{\curls}{\operatorname{\rm curl}}
\newcommand{\Cinf}{\mathscr{C}^\infty}

\newcommand\bB{{\mathbb B}}
\newcommand\C{{\mathbb C}}
\newcommand\R{{\mathbb R}}
\newcommand\N{{\mathbb N}}
\newcommand\Q{{\mathbb Q}}
\newcommand\T{{\mathbb T}}
\newcommand\Z{{\mathbb Z}}

\newcommand{\sC}{{\mathscr C}}
\newcommand{\sB}{{\mathscr B}}
\newcommand{\sD}{{\mathscr D}}
\newcommand{\sE}{{\mathscr E}}
\newcommand{\sL}{{\mathscr L}}
\newcommand{\sS}{{\mathscr S}}
\newcommand{\sZ}{{\mathscr Z}}

\renewcommand{\Re}{\operatorname{Re}}
\renewcommand{\Im}{\operatorname{Im}}

\newcommand\cA{{\mathcal{A}}}
\newcommand\cB{{\mathcal{B}}}
\newcommand\cC{{\mathcal{C}}}
\newcommand\cD{{\mathcal{D}}}
\newcommand\cE{{\mathcal{E}}}
\newcommand\cG{{\mathcal{G}}}
\newcommand\cI{{\mathcal{I}}}
\newcommand\cH{{\mathcal{H}}}
\newcommand\cJ{{\mathcal{J}}}
\newcommand\cK{{\mathcal{K}}}
\newcommand\cL{{\mathcal{L}}}
\newcommand\cM{{\mathcal{M}}}
\newcommand\cO{{\mathcal{O}}}
\newcommand\cR{{\mathcal{R}}}
\newcommand\cS{{\mathcal{S}}}
\newcommand\cT{{\mathcal{T}}}
\newcommand\cU{{\mathcal{U}}}
\newcommand\cV{{\mathcal{V}}}
\newcommand\cZ{{\mathcal{Z}}}

\newcommand{\rA}{{\mathsf{A}}}
\newcommand{\rB}{{\mathsf{B}}}
\newcommand{\mD}{{\mathrm{D}}}
\newcommand{\rF}{{\mathsf{F}}}
\newcommand{\rG}{{\mathsf{G}}}
\newcommand{\rH}{{\mathsf{H}}}
\newcommand{\rM}{{\mathsf{M}}}
\newcommand{\rL}{{\mathsf{L}}}
\newcommand{\mP}{{\mathrm{P}}}
\newcommand{\rP}{{\mathsf{P}}}
\newcommand{\rQ}{{\mathsf{Q}}}
\newcommand{\rR}{{\mathsf{R}}}
\newcommand{\mS}{{\mathrm{S}}}
\newcommand{\rS}{{\mathsf S}}
\newcommand{\rT}{{\mathsf{T}}}
\newcommand{\rU}{{\mathsf{U}}}
\newcommand{\rV}{{\mathsf{V}}}
\newcommand{\rW}{{\mathsf{W}}}

\newcommand{\rb}{{\mathsf{b}}}
\newcommand{\rc}{{\mathsf{c}}}
\newcommand{\rd}{{\mathrm d}}
\newcommand{\re}{{\mathrm e}}
\newcommand{\rh}{{\mathsf{g}}}
\newcommand{\rs}{{\mathsf{s}}}
\newcommand{\ru}{{\mathsf{u}}}
\newcommand{\rv}{{\mathsf{v}}}

\newcommand {\gA}{\mathfrak{A}}
\newcommand {\gC}{\mathfrak{C}}
\newcommand {\gE}{\mathfrak{E}}
\newcommand {\gF}{\mathfrak{F}}
\newcommand {\gL}{{\mathfrak L}}
\newcommand {\gM}{{\mathfrak M}}
\newcommand {\gP}{{\mathfrak P}}
\newcommand {\gR}{{\mathfrak R}}
\newcommand {\gU}{{\mathfrak U}}
\newcommand {\gW}{{\mathfrak W}}
\newcommand {\gZ}{{\mathfrak Z}}

\newcommand{\tu}{{\tilde{u}}}

\newcommand {\Id}{\mathbb I}
\newcommand{\comp}{\mathrm{comp}}
\newcommand{\diam}{\mathrm{diam}}
\newcommand{\dist}{\mathrm{dist}}
\newcommand{\meas}{\mathrm{meas}}

\usepackage{soul}

\begin{document}

\title[Converging expansions]{\bf Converging expansions for Lipschitz\\ self-similar perforations of a plane sector}

\author{Martin Costabel}
\address{IRMAR UMR 6625 du CNRS, Universit\'{e} de Rennes 1, Campus de Beaulieu,
35042 Rennes Cedex, France}
\email{martin.costabel@univ-rennes1.fr}
\author{Matteo Dalla Riva}
\address{Department of Mathematics, The University of Tulsa, 800 South Tucker Drive, Tulsa, Oklahoma 74104, USA}
\email{matteo-dallariva@gmail.com}
\author{Monique Dauge}
\address{IRMAR UMR 6625 du CNRS, Universit\'{e} de Rennes 1, Campus de Beaulieu,
35042 Rennes Cedex, France}
\email{monique.dauge@univ-rennes1.fr}
\author{Paolo Musolino}
\address{Department of Mathematics, Aberystwyth University, Ceredigion SY23 3BZ, Wales, UK}
\email{musolinopaolo@gmail.com}



\begin{abstract}
 In contrast with the well-known methods of matching asymptotics and multiscale (or compound) asymptotics, the ``functional analytic approach'' of Lanza de Cristoforis (Analysis 28, 2008) allows to prove convergence of expansions around interior small holes of size $\varepsilon$ for solutions of elliptic boundary value problems. Using the method of layer potentials, the asymptotic behavior of the solution as $\varepsilon$ tends to zero is described not only by asymptotic series in powers of $\varepsilon$, but by convergent power series. 
Here we use this method to investigate the Dirichlet problem for the Laplace operator where holes are collapsing at a polygonal corner of opening $\omega$. Then in addition to the scale $\varepsilon$ there appears the scale $\eta=\varepsilon^{\pi/\omega}$. We prove that when $\pi/\omega$ is irrational, the solution of the Dirichlet problem is given by convergent series in powers of these two small parameters. 
Due to interference of the two scales, this convergence is obtained, in full generality, by grouping together integer powers of the two scales that are very close to each other. Nevertheless, there exists a dense subset of openings $\omega$ (characterized by Diophantine approximation properties), for which real analyticity in the two variables $\varepsilon$ and $\eta$ holds and the power series converge unconditionally. When $\pi/\omega$ is rational, the series are unconditionally convergent, but contain terms in $\log\varepsilon$. 
\end{abstract}
\maketitle

{
\parskip 1pt
\tableofcontents
}


\section*{Introduction}
\label{s:1}

Domains with small holes are fundamental examples of singularly perturbed domains. The analysis of the asymptotic behavior of elliptic boundary value problems in such perforated domains as the size of the holes tends to zero lays the basis for numerous applications in more involved situations that can be found in the classical monographs \cite{Il92}, \cite{MaNaPl00}, \cite{KoMaMo99} and the more recent \cite{AmKa07}. The two methods that are most widely spread are the {\em matching of asymptotic expansions} as exposed by Il'in \cite{Il92}, and the method of {\em multiscale} (or {\em compound}) expansions as in Maz'ya, Nazarov, and Plamenevskij \cite{MaNaPl00} or Kozlov, Maz'ya, and Movchan \cite{KoMaMo99}. Ammari and Kang \cite{AmKa07} use the method of layer potentials to construct asymptotic expansions.
When the holes are shrinking to the corner of a polygonal domain, one encounters the class of self-similar singular perturbations, a case that has been treated by Maz'ya, Nazarov, and Plamenevskij \cite[Ch.2]{MaNaPl00} with the method of compound expansions, and by Dauge, Tordeux, and Vial \cite{DaToVi10} with both methods of matched and compound expansions. 
 The common feature of these methods is their algorithmic and constructive nature: The terms of the asymptotic expansions are constructed according to a sequential order. At each step of the construction, remainder estimates are proved, but there is no uniform control of the remainders, 
and this does therefore not lead to a convergence proof.

Another method appeared recently, based on the ``functional analytic approach'' introduced by Lanza de Cristoforis \cite{La02}. This method has so far mainly been applied to  the Laplace equation on domains with holes collapsing to interior points. The core feature is a description of the solution as a real analytic function of one or several variables depending on the small parameter $\varepsilon$ that characterizes the size of the holes. This is proved by a reduction to the boundary via integral equations, and after a careful analysis of the boundary integral operators the analytic implicit function theorem can be invoked, providing the expansion of the solutions into convergent series. Analytic functions of several variables appear for example in \cite{La08,DaMu15}, where the two-dimensional Dirichlet problem leads to the introduction of the scale $1/\log\varepsilon$ besides $\varepsilon$, or in \cite{DaMu16}, where a boundary value problem in a domain with moderately close holes is studied 
and the size of the holes and their distance are defined by small parameters that may be of different size.

Our aim in this paper is to understand how this method would apply to the Dirichlet problem in a polygonal domain when holes are shrinking to the corner in a self-similar manner. In the limit $\varepsilon\to0$, the singular behavior of solutions at corners without holes will combine with the singular perturbation of the geometry. 
In contrast to what happens in the case of holes collapsing at interior points or at smooth boundary points \cite{BoDaDaMu16}, we find that the series expansions in powers of $\varepsilon$ that correspond to the asymptotic expansions of \cite{DaToVi10} are only  ``stepwise convergent''. For  corner opening angles $\omega$ that are rational multiples of $\pi$, the series will be unconditionally convergent, but in general for irrational multiples of $\pi$, certain pairs of terms in the series may have to be grouped together in order to achieve convergence. This is a peculiar feature similar to, and in the end caused by, the stepwise convergence of the asymptotic expansion of the solution of boundary value problems near corners when the data are analytic \cite{BraDau82,Dauge84}.

\subsection{Geometric setting}\label{ss:geomset}
We consider perforated domains where the holes are shrinking towards a point of the boundary that is the vertex of a plane sector. For the sake of simplicity, we try to concentrate on the essential features and avoid unnecessary generality. Therefore we consider only one corner, but we admit several holes.

We denote by $t=(t_1,t_2)$  the Cartesian coordinates in the plane $\R^2$, and by $(\rho=|t|,\vartheta=\arg(t))$ the polar coordinates. The open ball with center $0$ and radius $\rho_0$ is denoted by $\sB(0,\rho_0)$. Let the opening angle $\omega$ be chosen in $(0,2\pi)$ and denote by $\rS_\omega$ the infinite sector
\begin{equation}
\label{eq:1E1}
   \rS_\omega = \{t\in\R^2,\quad \vartheta\in(0,\omega)\}.
\end{equation}
The case $\omega=\pi$ is degenerate and corresponds to a half-plane. 

The perforated domains $\rA_{\varepsilon}$ are determined by an unperforated domain $\rA$, a hole pattern $\rP$ and scale factors $\varepsilon$, about which we make some hypotheses.

The \emph{unperforated domain} $\rA$ satisfies the following assumptions, see Fig.\ref{fig:1} left,
\begin{enumerate}
\item $\rA$ is a subset of the sector $\rS_\omega$ and coincides with it near its vertex:
\begin{equation}
\label{eq:1E2}
   \exists\rho_0>0 \quad\mbox{such that}\quad \sB(0,\rho_0)\cap \rA = \sB(0,\rho_0)\cap \rS_\omega,
\end{equation}
\item $\rA$ is bounded, simply connected, and has a Lipschitz boundary,
\item $\rS_\omega\setminus\rA$ has a Lipschitz boundary,
\item $\partial\rA\cap\partial\rS_\omega$ is connected.
\end{enumerate}

The \emph{hole pattern} $\rP$ satisfies, see Fig.\ref{fig:1} right,
\begin{enumerate}
\item $\rP$ is a subset of the sector $\rS_\omega$ and its complement $\rS_\omega\setminus\rP$ coincides with $\rS_\omega$ at infinity:
\begin{equation}
\label{eq:1E3}
   \exists\rho'_0>0 \quad\mbox{such that}\quad 
   \rP\subset\rS_\omega\cap\sB(0,\rho'_0).
\end{equation}
\item $\rP$ is a finite union of bounded simply connected Lipschitz domains $\rP_j$, $j=1,\ldots,J$,
\item $\rS_\omega\setminus\rP$ has a Lipschitz boundary,
\item For any $j\in\{1,\ldots,J\}$, $\partial\rP_j\cap\partial\rS_\omega$ is connected.
\end{enumerate}

\begin{figure}[ht]
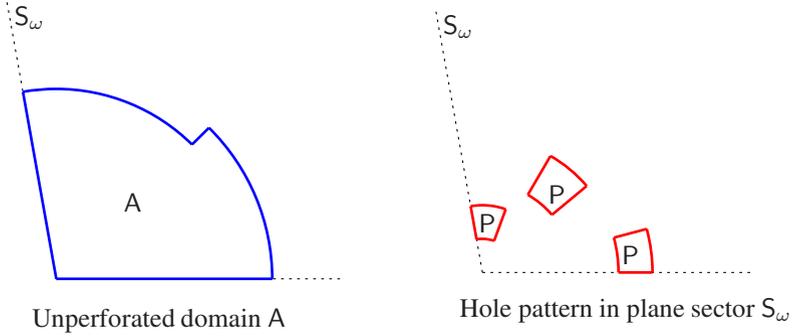

\begin{minipage}{0.49\textwidth}
\input Fig4_perfsector.tex
\end{minipage}
\begin{minipage}{0.49\textwidth}
\input Fig1_perfsector.tex
\end{minipage}
\caption{Limit domain $\rA$ and hole pattern $\rP$.}
\label{fig:1}
\end{figure}

Let $\varepsilon_0=\rho_0/\rho'_0$. 
The family of perforated domains  $\big(\rA_\varepsilon\big)_{0<\varepsilon<\varepsilon_0}$ is defined by, see Fig.\ref{fig:2},
\begin{equation}
\label{eq:Aeps}
   \rA_\varepsilon = \rA \setminus \varepsilon\overline\rP,\quad\mbox{for}\quad 
   0<\varepsilon<\varepsilon_0.
\end{equation}
The family $\varepsilon\rP$ can be seen as a self-similar collection of holes concentrating at the vertex of the sector.
Here, in contrast with \cite{DaMu15} we do not assume that $0$ belongs to $\rP$. We do not even assume that $0$ does not belong to $\partial\rP$. 

Our assumptions \eqref{eq:1E2}, \eqref{eq:1E3} exclude some classes of self-similar perturbations of corner domains that are also interesting to study and have been analyzed using different methods, see \cite{DaToVi10}. For example, condition (2) in \eqref{eq:1E3} excludes the case of the approximation of a sharp corner by rounded corners constructed with circles of radius $\varepsilon$. 
Condition (3) in \eqref{eq:1E3} excludes holes touching the boundary in a point.
The Lipschitz regularity conditions (2) and (3) in \eqref{eq:1E2}, \eqref{eq:1E3} are essential for our boundary integral equation approach. 
On the other hand, the condition that $\rA$ is simply connected and the related connectivity conditions (4) in \eqref{eq:1E2}, \eqref{eq:1E3} are not essential, they are merely made for simplicity of notation.

\begin{figure}[ht]
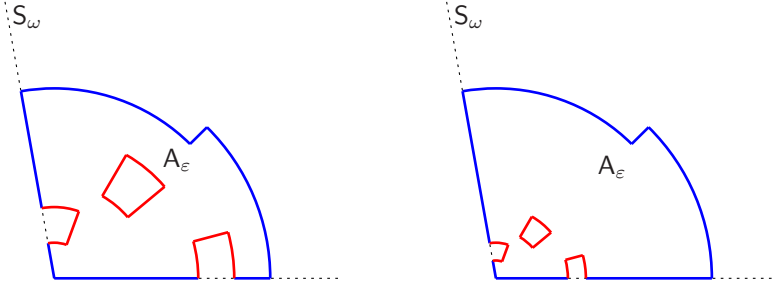

\begin{minipage}{0.49\textwidth}
\input Fig5_perfsector.tex
\end{minipage}
\begin{minipage}{0.49\textwidth}
\input Fig6_perfsector.tex
\end{minipage}
\caption{Perforated domain $\rA_\varepsilon$ for two values of $\varepsilon$.}
\label{fig:2}
\end{figure}

\subsection{Dirichlet problems and mutiscale expansions}
We are interested in the collective behavior of solutions of the family of Poisson problems
\begin{equation}
\label{eq:poisson}
\begin{cases}
\begin{array}{rcll}
   \Delta u_\varepsilon &=& f \quad& \mbox{in}\quad \rA_\varepsilon,\\
   u_\varepsilon &=& 0 \quad& \mbox{on}\quad \partial\rA_\varepsilon.
\end{array}
\end{cases}
\end{equation}
We assume that the common right hand side $f$ is an element of $L^2(\rA)$, which, by restriction to $\rA_\varepsilon$, defines an element of $L^2(\rA_\varepsilon)$ and provides a unique solution $u_\varepsilon\in H^1_0(\rA_\varepsilon)$ to problem \eqref{eq:poisson}. 

If moreover $f$ is infinitely smooth on $\overline\rA$ {in a neighborhood of the origin}, then a description of the $\varepsilon$-behavior of $u_\varepsilon$ can be performed in terms of {\em multiscale asymptotic expansions}. We refer to \cite{MaNaPl00,DaToVi10} which apply to the present situation. As a result of this approach, cf \cite[Th.\,4.1 \& Sect.\,7.1]{DaToVi10}, $u_\varepsilon$ can be described by an asymptotic expansion containing two sorts of terms:
\begin{itemize}
\item Slow terms $u^\beta(t)$, defined in the standard variables $t$
\item Rapid terms, or {\em profiles}, $U^\beta(\frac{t}{\varepsilon})$, defined in the rapid variable $\frac{t}{\varepsilon}$.
\end{itemize}
Here the exponent $\beta$ runs in the set 
$\,\N+\frac{\pi}{\omega}\N=\{\ell+k\frac\pi\omega,\: k,\ell\in\N\}$. 

If $\frac{\pi}{\omega}$ is not a rational number, $u_\varepsilon$ can be expanded in powers of $\varepsilon$ 
\begin{equation}
\label{eq:multisca}
   u_\varepsilon(t) \simeq \sum_{\beta\in\N +\frac{\pi}{\omega}\N} \varepsilon^\beta\,u^\beta(t)
   + \sum_{\beta\in\N +\frac{\pi}{\omega}\N} \varepsilon^\beta\,U^\beta(\tfrac{t}{\varepsilon})\,.
\end{equation}
The sums are asymptotic series, which means the following here: 
Let $(\beta_n)_{n\in\N}$ be the strictly increasing enumeration of $\N+\frac{\pi}{\omega}\N$ and define the $N$th partial sum by
\begin{equation}
\label{eq:multiscN}
    u_\varepsilon^{[N]}(t) = \sum_{n=0}^N \varepsilon^{\beta_n}\,u^{\beta_n}(t)
   + \sum_{n=0}^N \varepsilon^{\beta_n}\,U^{\beta_n}(\tfrac{t}{\varepsilon})
\,.
\end{equation}
Then for all $ N\in\N$ there exists $C_{N}$ such that for all $\varepsilon\in(0,\varepsilon_1]$
\begin{equation}
\label{eq:multiscb}
   \big\Vert u_\varepsilon - u_\varepsilon^{[N]} \big\Vert_{H^1(\rA_\varepsilon)} \le C_N\,\varepsilon^{\beta_{N+1}}
\end{equation}
where we have chosen $\varepsilon_1<\varepsilon_0$. 

If $\frac{\pi}{\omega}$ is a rational number, the terms corresponding to $\beta$ in the intersection $\beta\in\N\cap\frac{\pi}{\omega}\N_*$ contain a $\log\varepsilon$ and the estimate \eqref{eq:multiscb} has to be modified accordingly.

\subsection{Convergence analysis}

{If we want to have convergence of the series \eqref{eq:multisca}, it is not enough that the right hand side $f$ belongs to $L^2(\rA)$ and not even that it is infinitely smooth near the origin, but in addition its asymptotic expansion (Taylor series) at the origin needs to be a convergent series converging to $f$. Thus we have to assume, and we will do this from now on,  
that $f$ 
}
has an extension as a real analytic function in a neighborhood of the origin.
More specifically, we assume that there exist two positive constants $M_f$ and $C$ so that {$f\in L^2(\rA)$ and}
\begin{equation}
\label{eq:f}
   f(t) = \sum_{\alpha\in\N^2} f_\alpha\, t_1^{\alpha_1}t_2^{\alpha_2} ,
   \quad \forall t\in \sB(0,M_f^{-1})\cap\rA,\quad\mbox{with}\quad
   |f_\alpha| \le C M_f^{|\alpha|}.
\end{equation}
{A simple special case would be a right hand side $f\in L^2(\rA)$ that vanishes in a neighborhood of the origin.  Likewise, one could consider, as in \cite{La08,DaMu15,BoDaDaMu16}, a variant of the boundary value problem \eqref{eq:poisson} that is driven not by a domain force $f$, but by a given trace on the boundary.}

In the present work we address the question of the convergence of the series \eqref{eq:multisca}
{under the assumption \eqref{eq:f}. In the above references \cite{MaNaPl00,DaToVi10} the recursive construction of the terms $u^\beta$ and $U^\beta$ of \eqref{eq:multisca} is performed without control of the constants $C_N$ in function of $N$, thus without providing any information on the convergence of the asymptotic series. 
We will exploit the ``functional analytic approach'' to obtain this convergence.
}
 
It follows from general properties of power series that convergence of \eqref{eq:multisca} in the sense that 
\[
 \lim_{N\to\infty}  \big\Vert u_\varepsilon -  u_\varepsilon^{[N]} \big\Vert =0
\]
in some norm and for some $\varepsilon=\varepsilon_1>0$ implies that the series converges absolutely and unconditionally for any $\varepsilon\in(-\varepsilon_1,\varepsilon_1)$.
It will follow from our analysis that there exists a set $\Lambda_\rs$ of real irrational numbers (super-exponential Liouville numbers, see Definition~\ref{def:Liouville}) with the property that whenever the opening angle $\omega$ does not belong to $\pi\Lambda_\rs$, then such an $\varepsilon_1>0$ does indeed exist. For $\omega\in \pi\Lambda_\rs$ on the other hand, in general the series \eqref{eq:multisca} does not converge for any $\varepsilon\ne0$. It is known from classical number theory that both $\Lambda_\rs$ and its complement are uncountable and dense in $\R$ and $\Lambda_\rs$ is of Lebesgue measure zero and even of Hausdorff dimension zero. 
The series can be made convergent, however, for any $\omega\in(0,2\pi)$ by grouping together certain pairs of terms in the sums for which $\beta_{n+1}-\beta_n$ is small. This situation can also be expressed by the fact that there exists a subsequence $(N_k)_{k\in\N}$ of $\N$ such that for any $\varepsilon\in(-\varepsilon_1,\varepsilon_1)$
\[
 \lim_{k\to\infty}  \big\Vert u_\varepsilon -  u_\varepsilon^{[N_k]} \big\Vert =0 \,.
\]
We call this kind of convergence ``stepwise convergence'', and the main result of this paper is the construction of a convergent series in this sense.

Our analysis relies on four main steps, developed in the four sections of this paper. 

{\sc Step 1.} We set $\tu_\varepsilon = u_\varepsilon-u_0\on{\rA_\varepsilon}$, where
$u_0\in H^1_0(\rA)$ is the solution of the limit problem
\begin{equation}
\label{eq:u0}
\begin{cases}
\begin{array}{rcll}
   \Delta u_0 &=& f \quad& \mbox{in}\quad \rA,\\
   u_0 &=& 0 \quad& \mbox{on}\quad \partial\rA\,.
\end{array}
\end{cases}
\end{equation}
Doing this, we reduce our investigation to the harmonic function $\tu_\varepsilon$, solution of the problem
\begin{equation}
\label{eq:tue}
\begin{cases}
\begin{array}{rcll}
   \Delta \tu_\varepsilon &=& 0 \quad& \mbox{in}\quad \rA_\varepsilon,\\
   \tu_\varepsilon &=& -u_0 \quad& \mbox{on}\quad \partial\rA_\varepsilon,
\end{array}
\end{cases}
\end{equation}
Since $u_0$ is zero on $\partial\rA$, the trace of $u_0$ on $\partial\rA_\varepsilon$ can be nonzero only on the boundary of the holes $\varepsilon\partial\rP$. In order to analyze this trace, we expand $u_0$ near the origin in quasi-homogeneous terms with respect to the distance $\rho$ to the vertex according to the classical Kondrat'ev theory \cite{Kondratev67}. The investigation of the possible convergence of this series is far less classical, see \cite{BraDau82}, and it may involve stepwise convergent series.
 This issue is also related to the stability of the terms in the expansion with respect to the opening, cf \cite{CostabelDauge93c,CostabelDauge94}. We provide rather explicit formulas for such expansions in complex variable form.

{\sc Step 2.} We transform problem \eqref{eq:tue} into a similar problem on a perforated domain for which the holes shrink to an \emph{interior} point of the limit domain, a situation studied in \cite{La08,DaMu15,DaMu16}. To get there,  we compose two transformations that are compatible with the Dirichlet Laplacian, 
\begin{itemize}
\item A conformal map of power type,
\item An odd reflection.
\end{itemize}
In this way the unperturbed sector domain $\rA$ is transformed into a bounded simply connected Lipschitz domain $\rB$ that contains the origin, and the hole pattern $\rP$ is transformed into another hole pattern $\rQ$ that is a finite union of simply connected bounded Lipschitz domains $\rQ_j$. The small parameter is transformed into another small parameter $\eta$ by the power law
\[
   \eta = \varepsilon^{\pi/\omega},
\]
and the new perforated domains $\rB_\eta$ have the form
\[
   \rB_\eta = \rB\setminus\eta\overline\rQ,\quad \eta\in(0,\eta_0).
\]
The boundary of $\rB_\eta$ is the disjoint union of the external part $\partial\rB$ and the boundary $\eta \ee\partial\rQ$ of the holes $\eta\rQ$. The holes shrink to the origin $0$, which now lies in the interior of the unperforated domain $\rB$.

In this way problems \eqref{eq:tue} are transformed into Dirichlet problems on $\rB_\eta$
\begin{equation}
\label{eq:pB}
\begin{cases}
\begin{array}{rcll}
   \Delta v_\eta &=& 0 \quad& \mbox{in}\quad \rB_\eta,\\
   v_\eta &=& \mu_\eta \quad& \mbox{on}\quad \partial\rB_\eta\,.
\end{array}
\end{cases}
\end{equation}
The family of Dirichlet traces $\mu_\eta$ are determined by the trace of $u_0$ on the family of boundaries $\varepsilon\partial\rP$ of the holes. They have a special structure due to the mirror symmetry.

{\sc Step 3.} We study analytic families of model problems of this type where $\mu_\eta$ depends on $\eta$ as follows
\begin{equation}
\label{eq:psi}
\begin{cases}
   \mu_\eta(x) = \psi(x) & \mbox{if $x\in\partial\rB$,} \\
   \mu_\eta(x) = \Psi(\frac{x}{\eta}) & \mbox{if $x\in\eta \ee\partial\rQ$.}\\
\end{cases}
\end{equation}
Here we have a clear separation between the external boundary $\partial\rB$ where $\mu_\eta$ does not depend on $\eta$, and the internal boundary $\eta \ee\partial\rQ$ of $\partial\rB_\eta$ that is the boundary of the scaled holes $\eta\rQ$. Via representation formulas involving the double layer potential, 
we transform the problem \eqref{eq:pB} with right hand side \eqref{eq:psi} into an equivalent system of boundary integral equations \eqref{eq:Meta} with a matrix of boundary integral operators $\cM(\eta)$ depending analytically on $\eta$ in a neighborhood of zero and such that $\cM(0)$ is invertible, see Theorem~\ref{thm:Meta=0}.

The crucial property making this possible is the homogeneity of the double layer kernel, which allows to write $\cM(\eta)$ in such a form that its diagonal terms are independent of $\eta$ and the off-diagonal terms vanish at $\eta=0$. The problem corresponding to the boundary integral operator $\cM(0)$ can be interpreted as a decoupled system of Dirichlet problems, one (in slow variables $x$) on the unperturbed domain $\rB$ and a second one (in rapid variables $X=\frac{x}\eta$) on the complement $\R^2\setminus\overline\rQ$ of the holes at $\eta=1$.

{\sc Step 4.}
From the formulation via an analytic family of boundary integral equations in Step~ 3 follows that there exists $\eta_1>0$ such that the solutions $v_\eta$ of problems \eqref{eq:pB}-\eqref{eq:psi} depend on the data $\psi\in H^{1/2}(\partial\rB)$ and $\Psi\in H^{1/2}(\partial\rQ)$ via a solution operator $\cL(\eta)$ that is analytic in $\eta$ for $\eta\in(-\eta_1,\eta_1)$ and, therefore, is given by a convergent series around $0$
\begin{equation}
\label{eq:gL}
   \cL(\eta) = \sum_{n=0}^\infty \eta^n\cL_n,\quad
   |\eta|\le\eta_1\,.
\end{equation}

Combining this with the results of Steps 1 and 2, we obtain expansions of the solutions $u_\varepsilon$ of problem \eqref{eq:poisson} in slow and rapid variables similar to \eqref{eq:multisca} that are not only asymptotic series as $\varepsilon\to0$ like in \eqref{eq:multiscb}, but convergent for $\varepsilon$ in a neighborhood of zero. The convergence is shown in weighted Sobolev norms and it is, in general, ``stepwise'' in the same sense as had been known for the convergence of the expansion in corner singular functions of the solution $u_0$ of the unperturbed problem \eqref{eq:u0}. 
The main results on this kind of convergent expansions are given in Theorems~\ref{thm:ueps}, \ref{thm:outer} and \ref{thm:inner}.

While the series in powers of $\varepsilon$ are, under our general conditions, not unconditionally convergent due to the interaction of integer powers coming from the Taylor expansion of the right hand side $f$ in our problem \eqref{eq:poisson} and the powers of the form $k\pi/\omega$, $k\in\N$, coming from the corner singularities, there are two situations where the convergence is, in fact, unconditional. \\[0.5ex]
The first such situation is met when the opening angle $\omega$ is such that $\pi/\omega$ is either a rational number or, conversely, is not approximated too fast by rational numbers, namely not a super-exponential Liouville number as defined in Definition~\ref{def:Liouville}. In this case, the right hand side $f$ can be arbitrary, as long as it is analytic in a neighborhood of the corner,
see Corollaries \ref{cor:omegaQ} and \ref{cor:irrat}. 
\\[0.5ex]
The second situation where we find unconditional convergence is met for arbitrary opening angles $\omega$ when the right hand side $f$ in \eqref{eq:poisson} vanishes in a neighborhood of the corner: Then we have the {\em converging expansion} in $L^\infty(\Omega)$
\begin{equation}
\label{eq:conv}
   u_\varepsilon(t) = u_0(t)+\sum_{\beta\in \frac{\pi}{\omega}\N_*} \varepsilon^\beta\,u^\beta(t)
   + \sum_{\beta\in\frac{\pi}{\omega}\N_*} \varepsilon^\beta\,U^\beta(\tfrac{t}{\varepsilon}).
\end{equation}
Thus both parts of this two-scale decomposition of $u_\varepsilon$ are given by functions that are real analytic near zero in the variable $\eta=\varepsilon^{\pi/\omega}$, see Corollary~\ref{cor:f=0}.


\section{Unperturbed problem on a plane sector}
\label{s:2}
We are going to analyze the solution $u_0$ of problem \eqref{eq:u0} when the right hand side satisfies the assumption \eqref{eq:f}. We represent $u_0$ as the sum of three series converging in a neighborhood of the vertex:
\[
   u_0 = u_{f} + u_{\partial} + u_{\sf rm}
\]
where
\begin{enumerate}
\item $u_{f}$ is a particular solution of $\Delta u = f$, 
\item $u_{\partial}$ is a particular solution of $\Delta u = 0$, with $u_{\partial}+u_{f}=0$ on the sides $\vartheta = 0$ or $\omega$,
\item $u_{\sf rm}$ is the remaining part of $u_0$.
\end{enumerate}
We use the complex variable form of Cartesian coordinates
\begin{equation}
\label{eq:zeta}
   \zeta = t_1 + it_2,\ \ \bar \zeta = t_1-it_2\quad\mbox{i.e.}\quad
   \zeta = \rho e^{i\vartheta}.
\end{equation}
In particular, instead of \eqref{eq:f}, we write the Taylor expansion at origin of $f$ in the form
\begin{equation}
\label{eq:fcomp}
   f(t) = \sum_{\alpha\in\N^2} \tilde f_\alpha\, \zeta^{\alpha_1}\bar\zeta{}^{\alpha_2} 
   \ \mbox{ in } \ \sB(0,M_f^{-1}),\quad\mbox{with}\quad
   |\tilde f_\alpha| \le  C_M M^{|\alpha|},\ (M>M_f).
\end{equation}

\subsection{Interior particular solution}
The existence of a real analytic particular solution to the equation $\Delta u = f$ is a consequence of classical regularity results (cf.~Morrey and Nirenberg \cite{MoNi57}). Nevertheless, we can also provide an easy direct proof by an explicit formula using the complex variable representation \eqref{eq:fcomp}: It suffices to set
\begin{equation}
\label{eq:uf}
   u_{f}(t) = \sum_{\alpha\in\N^2} \, 
   \frac{\tilde f_\alpha}{4(\alpha_1+1)(\alpha_2+1)}\,\zeta^{\alpha_1+1}\bar\zeta{}^{\alpha_2+1} 
   \ \mbox{ in } \ \sB(0,M_f^{-1}).
\end{equation}
to obtain a particular real analytic solution to the equation $\Delta u = f$ in $\sB(0,M_f^{-1})$.

\subsection{Lateral particular solution}\label{ss:1.2}
Set $\rho_1=\min\{\rho_0,M^{-1}\}$ for a chosen $M>M_f$. In the finite sector $\rA\cap\sB(0,\rho_1)$, the difference $\tilde u\equiv u_0-u_f$ is a harmonic function and its traces on the sides $\vartheta=0$ and $\vartheta=\omega$ coincide with $-u_f$. 
Denote by $g^0$ and $g^\omega$ the restriction of $-u_f$ on the rays $\vartheta=0$ and $\vartheta=\omega$. These two functions are analytic in the variable $\rho$:
\begin{equation}
\label{eq:estG}
   g^0 = \sum_{\ell\in\N_*} g^0_\ell \rho^\ell,\quad
   g^\omega = \sum_{\ell\in\N_*} g^\omega_\ell \rho^\ell,\quad
   |g^0_\ell| + |g^\omega_\ell| \le C \rho_1^{-\ell}.
\end{equation}
The constant $\rho_1>0$, which is a lower bound for the convergence radius of the power series \eqref{eq:estG}, is by construction less than $M_f^{-1}$, where $M_f$ is the constant in the assumption \eqref{eq:f} on the analyticity of the right hand side $f$.  
Note that, by construction, $g^0$ and $g^\omega$ vanish at the origin.

As a next step in the analysis of $u_0$, we now construct a particular solution $u_{\partial}$ of the problem satisfied by $\tilde u$
\begin{equation}
\label{eq:Dirlat}
\left\{
\begin{array}{rcll}
   \Delta u_{\partial}(t)&=&0 \quad& \forall t \in \rA\cap \sB(0,\rho_1)\,,\\
   u_{\partial}(t)&=&-u_{f}(t) \quad& \forall t \in (\rT_0\cup \rT_\omega) \cap \sB(0,\rho_1)\,. \\
\end{array}
\right.
\end{equation}
This solution uses the convergent series expansion \eqref{eq:estG} and will be given as a convergent series, too.

Following \cite{BraDau82}, for any positive integer $\ell\in\N_*$ we can write explicit particular solutions $w_\ell$ to the Dirichlet problem in the infinite sector $\rS_\omega$
\begin{equation}
\label{eq:bvpellsect}
\left\{
\begin{array}{rcll}
   \Delta w_\ell(t)&=&0 \quad& \forall t \in \rS_\omega\,,\\
   w_\ell(t)&=&g^0_\ell \rho^\ell \quad& \forall t \in \rT_0 \,, \\
   w_\ell(t)&=&g^\omega_\ell \rho^\ell \quad& \forall t \in \rT_\omega \,,
\end{array}
\right.
\end{equation}
where $\rT_0$ and $\rT_\omega$ are the two sides of the sector $\rS_\omega$.

The idea is then to give estimates of the $w_\ell$ that show convergence of the series
$  u_\partial=\sum_{\ell\in\N_*}w_\ell\,.
$

There exists always a (quasi-)homogeneous solution of degree $\ell$. The harmonic functions that are homogeneous of degree $\ell$ are
\[
   \Im\zeta^\ell\quad\mbox{and}\quad\Re\zeta^\ell
\]
They are given in polar coordinates by $\rho^\ell\sin\ell\vartheta$ and $\rho^\ell\cos\ell\vartheta$. The determinant of their boundary values is  $\sin \ell\omega$. If this is zero, we cannot solve \eqref{eq:bvpellsect} in homogeneous functions (except in the smooth case, i.e. when $\omega=\pi$, where we find that $w_\ell = b_\ell \, \Re\zeta^\ell$ with $b_\ell = g^0_\ell$ is a solution), but we need the quasihomogeneous function 
\[
  \Im(\zeta^\ell\log\zeta)\,. 
\]
We find the solution

{\em(i)} If $\sin \ell\omega\neq0$, i.e. if $\ell\omega\not\in \pi \N$
\begin{equation}
\label{eq:solel}
\left\{\ 
\begin{aligned}
   & w_\ell(t) = a_\ell \,\Im\zeta^\ell + b_\ell \, \Re\zeta^\ell
   \quad \\
   & \mbox{with}\quad 
   a_\ell = \frac{g^\omega_\ell - g^0_\ell\,\cos\ell\omega}{\sin \ell\omega}
   \quad\mbox{and}\quad 
   b_\ell = g^0_\ell\,.
\end{aligned}
\right.
\end{equation}
In this case the solution to \eqref{eq:bvpellsect} is unique in the space of homogeneous functions of degree $\ell$.

{\em(ii)} If $\sin \ell\omega=0$, i.e. if $\ell\omega = k\pi$ with $k\in \N$, so $\cos \ell\omega=(-1)^k$,
\begin{equation}
\label{eq:solellog}
\left\{\ 
\begin{aligned}
   & w_\ell(t) = a_\ell \,\Im(\zeta^\ell\log\zeta)
   + b_\ell \, \Re\zeta^\ell \\
   & \mbox{with}\quad
   a_\ell = \frac{g^\omega_\ell - g^0_\ell\,\cos\ell\omega}{\omega\cos \ell\omega}
   \quad\mbox{and}\quad
   b_\ell = g^0_\ell\,.
\end{aligned}
\right.
\end{equation}

We draw the following consequences  according to whether $\frac\pi\omega$ is rational or not:

{\em(a)} If $\frac\pi\omega\in\Q$, then the coefficients $a_\ell$ and $b_\ell$ in \eqref{eq:solel}-\eqref{eq:solellog} are controlled since $\sin\ell\omega$ spans a finite set of values: There exists $C'$ such that
\begin{equation}
\label{eq:estaell}
   |a_\ell| +  |b_\ell| \le C' \rho_1^{-\ell},\quad\ell\in\N_*\,.
\end{equation}

{\em(b)} If $\frac\pi\omega\not\in\Q$,  
 estimating $a_\ell$ is hindered by the possible appearance of small denominators $\sin\ell\omega$. In Appendix~\ref{app:Liouville} we show that there exists a dense set of angles $\omega$ such that $\sin\ell\omega$ takes such small values that the series with $w_\ell$ defined by \eqref{eq:solel} will not converge, in general. The criterion is that $\pi/\omega$ belongs to the set $\Lambda_\rs$ of super-exponential Liouville numbers, defined in Definition~\ref{def:Liouville} by their very fast approximability by rational numbers. We can restore the control of $w_\ell$ by modifying it as proposed in \cite{BraDau82,Dauge84}. For this, we ``borrow'' a term from the expansion \eqref{eq:urm} of $u_{\sf rm}$ in Section~\ref{ss:urm} below, namely a solution of the problem with zero lateral boundary conditions, a  Laplace-Dirichlet singularity
\[
   \Im\zeta^{k\pi/\omega} \quad\mbox{(harmonic in $\rS_\omega$, zero on $\rT_0$ and $\rT_\omega$)}.
\]  
Using this with $k
=\round{\ell\omega/\pi}
\in\N_*$ such that $|\ell\omega-k\pi|$ is minimal, we introduce a variant of $w_\ell$ from \eqref{eq:solel} by defining
\[
   \widetilde w_\ell(t) = 
   a_\ell \,\big(\Im\zeta^\ell - \Im\zeta^{k\pi/\omega} \big)
   + b_\ell \, \Re\zeta^\ell
\]
with $a_\ell$ and $b_\ell$ as in \eqref{eq:solel}.
We note that
\[
   a_\ell \,\big(\Im\zeta^\ell - \Im\zeta^{k\pi/\omega} \big) =
   (g^\omega_\ell - g^0_\ell\,\cos\ell\omega) \Im\,
   \frac{\zeta^\ell- \zeta^{k\pi/\omega}}{\sin\ell\omega}
   \;.
\]
The quotient on the right is \emph{stable} because it can be expressed by {\em divided differences}: 
\[
   \frac{\zeta^\ell - \zeta^{k\pi/\omega}}{\sin\ell\omega} = \ 
   \frac{\zeta^\ell - \zeta^{k\pi/\omega}}{\ell - k\pi/\omega} 
   \ \frac{\ell - k\pi/\omega}{\sin\ell\omega - \sin k\pi}\;.
\]
For fixed $\ell$, this is continuous in $\omega$, even if $\ell\omega\to k\pi$, and we recover the logarithmic term from  \eqref{eq:solellog}:
\[
   \lim_{\ell\omega\to k\pi} \Im \frac{\zeta^\ell - \zeta^{k\pi/\omega}}{\sin\ell\omega}  = 
   \Im(\zeta^\ell\log\zeta) \ \frac{1}{\omega\cos\ell\omega}\;.
\]
For fixed $\omega$, we find a bound for the coefficient uniformly in $\ell$ if $|\ell\omega-k\pi|\le\pi/2$:
\[
 \left| \frac{\ell - k\pi/\omega}{\sin\ell\omega} \right|
 = \frac1\omega 
  \, \left| \frac{\ell\omega - k\pi}{\sin(\ell\omega- k\pi)} \right|
 \le \frac\pi{2\omega}\,.
\]
The stable variant of \eqref{eq:solel}, which contains the logarithmic expressions  \eqref{eq:solellog}, is therefore
\begin{equation}
\label{eq:solelk}
\left\{\ 
\begin{aligned}
   & \widetilde w_\ell(t) = \tilde a_\ell \,\Im \frac{\zeta^\ell - \zeta^{k\pi/\omega}}{\ell - k\pi/\omega} 
    + b_\ell \, \Re\zeta^\ell
   \quad \\
   & \mbox{with}\quad 
   \tilde a_\ell =  (g^\omega_\ell - g^0_\ell\,\cos\ell\omega) 
   \ \frac{\ell - k\pi/\omega}{\sin\ell\omega}
   \quad\mbox{and}\quad 
   b_\ell = g^0_\ell\,.
\end{aligned}
\right.
\end{equation}
We need this variant only when $|\ell\omega-k\pi|$ is small. We fix a threshold
\begin{equation}
\label{eq:thresh}
  0<\delta_\omega<\tfrac12\,\min\{\omega,\pi\}
\end{equation}
and replace $w_\ell$ by $\widetilde w_\ell$ if there exists $k\in\N_*$ such that $|\ell\omega-k\pi|\le \delta_\omega$.

The bounds on $\delta_\omega$ imply on one hand that in this definition $k$ is defined uniquely by $\ell$, but $\ell$ is also uniquely determined by $k$. On the other hand, we can check that the coefficients $a_\ell$ and $\tilde a_\ell$ are uniformly controlled: There exists $C'$ independent of $\ell$ such that
\begin{equation}
\label{eq:esttaell}
   |a^\flat_\ell| \le C' \rho_1^{-\ell}\,,\quad
   \mbox{ where }
   a^\flat_\ell=\Big\{\begin{array}{ll}a_\ell\\
                 \widetilde a_\ell\end{array}\;
  \mbox{ if }\; \dist(\ell\omega,\pi\N)\,
             \Big\{\begin{array}{ll}>\delta_\omega\,,\\
                 \le \delta_\omega\,.\end{array}
\end{equation}
Thus, choosing for each value of $\ell$ solutions $w_\ell$ or $\widetilde w_\ell$, cf \eqref{eq:solel}-\eqref{eq:esttaell}, we obtain a convergent series expansion 
for a particular solution $u_{\partial}$ of the (partial) Dirichlet problem
\eqref{eq:Dirlat}.

\subsection{Remaining boundary condition and convergence}\label{ss:urm}
Let us write in $\rA\cap \sB(0,\rho_1)$:
\begin{equation}
\label{eq:tu3}
   u_0 = u_{f} + u_{\partial} + u_{\sf rm}\,.
\end{equation}
Now the function $u_{\sf rm}$ resolves the remaining boundary condition (here we choose $\rho'_1\in(0,\rho_1)$)
\begin{equation}
\label{bvpremain}
\left\{
\begin{array}{rcll}
   \Delta u_{\sf rm}(t)&=&0 & \forall t \in \rA\cap \sB(0,\rho'_1)\,,\\
   u_{\sf rm}(t)&=&0 & \forall t \in (\rT_0\cup \rT_\omega) \cap \sB(0,\rho'_1)\,, \\
   u_{\sf rm}(t)&=&g(t) & \forall t \in \rS_\omega, \ \ |t|=\rho'_1\,,
\end{array}
\right.
\end{equation}
where 
\[
   g(t) \equiv u_0(t)-u_{f}(t)-u_{\partial}(t) \quad\mbox{for}\quad |t|=\rho'_1 \,.
\]
Denoting by $\Pi$ the arc $\vartheta\in(0,\omega)$, $\rho=1$, we can see that the trace $g$ belongs to $H^{1/2}_{00}(\rho'_1\Pi)$. By partial Fourier expansion with respect to the eigenfunction basis $\big(\sin\frac{k\pi}{\omega}\vartheta\big)_{k\in\N_*}$ we find  
\[
   g(t) = \sum_{k\ge1} g_k \sin\frac{k\pi}{\omega}\vartheta,\quad t\in \rho'_1\Pi,
\]
with a bounded\footnote{In fact the sequence $\big(\sqrt{k}\,g_k\big)_{k\in\N_*}$ belongs to $\ell^2(\N_*)$.} sequence $\big(g_k\big)_{k\in\N_*}$, and we deduce the representation
\[
   u_{\sf rm}(t) = \sum_{k\ge1} g_k \Big(\frac{\rho}{\rho'_1}\Big)^{k\pi/\omega}
   \sin\frac{k\pi}{\omega}\vartheta,
   \quad t\in\overline\rA\cap \sB(0,\rho'_1).
\]
Setting $c_{k\pi/\omega}=g_k(\rho'_1)^{-k\pi/\omega}$, we find that the expansion for the remaining term can be written as the converging series
\begin{equation}
\label{eq:urm}
   u_{\sf rm}(t) = \sum_{\gamma\in\frac\pi\omega\N_*} 
   c_\gamma \Im\zeta^{\gamma},
   \quad t\in\overline\rA\cap \sB(0,\rho'_1),
\end{equation}
with the estimates
\begin{equation}
\label{eq:estcg}
   |c_\gamma| \le C (\rho'_1)^{-\gamma}.
\end{equation}
The collection of formulas and estimates \eqref{eq:uf}, \eqref{eq:solel}-\eqref{eq:esttaell}, and \eqref{eq:urm}-\eqref{eq:estcg} motivates the following unified notation.

\begin{notation}
\label{not:1}
Let $\gA$ be the set of indices (here $\N^2_*$ denotes $\N^2\setminus\{(0,0)\}$)
\[
   \gA = \N^2_* \cup \tfrac\pi\omega\N_*,
\]
and let $\gA_0$ be the subset of $\gA$ of elements of the form $(\ell,0)$ with $\ell\in\N_*$ such that there exists
\begin{equation}
\label{eq:ellk}
   k\in\N_*\quad\mbox{with}\quad  
   |\ell\omega-k\pi| \le \delta_\omega  
    \quad\mbox{ (see \eqref{eq:thresh})}.
\end{equation}
For any $\gamma\in\gA$, we define the function $t\mapsto\sZ_\gamma(t)$ as follows
\begin{enumerate}
\item If $\gamma\in\frac\pi\omega\N_*$, set $\sZ_\gamma(t) = \zeta^\gamma$,
\item If $\gamma=(\alpha_1,\alpha_2)\in\N^2_*$ and $\gamma\not\in\gA_0$, set $\sZ_\gamma(t) = \zeta^{\alpha_1}\bar\zeta{}^{\alpha_2}$,
\item If $\gamma=(\ell,0)\in\gA_0$, let $k$ be the unique integer such that \eqref{eq:ellk} holds. Set
\begin{equation}
\label{eq:Zgam}
\begin{cases}
   \sZ_\gamma(t) = \zeta^{\ell}\log \zeta & \mbox{if } \ell=\frac{k\pi}{\omega}, \\
   \displaystyle \sZ_\gamma(t) = \frac{\zeta^\ell - \zeta^{k\pi/\omega}}{\ell - k\pi/\omega}
   & \mbox{if } \ell\neq\frac{k\pi}{\omega}. \\
\end{cases}
\end{equation}
\end{enumerate}
\end{notation}

We are ready to prove the main result of this section:

\begin{theorem}
\label{th:tu}
Let $u_0$ be the solution of the unperturbed problem \eqref{eq:u0} with right hand side $f\in L^2(\rA)$ satisfying \eqref{eq:f}.
We can represent $u_0$ as the sum of a convergent series in a neighborhood of the vertex $0$ 
\begin{equation}
\label{eq:tu}
   u_0(t) = \Im \sum_{\gamma\in\gA} a_\gamma \sZ_\gamma(t), \quad t\in\rA\cap\sB(0, \rho_1)
\end{equation}
where the set $\gA$ and the special functions $\sZ_\gamma$ are introduced in Notation \ref{not:1}, and the coefficients $a_\gamma$ satisfy the analytic type estimates: for all $M>\rho_1^{-1}$ there exists $C$ such that
\begin{equation}
\label{eq:tuest}
   |a_\gamma| \le C M^{|\gamma|},\quad \gamma\in\gA,
\end{equation}
where $|\gamma|=\alpha_1+\alpha_2$ if $\gamma=(\alpha_1,\alpha_2)\in(\N_*)^2$, and $|\gamma|=\gamma$ if $\gamma\in\frac\pi\omega\N_*$. The coefficients $a_\gamma$ are real if $\gamma\in\frac{\pi}{\omega}\N_*$ or if $\gamma=(\ell,0)\in\N^2_*$.
\end{theorem}

\begin{proof}
We start from the representation \eqref{eq:tu3} of $u_0$ in the three parts $u_{f}$, $u_{\partial}$ and $u_{\sf rm}$.

1) $u_{f}$ has the explicit expression \eqref{eq:uf} that can be written as\\
 $\sum_{\alpha_1\in\N_*}\sum_{\alpha_2\in\N_*} b_\alpha \zeta^{\alpha_1}\bar\zeta{}^{\alpha_2}$ with suitable estimates for the coefficients $b_\alpha$:
\[
   |b_\alpha| \le C M^{|\alpha|}.
\]
We notice that the set of indices $(\N_*)^2$ has an empty intersection with $\gA_0$. So
\[
   u_{f}(t) = \sum_{\gamma\in(\N_*)^2} b_\gamma \sZ_\gamma(t).
\]
Since $u_{f}$ is real, we can 
set $a_\gamma=ib_\gamma$ and get
\[
   u_{f}(t) = \Im \sum_{\gamma\in(\N_*)^2} a_\gamma \sZ_\gamma(t).
\]

2) $u_{\partial}$ is equal to $\sum_{\ell\in\N_*}w^\flat_\ell$ with 
\begin{itemize}
\item[{\em i)}] $w^\flat_\ell=w_\ell$ with $w_\ell$ given by \eqref{eq:solel} if $(\ell,0)\not\in\gA_0$,
\item[{\em ii)}] $w^\flat_\ell=w_\ell$ with $w_\ell$ given by \eqref{eq:solellog} if $\ell=\frac{k\pi}{\omega}$ for some $k\in\N_*$,
\item[{\em iii)}] $w^\flat_\ell=\widetilde w_\ell$ with $\widetilde w_\ell$ given by \eqref{eq:solelk} if $(\ell,0)\in\gA_0$  and $\ell\neq\frac{k\pi}{\omega}$, where $k$ is the integer such that \eqref{eq:ellk} holds.
\end{itemize}
We parse each of these three cases

{\em i)} $(\ell,0)\not\in\gA_0$: Then $\sZ_{(\ell,0)}=\zeta^\ell$ and $\sZ_{(0,\ell)}=\bar\zeta{}^\ell$. We use formula \eqref{eq:solel} to obtain
\begin{equation}
\label{eq:well}
   w_\ell = \Im ( a_\ell \sZ_{(\ell,0)} + ib_\ell \sZ_{(0,\ell)})
\end{equation}
The coefficients $a_{(\ell,0)}=a_\ell$ and $a_{(0,\ell)}=ib_\ell$ satisfy the desired estimates, because $(\ell,0)\not\in\gA_0$ implies $|\sin\ell\omega|\ge\sin\frac{\omega}{2}$ .

{\em ii)} There exists $k\in\N_*$ such that $\ell=\frac{k\pi}{\omega}$: Then $\sZ_{(\ell,0)}=\zeta^\ell\log\zeta$ and $\sZ_{(0,\ell)}=\bar\zeta{}^\ell$. We use formula \eqref{eq:solellog} to obtain the representation \eqref{eq:well} again.

{\em iii)} $(\ell,0)\in\gA_0$ and $\ell\neq\frac{k\pi}{\omega}$, with the integer $k$  for which \eqref{eq:ellk} holds. Now we start from formula \eqref{eq:solelk} and find once more the representation \eqref{eq:well} with $a_\ell$ replaced by $\tilde a_\ell$.

3) Finally $u_{\sf rm}$ given by \eqref{eq:urm} is already written in the desired form.
\end{proof}

\begin{remark}
\label{rem:imre}
1) Examining the structure of the terms in \eqref{eq:tu} we can see that a real valued basis for the expansion of $u_0$ is the union of
\begin{itemize}
\item $\Im\sZ_\gamma$ if $\gamma\in\frac\pi\omega\N_*$ or if $\gamma=(\alpha_1,\alpha_2)\in\N^2_*$ with $\alpha_1>\alpha_2$,

\item $\Re\sZ_\gamma$ if $\gamma=(\alpha_1,\alpha_2)\in\N^2_*$ with $\alpha_1\le\alpha_2$.
\end{itemize}

2) The traces of the function $\Im\sZ_{k\pi/\omega}$ are zero on $\partial\rS_\omega$ for all $k\in\N_*$. If we write the expansion \eqref{eq:tu} in the form
\begin{equation}
\label{eq:tup}
   u_0 = \sum_{k\in\N_*} a_{k\pi/\omega} \Im\sZ_{k\pi/\omega}
   + \sum_{\ell\in\N_*} \Big(\Im \sum_{\substack {\gamma\in\N^2\\ |\gamma|=\ell}} a_\gamma \sZ_\gamma\Big)
\end{equation}
we obtain terms $\Im\sZ_{k\pi/\omega}$, or packets of terms $\Im\sum_{|\gamma|=\ell} a_\gamma \sZ_\gamma$ that have zero traces on $\partial\rS_\omega$.
\end{remark}

\begin{remark}
\label{rem:pi}
If $\omega=\pi$, then $u_0$ has a converging Taylor expansion at the origin.
\end{remark}

\begin{remark}
\label{rem:f=0}
If $f=0$ in a neigborhood of the origin, then in the above construction we find 
that $u_{f}$ and $u_{\partial}$ vanish identically, hence $u_0=u_{\sf rm}$. For the latter we have the convergent expansion \eqref{eq:urm}, and therefore $u_0$ has an expansion in terms of $\Im\zeta^{k\frac\pi\omega}$, $k\in\N_*\,$, that is convergent in a neighborhood of the origin. 
\end{remark}

\begin{remark}
\label{rem:other}
 
The definition of $\gA_0$ depends on the choice of the threshold $\delta_\omega$, see \eqref{eq:thresh}. This influences which pairs of terms $\zeta^\ell$ and $\zeta^{k\pi/\omega}$ are grouped together into $\sZ_\gamma$ in the sum \eqref{eq:tu}, but changing $\gA_0$ does not change the sum. One can also omit a finite number of indices from $\gA_0$ without changing the sum. From Appendix~\ref{app:Liouville} follows that we can even set $\delta_\omega=0$ and therefore reduce $\gA_0$ to the empty set if $\pi/\omega$ is irrational, but not a super-exponential Liouville number. The resulting series in which no pairs of terms are regrouped will then converge, with a possibly smaller convergence radius than $\rho_1$ if $\pi/\omega$ is an exponential, but not super-exponential Liouville number. The full convergence radius $\rho_1$ is retained if $\pi/\omega$ is not an exponential Liouville number, in particular if it is not a Liouville number. If $\pi/\omega$ is a super-exponential Liouville number, then there exist right hand sides $f$ such that the unmodified series does not converge for $t\ne0$. 
\end{remark}

\subsection{Residual problem on the perforated domain}

Setting $\tilde u_\varepsilon=u_\varepsilon-u_0$ with $u_\varepsilon$ and $u_0$ the solutions of problems \eqref{eq:poisson} and \eqref{eq:u0}, respectively, we obtain that $\tilde u_\varepsilon$ solves the residual problem
\begin{equation}
\label{eq:pepsgen}
\begin{cases}
\begin{array}{rcll}
   \Delta \tilde u_\varepsilon &=& 0 \quad& \mbox{in}\quad \rA_\varepsilon,\\
   \tilde u_\varepsilon &=& -u_0 \quad& \mbox{on}\quad \partial\rA_\varepsilon,
\end{array}
\end{cases}
\end{equation}
By construction, $u_0$ is zero on $\partial\rA$, therefore on $\partial\rA_\varepsilon\cap\partial\rA$. Thus  the trace of $u_0$ on $\partial\rA_\varepsilon$ can be nonzero only on the part $\varepsilon\partial\rP\cap\rS_\omega$ of the boundary of the perforations, compare Fig.\ref{fig:2}. The converging expansion \eqref{eq:tu} allows us to interpret traces of $u_0$ on $\varepsilon\partial\rP\cap\rS_\omega$ as a series of traces on $\partial\rP\cap\rS_\omega$ with coefficients depending on $\varepsilon$.
To describe this dependence, we recall Notation \ref{not:1} and introduce corresponding combinations of powers of $\varepsilon$.

\begin{notation}
\label{not:2}
Let $\gA$ and $\gA_0$ be the sets of indices introduced in Notation \ref{not:1}. For any $\gamma\in\gA$ we define the function $\varepsilon\mapsto\sE_\gamma(\varepsilon)$ as follows
\begin{enumerate}
\item If $\gamma\in\frac\pi\omega\N_*$, set $\sE_\gamma(\varepsilon) = \varepsilon^\gamma$,
\item If $\gamma=(\alpha_1,\alpha_2)\in\N^2_*$ and $\gamma\not\in\gA_0$, set $\sE_\gamma(\varepsilon) = \varepsilon^{|\gamma|}$,
\item If $\gamma=(\ell,0)\in\gA_0$, let $k$ be the unique integer such that \eqref{eq:ellk} holds. Set
\begin{equation}
\label{eq:Egam}
\begin{cases}
   \sE_\gamma(\varepsilon) = \varepsilon^{\ell}\log \varepsilon & \mbox{if } \ell=\frac{k\pi}{\omega}, \\
   \displaystyle \sE_\gamma(\varepsilon) = 
   \frac{\varepsilon^\ell - \varepsilon^{k\pi/\omega}}{\ell - k\pi/\omega}
   & \mbox{if } \ell\neq\frac{k\pi}{\omega}. \\
\end{cases}
\end{equation}
\end{enumerate}
\end{notation}

The functions $\sZ_\gamma$ \eqref{eq:Zgam} are pseudo-homogeneous in the following sense.

\begin{lemma}
\label{lem:homo}
Let $\gamma\in\gA$ and $T\in\rS_\omega$.
\begin{itemize}
\item If $\gamma\not\in\gA_0$, then 
\[
   \sZ_\gamma(\varepsilon T) = \varepsilon^{|\gamma|} \sZ_\gamma(T) = \sE_\gamma(\varepsilon)\sZ_\gamma(T)\,.
\]
\item If $\gamma=(\ell,0)\in\gA_0$, let $k$ be the unique integer such that \eqref{eq:ellk} holds. We set $\gamma'=\frac{k\pi}{\omega}$ and we have
\[
   \sZ_\gamma(\varepsilon T) = 
   \varepsilon^{|\gamma|} \sZ_{\gamma}(T) + \sE_\gamma(\varepsilon) \sZ_{\gamma'}(T)\,.
\]
\end{itemize}
\end{lemma}

\begin{corollary}
\label{cor:tu}
Under the conditions of Theorem \ref{th:tu}, using the packet expansion \eqref{eq:tup}, we find
\begin{multline}
\label{eq:tupe}
   u_0(\varepsilon T) = \!\!
   \sum_{\gamma\in\frac{\pi}{\omega}\N_*} \!\! a_{\gamma} \varepsilon^{\gamma} \Im\sZ_{\gamma}(T)\\
   + \sum_{\ell\in\N_*} \varepsilon^\ell
   \Big(\Im \sum_{\substack {\gamma\in\N^2\\ |\gamma|=\ell}} a_\gamma \sZ_\gamma(T)\Big)
   + \sum_{\gamma\in\gA_0}  a_\gamma \sE_\gamma(\varepsilon)\Im\sZ_{\gamma'}(T).
\end{multline}
Each of the terms or packets has zero trace on $\partial\rS_\omega$.
\end{corollary}


\section{From a perforated sector to a domain with interior holes}
\label{s:3}

In this section, we transform the residual Dirichlet problem \eqref{eq:pepsgen} into a problem on a perforated domain with holes shrinking towards an interior point, so that to be able to use integral representations for its solution.
A suitable transformation is obtained as the composition of two operations, see Fig.\ref{fig:3}:
\begin{itemize}
\item A conformal map $\cG_\kappa$:  $\zeta\mapsto z = \zeta^{\kappa}$ with $\kappa=\frac\pi\omega$ that transforms the sector $\rS_\omega$ into the upper half-plane $\rS_\pi=\R\times\R_+$,
\item The odd reflection operator $\cE$  that extends domains and functions from $\rS_\pi$ to $\R^2$.
\end{itemize}
\begin{figure}[ht]
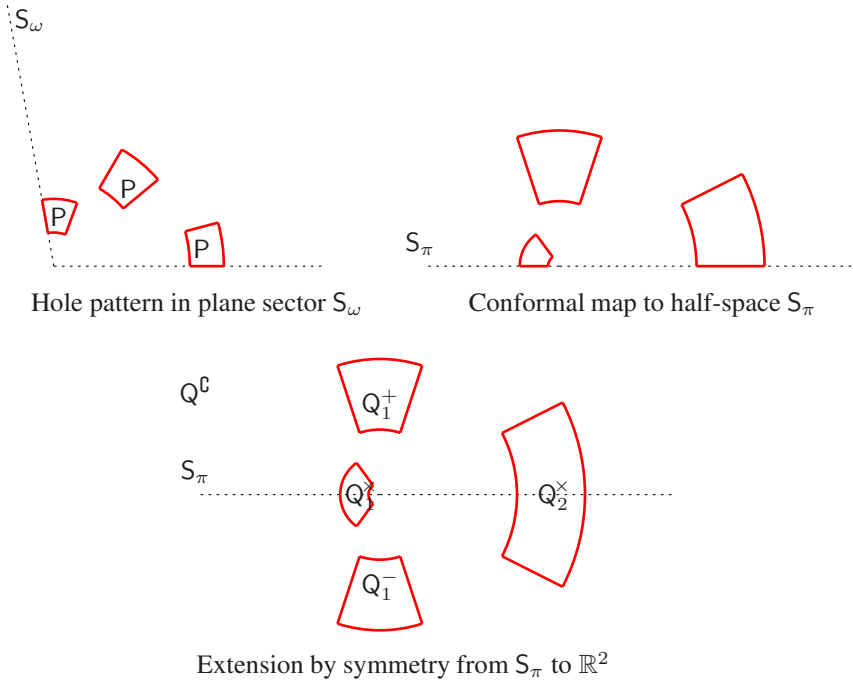

\begin{minipage}[b]{0.49\textwidth}
\input Fig1_perfsector.tex
\end{minipage}
\begin{minipage}[b]{0.49\textwidth}
\input Fig2_perfsector.tex
\end{minipage}
\vskip 1.5em

\begin{minipage}[b]{0.49\textwidth}
\input Fig3_perfsector.tex
\end{minipage}
\caption{Conformal map and symmetry acting on hole pattern $\rP$.}
\label{fig:3}
\end{figure}

We introduce these two operations and list some of their properties before composing them in view of the transformation of problem \eqref{eq:pepsgen}.

\subsection{Conformal mapping of power type}\label{ss:conf}
Let $\omega\in(0,2\pi)$ and $\kappa>0$ be chosen so that $\kappa\omega<2\pi$. The conformal map $\cG_\kappa$:  $\zeta\mapsto z = \zeta^{\kappa}$ transforms Cartesian coordinates $t$ into Cartesian coordinates $x$ with
\[
   \zeta = t_1+it_2\quad\mbox{and}\quad z = x_1+ix_2
\]
and polar coordinates $(\rho,\vartheta)$ into $(r,\theta)$ with
\[
   r = \rho^\kappa\quad\mbox{and}\quad \theta = \kappa\vartheta.
\]

\begin{lemma}
\label{lem:LipGk}
Assume that $\Omega\subset\rS_\omega$ and that $\Omega$ and $\rS_\omega\setminus\overline\Omega$ have a Lipschitz boundary. Then $\cG_\kappa\Omega\subset\rS_{\kappa\omega}$ and, moreover, $\cG_\kappa\Omega$ and $\rS_{\kappa\omega}\setminus\overline{\cG_\kappa\Omega}$ have a Lipschitz boundary.
\end{lemma}

A function $u$ defined on such a domain $\Omega\subset\rS_\omega$ is transformed into a function $\cG_\kappa^*u$ defined on $\cG_\kappa\Omega$ through the composition
\[
   \cG_\kappa^*u = u\circ\cG^{-1}_\kappa = u\circ\cG_{1/\kappa}\,.
\]

If $\partial\Omega$ is disjoint from the origin, then for any real $s$, the transformation $\cG_\kappa^*$ defines an isomorphism from the standard Sobolev space $H^s(\Omega)$ onto $H^s(\cG_\kappa\Omega)$. 

If, on the contrary, $0\in\partial\Omega$, there is no such simple transformation law for standard Sobolev spaces. Nevertheless, weighted Sobolev spaces of Kondrat'ev type can be equivalently expressed using polar coordinates and support such transformation: For real $\beta$ and natural integer $m$, the space $K^m_\beta(\Omega)$ is defined as
\begin{equation}
\label{eq:Kmb}
   K^m_\beta(\Omega) = \{u\in L^2_{\rm loc}(\Omega),\quad
   \rho^{\beta+|\alpha|}\partial^\alpha_t u\in L^2(\Omega),\ \ \forall\alpha\in\N^2,\ |\alpha|\le m\}.
\end{equation}
We have the equivalent definition in polar coordinates
\[
   K^m_\beta(\Omega) = \{u\in L^2_{\rm loc}(\Omega),\quad
   \rho^{\beta}(\rho\partial_\rho)^{\alpha_1} \partial^{\alpha_2}_\vartheta u\in L^2(\Omega),\ \ 
   \forall\alpha\in\N^2,\ |\alpha|\le m\}.
\]
\begin{lemma}
\label{lem:Km}
The conformal map $\cG_\kappa$ defines an isomorphism
\[
   \cG^*_\kappa : K^m_\beta(\Omega) \quad\mbox{onto}\quad K^m_{\frac{1+\beta}{\kappa}-1}(\cG_\kappa\Omega).
\]
\end{lemma}

The proof is based on the formulas
\[
   \rho\partial_\rho = \kappa\, r\partial_r \quad\mbox{and}\quad
   \rho \rd\rho\rd\vartheta = r^{\frac2\kappa-2} r\rd r\rd\theta.
\]
Details are left to the reader.

A relation between standard Sobolev spaces $H^s$ for real positive $s$ and the weighted scale $K^m_\beta$ is the following \cite[Appendix A]{Dauge88}
\begin{equation}
\label{eq:KmHs}
   K^m_\beta(\Omega) \subset H^s(\Omega)\quad\mbox{if}\quad m\ge s\quad\mbox{and}\quad \beta<-s.
\end{equation}

Coming back to the solution $u_0$ of problem \eqref{eq:u0}, we check that as a consequence of \eqref{eq:tu} and of Lemma \ref{lem:Km}, there holds:

\begin{lemma}
\label{lem:u0Km}
Let $m\ge1$ be an integer. Then the solution $u_0$ of problem \eqref{eq:u0} satisfies
\begin{equation}
\label{eq:u0Km}
   u_0\in K^m_\beta(\rA)\quad\forall\beta\quad\mbox{such that}\quad 
   1+\beta>-\min\{\tfrac{\pi}{\omega},2\}.
\end{equation}
Let $\kappa>0$. Then
\begin{equation}
\label{eq:u0KmGk}
   \cG_\kappa u_0\in K^m_{\beta'}(\cG_\kappa\rA)\quad\forall\beta'\quad\mbox{such that}\quad 
   1+\beta'>-\tfrac{1}{\kappa}\,\min\{\tfrac{\pi}{\omega},2\}.
\end{equation}
\end{lemma}

\subsection{Reflection and odd extension}\label{Ss:Reflection}
We first denote by $\cR$ the mapping from $\mathbb{R}^2$ to itself defined as the reflection across the $x_1$ axis
\[
   \cR(x_1,x_2)\equiv (x_1,-x_2) \qquad \forall x=(x_1,x_2) \in \mathbb{R}^2\, .
\]
Then if $\Omega$ is a subset of $\mathbb{R}^2$ and $g$ a function defined on $\Omega$, we denote by $\cR^\ast[g]$ the function on $\cR(\Omega)$ defined by
\[
   \cR^\ast[g](x)=g(\cR(x)) \qquad \forall x \in \cR(\Omega)\, .
\]

Let $\Omega$ be a subdomain of the half-plane $\rS_\pi$. Let us set
\[
   \Gamma = {\rm interior}\,(\partial\Omega\cap\partial\rS_\pi).
\]
We denote by $\cE(\Omega)$ the symmetric extension of $\Omega$ across the $x_1$ axis
\begin{equation}
\label{eq:EOm}
   \cE(\Omega) = \Omega \cup \cR(\Omega) \cup \Gamma.
\end{equation}

Since we want to stay within the category of Lipschitz domains, we need here the assumption that both $\Omega$ and its complement in $\rS_\pi$ are Lipschitz. 
Note that whereas for a Lipschitz domain its complement in $\R^2$ is automatically Lipschitz, too, this is not the case, in general, for the complement in $\rS_\pi$.
This is the reason why we had to make corresponding assumptions in Subsection~\ref{ss:geomset}, see assumptions (2) and (3) on the domain $\rA$ and the perforations $\rP$. Under these assumptions $\cE(\Omega)$ is a Lipschitz domain. Since the proof of this fact is rather technical, we present it in Appendix~\ref{app:symextlip}, Lemma~\ref{lem:LipE}.

If $g$ is a function defined on $\Omega$, the \emph{odd extension} of $g$ to $\cE(\Omega)$ is defined as
\[
\cE^*[g](x)\equiv
\left\{
\begin{array}{ll}
   g(x)\qquad &   \forall x\in\Omega \\
  -g(\cR(x))\qquad &   \forall x\in\cR(\Omega)\\
   0 &   \forall x\in\Gamma\,.
\end{array}
\right.
\]
Let us denote by $H^1_{0,\Gamma}(\Omega)$ the following subspace of $H^1(\Omega)$
\[
   H^1_{0,\Gamma}(\Omega) = \{u\in H^1(\Omega),\quad u\on{\Gamma}=0\}.
\]

\begin{lemma}
\label{lem:KmE}
Assume that $\Omega\subset\rS_\pi$ and that $\Omega$ and $\rS_\pi\setminus\overline\Omega$ have a Lipschitz boundary. Then the odd extension $\cE^*$ defines a bounded embedding
\[
   \cE^*: K^2_\beta(\Omega)\cap H^1_{0,\Gamma}(\Omega) \longrightarrow K^2_\beta(\cE(\Omega))\quad
   \forall \beta\in\R.
\]
\end{lemma}

\begin{proof}
If $u$ belongs to $K^2_\beta(\Omega)\cap H^1_{0,\Gamma}(\Omega)$, the jumps of $\cE^*[u]$ and of $\partial_2\cE^*[u]$ across $\Gamma$ are zero. Hence for all multiindices $\alpha$, $|\alpha|\le2$, the partial derivative $\partial^\alpha \cE^*[u]$ has no density across $\Gamma$ and
\[
   \DNormc{r^{|\alpha|+\beta}\partial^\alpha \cE^*[u]}{L^2({\cE(\Omega)})} = 
   2\DNormc{r^{|\alpha|+\beta}\partial^\alpha u}{L^2(\Omega)}.
\]
\end{proof}

\subsection{Transformation of the residual problem}\label{ss:residual}
We come back to our main setting, with unperforated domain $\rA$, hole pattern $\rP$, and family of perforated domains $\rA_\varepsilon$. We denote by $\cT$ and $\cT^*$ the composition of the conformal map $\cG_{\pi/\omega}$ and the odd extension acting on domains and functions respectively
\begin{equation}
\label{eq:TT}
   \cT = \cE\circ\cG_{\pi/\omega}\quad\mbox{and}\quad \cT^* = \cE^*\circ\cG^*_{\pi/\omega}\,.
\end{equation}
Then we denote
\[
   \rB = \cT(\rA),\quad
   \rQ = \cT(\rP).
\]
As a consequence of assumptions on $\rA$ and $\rP$, and of Lemmas \ref{lem:LipGk} and \ref{lem:LipE}, $\rB$ is a bounded simply connected Lipschitz domain containing the origin, and $\rQ$ is a finite union of bounded simply connected Lipschitz domains. 

The perforated sector $\rA_\varepsilon$ is transformed by $\cT$ into the perforated domain $\rB_\eta$ with
\begin{equation}
\label{eq:eta}
   \eta = \varepsilon^{\pi/\omega}\quad\mbox{and}\quad \rB_\eta = \rB\setminus\eta\overline\rQ.
\end{equation}
We note that the boundary of $\rB_\eta$ is the disjoint union of $\partial\rB$ and $\eta \ee\partial\rQ$, see Fig.\ref{fig:4}:
\begin{equation}
\label{eq:dBe}
   \partial\rB_\eta = \partial\rB\cup\eta \ee\partial\rQ.
\end{equation}

\begin{figure}[ht]
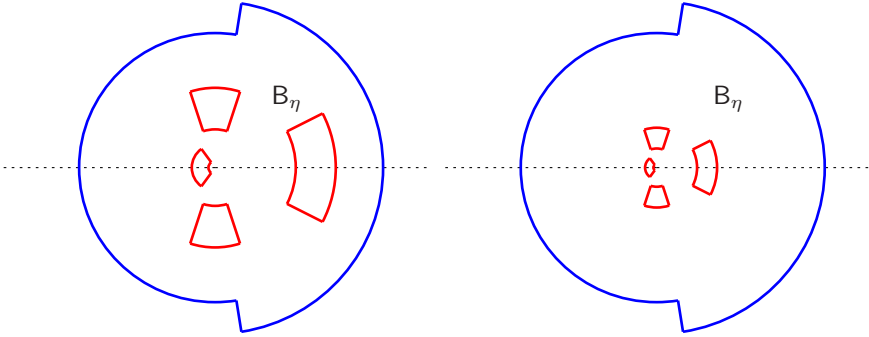

\begin{minipage}[b]{0.49\textwidth}
\input Fig8_perfsector.tex
\end{minipage}
\begin{minipage}[b]{0.49\textwidth}
\input Fig9_perfsector.tex
\end{minipage}
\caption{Transformed perforated domain $\rB_\eta$ for two values of $\eta$.}
\label{fig:4}
\end{figure}

The residual problem \eqref{eq:pepsgen} on $\rA_\varepsilon$ is transformed into the Dirichlet problem on $\rB_\eta$
\begin{equation*}
\begin{cases}
\begin{array}{rcll}
   \Delta v_\eta &=& 0 \quad& \mbox{in}\quad \rB_\eta,\\
   v_\eta &=& -\cT^*[u_0] \quad& \mbox{on}\quad \partial\rB_\eta,
\end{array}
\end{cases}
\end{equation*}
where we have set $v_\eta=\cT^*[\widetilde u_\varepsilon]$. We note that $v_\eta$ belongs to $H^1(\rB_\eta)$ and that its trace is zero on $\partial\rB$. We analyze now the structure of the trace of $\cT^*[u_0]$ on $\eta \ee\partial\rQ$. We take advantage of the converging expansion \eqref{eq:tu} and of the pseudo-homogeneity of its terms. 

We recall from Notation \ref{not:1} that the set of indices $\gA$ is the union of $\frac{\pi}{\omega}\N_*$ and $\N^2_*$, and from Notation \ref{not:2} that the pseudo-homogeneous functions $\sE_\gamma(\varepsilon)$ are defined as $\varepsilon^{|\gamma|}$ if $\gamma$ does not belong to the set of exceptional indices $\gA_0$, and by a divided difference or a logarithmic term in the opposite case.

\begin{theorem}
\label{th:Tu0}
Let $u_0$ be the solution of the unperturbed problem \eqref{eq:u0} with right hand side $f\in L^2(\rA)$ satisfying \eqref{eq:f}. The residual problem \eqref{eq:pepsgen} on $\rA_\varepsilon$ is transformed by the transformation $\cT$ \eqref{eq:TT} into the Dirichlet problem on $\rB_\eta$, with $\eta=\varepsilon^{\pi/\omega}$:
\begin{equation}
\label{eq:petagen}
\begin{cases}
\begin{array}{rcll}
   \Delta v_\eta &=& 0 \quad& \mbox{in}\quad \rB_\eta,\\
   v_\eta &=& 0 \quad& \mbox{on}\quad \partial\rB,\\
   v_\eta &=& -\cT^*[u_0] \quad& \mbox{on}\quad \eta\partial\rQ\,.
\end{array}
\end{cases}
\end{equation}
The trace of $\cT^*[u_0]$ can be written as a convergent sum for $\varepsilon\in(0,\varepsilon_1]$ for some positive $\varepsilon_1$
\begin{equation}
\label{eq:Tu0}
   \cT^*[u_0](\eta X) = \sum_{\gamma\in\gA} \sE_\gamma(\varepsilon)\Psi_\gamma(X), \quad X\in\partial\rQ\,,
\end{equation}
where the set $\gA$ and the functions $\sE_\gamma$ are introduced in Notations \ref{not:1} and \ref{not:2}.
There exists a positive number $\tau\in(0,1/2)$ such that the convergence takes place in the trace Sobolev space $H^{\tau+1/2}(\partial\rQ)$: There exist positive constants $C$ and $M$ such that
\begin{equation}
\label{eq:Tu0est}
   \DNorm{\Psi_\gamma}{H^{\tau+1/2}(\partial\rQ)} \le C M^{|\gamma|},\quad \gamma\in\gA\,.
\end{equation}
\end{theorem}

\begin{proof}
We use the expansion of $u_0(\varepsilon T)$ as written by packets in \eqref{eq:tupe}. Applying the transformation $\cT^*$ we find
\begin{equation}
\label{eq:tupeT}
\begin{aligned}
   \cT^*[u_0](\eta X) =& 
   \sum_{\gamma\in\frac{\pi}{\omega}\N_*} a_{\gamma} \varepsilon^{\gamma} \cT^*[\Im\sZ_{\gamma}](X) \\
   +& \sum_{\ell\in\N_*} \varepsilon^\ell \cT^*
   \Big[\Im \sum_{\substack {\gamma\in\N^2\\ |\gamma|=\ell}} a_\gamma \sZ_\gamma\Big](X)
   + \sum_{\gamma\in\gA_0}  a_\gamma \sE_\gamma(\varepsilon) \cT^*[\Im\sZ_{\gamma'}](X).
\end{aligned}
\end{equation}
We define for $\gamma\in\gA$
\begin{equation}
\label{eq:Phi}
   \Phi_\gamma = 
   \begin{cases}
   \ a_{\gamma} \Im\sZ_{\gamma}
   \quad & \mbox{if}\quad \gamma\in\frac{\pi}{\omega}\N_*,\\
   \ \Im \sum_{\tilde\gamma\in\N^2,\ |\tilde\gamma|=\ell} a_{\tilde\gamma} \sZ_{\tilde\gamma}
   \quad & \mbox{if}\quad \gamma =(0,\ell),\ \ell\in\N_*,\\
   \ a_\gamma  \Im\sZ_{\gamma'}
   \quad & \mbox{if}\quad \gamma =(\ell,0)\in\gA_0,\\
   \ 0 \quad & \mbox{for remaining $\gamma$'s , }
   \end{cases}
\end{equation}
and set
\begin{equation}
\label{eq:Psi}
   \Psi_\gamma = \cT^*[\Phi_\gamma],\quad\forall\gamma\in\gA.
\end{equation}
Thus \eqref{eq:tupeT}-\eqref{eq:Psi} imply \eqref{eq:Tu0}. 

Let us prove estimates \eqref{eq:Tu0est}. Let us choose $\beta<-1$ such that $1+\beta>-\min\{\tfrac{\pi}{\omega},2\}$, cf \eqref{eq:u0Km}. Relying on the explicit form of the functions $\Phi_\gamma$ and on the boundedness of the domain $\rP$, we find that there exist constants $C$ and $M$ such that
\begin{equation}
\label{eq:PhiNorm}
   \DNorm{\Phi_\gamma}{K^2_\beta(\rP)} \le C M^{|\gamma|},\quad \gamma\in\gA\,.
\end{equation}
Let $\beta'=\frac{\omega}{\pi}(1+\beta)-1$. Then $\beta'<-1$ and by Lemma \ref{lem:Km} the conformal map $\cG^*_{\pi/\omega}$ is bounded from $K^2_{\beta}(\rP)$ to $K^2_{\beta'}(\cG_{\pi/\omega}\rP)$. Then by Lemma \ref{lem:KmE}, the odd extension $\cE^*$ is bounded from $K^2_{\beta'}(\cG_{\pi/\omega}\rP)$ to $K^2_{\beta'}(\rQ)$. If we choose $\tau$ such that
\[
   \tau \le -(1+\beta') \quad\mbox{and}\quad \tau\in(0,\tfrac{1}{2}),
\]
we find that by \eqref{eq:KmHs}, the space $K^2_{\beta'}(\cG_{\pi/\omega}\rP)$ is continuously embedded in $H^{\tau+1}(\rQ)$. The trace theorem for Lipschitz domains then yields the continuity of the trace from $H^{\tau+1}(\rQ)$ to $H^{\tau+1/2}(\partial\rQ)$. Hence there exists $C'$ such that
\[
   \DNorm{\Psi_\gamma}{H^{\tau+1/2}(\partial\rQ)} \le C'\DNorm{\Phi_\gamma}{K^2_\beta(\rP)},
   \quad\forall\gamma\in\gA \,.
\]
Combining this with the previous estimate of $\DNorm{\Phi_\gamma}{K^2_\beta(\rP)}$ gives estimates \eqref{eq:Tu0est}.

We finally notice that estimates \eqref{eq:Tu0est} imply the convergence of the series \eqref{eq:Tu0} for $\varepsilon\in[0,\varepsilon_1]$ if $\varepsilon_1$ is chosen such that $\varepsilon_1M<1$.
\end{proof}

\begin{remark}
\label{rem:tau}
From the proof one finds a bound for the regularity index $\tau$
\begin{equation}
\label{eq:boundontau}
 \tau<\min\{\tfrac12,\tfrac{2\omega}\pi\}\,.
\end{equation}
This inequality, which is restrictive for small angles $\omega< \tfrac\pi4$, is mainly due to the use of the embedding \eqref{eq:KmHs} of the weighted Sobolev spaces into unweighted Sobolev spaces. It is needed only in the special situation where the boundary of the holes touches the origin. If, conversely, $0\not\in\partial\rP$, then we can use the trace theorem directly without passing by the embedding \eqref{eq:KmHs}, and the statement of the theorem is true for all $\tau\in(0,1/2)$.
\end{remark}


\section{Symmetric perforated Lipschitz domains}
\label{s:symperf}

In this section we investigate the asymptotic behavior of the solution of a Dirichlet problem in a symmetric Lipschitz domain with small holes. The analysis here performed will allow to study the behavior of the solution of problem \eqref{eq:petagen}.

More precisely, we will consider the case where 
the domain and its holes are symmetric with respect to the horizontal axis and the boundary data are antisymmetric. 
We use the technique with which the behavior of harmonic functions in perforated planar domains was studied in Lanza de Cristoforis \cite{La08} and in Dalla Riva and Musolino \cite{DaMu15}. As in \cite{DaMu15}, we employ boundary integral equations, but there are some differences in the assumptions:
In \cite{DaMu15}, perforations were of class $C^{1,\alpha}$ and connected, whereas we consider here perforations with Lipschitz boundaries and a finite number of connected components. This generalization is naturally implied by the construction of the perforation $\rQ$ from $\rP$ by the conformal transformations and reflections described in the previous sections, because even for a smooth and connected hole $\rP$ in the sector $\rS_\omega$, the resulting perforation $\rQ$ in $\R^2$ may have corners or several connected components.
On the other hand, our symmetry assumptions will allow to simplify notably the treatment of the problem. In particular, we find that we do not have to deal with the logarithmic behavior which arises in the general setting for two-dimensional perforated domains. 

\subsection{Some notions of potential theory on Lipschitz domains}\label{ss:prelLip}

We collect here some known results about harmonic double layer potentials on Lipschitz domains in the plane. Main references for these facts are the paper by Costabel \cite{Co88} and the books by Folland  \cite{Fo95} and McLean \cite{Mc00}. 

We assume that $\Omega\subset\R^2$ is a bounded Lipschitz domain (a role that will mainly be played by the perforated domain $\rB\setminus\eta\rQ$). Furthermore, $\Omega$ will be connected, but its complement 
$\Omega^\complement=\mathbb{R}^2 \setminus \overline{\Omega}$ 
may be not connected. 
Let $\Omega^\complement_{(1)}, \dots, \Omega^\complement_{(m)}$ be the bounded connected components of $\Omega^\complement$ and $\Omega^\complement_{(0)}$ the unbounded connected component of $\Omega^\complement$. Thus the boundary $\partial\Omega$ has the  $m+1$ connected components
$\partial\Omega^\complement_{(0)}, \dots, \partial\Omega^\complement_{(m)}$.

Let $E $ be the function from $\mathbb{R}^2\setminus\{0\}$ to $\mathbb{R}$ defined by
\[
E (x) \equiv
-\frac{1}{2\pi}\log |x| \qquad    \forall x\in 
{\mathbb{R}}^{2}\setminus\{0\}.
\]
 As is well known, $E $ is a fundamental solution of $-\Delta$ on $\mathbb{R}^2$.

If $\phi$ is an integrable function on $\partial \Omega$, we define the double layer potential $\mathcal{D}_{\partial\Omega}[\phi]$ by setting
\[
\mathcal{D}_{\partial\Omega}[\phi] (x)\equiv
 -\int_{\partial\Omega}\phi(y)\, n(y)\cdot \nabla E (x-y)\,\mathrm{d}s_y
 \qquad\forall x\in\mathbb{R}^2 \setminus \partial \Omega\, ,
\] 
where $\mathrm{d}s$ denotes the length element on $\partial\Omega$ and $n$ denotes the outward unit normal to $\Omega$, which exists almost everywhere on $\partial\Omega$. 
In $\R^2\setminus\partial\Omega$, the double layer potential $\mathcal{D}_{\partial\Omega}[\phi]$ is a harmonic function, vanishing at infinity.
By Costabel \cite[Thm.~1]{Co88}, if $\tau \in [-1/2,1/2]$ and $\phi \in H^{1/2+\tau}(\partial \Omega)$ then
\begin{equation}
\label{eq:dblreptau}
 \mathcal{D}_{\partial\Omega}[\phi]\on{\Omega} \in H^{1+\tau}(\Omega)\, , \qquad   \mathcal{D}_{\partial\Omega}[\phi]\on{\Omega^\complement} \in H^{1+\tau}_{\mathrm{loc}}(\Omega^\complement)\, .
\end{equation}
We denote by $\gamma_{0}$ and $\gamma_{0}^\complement$ the interior and exterior traces on $\partial \Omega$, respectively, and by $\gamma_{1}$ and $\gamma_{1}^\complement$ the interior and exterior normal derivatives on $\partial \Omega$, respectively, (both taken with respect to the exterior normal $n$). Then we have the jump relations \cite[Lem.~4.1]{Co88}
\[
\gamma_{0}^\complement\mathcal{D}_{\partial\Omega}[\phi]-\gamma_{0}\mathcal{D}_{\partial\Omega}[\phi]=\phi \ , \qquad \gamma_{1} \mathcal{D}_{\partial\Omega}[\phi]=\gamma_{1}^\complement\mathcal{D}_{\partial\Omega}[\phi]\, ,
\]
for all $\phi \in H^{1/2+\tau}(\partial \Omega)$. We introduce the boundary operators $\mathcal{K}_{\partial\Omega}$ and $\mathcal{W}_{\partial \Omega}$ by setting
\[
\mathcal{K}_{\partial \Omega}[\phi]\equiv \frac{1}{2}\bigg(\gamma_{0}\mathcal{D}_{\partial\Omega}[\phi]+\gamma_{0}^\complement\mathcal{D}_{\partial\Omega}[\phi]\bigg)\, , \quad
\mathcal{W}_{\partial \Omega}[\phi]\equiv-\gamma_{1}\mathcal{D}_{\partial\Omega}[\phi] =-\gamma_{1}^\complement\mathcal{D}_{\partial\Omega}[\phi]   
\]
for all $\phi \in H^{1/2+\tau}(\partial \Omega)$. As a consequence,
\begin{equation}\label{eq:jump}
\gamma_{0}\mathcal{D}_{\partial \Omega}[\phi]=- \tfrac{1}{2}\phi + \mathcal{K}_{\partial \Omega}[\phi]\, , \quad \gamma_{0}^\complement\mathcal{D}_{\partial\Omega} [\phi]= \tfrac{1}{2}\phi + \mathcal{K}_{\partial \Omega}[\phi]\, , \quad \forall \phi \in H^{1/2+\tau}(\partial \Omega)\, .
\end{equation}
Thus the boundary integral operator associated with the Dirichlet problem in $\Omega$ is $-\frac{1}{2}I + \mathcal{K}_{\partial \Omega}$, whose mapping properties we therefore want to summarize in the  sequel.  

From Costabel and Wendland \cite[Remark 3.15]{CoWe86} (see also Steinbach and Wendland \cite{StWe01} and Mayboroda and Mitrea \cite{MaMi06}), we deduce the validity of the following. 

\begin{lemma}\label{thm:Fred}
For any $\tau \in [-1/2,1/2]$, 
the operators $\pm\frac{1}{2}I + \mathcal{K}_{\partial \Omega} \colon H^{1/2+\tau}(\partial \Omega)\to H^{1/2+\tau}(\partial \Omega)$ are Fredholm operators of index 0.
\end{lemma}
The kernels and cokernels of $\pm\frac{1}{2}I + \mathcal{K}_{\partial \Omega}$ are also known; they are independent of $\tau$. They are described in terms of the characteristic functions of the connected components of $\partial \Omega$. Here, if $\mathcal{O}$ is a subset of $\partial \Omega$ we denote by $\chi_{\mathcal{O}}$ the function from $\partial \Omega$ to $\mathbb{R}$ defined by
\[
\chi_{\mathcal{O}}(x)\equiv \left\{
\begin{array}{ll}
1 &\text{if $x \in \mathcal{O}$}\, ,\\
0 &\text{if $x \in \partial \Omega \setminus \mathcal{O}$}\, .
\end{array}
\right.
\] 
The value of the double layer potential of a constant density is well known. 
\[
 \mathcal{D}_{\partial\Omega}[\chi_{\partial\Omega}](x) = \left\{
\begin{array}{ll}
0 &\text{if $x \in \Omega^\complement$}\\
-1 &\text{if $x \in  \Omega $}
\end{array}
\right.\quad 
\text{ hence } \quad 
\mathcal{K}_{\partial \Omega}[\chi_{\partial\Omega}]= -\tfrac12 \chi_{\partial\Omega}\,.
\]
Applying this to the components $\Omega^\complement_{(j)}$, $j=1,\dots,m$, we see that the characteristic functions $\chi_{\partial\Omega^\complement_{(j)}}$ generate double layer potentials that vanish in $\Omega$ and that they are therefore in the kernel of the operator $-\frac{1}{2}I + \mathcal{K}_{\partial \Omega}$. In fact, by arguing as in Folland \cite[Ch.~3]{Fo95} one can prove the following.
\begin{lemma}\label{lem:V-}
Let
\[
 \mathfrak{V}_{\pm}\equiv \bigg\{\phi \in H^{1/2}(\partial \Omega) \colon \pm \tfrac{1}{2}\phi + \mathcal{K}_{\partial \Omega}[\phi] =0 \bigg\}\, .
\]
Then $\mathfrak{V}_+$ has dimension $1$ and consists of constant functions on $\partial\Omega$.
The space $\mathfrak{V}_-$ has dimension $m$ and is generated by $\{\chi_{\partial \Omega^\complement_{(j)}}\}_{j=1}^{m}$. 
\end{lemma}

In order to characterize the range of the double layer potential operator, we use the description of the mapping properties of the operator  of the normal derivative of the double layer potential 
$\mathcal{W}_{\partial \Omega}=-\gamma_1\mathcal{D}_{\partial \Omega}$
as given in  McLean \cite[Thm.~8.20]{Mc00}.
\begin{lemma}\label{lem:kerW}
The operator $\mathcal{W}_{\partial \Omega}$ is a bounded selfadjoint operator from $H^{1/2}(\partial \Omega)$ to its dual space $H^{-1/2}(\partial \Omega)$. The kernel of $\mathcal{W}_{\partial \Omega}$ consists of locally constant functions in $H^{1/2}(\partial \Omega)$. Its dimension is $m+1$, and it is generated by  $\{\chi_{\partial \Omega^\complement_{(j)}}\}_{j=0}^{m}$.
\end{lemma}
For a bounded selfadjoint operator, the kernel determines the range, the latter being the orthogonal complement of the former. We thus obtain the following description of the range of the double layer potential operator.
\begin{corollary}\label{cor:ranD}
Let $\tau\in[-1/2,1/2]$. Let $u \in H^{1+\tau}(\Omega)$ be such that $\Delta u=0$ in $ \Omega$. Then there exists $\mu \in H^{1/2+\tau}(\partial\Omega)$ such that $u=\mathcal{D}_{\partial \Omega}[\mu]\on{\Omega}$ if and only if
\begin{equation}\label{eq:ranD}
 \big\langle \gamma_{1}u \,,\, \chi_{\partial \Omega^\complement_{(j)}} \big\rangle =0
 \qquad \forall j\in \{1,\dots,m\}\, .
\end{equation}
\end{corollary}
Here we use the brackets $\big\langle\cdot,\cdot\big\rangle$ to denote the natural duality between a Sobolev space $H^{s}(\partial \Omega)$ and $H^{-s}(\partial \Omega)$. Thus the condition in \eqref{eq:ranD} means that the integral of the normal derivative 
\[
 \partial_n u = n\cdot\nabla u
\]
of $u$ over each component of $\partial\Omega$ vanishes. The condition for $j=0$ is implied by the others, because by Green's formula any harmonic function $u$ in $\Omega$ satisfies
\[
  \sum_{j=0}^m \big\langle \gamma_{1}u \,,\, 
       \chi_{\partial \Omega^\complement_{(j)}} \big\rangle =
  \int_{\partial\Omega} \partial_n u\, \mathrm{d}s = 0\,.
\]

\begin{remark}\label{rem:ranK}
The function $\mu$ represents $u$ as a double layer potential if and only if $\mu$ is a solution of the boundary integral equation
\begin{equation}\label{eq:BIE}
 -\tfrac12\mu + \mathcal{K}_{\partial \Omega}[\mu] = \gamma_{0} u \qquad\text{ on } \partial\Omega\,.
\end{equation}
This follows from the jump relation \eqref{eq:jump} and from the uniqueness of the solution of the Dirichlet problem. Thus Corollary~\ref{cor:ranD} can be seen as a statement on the solvability of the boundary integral equation~\eqref{eq:BIE}, and conditions~\eqref{eq:ranD} characterize the cokernel of the operator 
$-\frac12I + \mathcal{K}_{\partial \Omega}$.
\end{remark}

\subsection{The double layer potential for symmetric planar domains}\label{ss:prelsym}
We now come back to the geometric situation found at the end of Section~\ref{s:3}. This means that in this subsection $\Omega=\rB_\eta=\rB\setminus\eta\overline\rQ$, where $\rB$ is a simply connected bounded Lipschitz domain in $\R^2$ containing the origin, $\eta\in(0,\eta_0)$ is a small positive real number, and $\rQ$ is a finite union of bounded simply connected Lipschitz domains such that $\eta_0\rQ$ is contained in $\rB$. In addition, $\rB$ and $\rQ$ are symmetric with respect to reflection at the horizontal axis. From Subsection~\ref{Ss:Reflection} we recall the notation for the reflection and the corresponding pullback
\[
   \cR(x_1,x_2)\equiv (x_1,-x_2) \qquad \forall x=(x_1,x_2) \in \mathbb{R}^2\, 
   \qquad\text{and}\quad
   \cR^\ast[g]=g\circ\cR\,.
\]
Thus we assume
\[
  \rB = \cR(\rB) \qquad\text{and}\quad \rQ = \cR(\rQ)\,.
\]
The symmetry of $\rQ$ implies that there exist two natural numbers $m^{\times}$ and $m^{\#}$, such that  $\rQ$ has $m=m^{\times}+2m^{\#}>0$ connected components
\[
\rQ^{\times}_{1}, \dots, \rQ^{\times}_{m^{\times}}, \rQ^{+}_{1}, \dots, \rQ^{+}_{m^{\#}}, \rQ^{-}_{1}, \dots,  \rQ^{-}_{m^{\#}}\,,
\]
satisfying
\begin{equation}\label{eq:symass}
\begin{aligned}
& \rQ^{\times}_i = \cR( \rQ^{\times}_i) 
   \qquad &\forall i \in \{1,\dots,m^{\times}\}\,,\\ 
& \rQ^{+}_j = \cR(\rQ^{-}_j)\,\quad\text{and}\quad
   \overline{\rQ^{+}_{j}}\subset \rS_\pi \equiv \R\times\R_+ 
   \qquad &\forall j \in \{1,\dots,m^{\#}\}\,.
\end{aligned}
\end{equation}
See Figure~\ref{fig:3} for an example with $m^{\times}=2$ and $m^{\#}=1$.

We introduce the following definition.

\begin{definition}\label{def:odd}
Let $\tau\in[-1/2,1/2]$. If $\partial \Omega = \cR(\partial \Omega)$, then $ H^{1/2+\tau}_{\mathrm{odd}}(\partial\Omega)$ denotes the closed subspace of $H^{1/2+\tau}(\partial \Omega)$ defined by
\[
 H^{1/2+\tau}_{\mathrm{odd}}(\partial\Omega)\equiv\left\{g \in H^{1/2+\tau}(\partial \Omega)\,:\; g=-\cR^\ast[g] \quad \text{on $\partial \Omega$} \right\}\,.
\] 
\end{definition}

The mapping properties of the boundary integral operator $-\frac12I+\mathcal{K}_{\partial\rB_\eta}$ on odd functions can be summarized as follows.

\begin{lemma}\label{lem:Kodd}
Let $\tau\in[-1/2,1/2]$ and $\eta\in(0,\eta_0)$.
\begin{itemize}
\item[(i)]
 The operator $-\frac12I+\mathcal{K}_{\partial\rB_\eta}$ defines a Fredholm operator of index zero from $H^{1/2+\tau}_{\mathrm{odd}}(\partial\rB_\eta)$ to itself.
\item[(ii)]
If $m^{\#}=0$, this operator is an isomorphism. More generally,
its kernel has dimension $m^{\#}$ and is generated by the functions
$\{\chi_{\eta\partial\rQ^+_j}-\chi_{\eta\partial\rQ^-_j}\}_{j=1}^{m^{\#}}$, where
\[
  (\chi_{\eta\partial\rQ^+_j}-\chi_{\eta\partial\rQ^-_j})(x) =
  \left\{\begin{array}{cl}
  +1, &\qquad x\in\eta\ee\partial\rQ^+_j\,,\\
  -1, &\qquad x\in\eta \ee\partial\rQ^-_j\,,\\
  0,  &\qquad x\in \partial\rB_\eta\setminus(\eta \ee\partial\rQ^+_j\cup\eta \ee\partial\rQ^-_j)  \,.
  \end{array}\right. 
\]
\item[(iii)]
If $u\in H^{1+\tau}(\rB_\eta)$ is such that $\Delta u=0$ in $\rB_\eta$ and 
$\cR^\ast[u]=-u$ , then there exists $\mu \in H^{1/2+\tau}_{\mathrm{odd}}(\partial\rB_\eta)$ such that $u=\mathcal{D}_{\partial \rB_\eta}[\mu]\on{\rB_\eta}$ if and only if
\begin{equation}\label{eq:ranDodd}
 \big\langle \gamma_{1}u \,,\, \chi_{\eta\partial\rQ^+_j} \big\rangle =0
 \qquad \forall j\in \{1,\dots,m^{\#}\}\, .
\end{equation}
\item[(iv)]
If $u \in H^{1+\tau}_{\mathrm{loc}}(\rQ^\complement)$ is such that $\Delta u=0$ in $\rQ^\complement$, $u$ is harmonic at infinity, and $\cR^\ast[u]=-u$ ,
then there exists $\mu \in H^{1/2+\tau}_{\mathrm{odd}}(\partial\rQ)$ such that $u=\mathcal{D}_{\partial\rQ}[\mu]\on{\rQ^\complement}$ if and only if
\[
\big\langle \gamma_{1}u \,,\, \chi_{\partial\rQ^+_j} \big\rangle =0
 \qquad \forall j\in \{1,\dots,m^{\#}\}\, .
\]
\end{itemize}
\end{lemma}
\begin{proof}
By the rule of change of variables in integrals, we have 
$\mathcal{D}_{\partial\rB_\eta}\big[\cR^\ast[\psi]\big]=\cR^\ast\big[\mathcal{D}_{\partial\rB_\eta}[\psi]\big]$, hence
\begin{equation}\label{eq:reflK}
\quad 
\mathcal{K}_{\partial\rB_\eta}\big[\cR^\ast[\psi]\big]=\cR^\ast\big[\mathcal{K}_{\partial\rB_\eta}[\psi]\big]\quad \forall \psi \in H^{1/2+\tau}(\partial\rB_\eta) \, . 
\end{equation}
Thus $\mathcal{K}_{\partial\rB_\eta}$ maps odd functions to odd functions. 
By Lemma~\ref{lem:V-} the kernel consists of odd functions in $\mathfrak{V}_-$, that is, odd linear combinations of the characteristic functions $\chi_{\eta\partial\rQ^\times_i}$ and $\chi_{\eta\partial\rQ^\pm_j}$. This space is generated by 
$\chi_{\eta\partial\rQ^+_j}-\chi_{\eta\partial\rQ^-_j}$, $j=1,\dots,m^{\#}$. If $m^{\#}=0$, the operator is therefore injective.\\
Let us now show (iii). For an odd function $u$, the $m^{\#}$ conditions \eqref{eq:ranDodd} are equivalent to the whole set of $m^\times+2m^{\#}$ conditions \eqref{eq:ranD}, because the integrals over $\eta \ee\partial\rQ^\times_i$ vanish by symmetry, and the integrals over $\eta \ee\partial\rQ^-_j$ equal the negatives of the integrals over $\eta \ee\partial\rQ^+_j$. Thus \eqref{eq:ranDodd} is equivalent to the existence of $\mu\in H^{1/2+\tau}(\partial\rB_\eta)$ such that $u=\mathcal{D}_{\partial \rB_\eta}[\mu]\on{\rB_\eta}$. If $\mu$ is not yet odd, we can replace it by its odd part
\[
  \tilde\mu \equiv \tfrac12\big(\mu-\mathcal{R}^\ast[\mu]\big)\,,
\]
which will also represent $u$, that is $u=\mathcal{D}_{\partial \rB_\eta}[\tilde\mu]\on{\rB_\eta}$. Statement (iii) is proved.\\
To prove that $-\tfrac12I+\mathcal{K}_{\partial\rB_\eta}$ is a Fredholm operator of index zero, it remains to show that its range has codimension $m^{\#}$ (we recall that $\mathcal{K}_{\partial\rB_\eta}$ is not compact, in general, when $\partial\rB_\eta$ is a only required to be Lipschitz). We use the observation noted in Remark~\ref{rem:ranK}.
For $g\in H^{1/2+\tau}_{\mathrm{odd}}(\partial\rB_\eta)$, let $u\in H^{1+\tau}(\rB_\eta)$ be its harmonic extension, that is, the unique harmonic function in $\rB_\eta$ satisfying $\gamma_{0}u=g$. It is clear that $u$ is an odd function. The $m^{\#}$ conditions \eqref{eq:ranDodd} that guarantee the representability of $u$ as a double layer potential with a density $\mu \in H^{1/2+\tau}_{\mathrm{odd}}(\partial\rB_\eta)$  can be considered as $m^{\#}$ continuous linear functionals acting on $g$ and defining solvability conditions for the boundary integral equation
\[
 (-\tfrac12I+\mathcal{K}_{\partial\rB_\eta})[\mu]=g\,.
\] 
We have shown that the cokernel of $-\frac12I+\mathcal{K}_{\partial\rB_\eta}$ in $H^{1/2+\tau}_{\mathrm{odd}}(\partial\rB_\eta)$  has dimension $m^{\#}$, and thus the remaining statements of (i) and (ii) follow.\\
Statement (iv) is proved in the same way as (iii) by relying on the characterization of the cokernel of the operator $\mathcal{W}_{\partial\rQ}$ as given by McLean \cite[Thm.~8.20]{Mc00}. 
\end{proof}
In addition to the boundary integral operators, we will also need mapping properties of the double layer potential restricted to some subsets of the domain. 

The first result is global and concerns the entire interior of $\partial\rB$ or exterior of $\partial\rQ$. For this we use the weighted Sobolev spaces of Kondrat'ev type $K^s_\beta$ introduced in Section~\ref{ss:conf} and defined for integer regularity exponent $s=m\in\N$ in \eqref{eq:Kmb}. For non-integer $s$ there exist several equivalent ways to define this space (see \cite{Dauge88}), the shortest (and for us, easiest to use) is by Hilbert space interpolation: 
If $s=m+\tau$, $m\in\N$ and $\tau\in[0,1]$, then 
\[
  K^s_\beta(\Omega) = \big[K^m_\beta(\Omega),K^{m+1}_\beta(\Omega)\big]_\tau\,.
\]
It is clear that if $\Omega$ is a bounded domain with a positive distance to the origin, then the norm in $K^s_\beta(\Omega)$ is equivalent to the norm in the standard Sobolev space $H^s(\Omega)$, and the weight exponent $\beta$ influences only the behavior near the origin and at infinity. Thus let $\chi\in C^\infty_0(\R^2)$ be a cutoff function that is equal to $1$ on the ball $\sB(0,1/2)$ and $0$ outside of $\sB(0,1)$ and for $R>0$ define
\[
 \chi_R(x)=\chi(\tfrac xR)\,;
\]
let further $R_0,R_1$ be such that $0<2R_0\le R_1$ and let $\beta_0,\beta_1\in\R$ be two weight indices. Then if we define  
\begin{equation}
\label{eq:Ksb0b1}
 \DNorm{u}{K^s_{\beta_0\beta_1}(\Omega)} \equiv
 \DNorm{\chi_{R_0}u}{K^s_{\beta_0}(\Omega)} +
 \DNorm{(1-\chi_{R_0})\chi_{R_1}u}{H^s(\Omega)} +
 \DNorm{(1-\chi_{R_1})u}{K^s_{\beta_1}(\Omega)} \,,
\end{equation}
we see that for any Lipschitz domain $\Omega$ 
the norms $\DNorm{u}{K^s_\beta(\Omega)}$ and $\DNorm{u}{K^s_{\beta\beta}(\Omega)}$ are equivalent;
for any bounded $\Omega$  
the norm $\DNorm{u}{K^s_{\beta_0\beta_1}(\Omega)}$ is equivalent to
$\DNorm{u}{K^s_{\beta_0}(\Omega)}$, and
for any $\Omega$ that has a positive distance to the origin 
it is equivalent to
$\DNorm{u}{K^s_{\beta_1}(\Omega)}$.
With this preparation, we can now prove the boundedness of the double layer representation.

\begin{lemma}
\label{lem:Drep}
 For any $\tau\in[-1/2,1/2]$, $\beta_0>-2$ and $\beta_1<0$, the following operators are bounded:
 \begin{align}
 \label{eq:DdB}
 \cD_{\partial\rB} &: H^{1/2+\tau}_{\mathrm{odd}}(\partial\rB)
   \to K^{1+\tau}_{\beta_0}(\rB) \,, \\
 \label{eq:DdQ}
 \cD_{\partial\rQ} &: H^{1/2+\tau}_{\mathrm{odd}}(\partial\rQ)
   \to K^{1+\tau}_{\beta_0\beta_1}(\rQ^\complement) 
   \quad\text{ if }\quad 0\not\in\partial\rQ \,,\\
 \label{eq:DdQ0}
 \cD_{\partial\rQ} &: H^{1/2+\tau}_{\mathrm{odd}}(\partial\rQ)
   \to K^{1+\tau}_{\beta_0\beta_1}(\rQ^\complement) 
   \quad\text{ if }\quad 0\in\partial\rQ \;\,\text{ and }\;
     \beta_0\ge-1-\tau\,.
 \end{align}
\end{lemma}
\begin{proof}
In \eqref{eq:dblreptau} we already quoted from  \cite[Thm.~1]{Co88} that for a bounded domain $\Omega$, $\cD_{\partial\Omega}$ maps $H^{1/2+\tau}(\partial\Omega)$ boundedly into $H^{1+\tau}(\Omega)$. This implies that
\[
  \cD_{\partial\rB} : H^{1/2+\tau}_{\mathrm{odd}}(\partial\rB)
   \to H^{1+\tau}(\rB) 
\]
is bounded. Let $R_0$ be such that $\sB(0,R_0)\subset\rB$ and let
$g$ be the trace of $u\equiv\cD_{\partial\rB}[\psi]$ on $\partial\sB(0,R_0)$. 
Since $u$ is harmonic in $\rB$ and odd, it can be expanded in a Fourier series of the form
\[
  u(x)= \sum_{k=1}^\infty g_k \Big(\frac r{R_0}\Big)^k \sin k\theta \,.
\]
Here $g_k$ are the Fourier coefficients of $g$, and the Sobolev norms of $g$ can equivalently be expressed by weighted norms of the sequence $g_k$.
\[
 \DNormc{g}{H^s(\partial\sB(0,R_0))} =  \sum_{k=1}^\infty k^{2s} |g_k|^2\,.
\]
 It can be verified by explicit computation that for $\beta>-2$ and any $m\in\N$
\[
 \DNormc{u}{K^m_{\beta}(\sB(0,R_0))} = 
   \sum_{k\ge1}c_{k,m}|g_k|^2
   \qquad\text{ with }\quad
   c\,k^{2m-1} \le c_{k,m} \le C\,k^{2m-1}\,.
\]
The constants here depend on $R_0$ and $\beta$, but not on $u$.
Thus $\DNorm{u}{K^m_{\beta}(\sB(0,R_0))}$ is equivalent to 
$\DNorm{g}{H^{m-1/2}(\partial\sB(0,R_0))}$. By interpolation it follows that this is also true for $m$ replaced by $1+\tau$. Thus 
\[
 \DNorm{u}{K^{1+\tau}_{\beta_0}(\sB(0,R_0))} \le
 C\, \DNorm{g}{H^{1/2+\tau}(\partial\sB(0,R_0))} \le
 C\, \DNorm{u}{H^{1+\tau}(\rB)}\,.
\]
Adding $\DNorm{u}{H^{1+\tau}(\rB)}$, we find 
\[
 \DNorm{u}{K^{1+\tau}_{\beta_0}(\rB)} \le
 C\, \DNorm{g}{H^{1/2+\tau}(\partial\rB)}\,,
\]
hence \eqref{eq:DdB}.

For the proof of \eqref{eq:DdQ}, we use a similar argument: 
Let $U=\cD_{\partial\rQ}[\Psi]$ in $\rQ^\complement$ and let $G$ be the trace of $U$ on some $\partial\sB(0,R_1)$ with $R_1$ chosen such that $\rQ\subset\sB(0,R_1)$. Then
\[
 \DNorm{G}{H^{1/2+\tau}(\partial\sB(0,R_1))} \le
 \DNorm{U}{H^{1+\tau}(\rQ^\complement\cap\sB(0,R_1))}\le
 C\, \DNorm{\Psi}{H^{1/2+\tau}(\partial\rQ)}\,. 
\]
Now we write $U$ in $\sB(0,R_1)^\complement$ as a Fourier series, using that it is harmonic in $\rQ^\complement$ and vanishes at infinity (for this we do not even need that $U$ is odd), and prove by explicit calculation of weighted Sobolev norms and interpolation that for any $\beta<0$ there is an estimate
\[
  \DNorm{U}{K^{1+\tau}_{\beta}(\sB(0,R_1)^\complement)}
  \le C\, \DNorm{G}{H^{1/2+\tau}(\partial\sB(0,R_1))} \,.
\]
If $0\in\rQ$, we do not need to estimate $U$ in a neighborhood of $0$. If $0\in\rQ^\complement$, we can get an estimate of $U$ in a neighborhood of $0$ as above for $u$.
Together, this implies \eqref{eq:DdQ}. 

For \eqref{eq:DdQ0}, we use the previous estimate outside of a neighborhood of the origin, but now we additionally have to estimate 
$\DNorm{U}{K^{1+\tau}_{\beta_0}(\rQ^\complement\cap\sB(0,R_0))}$
for some $R_0>0$. We cannot apply the same argument as for $u$ above, because $0$ is on the boundary of $\rQ$, and $U$ is not harmonic in a whole neighborhood of $0$. Instead we will use the fact that if $\beta_0\ge-1-\tau$ then there is a continuous inclusion
\[
 H^{1+\tau}_{\mathrm{odd}}(\rQ^\complement\cap\sB(0,R_0))
 \subset
 K^{1+\tau}_{\beta_0}(\rQ^\complement\cap\sB(0,R_0))\,.
\]
For $\tau\ne0$ this follows from Dauge \cite[Theorem (AA.7)]{Dauge88}. 
It is also true for $\tau=0$ as follows easily from the well known Hardy inequality
\[
 \DNorm{\frac{U(\cdot)}{x_2}}{L^2(\rQ^\complement)}
 \le 2\, \DNorm{\partial_{x_2}U}{L^2(\rQ^\complement)}
\]
for all $U\in H^1(\rQ^\complement)$ satisfying $U=0$ for $x_2=0$.
This inclusion together with the previous estimates that led to \eqref{eq:DdQ} proves \eqref{eq:DdQ0} and ends the proof of the lemma.
\end{proof}
\begin{remark}
\label{rem:DLinf}
If $\tau\in(0,1/2]$, then one also has bounded mappings
\begin{equation}
\label{eq:Dlinf}
 \cD_{\partial\rB} : H^{1/2+\tau}(\partial\rB)
   \to L^\infty(\rB) 
   \;\text{ and }\;
 \cD_{\partial\rQ} : H^{1/2+\tau}(\partial\rQ)
   \to L^\infty(\rQ^\complement) \,.
\end{equation}
This follows for $\cD_{\partial\rB}$ from the Sobolev inclusion
$
 H^{1+\tau}(\rB)\subset L^\infty(\rB)
$
and for $\cD_{\partial\rQ}$ from the Sobolev inclusion on
$\rQ^\complement\cap\sB(0,R_1)$ combined with Fourier series (or simply the maximum principle) on $\sB(0,R_1)^\complement$.
\end{remark}

The second class of results is local in nature and describes the analyticity of the double layer potential near the origin and near infinity in a form that is suitable for our situation of a symmetric domain with small perforations. The result can be considered simply to be a consequence of the analyticity of the fundamental solution $E$ on $\R^2\setminus\{0\}$, and it is similar to the subject studied in Lanza de Cristoforis and Musolino \cite{LaMu13}, but there are some particularities related to the symmetry and weak smoothness of the domains studied here.
\begin{lemma}
\label{lem:Danal}
Let $\tau\in[-1/2,1/2]$.
\begin{itemize}
\item[(i)]
Let $\Omega\subset\R^2$ be a bounded Lipschitz domain.
For positive $\eta$ sufficiently small so that $\eta\Omega\subset\rB$, we define the restriction $\cD_{\partial\rB,\Omega}(\eta)$ of the double layer potential $\cD_{\partial\rB}$ to $\eta\Omega$ written in ``fast'' variables
\[
 \cD_{\partial\rB,\Omega}(\eta)[\psi](X)
 \equiv \cD_{\partial\rB}[\psi]\on{\eta\Omega}(\eta X) 
 \quad (X\in\Omega) \qquad \forall \psi \in H^{1/2+\tau}_{\mathrm{odd}}(\partial\rB)\,.
\] 
Then there exists $\eta_1>0$ such that the function 
$\eta\mapsto\cD_{\partial\rB,\Omega}(\eta)$ has for any $s\in\R$ a continuation to $\eta\in(-\eta_1,\eta_1)$ as an analytic function with values in
$\sL\big(H^{1/2+\tau}_{\mathrm{odd}}(\partial\rB),\,H^s(\Omega)\big)$.
\item[(ii)]
Let $\Omega\subset\R^2$ be a bounded Lipschitz domain such that $0\not\in\overline\Omega$.
For positive $\eta$ sufficiently small so that $(1/\eta)\Omega\subset\rQ^\complement$, we define the restriction $\cD^\complement_{\partial\rQ,\Omega}(\eta)$ of the double layer potential $\cD_{\partial\rQ}$ to $(1/\eta)\Omega$ written in ``slow'' variables
\[
 \cD^\complement_{\partial\rQ,\Omega}(\eta)[\Psi](x)
  \equiv \cD_{\partial\rQ}[\Psi]\on{(1/\eta)\Omega}(\tfrac{x}{\eta}) 
 \quad (x\in\Omega) \qquad \forall \Psi \in H^{1/2+\tau}_{\mathrm{odd}}(\partial\rQ) \,.
\] 
Then there exists $\eta_1>0$ such that the function 
$\eta\mapsto\cD^\complement_{\partial\rQ,\Omega}(\eta)$ has for any $s\in\R$  a continuation to $\eta\in(-\eta_1,\eta_1)$ as an analytic function with values in
$\sL\big(H^{1/2+\tau}_{\mathrm{odd}}(\partial\rQ),\,H^s(\Omega)\big)$.
\item[(iii)]
In addition, $\cD_{\partial\rB,\Omega}(0)=0$ and $\cD^\complement_{\partial\rQ,\Omega}(0)=0$.
\end{itemize}
\end{lemma}
\begin{proof}
The proofs for (i) and (ii) are similar. Both use the fact that the double layer potential is analytic outside of the boundary and vanishes at infinity, and for an odd density it vanishes at the origin. We give the proof of (ii) and leave the proof of (i) and (iii) to the reader. \\
Let $\Psi\in H^{1/2+\tau}_{\mathrm{odd}}(\partial\rQ)$ and define 
$W=\cD_{\partial\rQ}[\Psi]$ in $\rQ^\complement$. We can choose $R_0$ such that 
$\overline\rQ\subset\sB(0,R_0)$ and $R_1,R_2$ such that
$\Omega\subset\sB(0,R_2)\cap\sB(0,R_1)^\complement$.
For $|X|\ge R_0$, we can expand the harmonic and odd function $W$ in a Fourier series
\begin{equation}
\label{eq:Wexp}
  W(X)= \sum_{k=1}^\infty w_k \Big(\frac R{R_0}\Big)^{-k} \sin k\theta \,.
\end{equation}
Here $(R,\theta)$ denote polar coordinates for $X$, and from the fact that $\cD_{\partial\rQ}$ maps $H^{1/2+\tau}_{\mathrm{odd}}(\partial\rQ)$ to $H^{1+\tau}_{\mathrm{odd}}(\rQ^\complement\cap\sB(0,R_0))$ we deduce the (crude) estimate that the $w_k$ are bounded and satisfy an estimate
\[
 \sup_k |w_k| \le C\, \DNorm{\Psi}{H^{1/2+\tau}(\partial\rQ)}\,.
\]
If $\eta\in(0,R_1/R_0)$, then 
$X\in(1/\eta)\Omega \subset\sB(0,R_2/\eta)\cap\sB(0,R_1/\eta)^\complement$ implies 
$|X|>R_0$, so that we can use the expansion
\eqref{eq:Wexp} for the restriction of $W$ to $(1/\eta)\Omega$. Writing this in slow variables $x=\eta X$, or in polar coordinates with $|x|=r=\eta R$, we get
\[
 \cD^\complement_{\partial\rQ,\Omega}(\eta)[\Psi](x) = W(\frac x\eta) 
 = \sum_{k=1}^\infty \eta^{k}\, w_k\, p_k(x)
 \quad\text{ with }\;
 p_k(x)=\Big(\frac r{R_0}\Big)^{-k} \sin k\theta
\]
By explicit computation for any chosen $m\in\N$, we can estimate the $H^m$ norm of $p_k$
\[
 \DNorm{p_k}{H^m(\Omega)}\le \DNorm{p_k}{H^m(\sB(0,R_2)\cap\sB(0,R_1)^\complement)}
 \le C\, \big(\frac{R_0}{R_1}\big)^{k}k^{2m-1}\,,
\]
with $C$ independent of $k$.
We conclude that $\cD^\complement_{\partial\rQ,\Omega}(\eta)$ has a convergent  expansion
\begin{equation}
\label{eq:Dexp}
 \cD^\complement_{\partial\rQ,\Omega}(\eta) =
 \sum_{k=1}^\infty \eta^{k} \, D_k \,,
\end{equation}
where the $D_k$ are bounded linear operators from 
$H^{1/2+\tau}_{\mathrm{odd}}(\partial\rQ)$ to $H^m(\Omega)$ satisfying 
\[
 \DNorm{D_k}{\sL\big(H^{1/2+\tau}_{\mathrm{odd}}(\partial\rQ),H^m(\Omega)\big)}
 \le C\, \big(\frac{R_0}{R_1}\big)^{k} k^{2m-1}\,.
\]
It follows that the expansion \eqref{eq:Dexp} converges for $|\eta|<R_1/R_0$ and 
this proves the analyticity as claimed in (ii).
\end{proof}

\subsection{The Dirichlet problem in a symmetric perforated domain}\label{ss:Dir}

In this subsection we apply the double layer representation to the solution of the Dirichlet problem in our perforated symmetric domain $\rB_\eta$. 

Thus we assume that we are given odd functions 
$\psi\in H^{1/2+\tau}_{\mathrm{odd}}(\partial\rB)$ and
$\Psi\in H^{1/2+\tau}_{\mathrm{odd}}(\partial\rQ)$, and 
we denote by $u[\eta,\psi,\Psi]$ the unique solution in $H^{1+\tau}(\rB_\eta)$ of the boundary value problem
\begin{equation}\label{eq:symdir}
\left\{
\begin{array}{ll}
\Delta u=0&\text{ in }\;\rB_\eta\,,\\
\gamma_{0}u=\psi &\text{ on }\;\partial\rB\,,\\
\gamma_{0}u=\Psi(\cdot/\eta) &\text{ on } \;\eta \ee\partial\rQ\,.\\
\end{array}
\right.
\end{equation}
We would like to represent $u[\eta,\psi,\Psi]$ as a double layer potential.
It is clear that $u$ is an odd function. The conditions \eqref{eq:ranDodd} will, however, not be satisfied, in general, if the number $m^{\#}$ of ``paired holes'' is non-zero. As a remedy for this problem, we introduce harmonic functions $\Xi_1,\dots,\Xi_{m^{\#}}$ that span a complement of the range of the double layer potential operator. We define $\Xi_j$ as the 
 unique function in $H^{1+\tau}_{\mathrm{loc}}\big(\mathbb{R}^2 \setminus( \overline{\rQ^+_j}\cup\overline{\rQ^-_j})\big)$ such that
\begin{equation}\label{eq:Xi}
\left\{
\begin{array}{ll}
\Delta \Xi_j=0&\text{ in }\mathbb{R}^2 \setminus\big( \overline{\rQ^+_j}\cup\overline{\rQ^-_j}\big)\,,\\
\gamma_{0}\Xi_j=\pm 1&\text{ on }\partial\rQ_j^\pm\,,\\
 \|\Xi_j\|_\infty<+\infty\, .&
\end{array}
\right.
\end{equation}
A simple argument for the existence of such functions $\Xi_j$ is to use the Kelvin transformation with origin in $\rQ^+_j$ that reduces the exterior Dirichlet problem problem~\eqref{eq:Xi} to a Dirichlet problem on a bounded domain (see Folland \cite[Ch. 2.I]{Fo95}), and then invoke the existence and uniqueness of solution of the Dirichlet problem on a bounded domain.

The uniqueness of $\Xi_j$ implies in particular that $\Xi_j$ is odd,
\begin{equation}\label{eq:symphi}
\Xi_j(X)=-\cR^\ast[\Xi_j](X) \qquad \text{for $X \in\mathbb{R}^2 \setminus\big( \overline{\rQ^+_j}\cup\overline{\rQ^-_j}\big)$.}
\end{equation}
Then by \eqref{eq:symphi} and by the harmonicity in $\mathbb{R}^2 \setminus\big( \overline{\rQ^+_j}\cup\overline{\rQ^-_j}\big)$ and at infinity of $\Xi_j$ it follows that
\begin{equation}\label{eq:limphi}
\lim_{X\to\infty}\Xi_j(X)=0\,.
\end{equation}
Concerning the integrals of the normal derivative of $\Xi_j$ over the boundaries of the connected components of $\rQ$, it follows from the harmonicity that they vanish except for the components $\rQ_j^\pm$, in particular 
\begin{equation}\label{eq:phiQpm0}
 \big\langle \gamma_{1}\Xi_j\,,\, \chi_{\partial\rQ^\pm_k} \big\rangle =0
 \quad \forall  k\in \{1,\dots,m^{\#}\}\setminus\{j\}
 \, .
\end{equation}
From the harmonicity at infinity and \eqref{eq:limphi} follows that 
$\nabla \Xi_j \in L^2\Big(
 \mathbb{R}^2 \setminus\big( \overline{\rQ^+_j}\cup\overline{\rQ^-_j}\big)\Big)$ 
and that we can use the Divergence Theorem, which gives
\begin{multline}\label{eq:phiQpm}
0<\int_{\mathbb{R}^2 \setminus\big( \overline{\rQ^+_j}\cup\overline{\rQ^-_j}\big)}
 |\nabla\Xi_j(X)|^2\, \mathrm{d}X\\
 =-\int_{\partial \rQ^{+}_j} \partial_n \Xi_j\, \mathrm{d}s
  +\int_{\partial \rQ^{-}_j} \partial_n \Xi_j\, \mathrm{d}s
= -2\big\langle \gamma_{1}\Xi_j\,,\, \chi_{\partial\rQ^+_j} \big\rangle\,.
\end{multline}

We can now show the following augmented double layer representation for the solution $u[\eta,\psi,\Psi]$ of problem \eqref{eq:symdir}. 

\begin{lemma}\label{lem:equiv}
Let $\tau\in[-1/2,1/2]$. Let $\eta \in (0,\eta_0)$. Then the following statements hold.
\begin{itemize}
\item[(i)] If $m^{\#}=0$, then there exists a unique function $\mu\in  H^{1/2+\tau}_{\mathrm{odd}}(\partial \rB_\eta)$ such that
\begin{equation}\label{eq:equiv=}
u[\eta,\psi,\Psi]=\mathcal{D}_{\partial \rB_\eta}[\mu] \qquad \text{in $\rB_\eta$.}
\end{equation}
\item[(ii)] If $m^{\#}>0$, then there exists a unique 
$m^{\#}$-tuple $\boldsymbol{c}=(c_1,\dots,c_{m^{\#}})\in\mathbb{R}^{m^{\#}}$ and
a unique function $\mu\in  H^{1/2+\tau}_{\mathrm{odd}}(\partial \rB_\eta)$
satisfying
\begin{equation}\label{eq:equiv>}
\left\{
\begin{array}{ll}
u[\eta,\psi,\Psi]=\mathcal{D}_{\partial \rB_\eta}[\mu]+\sum_{j=1}^{m^{\#}}c_j\Xi_j(\cdot/\eta) & \text{in $\rB_\eta$}\\
\int_{\eta \partial \rQ^{+}_j}\mu\, \mathrm{d}s=0 & \forall j \in \{1,\dots,m^{\#}\}\, . 
\end{array}
\right.
\end{equation}
\end{itemize}
\end{lemma}

\begin{proof}
Statement (i) follows from Lemma \ref{lem:Kodd} (ii). We now consider statement (ii). We first note that by \eqref{eq:phiQpm} for each $j \in \{1,\dots,m^{\#}\}$ there exists a unique $c_j \in \mathbb{R}$ such that
\[
\big\langle \gamma_{1} u[\eta,\psi,\Psi]\,,\, \chi_{\eta\partial\rQ^+_j} \big\rangle
-c_j\, \big\langle \gamma_{1} \Xi_j(\cdot/\eta) \,,\, \chi_{\eta\partial\rQ^+_j} \big\rangle =0\,.
\] 
Using \eqref{eq:phiQpm0}, it follows that the function 
$u[\eta,\psi,\Psi]-\sum_{j=1}^{m^{\#}}c_j\Xi_j(\cdot/\eta)$ satisfies the conditions of Lemma~\ref{lem:Kodd} for the existence of a representation as a double layer potential.
As a consequence, there exists 
 $\tilde{\mu} \in H^{1/2+\tau}_{\mathrm{odd}}(\partial \rB_\eta)$ such that
\[
\mathcal{D}_{\partial \Omega_\eta}[\tilde{\mu}]=u[\eta,\psi,\Psi]-\sum_{j=1}^{m^{\#}}c_j\Xi_j(\cdot/\eta) \qquad \text{in $\rB_\eta$}\, .
\]
Recalling from Lemma~\ref{lem:Kodd}(ii) that the kernel 
$\mathfrak{V}_{-,\mathrm{odd}} \equiv \mathfrak{V}_{-}\cap H^{1/2+\tau}_\mathrm{odd}(\partial\rB_\eta)$ of the operator 
$-\frac{1}{2}I+\mathcal{K}_{\partial \rB_\eta}$ acting on odd functions is spanned by 
the functions
$\{\chi_{\eta\partial\rQ^+_j}-\chi_{\eta\partial\rQ^-_j}\}_{j=1}^{m^{\#}}$, we find that among the functions 
$\mu \in \tilde\mu + \mathfrak{V}_{-,\mathrm{odd}}$ that satisfy the first line of \eqref{eq:equiv>} there is exactly one satisfying the side conditions of the second line of \eqref{eq:equiv>}.
\end{proof}

With the help of the augmented double layer potential representation \eqref{eq:equiv>} we can now rewrite our Dirichlet problem \eqref{eq:symdir} as an equivalent boundary integral equation on $\partial\rB_\eta$. This is still a problem on an $\eta$-dependent domain, but it is possible to interpret it as a system of boundary integral equations in the function space $H^{1/2+\tau}_{\mathrm{odd}}(\partial \rB)\times  H^{1/2+\tau}_{\mathrm{odd}}(\partial \rQ)$ defined on the fixed domain $\partial\rB\times \partial\rQ$. Owing to the special form of the double layer kernel, this system has a simple form that makes it natural to study the dependence on $\eta$ in the limit $\eta\to0$ and even to extend it in an analytic way to a neighborhood of $\eta=0$.

The formulation \eqref{eq:symdir} of our Dirichlet problem already makes use of the identification of a function defined on $\partial\rB_\eta$ with a pair $\big(\psi,\Psi(\cdot/\eta)\big)$ of functions, the first one defined on $\partial\rB$ and depending on standard or ``slow'' variables $x$, the second one defined on $\partial\rQ$ and depending on ``fast'' variables $X=x/\eta$. Let $\cJ_\eta$ denote this mapping from $H^{1/2+\tau}_{\mathrm{odd}}(\partial \rB)\times  H^{1/2+\tau}_{\mathrm{odd}}(\partial \rQ)$ to $H^{1/2+\tau}_{\mathrm{odd}}(\partial \rB_\eta)$, which obviously is an isomorphism.
\begin{equation}\label{eq:cJ}
 \cJ_\eta[\phi,\Phi](x) \equiv \left\{
\begin{array}{ll}
  \phi(x) & \text{on $\partial \rB$}\, ,\\
  \Phi(x/\eta) & \text{on $\eta \partial  \rQ$}\, .
\end{array}
\right.
\end{equation}

The boundary integral equation for \eqref{eq:symdir} is obtained from the representation formula \eqref{eq:equiv>} by taking traces on $\partial\rB_\eta$. 
In order to treat simultaneously the case $m^{\#}=0$ and the case $m^{\#}>0$, from now on we will assume that the symbols $c_1,\dots,c_{m^{\#}}$ and $\sum_{j=1}^{m^{\#}}c_j\phi_j$ are omitted if $m^{\#}=0$.
In addition we find it convenient to set
\[
 H^{1/2+\tau}_{\mathrm{odd}}(\partial \rQ)_{\#}\equiv \Big\{\mu \in  H^{1/2+\tau}_{\mathrm{odd}}(\partial \rQ) \colon \int_{\partial \rQ^{+}_j}\mu\, \mathrm{d}s=0\;\ \forall j \in \{1,\dots,m^{\#}\}\Big\}\, .
\]
Clearly, if $m^{\#}=0$ then $H^{1/2+\tau}_{\mathrm{odd}}(\partial \rQ)_{\#}= H^{1/2+\tau}_{\mathrm{odd}}(\partial \rQ)$.
We can then write the trace of \eqref{eq:equiv>} as the problem of finding 
$\mu\in\cJ_\eta\left[H^{1/2+\tau}_{\mathrm{odd}}(\partial \rB)\times  H^{1/2+\tau}_{\mathrm{odd}}(\partial \rQ)_{\#}\right]$  and 
$\boldsymbol{c}\in\mathbb{R}^{m^{\#}}$
such that
\begin{equation}\label{eq:BIEdBeta}
 (-\tfrac12I + \mathcal{K}_{\partial\rB_\eta})[\mu]   
   +\sum_{j=1}^{m^{\#}}c_j\,\gamma_{0}\,\Xi_j(\cdot/\eta) = g 
   \quad\text{ on } \partial\rB_\eta
\end{equation}
where $g=\cJ_\eta[\psi,\Psi]$. With $\mu=\cJ_\eta[\phi,\Phi]$ we find a first form of the equivalent system of boundary integral equations on $\partial\rB\times \partial\rQ$.
\begin{equation}\label{eq:BIEdBdQ}
 \cJ_\eta^{-1}\circ\biggl[(-\tfrac12I + \mathcal{K}_{\partial\rB_\eta})
   \circ\cJ_\eta[\phi,\Phi] +\sum_{j=1}^{m^{\#}}c_j\,\gamma_{0}\,\Xi_j(\cdot/\eta)\biggr] = \big(\psi,\Psi\big).
\end{equation}
We will now describe this system in more detail.

Changing variables $y\mapsto\eta Y$ in the double layer integral and using the fact that 
\[
  \nabla E(x) = -\frac{x}{2\pi|x|^2}
\]
is a function homogeneous of degree $-1$, we can write
\[
\begin{split}
\mathcal{D}_{\partial \rB_\eta}\Big[\cJ_\eta[\phi,\Phi]\Big]&=
 \mathcal{D}_{\partial\rB}[\phi]
 -\int_{\eta\partial\rQ} \Phi(y/\eta) 
 \partial_{n(y)} E (\cdot-y)\, \mathrm{d}s_y\\
 &=\mathcal{D}_{\partial\rB}[\phi]
 +\eta\int_{\partial\rQ}\Phi(Y)\, n(Y)\cdot \nabla E(\cdot-\eta Y)\, \mathrm{d}s_Y\quad \text{in $\rB_\eta$}\, ,
\end{split}
\]
and then express the representation formula \eqref{eq:equiv>} both in ``slow'' variables:
\begin{equation}\label{eq:rep}
 u[\eta,\psi,\Psi]=\mathcal{D}_{\partial \rB}[\phi]
 +\eta\int_{\partial \rQ}\Phi(Y)\, n(Y)\cdot \nabla E (\cdot-\eta Y)\, \mathrm{d}s_Y
 +\sum_{j=1}^{m^{\#}}c_j\Xi_j(\cdot/\eta) \quad \text{in $\rB_\eta$}\, 
\end{equation}
and in ``fast'' variables:
\begin{equation}\label{eq:rep-bis}
\begin{split}
 u[\eta,\psi,\Psi](\eta X)&=\mathcal{D}_{\partial\rB}[\phi](\eta X)\\ &\quad
 +\eta\int_{\partial\rQ}\Phi(Y)\, n(Y)\cdot \nabla E(\eta(X-Y))\, \mathrm{d}s_Y
 +\sum_{j=1}^{m^{\#}}c_j\Xi_j(X) \\
 &=\mathcal{D}_{\partial\rB}[\phi](\eta X)\\ &\quad
 +\int_{\partial\rQ}\Phi(Y)\, n(Y)\cdot \nabla E(X-Y)\, \mathrm{d}s_Y
 +\sum_{j=1}^{m^{\#}}c_j\Xi_j(X)\\
 &=-\mathcal{D}_{\partial\rQ}[\Phi](X)+\mathcal{D}_{\partial\rB}[\phi](\eta X)
 +\sum_{j=1}^{m^{\#}}c_j\Xi_j(X)\, .
\end{split}
\end{equation}

We obtain the concrete form of the system \eqref{eq:BIEdBdQ} by taking traces of the 
equalities \eqref{eq:rep} and \eqref{eq:rep-bis} on $\partial \rB$ and on $\partial \rQ$, respectively. We deduce with \eqref{eq:symdir} that the unique element 
$(\phi,\Phi,\boldsymbol{c})$ of $H^{1/2+\tau}_{\mathrm{odd}}(\partial \rB)\times  H^{1/2+\tau}_{\mathrm{odd}}(\partial \rQ)_{\#}\times \mathbb{R}^{m^{\#}}$ such that  \eqref{eq:rep} holds is the (unique) solution in 
$ H^{1/2+\tau}_{\mathrm{odd}}(\partial \rB)\times  H^{1/2+\tau}_{\mathrm{odd}}(\partial \rQ)_{\#}\times \mathbb{R}^{m^{\#}}$ of
\begin{equation}\label{eq:sys}
\begin{split}
(-\tfrac{1}{2}I+\mathcal{K}_{\partial \rB})[\phi](x)
 +\eta\int_{\partial \rQ}\!\!\Phi(Y)\, n(Y)\cdot \nabla E(x-\eta Y)\, \mathrm{d}s_Y
 &\\
 +\sum_{j=1}^{m^{\#}}c_j\Xi_j(x/\eta)&=\psi(x)
 \, \text{ , $\,x \in \partial \rB$} ,\\
-(\tfrac{1}{2}I+\mathcal{K}_{\partial \rQ})[\Phi](X)
 +\mathcal{D}_{\partial \rB}[\phi](\eta X)
 +\sum_{j=1}^{m^{\#}}c_j\Xi_j(X)&=\Psi(X)  
 \, \text{, $X \in \partial \rQ$}\, .\\
\end{split}
\end{equation}
Problem \eqref{eq:symdir} is now converted into the equivalent system of boundary integral equations \eqref{eq:sys}, which we can write as an $\eta$-dependent family of problems 
\begin{equation}\label{eq:Meta}
  \cM(\eta)\begin{pmatrix}\phi\\\Phi\\\boldsymbol{c}\end{pmatrix}
  = \begin{pmatrix}\psi\\\Psi\end{pmatrix}
\end{equation}
with a block $(2\times3)$ operator
\[
 \cM(\eta)\equiv\
 \begin{pmatrix}
 \cM_{11}(\eta)&\cM_{12}(\eta)&\cM_{13}(\eta)\\
 \cM_{21}(\eta)&\cM_{22}(\eta)&\cM_{23}(\eta)
 \end{pmatrix}
\]
acting from 
$
 H^{1/2+\tau}_{\mathrm{odd}}(\partial \rB)\times  H^{1/2+\tau}_{\mathrm{odd}}(\partial \rQ)_{\#}\times \mathbb{R}^{m^{\#}}
$ to
$ H^{1/2+\tau}_{\mathrm{odd}}(\partial \rB)\times  H^{1/2+\tau}_{\mathrm{odd}}(\partial \rQ)$.
In this form \eqref{eq:sys}, it is now possible to extend the problem to $\eta=0$ and to analyze the analyticity of its dependence on $\eta$. The main result in this section is the following.
\begin{theorem}\label{thm:Meta=0}
 Let $\tau\in[-1/2,1/2]$ and let $\cM(\eta)$ be defined by \eqref{eq:sys}, \eqref{eq:Meta}. 
\begin{itemize}
\item[(i)]
 There exists $\eta_1\in(0,\eta_0)$ such that the operator valued function
 $\cM$ mapping $\eta$ to
 \[
 \cM(\eta)\in \sL\Big(
  H^{1/2+\tau}_{\mathrm{odd}}(\partial \rB)\times  H^{1/2+\tau}_{\mathrm{odd}}(\partial \rQ)_{\#}\times \mathbb{R}^{m^{\#}}
  ,\;
  H^{1/2+\tau}_{\mathrm{odd}}(\partial \rB)\times  H^{1/2+\tau}_{\mathrm{odd}}(\partial \rQ)
  \Big)
 \]
 admits a real analytic continuation to $(-\eta_1,\eta_1)$.
\item[(ii)]
 For each $\eta\in(-\eta_1,\eta_1)$, the operator $\cM(\eta)$ is an isomorphism.
\end{itemize}
\end{theorem}
\begin{proof}
We will show first that the matrix elements $\cM_{mn}(\eta)$ are operator functions of $\eta$ that can be extended as real analytic functions in a neighborhood of $\eta=0$, and then that for $\eta=0$ the operator $\cM(0)$ is an isomorphism. From this it will follow that $\cM(\eta)$ is an isomorphism for $\eta$ in a neighborhood of $0$. Note that our previous arguments leading to \eqref{eq:sys} already showed that $\cM(\eta)$ is an isomorphism for $\eta\in(0,\eta_0)$.

We will begin by analyzing the dependence of the matrix elements $\cM_{mn}(\eta)$ on $\eta$, in particular at $\eta=0$. Of these, $\cM_{11}$, $\cM_{22}$ and $\cM_{23}$ are independent of $\eta$. 

The matrix elements 
\[
 \cM_{12}(\eta): \left\{ 
 \begin{array}{l}
   H^{1/2+\tau}_{\mathrm{odd}}(\partial \rQ)_{\#} \to
   H^{1/2+\tau}_{\mathrm{odd}}(\partial \rB)\\  
   \Phi\mapsto 
   \eta\int_{\partial \rQ}\Phi(Y)\, n(Y)\cdot \nabla E(\cdot-\eta Y)\, \mathrm{d}s_Y
  \end{array} \right.
\]
and
\[
 \cM_{21}(\eta): \left\{ 
 \begin{array}{l}
   H^{1/2+\tau}_{\mathrm{odd}}(\partial \rB)\to
   H^{1/2+\tau}_{\mathrm{odd}}(\partial \rQ) 
   \\  
   \phi\mapsto 
   \mathcal{D}_{\partial \rB}[\phi](\eta\,\cdot\,)
  \end{array} \right.
\]
depend analytically on $\eta$ as long as $\partial\rB$ and $\eta \ee\partial\rQ$ do not intersect, and both vanish for $\eta=0$. This follows from Lemma~\ref{lem:Danal} by taking traces. 

For the remaining matrix element
\[
 \cM_{13}(\eta): \left\{ 
 \begin{array}{l}
   \mathbb{R}^{m^{\#}} \to
   H^{1/2+\tau}_{\mathrm{odd}}(\partial \rB)\\  
   \boldsymbol{c}\mapsto 
   \sum_{j=1}^{m^{\#}}c_j\Xi_j(\cdot/\eta)
  \end{array} \right.
\]
the analyticity and the vanishing at $\eta=0$ follow from the analyticity of $\Xi_j$ at infinity. In fact, we could simply make the same argument as for $\cM_{12}$, based on Lemma~\ref{lem:Danal}(ii), if we could represent $\Xi_j$ as a double layer potential 
$\mathcal{D}_{\partial\rQ}[\mu]$
with odd density $\mu$. But the $\Xi_j$ were precisely constructed to be in the complement of the range of $\mathcal{D}_{\partial\rQ}$, so that this is impossible. One can, however, choose a connected smooth domain $\Omega_\#$ that satisfies $\rQ\subset\Omega_\#$ and $\eta_0\Omega_\#\subset\rB$ and is symmetric with respect to the reflection $\cR$. Then $\Xi_j$ will be representable as a double layer potential 
$\mathcal{D}_{\partial \Omega_\#}[\mu_j]$ on $\Omega_\#^\complement$, because 
$\int_{\partial\Omega_\#}\partial_n \Xi_j\,\mathrm{d}s=0$ by symmetry and we can  apply Lemma~\ref{lem:Kodd}(iv).

We now turn to the proof that $\cM(0)$ is an isomorphism. From the description of the solvability of the exterior Dirichlet problem with data on $\partial\rQ$ implied by Lemma~\ref{lem:Kodd}(iv) one can deduce, by taking traces on $\partial\rQ$, the unique solvability of the corresponding augmented boundary integral equation. This follows from the same arguments that led to the unique solvability of the augmented boundary integral equation \eqref{eq:BIEdBeta} on $\partial\rB_\eta$ associated with the interior Dirichlet problem in $\rB_\eta$. Taking into account that the normal vector on $\partial\rQ$ is by our convention exterior to $\rQ$ and therefore interior to $\rQ^\complement$, we find a change in the sign of the operator $\cK_{\partial\rQ}$ and can state the following lemma that provides the remaining argument for the completion of the proof of Theorem~\ref{thm:Meta=0}.
\begin{lemma}\label{lem:BIEdQ}
 Let $\tau\in[-1/2,1/2]$. For any $\Psi\in H^{1/2+\tau}_{\mathrm{odd}}(\partial \rQ) $ there exist unique $\Phi\in H^{1/2+\tau}_{\mathrm{odd}}(\partial \rQ)_{\#}$, $\boldsymbol{c}\in \mathbb{R}^{m^{\#}}$ such that
\begin{equation}\label{eq:BIEdQ}
  (-\tfrac12I - \mathcal{K}_{\partial\rQ})[\Phi]   
   +\sum_{j=1}^{m^{\#}}c_j\,\gamma_{0}\,\Xi_j = \Psi 
   \quad\text{ on } \partial\rQ \, .
\end{equation}
\end{lemma}
Now if we use the values at $\eta=0$ of the matrix elements of $\cM(\eta)$, we find
\[
 \cM(0)=
 \begin{pmatrix}
 \cM_{11}&0&0\\
 0&\cM_{22}&\cM_{23}
 \end{pmatrix} \,.
\]
Here $\cM_{11}=-\tfrac12I + \mathcal{K}_{\partial\rB}$ is an isomorphism from
$H^{1/2+\tau}_{\mathrm{odd}}(\partial \rB)$ to itself according to Lemma~\ref{lem:Kodd}(ii). And the fact that the operator
$(\cM_{22},\,\cM_{23})$ is an isomorphism from 
$H^{1/2+\tau}_{\mathrm{odd}}(\partial \rQ)_{\#}\times \mathbb{R}^{m^{\#}}$ to 
$H^{1/2+\tau}_{\mathrm{odd}}(\partial \rQ)$ is precisely the statement of Lemma~\ref{lem:BIEdQ}. Together this shows that $\cM(0)$ is an isomorphism as claimed, and the proof of the theorem is complete.
\end{proof}

\section{The convergent expansion}
\label{s:convexp}

As a consequence of Theorem~\ref{thm:Meta=0}, we obtain that the unique solution of the boundary integral system \eqref{eq:Meta} depends analytically on $\eta\in(-\eta_1,\eta_1)$. Namely,
\begin{equation}
\label{eq:solEqInt}
 \begin{pmatrix}\phi\\\Phi\\\boldsymbol{c}\end{pmatrix} =
 \cM(\eta)^{-1}\, \begin{pmatrix}\psi\\\Psi\end{pmatrix} \, ,
\end{equation}
and the operator function $\eta\mapsto\cM(\eta)^{-1}$ is real analytic for $\eta\in(-\eta_1,\eta_1)$.
Inserting this form of the solution of \eqref{eq:Meta} into the representation formula \eqref{eq:equiv>}, we find that the solution $u$ of the Dirichlet problem \eqref{eq:symdir} depends analytically on $\eta\in(-\eta_1,\eta_1)$, too, and therefore has a convergent expansion in powers of $\eta$ in a neighborhood of $\eta=0$. 
From this, by comparing \eqref{eq:symdir} with the form \eqref{eq:petagen} of the Dirichlet data in the residual problem found in Section~\ref{ss:residual}, we will obtain a convergent double series for the solution $v_\eta$ of the residual problem. In this way we will then be able to complete the construction of a convergent expansion of the solution of the original problem \eqref{eq:poisson}.

\subsection{Analytic parameter dependence for the auxiliary Dirichlet problem \protect\eqref{eq:symdir}}
\label{ss:4.1}
In this subsection, we consider the Dirichlet problem \eqref{eq:symdir} on the perforated domain $\rB_\eta$, where the Dirichlet data are given by $(\psi,\Psi)$ independent of $\eta$.

Let us first note that the auxiliary functions $\Xi_j$ introduced in \eqref{eq:Xi} can be considered in the same weighted Sobolev spaces that appear in Lemma~\ref{lem:Drep}.
\begin{equation}
\label{eq:XiK}
 \Xi_j \in K^{1+\tau}_{\beta_0\beta_1}(\rQ^\complement) 
 \quad \text{ for all } \tau\in[-1/2,1/2]\,,\;
 \beta_0>-2, \beta_1<0\,.
\end{equation}
We can now prove the first result about the solution of the boundary value problem \eqref{eq:symdir}. It is a global decomposition of the solution into two terms that are analytic with respect to $\eta$ if the first one is written in slow coordinates and the second one in fast variables.
\begin{theorem}
\label{thm:u=w+W}
Let $\tau\in[-1/2,1/2]$,
$\beta_0>-2$ with $\beta_0\ge-1-\tau$ if $0\in\partial\rQ$, 
$\beta_1<0$. Then there exists a sequence of bounded linear operators
\[
 \cL_n \,:\;
 H^{1/2+\tau}_{\mathrm{odd}}(\partial\rB) \times 
   H^{1/2+\tau}_{\mathrm{odd}}(\partial\rQ)
   \,\to\;
 K^{1+\tau}_{\beta_0}(\rB) \times K^{1+\tau}_{\beta_0\beta_1}(\rQ^\complement)
 \,,\quad j\in\N
\]
such that the solution $u[\eta,\psi,\Psi]$ of the Dirichlet problem \eqref{eq:symdir} in $\rB_\eta$ has the following form
\[
  u[\eta,\psi,\Psi](x) = w(x) + W(x/\eta)
\] 
with $w\in K^{1+\tau}_{\beta_0}(\rB)$ and $W\in K^{1+\tau}_{\beta_0\beta_1}(\rQ^\complement)$ given by the convergent series
\begin{equation}
\label{eq:wWsum}
  \binom wW = \sum_{n=0}^\infty \eta^n\,  \cL_n \binom\psi\Psi \,.
\end{equation}
There exists $\eta_1>0$ such that for any $\eta\in(-\eta_1,\eta_1)$ the series 
\[
 \cL(\eta) \equiv \sum_{n=0}^\infty \eta^n\,  \cL_n
\]
converges in the operator norm of
$\sL\big(
  H^{1/2+\tau}_{\mathrm{odd}}(\partial\rB) \times 
   H^{1/2+\tau}_{\mathrm{odd}}(\partial\rQ)
   \,,\;
 K^{1+\tau}_{\beta_0}(\rB) \times K^{1+\tau}_{\beta_0\beta_1}(\rQ^\complement)
 \big)\,.
$
\end{theorem}
\begin{proof}
According to the representation formula \eqref{eq:equiv>} in the form \eqref{eq:rep},
we have  
\[
 u[\eta,\psi,\Psi](x) = w(x) + W(x/\eta) \text{ with }
 w=\mathcal{D}_{\partial \rB}[\phi] \text{ , }
 W=-\mathcal{D}_{\partial\rQ}[\Phi]+\sum_{j=1}^{m^{\#}}c_j\Xi_j\,,
\]
where $(\phi,\Phi,\boldsymbol{c})$ is the solution of our system of boundary integral equations \eqref{eq:Meta}.
Now we use the solution formula \eqref{eq:solEqInt} of this system, which involves the analytic resolvent
\[
 \cM(\eta)^{-1} = \sum_{n=0}^\infty \eta^n \cM_n\,,
\]
where $\cM_n$ is a sequence of operators such that, after possibly shrinking $\eta_1>0$, this series converges for
$\eta\in(-\eta_1,\eta_1)$ in the operator norm of
$\sL\big(
  H^{1/2+\tau}_{\mathrm{odd}}(\partial\rB) \times 
   H^{1/2+\tau}_{\mathrm{odd}}(\partial\rQ)
   \,,\;
 H^{1/2+\tau}_{\mathrm{odd}}(\partial \rB)\times  H^{1/2+\tau}_{\mathrm{odd}}(\partial \rQ)_{\#}\times \mathbb{R}^{m^{\#}}
  \big)\,.
$
We combine this with the mapping
\[
\cD:(\phi,\Phi,\boldsymbol{c})\mapsto(w,W)
 =\big(\mathcal{D}_{\partial \rB}[\phi],-\mathcal{D}_{\partial\rQ}[\Phi]+\sum c_j\Xi_j\big),
 \] 
which thanks to Lemma~\ref{lem:Drep} and \eqref{eq:XiK} is known to be continuous from
$H^{1/2+\tau}_{\mathrm{odd}}(\partial \rB)\times H^{1/2+\tau}_{\mathrm{odd}}(\partial \rQ)\times \mathbb{R}^{m^{\#}}$ to
$K^{1+\tau}_{\beta_0}(\rB) \times K^{1+\tau}_{\beta_0\beta_1}(\rQ^\complement)$. 
This gives the desired representation \eqref{eq:wWsum} as a convergent series where we set
$\cL_n = \cD\,\cM_n$.
\end{proof}

\begin{remark}
\label{rem:wWLinf}
Replacing Lemma \ref{lem:Drep} by Remark \ref{rem:DLinf} in this proof, we conclude that the series expansions \eqref{eq:wWsum} for the functions $w$ and $W$ converge also uniformly, the series for $w$ in $L^\infty(\rB)$ and the series for $W$ in $L^\infty(\rQ^\complement)$.
\end{remark}

The second result gives a convergent series expansion for the whole solution  $u[\eta,\psi,\Psi]$ when written in slow variables and thus shows that it is an analytic function of $\eta$ in a neighborhood of $0$.  It is valid outside of a  neighborhood of the origin (``outer expansion'').
\begin{theorem}
\label{thm:uexpslow}
Let $\tau\in[-1/2,1/2]$ and $s\in\R$.  Let $\Omega$ be a  Lipschitz subdomain of  $\rB$ such that $0\not\in\overline\Omega$. If $\overline\Omega\subset\rB$, then $s$ can be any real number. If, however, $\partial\Omega\cap\partial\rB\ne\emptyset$, then we assume $s\le1+\tau$.
Let $\eta_\Omega>0$  be such that $\overline{\Omega}\cap \eta\overline{\rQ}=\emptyset$ for all $\eta\in(0,\eta_\Omega)$. 
Then there exist $\eta_1 \in (0,\eta_\Omega)$ and a real analytic map ${\cU}_\rS$ from $(-\eta_1,\eta_1)$ to
$\sL\big(H^{1/2+\tau}_{\mathrm{odd}}(\partial\rB)\times  H^{1/2+\tau}_{\mathrm{odd}}(\partial \rQ),\, H^s(\Omega)\big)$
such that
\begin{equation}
\label{eq:rep:1}
u[\eta,\psi,\Psi]\on{\Omega}=\cU_\rS(\eta)\binom\psi\Psi
\quad\forall (\eta,\psi,\Psi) \in (0,\eta_1)\times  H^{1/2+\tau}_{\mathrm{odd}}(\partial \rB)\times  H^{1/2+\tau}_{\mathrm{odd}}(\partial \rQ)\,.
\end{equation}
Moreover,
\begin{equation}\label{eq:rep:2}
\cU_\rS(0)\binom\psi\Psi=w_{0,\psi}\on{|\Omega}\qquad \forall(\psi,\Psi)\in  H^{1/2+\tau}_{\mathrm{odd}}(\partial \rB)\times  H^{1/2+\tau}_{\mathrm{odd}}(\partial \rQ)\, ,
\end{equation}
where $w_{0,\psi}$ is the unique solution in $H^{1+\tau}(\rB)$ of the Dirichlet problem
\begin{equation}\label{eq:rep:3}
\left\{
\begin{array}{ll}
\Delta w_{0,\psi}=0&\text{ in }\rB\,,\\
\gamma_{0}w_{0,\psi}=\psi&\text{ on }\partial\rB\,.
\end{array}
\right.
\end{equation}
\end{theorem}

\begin{proof}
We write $u[\eta,\psi,\Psi](x) = w(x) + W(x/\eta)$ as in Theorem~\ref{thm:u=w+W} and use the analytic dependency on $\eta$ of 
$\cM(\eta)^{-1}: (\psi,\Psi)\mapsto (\phi,\Phi,\boldsymbol{c})$ as in the proof of that theorem. The map
$(\phi,\Phi,\boldsymbol{c})\mapsto w\on{\Omega}$ being independent of $\eta$, only its range is of interest. This is contained in $H^{1+\tau}(\Omega)$ for any $\Omega\subset\rB$, and since $w=\cD_{\partial\rB}[\phi]$ is harmonic in $\rB$, it is contained in $C^\infty(\overline\Omega)\subset H^s(\Omega)$ for any $s$ if $\overline\Omega\subset\rB$.\\
For the map $(\Phi,\boldsymbol{c})\mapsto 
W(\cdot/\eta)\on{\Omega} = -\cD^\complement_{\partial\rQ,\Omega}(\eta)[\Phi]
 + \sum_{j=1}^{m^{\#}}c_j\Xi_j (\cdot/\eta)$
we invoke Lemma~\ref{lem:Danal} (ii) to get the desired analyticity (see also the argument for the analyticity of $\cM_{13}$ in the proof of Theorem \ref{thm:Meta=0} ).
\end{proof}

The third result shows the analytic dependence on $\eta$ for the solution $u[\eta,\psi,\Psi]$ when written in fast variables. It concerns the solution on a subdomain of size $\eta$ (``inner expansion''). The proof is similar to the proof of Theorem~\ref{thm:uexpslow}, but simpler, because it is based on the formula
\[
 u[\eta,\psi,\Psi](\eta X)=w(\eta X)+W(X)
\]
and it therefore simply invokes the
harmonicity, hence analyticity of $w=\cD_{\partial\rB}[\phi]$ near the origin, see Lemma~\ref{lem:Danal}(i).

\begin{theorem}
\label{thm:uexpfast}
Let $\tau\in[-1/2,1/2]$ and $s\in\R$.  Let $\Omega\subset\rQ^{\complement}=\R^2\setminus\overline\rQ$ be a bounded Lipschitz domain. If $\overline\Omega\subset\rQ^{\complement}$, then $s$ can be any real number. If $\partial\Omega\cap\partial\rQ\ne\emptyset$, then we assume $s\le1+\tau$.
Let $\tilde\eta_\Omega>0$  be such that $\eta\Omega\subset\rB$ for all $\eta\in(0,\tilde\eta_\Omega)$. 
Then there exist $\eta_1 \in (0,\tilde\eta_\Omega)$ and a real analytic map ${\cU}_\rF$ from $(-\eta_1,\eta_1)$ to
$\sL\big(H^{1/2+\tau}_{\mathrm{odd}}(\partial\rB)\times  H^{1/2+\tau}_{\mathrm{odd}}(\partial \rQ),\, H^s(\Omega)\big)$ such that
\begin{multline}
\label{eq:repF:1}
u[\eta,\psi,\Psi](\eta\,\cdot\,)\on{\Omega}=\cU_\rF(\eta)\binom\psi\Psi
\\
\quad\forall (\eta,\psi,\Psi) \in (0,\eta_1)\times  H^{1/2+\tau}_{\mathrm{odd}}(\partial \rB)\times  H^{1/2+\tau}_{\mathrm{odd}}(\partial \rQ)\,.
\end{multline}
Moreover,
\begin{equation}\label{eq:repF:2}
\cU_\rF(0)\binom\psi\Psi=W_{0,\Psi}\on{\Omega}\qquad \forall(\psi,\Psi)\in  H^{1/2+\tau}_{\mathrm{odd}}(\partial \rB)\times  H^{1/2+\tau}_{\mathrm{odd}}(\partial \rQ)\, ,
\end{equation}
where $W_{0,\Psi}$ is the unique solution in $K^{1+\tau}_{\beta_0\beta_1}(\rQ^\complement)$ for $\beta_0\in(-2,0)$, $\beta_1\in(-1,0)$ of the exterior Dirichlet problem
\begin{equation}\label{eq:repF:3}
\left\{
\begin{array}{ll}
\Delta W_{0,\Psi}=0&\text{ in }\rQ^\complement\,,\\
\gamma_{0}W_{0,\Psi}=\Psi&\text{ on }\partial\rQ\,.
\end{array}
\right.
\end{equation}
\end{theorem}

\subsection{Convergent expansion of the solution of the original problem}
\label{ss:4.2}
In this subsection, we first insert into the expansion of the solution $u[\eta,\psi,\Psi]$ of the Dirichlet problem on the perforated domain $\rB_\eta$ obtained in the preceding subsection the knowledge about the Dirichlet data from Theorem~\ref{th:Tu0}, namely $\psi=0$ and 
$\Psi=-\cT^*[u_0]$. The latter is given by a convergent series in \eqref{eq:Tu0}.
We then have to write the resulting double series as a series in $\varepsilon$ by using $\eta=\varepsilon^{\pi/\omega}$ and we have to interpret the result as a series that converges in function spaces defined on the original domain $\rA_\varepsilon$, by 
undoing the conformal map $\cG^*_{\pi/\omega}$. This will give a convergent expansion for the solution $\tilde u_\varepsilon$ of the residual problem \eqref{eq:pepsgen}.
The final step is to add the function $u_0$ as described in Theorem~\ref{th:tu}, in order to find a convergent expansion for the solution $u_\varepsilon$ of the original problem \eqref{eq:poisson}. 

Corresponding to the three results about the Dirichlet problem in the perforated domain $\rB_\eta$, Theorems~\ref{thm:u=w+W}, \ref{thm:uexpslow} and \ref{thm:uexpfast}, we prove three different results about the solution of the original problem \eqref{eq:poisson}. 
For the notation describing the convergent series in powers of $\varepsilon$, we refer to Sections~\ref{s:2} and \ref{s:3}, in particular to Notation~\ref{not:1} for the definition of the index set $\gA$ and to Notation~\ref{not:2} for the powers and divided differences of powers of $\varepsilon$ abbreviated by the symbol $\sE_\gamma(\varepsilon)$. 

In view of Theorem \ref{th:Tu0} and Remark \ref{rem:tau}, we introduce a maximal regularity index
\begin{equation}
\label{eq:tau0}
 \tau_{0} = \tfrac12 \;\mbox{ if } 0\not\in\partial\rP\,,\quad
 \tau_{0} = \min\{\tfrac12,\tfrac{2\omega}\pi\} \;\mbox{ if } 0\in\partial\rP\,.
\end{equation}

The first result is a globally valid two-scale splitting of the solution $u_{\varepsilon}$, where the slow-variable part and the fast-variable part have separate convergent expansions, when written in their respective variables.

\begin{theorem}
\label{thm:ueps}
There exist $\varepsilon_1>0$ such that the solution $u_\varepsilon$ of Problem \eqref{eq:poisson} has the following structure.
\begin{equation}
\label{eq:u=u+U}
 u_\varepsilon(t) = u_0(t) + u(\varepsilon)(t) + U(\varepsilon)(\tfrac t\varepsilon)
 \qquad \forall\, t\in \rA_\varepsilon, \; \varepsilon\in(0,\varepsilon_1)\,.
\end{equation}
Here $u_0$ is the solution of the limit problem \eqref{eq:u0} on the unperforated corner domain $\rA$. Its singular behavior near the corner is described by the convergent series \eqref{eq:tu} in Theorem~\ref{th:tu}.\\
The functions $u(\varepsilon)(t)$ and $U(\varepsilon)(T)$ are defined for $t\in\rA$ and $T\in\rP^\complement$, respectively, and have a convergent series expansion of the following form. 
\begin{equation}
\label{eq:U=sumV}
 \binom uU 
  = \sum_{(n,\gamma)\in\N\times\gA} \varepsilon^{n\pi/\omega}\sE_\gamma(\varepsilon)
   \binom{v_{n\gamma}}{V_{n\gamma}}\,.
\end{equation}
Let $\tau\in(0,\tau_{0})$ and $\beta_0>-1-\pi/\omega$ with $\beta_0>-1-\tau\pi/\omega$ if\/ $0\in\partial\rP$, and let $\beta_1<-1+\pi/\omega$. The series converges in the weighted Sobolev spaces
$
 K^{1+\tau}_{\beta_0}(\rA)\times K^{1+\tau}_{\beta_0\beta_1}(\rP^\complement)\,,
$
and there exist constants $C$ and $M$ such that
\[
   \DNorm{v_{n\gamma}}{K^{1+\tau}_{\beta_0}(\rA)}+ 
   \DNorm{V_{n\gamma}}{K^{1+\tau}_{\beta_0\beta_1}(\rP^\complement)} 
   \le C M^{n+|\gamma|},\quad (n,\gamma)\in\N\times\gA\,.
\]
The series converge also uniformly, in
$L^\infty(\rA)\times L^\infty(\rP^\complement)$.
\end{theorem}
\begin{proof}
We have $u_\varepsilon=u_0+\tilde u_\varepsilon$, where $\tilde u_\varepsilon$ solves the residual problem \eqref{eq:pepsgen}. After applying the conformal mapping $\cG_{\pi/\omega}$ and the odd reflection, this was rewritten in Theorem~\ref{th:Tu0} as the problem \eqref{eq:petagen}, a special case of the boundary value problem \eqref{eq:symdir}. Thus we have the identification
$\tilde u_\varepsilon = u[\eta,\psi,\Psi]\circ \cG_{\pi/\omega}$, where $\psi=0$ and $\Psi=-\cT^*[u_0](\eta \cdot)$.
Combining the convergent expansion \eqref{eq:Tu0} for $\cT^*[u_0]$ with  the power series \eqref{eq:wWsum} for the solution operator of problem \eqref{eq:symdir}, we thus find a convergent expansion that has the form \eqref{eq:U=sumV}
\[
 \binom uU 
 =  \sum_{n=0}^\infty \sum_{\gamma\in\gA} \varepsilon^{n\pi/\omega}\sE_\gamma(\varepsilon)
  \cL_n \binom{0}{\Psi_\gamma} \circ \cG_{\pi/\omega}\,.
\]
The right choice of weighted Sobolev spaces for the convergence follows from Theorem~\ref{thm:u=w+W} with the transformation rule of Lemma~\ref{lem:Km}. Note that this transformation rule motivates the use of weighted Sobolev spaces instead of non-weighted spaces. For the uniform convergence finally, we notice that $L^\infty$ remains invariant under the conformal mappings $\cG^*_\kappa$.
\end{proof}

The second result is a convergent expansion of the whole solution $u_{\varepsilon}$ written in slow (macroscopic) variables. It is valid in any fixed subdomain of $\rA_{\varepsilon}$ that has a positive distance to the corner and thus is free of holes for sufficiently small $\varepsilon$. This corresponds to the outer expansion in the method of matched asymptotic expansions, compare \cite[Section 5]{DaToVi10}.  

\begin{theorem}
\label{thm:outer}
Let $\Omega$ be a  Lipschitz subdomain of  $\rA$ such that $0\not\in\overline\Omega$.
Let $\varepsilon_\Omega>0$  be such that $\overline{\Omega}\cap \varepsilon\overline{\rP}=\emptyset$ for all $\varepsilon\in(0,\varepsilon_\Omega)$.
Then there exists $\varepsilon_1 \in (0,\varepsilon_\Omega)$ such that for $\varepsilon\in(0,\varepsilon_{1})$ the solution $u_{\varepsilon}$ of Problem \eqref{eq:poisson} has the following expansion in $\Omega$:
\begin{equation}
\label{eq:outer}
 u_{\varepsilon}(t) = u_{0}(t) + 
    \sum_{(n,\gamma)\in\N_{*}\times\gA} \varepsilon^{n\pi/\omega}\sE_\gamma(\varepsilon) u^{\rS}_{n\gamma}(t)\,,
    \qquad t\in\Omega\,.
\end{equation}
Let $\tau<\tau_0$. Then the series converges for $|\varepsilon|<\varepsilon_{1}$ in $H^{1+\tau}(\Omega)$, and there exist constants $C$ and $M$ such that
\[
   \DNorm{u^{\rS}_{n\gamma}}{H^{1+\tau}(\Omega)}
   \le C M^{n+|\gamma|},\quad (n,\gamma)\in\N_{*}\times\gA\,.
\]
The series converges also uniformly in $\Omega$.
\end{theorem}
\begin{proof}
In the multiscale decomposition \eqref{eq:u=u+U} 
$u_\varepsilon=u_0+u(\varepsilon)+U(\varepsilon)\big(\tfrac\cdot\varepsilon\big)$, the term $u(\varepsilon)$ has the required expansion according to Theorem~\ref{thm:ueps}. For $U(\varepsilon)$ we write it as 
\[
 U(\varepsilon)=W\circ \cG_{\pi/\omega}\,,
\]
where $W$ is the function defined in Theorem~\ref{thm:u=w+W} in the special case where $\psi=0$ and $\Psi=-\cT^*[u_0](\eta \cdot)$. The analyticity of $W(\cdot/\eta)$ with respect to $\eta$ at $\eta=0$ in the case of $\eta$-independent $\Psi$ has been deduced in the proof of Theorem~\ref{thm:uexpslow} from Lemma~\ref{lem:Danal}. We have to combine this, as in the proof of Theorem~\ref{thm:ueps}, with the expansion 
\eqref{eq:Tu0} for $\cT^*[u_0]$ and set $\eta=\varepsilon^{\pi/\omega}$, ending up with the expansion required for \eqref{eq:outer}. The coefficient functions $u^{\rS}_{n\gamma}$ are the sum of the corresponding terms of the expansion of $u(\varepsilon)$ and of $U(\varepsilon)(\cdot/\varepsilon)$. For $n=0$ both of these terms vanish, because they correspond to  $u[\eta,\psi,\Psi]$ in \eqref{eq:rep:1} at $\eta=0$ and $\psi=0$, and according to \eqref{eq:rep:2}--\eqref{eq:rep:3}, this is zero. Therefore the sum over $n$ in \eqref{eq:outer} starts with $n\ge1$.
\end{proof}

The third result is a convergent expansion of the whole solution $u_{\varepsilon}$ written in fast (microscopic) variables. It is valid outside of the holes in a scaled family $\varepsilon\Omega$ of subdomains of $\rA_{\varepsilon}$. This corresponds to the inner expansion in the method of matched asymptotic expansions, compare \cite[Section 5]{DaToVi10}.

\begin{theorem}
\label{thm:inner}
Let $\Omega\subset\rS_{\omega}\setminus\overline\rP$ be a bounded Lipschitz domain. 
Let $\tilde\varepsilon_\Omega>0$  be such that $\varepsilon\Omega\subset\rA$ for all $\varepsilon\in(0,\tilde\varepsilon_\Omega)$. 
Then there exists $\varepsilon_1 \in (0,\tilde\varepsilon_\Omega)$ such that for $\varepsilon\in(0,\tilde\varepsilon_{1})$ the solution $u_{\varepsilon}$ of Problem \eqref{eq:poisson} has the following expansion in $\varepsilon\Omega$:
\begin{equation}
\label{eq:inner}
 u_{\varepsilon}(\varepsilon T) =  
    \sum_{(n,\gamma)\in\N\times\gA} \varepsilon^{n\pi/\omega}\sE_\gamma(\varepsilon) U^{\rF}_{n\gamma}(T)\,,
    \qquad T\in\Omega\,.
\end{equation}
Let $\tau<\tau_0$.
Then the series converges for $|\varepsilon|<\tilde\varepsilon_{1}$ in $H^{1+\tau}(\Omega)$, and there exist constants $C$ and $M$ such that
\[
   \DNorm{U^{\rF}_{n\gamma}}{H^{1+\tau}(\Omega)}
   \le C M^{n+|\gamma|},\quad (n,\gamma)\in\N_{*}\times\gA\,.
\]
The series converges also uniformly in $\Omega$.

\end{theorem}
\begin{proof}
As in the proof of Theorem~\ref{thm:ueps} we use the identity 
$u_\varepsilon=u_0+\tilde u_\varepsilon\equiv u_0+u[\eta,\psi,\Psi]\circ \cG_{\pi/\omega}$, where $\psi=0$ and $\Psi=-\cT^*[u_0](\eta \cdot)$. Together with Theorem~\ref{thm:uexpfast} for $u[\eta,\psi,\Psi](\eta\,\cdot\,)$, this gives the desired form \eqref{eq:inner} of the expansion for the second term $\tilde u_\varepsilon(\varepsilon\cdot)$. Here, as in the outer expansion \eqref{eq:outer}, the sum over $n$ lacks the term $n=0$.
It remains to analyze the first term $u_0(\varepsilon\cdot)$. 
Here we need the asymptotic behavior (expansion into corner singular functions) of $u_0$ that was described in \eqref{eq:tup} and used for expanding $u_0(\varepsilon T)$ into a convergent series in \eqref{eq:tupe}. With the notation introduced in \eqref{eq:Phi} in the proof of Theorem~\ref{th:Tu0}, this series can be written as
\[
 u_0(\varepsilon T) = \sum_{\gamma\in\gA} \sE_\gamma(\varepsilon)\Phi_\gamma(T)\,.
\]
This is a series of the form \eqref{eq:inner} with $n=0$. Explicitly estimating norms of the functions $\Phi_\gamma$ or relying on the estimate \eqref{eq:PhiNorm}, we see that the series converges in $H^{1+\tau}(\Omega)$.
\end{proof}

The fact that the series expansions in the last three theorems 
are only stepwise convergent, that is convergent when pairs of powers of $\varepsilon$ are grouped together into the terms $\sE_\gamma(\varepsilon)$ from Notation~\ref{not:2}(3), is caused entirely by the corresponding fact for the expansion of $u_0$ studied in Section~\ref{s:2}, see in particular Remarks~\ref{rem:pi}--\ref{rem:other}. Thus if we assume that one of the conditions mentioned in these Remarks is satisfied, we find convergent power series, and it is then possible to reformulate the statements of Theorems~\ref{thm:ueps}--\ref{thm:inner} in terms of analytic functions of $\varepsilon$ and $\varepsilon^{\pi/\omega}$.

\begin{corollary}
\label{cor:f=0}
Suppose that the right hand side $f$ vanishes in a neighborhood of the corner $0$. Denote by $u_\varepsilon$ the solution of Problem \eqref{eq:poisson}.\\  
\emph{(i)}
Let the parameters $\tau$, $\beta_0$ and $\beta_1$ be chosen as in Theorem~\ref{thm:ueps}. Then there exists $\eta_1>0$ and a real analytic function
\[
 (-\eta_1,\eta_1) \ni \eta\mapsto 
  \mathcal{V}[\eta]=\binom{v[\eta]}{V[\eta]} \in 
  K^{1+\tau}_{\beta_0}(\rA)\times K^{1+\tau}_{\beta_0\beta_1}(\rP^\complement)
\]
such that in the two-scale decomposition \eqref{eq:u=u+U}
$u_\varepsilon=u_0+u(\varepsilon)+U(\varepsilon)(\tfrac\cdot\varepsilon)$ we have
\[
  u(\varepsilon) = v[\varepsilon^{\pi/\omega}]\,,\qquad
  U(\varepsilon) = V[\varepsilon^{\pi/\omega}]
    \qquad \forall\, \varepsilon\in(0,\eta_1^{\omega/\pi})\,.
\]
\emph{(ii)}
Let $\Omega$ be a  Lipschitz subdomain of  $\rA$ such that $0\not\in\overline\Omega$ and let $\tau$ be chosen as in Theorem~\ref{thm:outer}. Then there exists $\eta_1>0$ and a real analytic function
\[
 (-\eta_1,\eta_1) \ni \eta\mapsto 
  u_\rS[\eta] \in H^{1+\tau}(\Omega)
\]
such that we have
$u_{\rS}[0] = u_0$ and
\[
 u_\varepsilon = u_\rS[\varepsilon^{\pi/\omega}] 
 \qquad \mbox{ in }\Omega\,,\quad\forall\, \varepsilon\in(0,\eta_1^{\omega/\pi})\,.
\]
\emph{(iii)}
Let $\Omega\subset\rS_{\omega}\setminus\overline\rP$ be a bounded Lipschitz domain and let $\tau$ be chosen as in Theorem~\ref{thm:inner}. Then there exists $\eta_1>0$ and a real analytic function
\[
 (-\eta_1,\eta_1) \ni \eta\mapsto 
  U_\rF[\eta] \in H^{1+\tau}(\Omega)
\]
such that we have
\[
 u_\varepsilon(\varepsilon T) = U_\rF[\varepsilon^{\pi/\omega}](T) 
 \qquad \forall \,T\in\Omega\,,\quad \varepsilon\in(0,\eta_1^{\omega/\pi})\,.
\]
\end{corollary}
\begin{proof}
As we have seen in Remark~\ref{rem:f=0}, if $f$ vanishes in a neighborhood of the corner, then in the series expansion of $u_0$ there appear only exponents that are of the form $k\pi/\omega$ with integer $k$, and the series is unconditionally convergent. In the resolution of the residual problem in Section~\ref{ss:Dir},  integer powers of $\eta=\varepsilon^{\pi/\omega}$ were incorporated, so that the final convergent series expansions \eqref{eq:U=sumV}, \eqref{eq:outer} and \eqref{eq:inner} also contain only exponents that are integer multiples of $\pi/\omega$. It follows that these series are convergent power series, hence analytic functions, in the variable $\eta=\varepsilon^{\pi/\omega}$. 
\end{proof}

Let now $\omega$ be  a rational multiple of $\pi$, i.e. $\pi/\omega=p/q$, where $p$ and $q$ are relatively prime positive integers. In this case, all the exponents of $\varepsilon$ appearing in the 
convergent series expansions \eqref{eq:U=sumV}, \eqref{eq:outer} and \eqref{eq:inner} can be seen to be integer multiples of $1/q$. The expressions $\sE_\gamma(\varepsilon)$ as defined in Notation~\ref{not:2} are now either integer powers of $\delta=\varepsilon^{1/q}$ or of the form $\varepsilon^{\ell}\log \varepsilon$ with integer $\ell$. They can therefore be expressed via two real analytic functions of one variable. We formulate this observation for the two-scale  decomposition \eqref{eq:u=u+U} of Theorem~\ref{thm:ueps} and its convergent series expansion \eqref{eq:U=sumV} and leave the corresponding reformulations of Theorems~\ref{thm:outer} and \ref{thm:inner} to the reader.

\begin{corollary}
\label{cor:omegaQ}
Let $\pi/\omega=p/q$. With the notations of Theorem~\ref{thm:ueps}, there exist $\delta_1>0$ and two real analytic functions (we set $\varepsilon_1\equiv(\delta_1)^q$)
\[
\begin{aligned}
 (-\delta_1,\delta_1) \ni \delta &\mapsto 
  \mathcal{V}_0[\delta] \in 
  K^{1+\tau}_{\beta_0}(\rA)\times K^{1+\tau}_{\beta_0\beta_1}(\rP^\complement)\\
(-\varepsilon_1,\varepsilon_1) \ni \varepsilon &\mapsto 
  \mathcal{V}_1[\varepsilon] \in 
  K^{1+\tau}_{\beta_0}(\rA)\times K^{1+\tau}_{\beta_0\beta_1}(\rP^\complement)\\
\end{aligned}
\]
such that 
\[
  \binom{u(\varepsilon)}{U(\varepsilon)} =
   \mathcal{V}_0[\varepsilon^{1/q}] + \mathcal{V}_1[\varepsilon^{p}]\log\varepsilon
   \qquad \forall\, \varepsilon\in(0,\varepsilon_1)\,.
\]
\end{corollary}

The third case where we find absolutely convergent expansions in powers of $\varepsilon$ is when $\omega$ is not a rational multiple of $\pi$ but is such that we can choose $\gA_0=\emptyset$. According to the discussion in Section~\ref{ss:1.2} and in Appenix~\ref{app:Liouville}, this is the case if and only if 
$\frac\pi\omega$ is not a super-exponential Liouville number. In this case we do not need the divided differences of Notation~\ref{not:2}(3), and the terms in the convergent expansions \eqref{eq:U=sumV}, \eqref{eq:outer} and \eqref{eq:inner} are simply monomials in the two variables $\varepsilon$ and $\varepsilon^{\pi/\omega}$, and the series therefore define real analytic functions of two variables. We formulate again the corresponding result for the two-scale expansion of Theorem~\ref{thm:ueps} and leave the reformulations of Theorems~\ref{thm:outer} and \ref{thm:inner} to the reader. 
\begin{corollary}
\label{cor:irrat}
Suppose that $\pi/\omega$ is irrational and not a super-exponential Liouville number in the sense of Definition~\ref{def:Liouville}. Then there exist $\varepsilon_1>0$ and a real analytic function of two variables (we set $\eta_1=\varepsilon_1^{\pi/\omega}$)
\[
 (-\varepsilon_1,\varepsilon_1)\times(-\eta_1,\eta_1)
  \ni (\varepsilon,\eta) \mapsto 
  \mathcal{V}[\varepsilon,\eta] \in 
  K^{1+\tau}_{\beta_0}(\rA)\times K^{1+\tau}_{\beta_0\beta_1}(\rP^\complement)
\]
such that 
\[
  \binom{u(\varepsilon)}{U(\varepsilon)} =
   \mathcal{V}[\varepsilon, \varepsilon^{\pi/\omega}]
   \qquad \forall\, \varepsilon\in(0,\varepsilon_1)\,.
\]
\end{corollary}

\appendix
\section{Symmetric extension of Lipschitz domains}\label{app:symextlip}
In this section we use the objects defined in Section~\ref{Ss:Reflection}, in particular the upper half-plane $\rS_\pi$ and the operation $\cE$ of symmetric extension of a subset of $\rS_\pi$ by reflection at the horizontal axis.
In general, the symmetric extension of a Lipschitz domain is not Lipschitz, and therefore the following result is not entirely obvious and merits a complete proof.   
\begin{lemma}
\label{lem:LipE}
Assume that $\Omega$ is a bounded subdomain of\/ $\rS_\pi$ and that $\Omega$ and $\rS_\pi\setminus\overline\Omega$ have Lipschitz boundaries. Then 
$\cE(\Omega)$ has a Lipschitz boundary.
\end{lemma}
\begin{proof}
As a characterization of a Lipschitz boundary we use the property that it is locally congruent to the graph of a Lipschitz continuous function.
A simple consequence of this property is that in 2 dimensions, each point of the boundary has a 2-dimensional neighborhood in which the boundary is a simple curve, in particular it is homeomorphic to an interval. 

Let us now first show that $\partial\Omega\cap\partial\rS_\pi$ has no isolated points. 
Suppose there were such a point $\boldsymbol{x}_0=(x_0,0)$. We show that then $\rS_\pi\setminus\overline\Omega$ cannot be a Lipschitz domain, contrary to the hypothesis. Since $\Omega$ is Lipschitz, there is a neighborhood $\cU$ of $\boldsymbol{x}_0$ in which $\partial\Omega$ coincides with a simple curve $\Gamma_0$  and such that $\cU\cap\partial\Omega\cap\partial\rS_\pi=\{\boldsymbol{x}_0\}$. This neighborhood can be chosen such that $\cU\cap \partial\rS_\pi$ is an interval $\Gamma_1$. Since
$(\partial\Omega\cup \partial\rS_\pi)\setminus(\partial\Omega\cap \partial\rS_\pi)$
is contained in the boundary of $\rS_\pi\setminus\overline\Omega$, the latter coincides in $\cU$ with the union of the two curves $\Gamma_0$ and 
$\Gamma_1$ that intersect in the interior point $\boldsymbol{x}_0$. Such a union is clearly not homeomorphic to an interval.

We will now use the following equivalent reformulation of the above definition of a Lipschitz boundary $\partial\Omega$ in two dimensions: To each of its points there is a neighborhood $\cU$ and a convex cone $\cC_{\alpha\beta}$ with the following property: 
If the curve $\Gamma_0=\partial\Omega\cap\cU$ is parametrized by an interval,
\[
  \gamma: (t_0,t_1)\to \Gamma_0\subset \cU \,,
\]
then for $\boldsymbol{x}=\gamma(s)$, $\boldsymbol{y}=\gamma(t)$ with 
$s<t$ (we say ``$\boldsymbol{x}$ precedes $\boldsymbol{y}$'' or $\boldsymbol{x}\prec\boldsymbol{y}$) we have 
$\boldsymbol{y}\in \boldsymbol{x} + \cC_{\alpha\beta}$.
Here the cone $\cC_{\alpha\beta}$ is defined by two angles $\alpha$, $\beta$ with $\alpha<\beta<\alpha+\pi$,
\[
 \cC_{\alpha\beta} = 
  \{(\rho\cos\theta, \rho\sin\theta) \colon 0<\rho<\infty, \alpha<\theta<\beta \}\,.
\]
One can observe that the rotation angles $\omega$ (modulo $2\pi$) of coordinate axes that allow the representation of $\Gamma_0$ as a graph are given by the complement of $\cC_{\alpha\beta}$, the condition being 
\[
\omega-\tfrac\pi2\in(\beta-\pi,\alpha)\cup(\beta,\alpha+\pi)\,.
\]

Let now $\cU$ be such a neighborhood of a point of $\partial\Omega$. If $\partial\Omega\cap\cU$ is entirely contained either in the upper half-plane $\rS_\pi$ or in the axis of symmetry $\partial\rS_\pi$, then there is nothing to prove, because in this case (after possibly choosing a smaller neighborhood), the set 
$\cU\cup\cR(\cU)$ will be a suitable neighborhood for the boundary of $\cE(\Omega)$. 

The nontrivial case is when $\cU$ is a neighborhood of a point 
$\boldsymbol{x}_0\in\partial\Omega\cap\partial\rS_\pi$ and both 
$\cU\cap\partial\Omega\cap\partial\rS_\pi$ and 
$\cU\cap\partial\Omega\cap\rS_\pi$ are non-empty.
Since, as we have seen, $\boldsymbol{x}_0$ is not an isolated point of $\partial\Omega\cap\partial\rS_\pi$, the structure of $\partial\Omega\cap\cU$ is (after possibly choosing a smaller neighborhood) the following:
\[
 \partial\Omega\cap\cU = \Gamma_1\cup\Gamma_0\,,
\]
where $\Gamma_1$ is an interval $I_1\times\{0\}\subset\partial\rS_\pi$, and $\Gamma_0$ is a Lipschitz curve contained in $\rS_\pi$. Locally, the boundary of the complement has the form
\[
 \partial(\rS_\pi\setminus\overline\Omega)\cap\cU = \Gamma_1'\cup\Gamma_0\,,
\]
where $\Gamma_1'$ is another interval $I_1'\times\{0\}\subset\partial\rS_\pi$. 
The intervals have one point in common, which we can assume to be $\boldsymbol{x}_0$
\[
 \Gamma_1\cap\Gamma_1' = \{\boldsymbol{x}_0\} = \overline\Gamma_0\cap\partial\rS_\pi\,.
\]
Since now $\Omega$ and $\rS_\pi\setminus\Omega$ play symmetric roles, it is no restriction to assume that $I_1=[x_0-\delta,x_0]$ and $I_1'=[x_0,x_0+\delta]$ with some $\delta>0$. We can also assume that the two parametrizations of $\partial\Omega\cap\cU$ and of $\partial(\rS_\pi\setminus\overline\Omega)\cap\cU$ are oriented such that in both cases the segment $\Gamma_1$ or $\Gamma_1'$, respectively, precedes the curve $\Gamma_0$.

Now from our definition of a Lipschitz boundary, we get a cone $\cC_{\alpha\beta}$ that satisfies
\[
 \boldsymbol{x},\boldsymbol{y}\in\partial\Omega\cap\cU
 \:\text{ and }\: \boldsymbol{x}\prec\boldsymbol{y}
 \quad\Longrightarrow\quad
 \boldsymbol{y}-\boldsymbol{x}\in \cC_{\alpha\beta} \,.
\]
In particular, this holds for $\boldsymbol{x},\boldsymbol{y}\in \Gamma_1$, and this implies that we have $-\pi<\alpha<0$ and $0<\beta<\alpha+\pi$.

Likewise, there is a cone $\cC_{\alpha'\beta'}$ that satisfies
\[
 \boldsymbol{x},\boldsymbol{y}\in\partial(\rS_\pi\setminus\overline\Omega)\cap\cU
 \:\text{ and }\: \boldsymbol{x}\prec\boldsymbol{y}
 \quad\Longrightarrow\quad
 \boldsymbol{y}-\boldsymbol{x}\in \cC_{\alpha'\beta'} \,.
\]
Since this holds for $\boldsymbol{x},\boldsymbol{y}\in \Gamma_1'$, we must have $0<\alpha'<\pi$ and $\pi<\beta'<\alpha'+\pi$.

For the curve $\Gamma_0$ we have both conditions,
\[
 \boldsymbol{x},\boldsymbol{y}\in\Gamma_0
 \:\text{ and }\: \boldsymbol{x}\prec\boldsymbol{y}
 \quad\Longrightarrow\quad
 \boldsymbol{y}-\boldsymbol{x}\in 
   \cC_{\alpha\beta}\cap \cC_{\alpha'\beta'} = \cC_{\alpha'\beta}\,.
\]
The latter cone $\cC_{\alpha'\beta}$ is contained in the upper half-plane $\rS_\pi$, and this implies that the curve $\Gamma_0$ can be represented as a graph in a coordinate system rotated by a right angle $\omega=\pi/2$. This means that there is a Lipschitz continuous function $\phi:(0,y_0)\to \R$ such that
\[
 \Gamma_0=\partial\Omega\cap\cU\cap\rS_\pi=
 \{(x,y)\in\R^2 \colon x=\phi(y), 0<y<y_0 \} \,.
\]  
Now we can execute our symmetric extension and find that the point
$\boldsymbol{x}_0\in\partial(\cE(\Omega))$ has $\cE(\cU)$ as a neighborhood in which the boundary
\[
  \partial(\cE(\Omega))\cap\cE(\cU) = 
  \Gamma_0\cup\{\boldsymbol{x}_0\}\cup\cR(\Gamma_0)
\]
is represented as the graph $\{x=\tilde\phi(y)\}$ of a Lipschitz continuous function $\tilde\phi$, namely the even extension of $\phi$, $\tilde\phi(y)=\phi(|y|)\}$, $-y_0 < y < y_0$, completed by the choice $\phi(0)=x_0$.
\end{proof}

\section{Convergence of the corner expansion for the  Dirichlet problem and Diophantine approximation}\label{app:Liouville}
In this section, we use the notation of Section~\ref{ss:1.2}. We find conditions on the opening angle $\omega$ for the convergence of the series of particular solutions constructed according to \eqref{eq:solel}
\begin{equation}
\label{eq:sumupartial}
 u_\partial(t) = \sum_{\ell\in\N_*} w_\ell(t)
  = \sum_{\ell\in\N_*}
   \big(\frac{g^\omega_\ell - g^0_\ell\,\cos\ell\omega}{\sin \ell\omega}
     \,\Im\zeta^\ell
   + g^0_\ell \, \Re\zeta^\ell
   \big)\,,
\end{equation}
provided the two power series with coefficients $g^0_\ell$ and $g^\omega_\ell$ have a nonzero convergence radius as in \eqref{eq:estG}. 
We will assume here that the number
$\kappa=\tfrac\pi\omega$ is irrational, so that the coefficients in the sum \eqref{eq:sumupartial} are well defined. 
As was observed already in \cite{BraDau82,Dauge84}, for certain angles $\omega$ for which $\kappa$ is irrational the small denominators $\sin\ell\omega$ pose a problem for the convergence of the series \eqref{eq:sumupartial}, and a procedure for reestablishing the convergence was found. The convergence of the sum depends on the rate of approximability of $\kappa$ by rational numbers, a question that has been a classical subject of number theory for a long time, see for example \cite[Chapter XI]{HardyWright2008}. A classical theorem by Liouville states that irrationals that can be fast approximated by rationals in a certain way are transcendental, and it was shown by Greenfield and Wallach in 1972 \cite{GreenfieldWallach72} that these Liouville numbers play a role in the study of global hypoellipticity of differential operators on manifolds. 
More recently, Himonas \cite{Himonas2001} and Bergamasco \cite{Bergamasco1999} introduced a subset of Liouville numbers, the exponential Liouville numbers, in the context of questions of global analytic hypoellipticity. 
For the situation in our present paper, it turns out that we need to consider an even smaller subset of irrationals that have a fast approximation by rationals. We call them super-exponential Liouville numbers.
\begin{definition}
\label{def:Liouville} Let $a\in\R\setminus\Q$. Then $a$ is said to be\\
(i) a \emph{Liouville number} if for every $n\in\N_*$, there exist $p\in\Z$ and $q\in\N_*$ such that
\[
    0<\left|a- \tfrac {p}{q} \right| < \tfrac {1}{q^{n}} \,,    
\]
(ii) an \emph{exponential Liouville number} if  there exists $c\in\R$, $c>0$, and infinitely many $p\in\Z$ and $q\in\N_*$ such that
\[
    0<\left|a- \tfrac {p}{q} \right| < e^{-cq} \,,    
\]
(iii) a \emph{super-exponential Liouville number} if for any $c\in\R$, $c>0$, there exist $p\in\Z$ and $q\in\N_*$ such that
\[
    0<\left|a- \tfrac {p}{q} \right| < e^{-cq} \,.    
\]
We denote the sets of Liouville, exponential Liouville and super-exponential 
Liouville numbers by $\Lambda$, $\Lambda_\re$ and $\Lambda_\rs$, respectively.
\end{definition}  
It is clear that $\Lambda_\rs\subset\Lambda_\re\subset\Lambda$. It is known that $\Lambda$ is dense in $\R$, uncountable and of measure zero \cite[Theorem 198]{HardyWright2008}. Using the same arguments, one can see that these properties are valid for $\Lambda_\re$ and $\Lambda_\rs$, too. Finally, it is worth noting that each of these sets is invariant with respect to taking inverses, addition of rational numbers and multiplication by nonzero rational numbers.

\begin{proposition}
\label{pro:Liouville}
Let $\kappa=\pi/\omega$ be irrational. 
Let the lateral boundary data $g^0$ and $g^\omega$ be given by series
\[
 g^0(\rho) = \sum_{\ell\in\N_*} g^0_\ell \rho^\ell,\quad
   g^\omega(\rho) = \sum_{\ell\in\N_*} g^\omega_\ell \rho^\ell
\]
that converge for $|\rho|<\rho_0$. 
Then the following two statements are equivalent:\\
\emph{(i)} For any such $g^0$ and $g^\omega$, the series \eqref{eq:sumupartial} for the particular solution $u_\partial$ of the Dirichlet problem in the sector converges for $|\zeta|<\rho_0$.\\
\emph{(ii)} $\kappa$ is not an exponential Liouville number.\\
Likewise, the following two statements are equivalent:\\
\emph{(iii)} There exists $\rho_1>0$ such that for any such $g^0$ and $g^\omega$, the series \eqref{eq:sumupartial} for the particular solution $u_\partial$ of the Dirichlet problem in the sector converges for $|\zeta|<\rho_1$.\\
\emph{(iv)} $\kappa$ is not a super-exponential Liouville number.
\end{proposition}
For the proof, we use the following elementary observation about power series: Let the series 
$\sum_{\ell\ge1} a_\ell\,x^\ell$ and $\sum_{\ell\ge1} b_\ell\,x^\ell$ have convergence radii $\rho_a$ and $\rho_b$, respectively. Then the series
$\sum_{\ell\ge1} a_\ell\,b_\ell\,x^\ell$ has convergence radius $\rho_a\rho_b$ or greater, with equality if, for example, $b_\ell=\rho_b^{-\ell}$ for all $\ell$. Applying this with $a_\ell=1/\sin\ell\omega$, we see that the proof of the proposition is achieved if we prove the following lemma.
\begin{lemma}
\label{lem:Liouville}
Let $\pi/\omega$ be irrational and let $\rho_s$ be the convergence radius of the power series
$$
 \sum_{\ell\in\N_*} \frac{x^\ell}{\sin\ell\omega}\,.
$$ 
Then $\rho_s=1$ if and only if $\pi/\omega$ is not an exponential Liouville number, and $\rho_s>0$ if and only if $\pi/\omega$ is not a super-exponential Liouville number.
\end{lemma}
\begin{proof}
We use Hadamard's characterization
\[
 \rho_s^{-1} = \limsup_{\ell\to\infty}|\sin\ell\omega|^{-1/\ell},
\]
and we freely use that  
\[
 \limsup_{\ell\to\infty}(c\,\ell^d)^{1/\ell}=1 \;\mbox{ for any $c>0$, $d\in\R$. }
\]
Rational approximations of $\kappa=\pi/\omega$ appear because for all $k\in\N$: $|\sin\ell\omega|=|\sin(\ell\omega-k\pi)|$, and we can choose $k$ such that the difference is minimal:
$$
  k=k(\ell)\equiv\round{\tfrac{\ell\omega}\pi}\in (\tfrac{\ell\omega}\pi-\tfrac12,\tfrac{\ell\omega}\pi+\tfrac12]
  \quad\Longrightarrow\quad
  \ell\omega- k\pi \in [-\tfrac\pi2,\tfrac\pi2)\,.
$$
Then, using $\tfrac2\pi\le\tfrac{\sin x}x\le1$ for $|x|\le\tfrac\pi2$, we get with the $k$ chosen as above,
$$
  \tfrac2\pi|\ell\omega- k\pi| \le |\sin\ell\omega| \le |\ell\omega- k\pi|\,.
$$ 
Thus $|\sin\ell\omega|\simeq |\ell\omega- k\pi| 
 = k\omega |\tfrac\ell k - \tfrac\pi\omega| 
  \simeq k\, |\tfrac\ell k - \tfrac\pi\omega|$, implying
 \[
 \limsup_{\ell\to\infty}|\sin\ell\omega|^{-1/\ell} 
 = \limsup_{\ell\to\infty}|\tfrac\ell{k(\ell)} - \kappa|^{-1/\ell}\,.
 \]
Therefore the condition $\rho_s=1$ is equivalent to (note that $\rho_s\le1$ in any case)
$$
\begin{aligned}
 &\qquad \quad \forall M>1\; \exists \ell_M \;:\;
  \ell\ge\ell_M\Rightarrow |\tfrac\ell{k(\ell)} - \kappa|^{-1/\ell}\le M\\
 &\Longleftrightarrow\quad
  \forall M>1\; \exists \ell_M \;:\;
  \ell\ge\ell_M\Rightarrow |\tfrac\ell{k(\ell)} - \kappa|\ge M^{-\ell}\\
 \quad&\Longleftrightarrow\quad 
 \forall c>0\; \mbox{ the inequality } |\tfrac\ell{k(\ell)} - \kappa|< e^{-c\ell}\\
 &\qquad\qquad\qquad
 \mbox{ has only finitely many solutions }\ell\in\N_*\\
 \quad&\Longleftrightarrow\quad 
 \forall c>0\; \mbox{ the inequality }  |\tfrac\ell k - \tfrac\pi\omega|< e^{-ck}\\
 &\qquad\qquad\qquad
 \mbox{ has only finitely many solutions }k,\ell\in\N_* \\
\end{aligned}
$$
The last condition means, according to Definition~\ref{def:Liouville}, that $\kappa$ is not an exponential Liouville number.

Likewise, $\rho_s>0$ is equivalent to 
$$
\begin{aligned}
 \limsup_{\ell\to\infty}|\tfrac\ell{k(\ell)} - \kappa|^{-1/\ell}<\infty
 \quad&\Longleftrightarrow\quad 
   \sup_{\ell}|\tfrac\ell{k(\ell)} - \kappa|^{-1/\ell}<\infty\\
 \quad&\Longleftrightarrow\quad  
  \exists c>0\;:\; 
  \forall \ell \;:\; |\tfrac\ell{k(\ell)} - \kappa|\ge e^{-c\ell}\\
 \quad&\Longleftrightarrow\quad  
  \exists c>0\;:\; \forall k,\ell\in\N_* \;:\; |\tfrac\ell k - \kappa|\ge e^{-ck}\,.
\end{aligned}
$$
Again comparing the negation of the last condition with Definition~\ref{def:Liouville}, we see that this is equivalent to the fact that that $\kappa$ is not a super-exponential Liouville number.
\end{proof}
Let us finally note that if $\kappa$ is a super-exponential Liouville number, one can give explicit examples for the right hand side $f$ such that the series \eqref{eq:sumupartial} for $u_\partial$ {diverges for almost all} $t\ne0$. One such example is 
$f(t)=1/(\rho_0-t_1)$.


\section*{Acknowledgement}\label{app:ack}

The four authors were partially supported by `INdAM GNAMPA Project 2015 - Un approccio funzionale analitico per problemi di perturbazione singolare e di omogeneizzazione'. M. Dalla Riva  acknowledges the support of HORIZON 2020 MSC EF project FAANon (grant agreement MSCA-IF-2014-EF - 654795) at the University of Aberystwyth, UK. P.~Musolino acknowledges the support of an INdAM Fellowship. P. Musolino is a S\^er CYMRU II COFUND fellow, also supported by the `S\^er Cymru National Research Network for Low Carbon, Energy and Environment'.


\end{document}